\pdfoutput=1
\documentclass[a4paper,11pt,reqno]{amsart}

\usepackage[USenglish]{babel}
\usepackage[T1]{fontenc}
\usepackage[utf8]{inputenc}

\usepackage{amsfonts}
\usepackage{amsmath}  
\usepackage{amsthm} 
\usepackage[colorlinks]{hyperref}
\hypersetup{
  linkcolor=[RGB]{192,0,0},
  citecolor=[RGB]{0, 120, 0},
  urlcolor=[rgb]{0.6, 0.2, 0.2}
}
\usepackage[alphabetic]{amsrefs} 

\usepackage{amssymb} 

\usepackage[nameinlink]{cleveref}

\usepackage{newtxtext}

\usepackage[dotinlabels]{titletoc}

\usepackage{xcolor} 

\usepackage{fix-cm}

\usepackage{bm}

\usepackage{derivative}

\usepackage{fancyhdr}

\usepackage{array}
\usepackage{braket}
\usepackage{scalefnt}
\usepackage{lipsum}

\usepackage{mathtools}
\usepackage{thmtools}
\usepackage{graphicx}
\usepackage[normalem]{ulem}
\usepackage{tabularx}

\usepackage{paralist}
\usepackage{tensor}

\usepackage[font=small,labelfont=bf,margin=1.5cm]{caption}

\usepackage{tikz}
\usepackage{tikz-cd}
\usepackage{tikz-3dplot}
\usetikzlibrary{cd}
\usetikzlibrary{arrows}
\usetikzlibrary{matrix}
\usepackage{nicematrix} 
\usetikzlibrary{positioning}

\tikzcdset{arrow style=math font}

\usepackage[subscriptcorrection]{newtxmath}

\DeclareMathAlphabet{\mathbb}{U}{msb}{m}{n}
\DeclareMathAlphabet{\mathfrak}{U}{euf}{m}{n} 

\DeclareMathAlphabet{\mathbb}{U}{msb}{m}{n}

\usepackage{xcolor}
\definecolor{red}{rgb}{1,0,0}
\definecolor{darkred}{RGB}{192,0,0}

\newcommand{\Sum}{\displaystyle\sum}

\newcommand{\transpose}[1]{\tensor[^{\mathrm{t}}]{#1}{}}

\newcommand{\bfa}{\mathbf{a}}
\newcommand{\bfb}{\mathbf{b}}
\newcommand{\bfc}{\mathbf{c}}
\newcommand{\bfd}{\mathbf{d}}
\newcommand{\bfe}{\mathbf{e}}

\newcommand{\bfv}{\mathbf{v}}
\newcommand{\bfw}{\mathbf{w}}
\newcommand{\bfx}{\mathbf{x}}
\newcommand{\bfy}{\mathbf{y}}

\newcommand{\calA}{\mathcal{A}}
\newcommand{\calB}{\mathcal{B}}

\newcommand{\calD}{\mathcal{D}}

\newcommand{\calH}{\mathcal{H}}
\newcommand{\calI}{\mathcal{I}}

\newcommand{\calR}{\mathcal{R}}
\newcommand{\calS}{\mathcal{S}}
\newcommand{\calT}{\mathcal{T}}

\newcommand{\bbC}{\mathbb{C}}

\newcommand{\bbN}{\mathbb{N}}

\newcommand{\bbP}{\mathbb{P}}

\newcommand{\bbR}{\mathbb{R}}

\newcommand{\bbZ}{\mathbb{Z}}

\newcommand{\mfS}{\mathfrak{S}}

\newcommand{\mfa}{\mathfrak{a}}

\newcommand{\mfg}{\mathfrak{g}}
\newcommand{\mfh}{\mathfrak{h}}

\newcommand{\rmS}{\mathrm{S}}

\newcommand{\rme}{\mathrm{e}}

\newcommand{\rmi}{\mathrm{i}}

\newcommand{\bfalpha}{\boldsymbol{\alpha}}
\newcommand{\bfbeta}{\boldsymbol{\beta}}
\newcommand{\bfgamma}{\boldsymbol{\gamma}}
\newcommand{\bfdelta}{\boldsymbol{\delta}}

\renewcommand{\bar}[1]{\overline{#1}}

\newcommand{\id}{\mathrm{id}}

\newcommand{\pder}{\mathop{}\!\mathrm{\partial}}

\newcommand{\diff}{\mathop{}\!\mathrm{d}}
\newcommand{\Lap}{\mathop{}\!\Delta}

\DeclareMathOperator{\Ann}{Ann}

\DeclareMathOperator{\brk}{brk}
\DeclareMathOperator{\Cat}{Cat}

\DeclareMathOperator{\codim}{codim}

\DeclareMathOperator{\Der}{D}

\DeclareMathOperator{\End}{End}

\DeclareMathOperator{\HF}{HF}

\DeclareMathOperator{\Ker}{Ker}
\DeclareMathOperator{\im}{Im}

\DeclareMathOperator{\rk}{rk}

\DeclareMathOperator{\tr}{tr}

\DeclarePairedDelimiter{\pint}{\lfloor}{\rfloor}
\DeclarePairedDelimiter{\pa}{\langle}{\rangle}
\DeclarePairedDelimiter{\pt}{(}{)}
\DeclarePairedDelimiter{\pq}{[}{]}
\DeclarePairedDelimiter{\pg}{\{ }{ \} }
\DeclarePairedDelimiter{\abs}{\lvert}{\rvert}

\newcommand{\Mat}{\mathrm{Mat}}
\newcommand{\GL}{\mathrm{GL}}
\newcommand{\Oa}{\mathrm{O}}

\newcommand{\SL}{\mathrm{SL}}
\newcommand{\SO}{\mathrm{SO}}

\newcommand{\mfsl}{\mathfrak{sl}}
\newcommand{\mfgl}{\mathfrak{gl}}
\newcommand{\mfso}{\mathfrak{so}}

\usepackage[left=1in, right=1in, top=1in, bottom=1in, includefoot, headheight=13.6pt]{geometry}

\usepackage{enumitem}

\usepackage{adjustbox}

\usepackage{titlesec}
\titleformat{name=\section}{
 \vspace{10pt}\scshape\fontfamily{ptm}\raggedright\LARGE}{}{0em}{\hspace{-0.4pt}\LARGE \thesection.\hspace{0.5em}}[\color{black}\titlerule \vspace{5pt}]

\titleformat{name=\section,numberless}
{\vspace{10pt}\scshape\fontfamily{ptm}\raggedright\LARGE}{}{0em}
  {\hspace{-0.4pt}\LARGE}[\color{black}\titlerule \vspace{5pt}]

\allowdisplaybreaks

\numberwithin{equation}{section}
\theoremstyle{definition}
\newtheorem{defn}[equation]{Definition}
\theoremstyle{plain}
\newtheorem{teo}[defn]{Theorem}

\newtheorem{prop}[defn]{Proposition}

\newtheorem{lem}[defn]{Lemma}

\newtheorem*{teor}{Theorem}

\newtheorem{cor}[defn]{Corollary}
\theoremstyle{remark}
\newtheorem{rem}[defn]{Remark}

\allowdisplaybreaks

\counterwithin{figure}{section}

\renewcommand{\sectionmark}[1]{\markleft{\normalfont \scshape\fontfamily{ptm}\selectfont Cosimo Flavi}}\renewcommand{\subsectionmark}[1]{\markright{\normalfont \scshape\fontfamily{ptm}\selectfont Decompositions of powers of quadrics}}
\dottedcontents{section}[1.1em]{\hypersetup{linkcolor=black}}{1.1em}{5pt}

\title{\LARGE Decompositions of powers of quadrics}
\author{Cosimo Flavi}
\address{{\normalfont (Cosimo Flavi)},
\normalfont \scshape\fontfamily{ptm}\selectfont Wydział Matematyki, Informatyki i Mechaniki, Uniwersytet Warszawski,
 \normalfont{ul.~Stefana Banacha
2, 02-097 Warsaw, Poland.}}
\email{c.flavi@uw.edu.pl}



\titleformat{name=\subsection}{\vspace{7pt}}{\bfseries\Large\thetitle.}{0.5em}{\Large\bfseries}

\titleformat{name=\subsection*,numberless}{\vspace{7pt}}{\bfseries\Large\thetitle.}{0.5em}{\Large\bfseries}

\usepackage[textwidth=0.8in]{todonotes}
\begin{document}
\renewcommand{\abstractname}{\normalfont \scshape\fontfamily{ptm}\selectfont{Abstract}}
\renewcommand{\contentsname}{\normalfont \scshape\fontfamily{ptm}\selectfont{Contents}}
\begin{abstract}
We analyze the problem of determining Waring decompositions of the powers of any quadratic form over the field of complex numbers. Our main goal is to provide information about their rank and also to obtain decompositions whose size is as close as possible to this value. This is a classical problem and these forms assume importance especially because of their invariance under the action of the special orthogonal group. We give the detailed procedure to prove that the apolar ideal of the $s$-th power of a quadratic form is generated by the harmonic polynomials of degree $s+1$. We also generalize and improve some of the results on real decompositions given by B.~Reznick in his notes of 1992, focusing on possibly minimal decompositions and providing new ones, both real and complex. We investigate the rank of the second power of a non-degenerate quadratic form in $n$ variables, which in most cases is equal to $(n^2+n+2)/2$, and also give some results on powers of ternary quadratic forms.
\end{abstract}
\maketitle
\thispagestyle{empty}

\tableofcontents

\section{Introduction}
\label{sec_introduction}
\markboth{\normalfont \scshape\fontfamily{ptm}\selectfont Cosimo Flavi}{\normalfont \scshape\fontfamily{ptm}\selectfont Working notes}
A \textit{Waring decomposition} of size $r$ of a homogeneous polynomial $f\in\bbC[x_1,\dots,x_n]$ of degree $d$ is an expression of $f$ as a linear combination of the $d$-th powers of $r$ linear forms. That is, $f$ can be written as
\[
f=\sum_{j=1}^r\ell_j^d,
\]
where $\ell_1,\dots,\ell_r\in{\bbC^n}^*$.
The minimum of such a natural number $r$ is known as the \textit{Waring rank} of $f$, which we denote by $\rk f$. The problem of determining this value for a certain polynomial is in general not easy, and currently there is no efficient method to approach the problem. In this paper, we 
focus on the Waring rank of powers of quadratic forms. That is, for any quadratic form $q_n$ of rank $n$, we analyze the problem of determining suitable decompositions of $q_n^s$, with $s\in\bbN$, focusing on the minimal ones. Our starting point is to extend some of the results given by B.~Reznick in \cite{Rez92}, where he focused only on real decompositions. After generalizing to the field of complex numbers, we also provide new decompositions and some estimates of the rank. The main cases we consider are the forms in three variables and the square of quadratic forms in any number of variables. We begin with a brief overview of the theory of Waring decompositions and their applications. We then describe the contents of this paper in more detail.

\subsection{Sum of powers and Waring rank}
The Waring rank of a polynomial takes its name from E.~Waring, who in 1770, with the famous \textit{Waring problem}, asked whether for every natural number $k$ there exists a positive integer $s$ such that every natural number is the sum of at most $s$ natural numbers raised to the power $k$ (see \cite{War91}). This fascinating problem of number theory, after remaining open for over a century, was solved in 1909 by D.~Hilbert, who provided a proof of the affirmative answer, also known as the Hilbert-Waring theorem (see \cite{Hil09}). Determining the rank of polynomials is a rather classical problem and is nowadays approached with the languages of symmetric tensors. Indeed, once a basis is fixed, the symmetric algebra of a vector space of dimension $n\in\bbN$ corresponds to a polynomial ring in $n$ variables (see \autoref{prop nat isom S^d}). 
Especially in recent years, tensor decomposition has found many applications in various branches of science, both theoretical and applied. 
Indeed scientific data are often collected in multidimensional arrays, which are treated as tensors.
Some overviews of various applications related to tensor decompositions are provided by J.~M.~Landsberg in \cite{Lan12}*{chapter 1} and by B.~W.~Bader and T.~G.~Kolda in \cite{KB09}. Among the more theoretical ones, we can highlight the complexity of matrix multiplication (see e.g.~\cite{Str83}), the P versus NP complexity problem (see e.g.~\cite{Val01}), and the entanglement in quantum physics (see e.g.~\cite{BC12}). Instead, for applications to more applied sciences, we can consider independent component analysis and blind identification in signal processing (see \cite{Com94} and \cite{SGB00}, respectively), the study of phylogenetic invariants (see e.g.~\cite{AR08}), bioinformatics and spectroscopy (see e.g.~\cite{CJ10}). 
However, applications of symmetric tensor decomposition are also of great interest. In \cite{BGI11}*{section 1}, A.~Bernardi, A.~Gimigliano, and M.~Idà briefly summarize some examples, such as telecommunications in electrical engineering (see \cite{Che11} and \cite{DC07}) or cumulant tensors in statistics (see \cite{McC87}).
For this reason, despite its classical origins, the determination of the Waring rank of a polynomial retains a special role even in more recent times. 

The analog of the Waring problem for homogeneous polynomials, also known as the \textit{Big Waring Problem}, concerns the determination of the minimum number $r$ such that a generic form of degree $d$ admits a decomposition of size $r$. By \textit{generic}, we mean any form belonging to an open and dense subset in the Zariski topology. 
As one might expect, the question was not easy and remained unanswered until 1995, when the problem was solved by J.~Alexander and
A.~Hirschowitz in \cite{AH95}. They established in the famous Alexander-Hirschowitz theorem that the rank of a generic polynomial in $n$ variables of degree $d$ is given by the formula
\begin{equation}
\label{formula:generic_rank}
\rk f=\biggl\lceil\frac{1}{n}\binom{d+n-1}{d}\biggr\rceil,
\end{equation}
except for a few cases with some lower values of the degree and the number of variables
(a simpler version of the proof was provided in \cite{BO08}). So the forms having a rank different from the value of formula \eqref{formula:generic_rank} are special and, in particular, it can be interesting in many cases to ask whether they have a subgeneric or a supergeneric rank.

Although the Big Waring problem is fully understood, determining the rank of a given polynomial in general remains a difficult problem and there is currently no general efficient method to solve it.
Nevertheless, many partial results and methods have been produced among the years. For a more detailed survey, see  \cites{BCC+18,BGI11,CGO14, Lan12, LO13}.

The problem of determining the Waring rank of binary forms, already analyzed by J.~J.~Sylvester in \cite{Syl51}, was completely solved by G.~Comas and M.~Seiguer in \cite{CS11}). Many algorithms lead to explicit decompositions, such as the \textit{Sylvester algorithm}, introduced in \cite{Syl86} and also presented in \cites{CS11, BGI11, BCMT10}. It has been further improved with several other variants (see e.g.~\cite{BGI11}*{Algorithm 2}).

For a larger number of variables, other algorithms may provide decompositions, under very specific hypotheses (see, e.g., \cite{BCMT10}*{Algorithm 7.1}), which do not solve the problem in general. However, for some classes of polynomials the Waring rank is completely known, such as for monomials (see \cite{CCG12}*{Proposition 3.1}). 

The theory of apolarity is a very useful classical topic in this context, mostly formalized in 1999 by A.~Iarrobino and V.~Kanev in \cite{IK99}, to which we refer to get further details.
For any $f\in S^d\bbC^n$, the \textit{catalecticant map} of $f$ is defined by the derivative action of differential polynomial operators on $f$. That is, for every monomial $\bfy^{\bfalpha}=y_1^{\alpha_1}\cdots y_n^{\alpha_n}$ with $\bfalpha=(\alpha_1,\dots,\alpha_n)\in\bbN^n$, the map is obtained by extending by linearity the function defined on monomials by
\[
\Cat_f(\bfy^{\bfalpha})=\frac{\pder^{\abs{\bfalpha}}f}{\pder\bfx^{\bfalpha}}.
\]
The kernel of this map is known as the \textit{apolar ideal}, or \textit{annihilator}, of $f$, denoted by $\Ann(f)$ (see \cref{section apolarity}). 
The \textit{apolarity lemma} (\autoref{Lem Apo}, \cite{IK99}*{Lemma 1.15}) states that the ideal of the points associated to the linear forms appearing in each decomposition of $f$ is contained in $\Ann(f)$.
Thanks to this correspondence, the constructions of such ideals lead to explicit decompositions. In particular, this fact allows to obtain some information about the decompositions through the analysis of algebraic invariants, such as the Hilbert function.

A related problem is the discussion of the uniqueness of decompositions. Probably the oldest result on this topic is due to J.~J.~Sylvester, who proved in 1851 (see \cite{Syl51}), that a generic binary form of odd degree $2n-1$ can be uniquely written as a sum of $n$ powers of linear forms, which is also called the \textit{canonical form}. Another related result is the one historically known as Sylvester’s pentahedral theorem, which first appeared  in \cite{Syl51a} in 1851 (see also \cite{Dol12}*{Theorem 9.4.1} for a more recent reference). It states that every generic form $f\in S^3\bbC^4$ can be written uniquely, up to scalars, as a sum of the $3$-rd powers of $5$ 
non-proportional linear forms.
Other examples have been given for ternary forms of degree $4$ by D.~Hilbert in \cite{Hil88}, and for ternary forms of degree $5$ by F.~Palatini in \cite{Pal03} and by H.~W.~Richmond in \cite{Ric04}. The same argument has been treated from a more modern point of view in several recent papers (see e.g.~\cites{GM19,Mel09,RS00}), with the theory of \textit{identifiability}. In our context, we are interested in the uniqueness of decompositions up to 
suitable transformations, for polynomials which are invariant under the action of a linear group. One of the most illustrative examples in this sense (cf.~\autoref{teo_uniqueness_icosahedron}) is given by the polynomial
\[
(x_1^2+x_2^2+x_3^2)^2.
\]
This form, as proved by B.~Reznick in \cite{Rez92}*{Theorem 9.13} for the real case, can be represented by a sum of $6$ different $4$-th powers of linear forms, which must be the vertices of a regular icosahedron.

When dealing with decompositions of polynomials, other notions of tensor rank have been introduced in the literature. One of these is the \textit{border rank} of a polynomial, first introduced in 
\cites{BCRL79,BLR80}, where the authors defined it as the minimum number
of decomposable tensors required to approximate a tensor with an arbitrarily small error. Especially for applications, it can be useful in cases where
tensors have a large gap between rank and border rank (see, for instance,
\cite{Zui17}). The analog for polynomials is defined in the same way.
Other examples are the concepts of \textit{cactus rank} and \textit{smoothable rank} (see \cite{RS11}).
Thanks to the theory of secant varieties (see \cite{Har95}*{lecture 8}), a strong connection between the rank of polynomials and algebraic varieties has been developed. For more details, the reader can see, for instance, \cites{Bal10b,Bal10a,BCC+18,BL13,CGO14, LO13}.

\subsection{Powers of quadratic forms}
Quadratic forms appear frequently in various branches of mathematics, including algebra, analysis, arithmetic, geometry, and number theory, both classical and modern. There are a large number of texts that the reader can consult for further details on the subject, such as \cites{Eic52,Jon50,Lam80,Ome63,Ser73,Sie63,Sie67,Syl52,Wat60}.
In particular, the foundations of the general theory of quadratic forms over the rationals and rational integers were introduced in the late 19\textsuperscript{th} century by H.~Minkowski (see \cites{Min84,Min85,Min90}). For a more detailed historical overview, see \cite{Sch22}. However, our interest in quadratic forms is due to their invariance through the action of the orthogonal groups.
 
Up to a linear change of variables, every quadratic form can be written as the homogeneous polynomial
\[
q_n=x_1^2+\cdots+x_n^2
\]
for an appropriate $n\in\bbN$ (see~\cref{sec_quadratic_forms}, \cite{Ser73}*{chapter 4}). Then $q_n^s$ is a $2s$-degree form which is invariant under the action of the orthogonal group $\Oa_n(\bbC)$, for every $s\in\bbN$. In particular, this means that each of its decompositions has an $n$-dimensional orbit under the action of the orthogonal group.
Determining the rank of the form $q_n^s$ is the main problem we deal with. While in the case of binary forms the problem is quite simple and completely solved, we cannot say the same for the general case in more variables, about which there is not much information in the literature. 
The most complete analysis on this subject so far is due to B.~Reznick, who gives in \cite{Rez92} a precise survey of both classical and more original results over the field $\bbR$. Following the greater relevance that real numbers used to have in applications, with respect to complex ones, B.~Reznick focuses in his notes on real Waring decompositions, which he calls \textit{representations}. In order to get a new vision of this problem, especially in view of recent applications to tensors, we extend the problem to complex Waring decompositions. For powers of quadratic forms, this represents a new
approach, different from the classical point of view.

Several uses of decompositions of powers of quadratic forms have been listed by B.~Reznick in \cite{Rez92}*{section 8}, such as in number theory, to study the Waring problem, or even in functional analysis.
The quadratic form $q_n$, in terms of differential operators, is exactly the well-known Laplace operator
\[
\Lap=\pdv[2]{}{x_1}+\cdots+\pdv[2]{}{x_n},
\]
which plays an important role in analysis.

By an algebraic point of view, a formula of great significance is the decomposition
\[
S^d\bbC^n=\bigoplus_{j=0}^{\left\lfloor\frac{d}{2}\right\rfloor}q_n^j\calH_{n,d-2j},
\]
where, for every $j\in\bbN$, $\calH_{n,d-2j}$ is the space of homogeneous harmonic polynomials of degree $d-2j$ in $n$ variables. This elegant decomposition, which has been shown in detail by R.~Goodman and N.~R.~Wallach in \cite{GW98}*{Corollary 5.2.5}, gains much importance by considering its invariance under the action of the orthogonal complex group $\Oa_n(\bbC)$. 
This decomposition is an important tool we will use to determine the apolar ideal of $q_n^s$ (see \cref{sec_apolar_ideal_harmonic}). 
B.~Reznick has already noted that the catalecticant matrices of $q_n^s$ are all of full rank. In particular, denoting by $T_{n,s}$ the size of the middle catalecticant matrix, equal to
\[
T_{n,s}=\binom{s+n-1}{s},
\]
we get the following lower bound (\autoref{cor_cat_lower_bound}):
\[
\rk (q_n^s)\geq T_{n,s}.
\]
We start in \autoref{sec:preliminaries} with some basic elements of representation theory. This is rather important in connection with the fact that the space of $d$-harmonic polynomials $\calH_{n,d}$ is an irreducible $\SO_n(\bbC)$-module (\cite{GW98}*{Theorem 5.2.4}). We then present some notions of both classical and recent theory of apolarity, such as the apolar ideal, catalecticant map, the apolarity lemma, and Waring decompositions. We conclude the section with some basic information about quadratic forms, showing the crucial fact that every quadratic form of rank $n$ is equivalent to the form
\[
q_n=x_1^2+\cdots+x_n^2.
\]

In \autoref{cha_apolarity_quadratic_forms}, we focus our attention on the apolarity action on the polynomial $q_n^s$, by determining the structure of catalecticant matrices for every $n,s\in\bbN$ and we give our first main result (see.~\autoref{Teo Apolar ideal}), by proving that the apolar ideal of $q_n^s$ is given by
\[
\Ann(q_n^s)=(\calH_{n,s+1}).
\]
This result is already shown in \cite{Fla23a}*{Theorem 3.8} in a less detailed version. The main result we use is the fact that the space of harmonic polynomials is an irreducible representation of $\SO_n(\bbC)$.
Once we know how the apolar ideal is formed, our problem turns into the analysis of ideals of points contained in it, thanks to the apolarity lemma (\autoref{Lem Apo}). This characterization is useful in determining explicit decompositions that have the same pattern, as we will see in \autoref{cha_general_decompositions_more_variables}. 

B.~Reznick provides in \cite{Rez92}*{chapters 8-9} both classical and new decompositions, and also gives proofs of minimality for some of them. In particular, using the language of spherical designs (see \cite{DGS77}), he analyzes the existence and the uniqueness of \textit{tight representations}, namely, real decompositions with size equal to the rank of the middle catalecticant matrix. He summarizes this in the following theorem, which, in particular, guarantees that there is no tight representation for $s\geq 6$.
\begin{teor}[\cite{Rez92}*{Proposition 9.2}]
If $q_n^s$ has a real tight decomposition, then one of the following conditions holds:
\begin{enumerate}[label=(\arabic*), left= 0pt, widest=*,nosep]
\item $s=1$ or $n=2$;
\item $s=2$ and $n=3$;
\item $s=2$ and $n=m^2-2$ for some odd $m\in\bbN$;
\item $s=3$ and $n=3m^2-4$ for some $m\in\bbN$;
\item $s=5$ and $n=24$.
\end{enumerate}
\end{teor}
We partially generalize this theorem for complex numbers. Indeed, we focus in 
\cref{cha_tight_decompositions} on tight decompositions. We start with an overview of the already known facts for real decompositions. Then we analyze all the decompositions in two variables, proving the uniqueness also for the complex case, and the appropriate tight decompositions for the exponents $s=2,3$ for an arbitrary number of variables.
In particular, in the case where $s=2$, we have the following theorem.
\begin{teor}
If $q_n^2$ has a tight decomposition, then $n=3$ or $n=m^2-2$ for some odd number $m\in\bbN$.
\end{teor}
In the case of $s=3$ we get a weaker result than the one proposed by B.~Reznick.
\begin{teor}
If $q_n^3$ has a tight decomposition, then $n\equiv 2\bmod 3$.
\end{teor}
The proof of these results is made by following exactly the same strategy that B.~Reznick uses in \cite{Rez92}*{pp.~130-132}, but the possibility of using it is not so immediate. Indeed, he uses the important fact that every tight representation must be \textit{first caliber}, namely that every point of such a decomposition must have the same norm. That is, if
\[
q_n^s=\sum_{j=1}^r(\bfa_j\cdot\bfx)^{2s}
\]
is a tight real decomposition with $r\in\bbN$ and $\bfa_1,\dots,\bfa_r\in\bbR^n$, then
\[
\abs{\bfa_j}=\abs{\bfa_k}
\]
for every $1\leq j,k\leq r$.
Denoting by $B_{n,s}$ the value of the norm of every point raised to $2s$, for such a tight decomposition, we have (see \cite{Rez92}*{Corollary 8.18})
\[
B_{n,s}=\abs{\bfa_j}^{2s}=\frac{1}{T_{n,s}}\prod_{j=0}^{s-1}\frac{2j+n}{2j+1}=\binom{s+n-1}{s}^{-1}\prod_{j=0}^{s-1}\frac{2j+n}{2j+1}.
\]
In view of this, B.~Reznick determines the kernel of the middle catalecticant of the polynomial
\[
q_n^2-B_{n,2}x_1^4,
\]
which must be necessarily nonzero. In particular, this imposes consequent conditions on remaining points of the decompositions.

The new fact, which allows us to generalize the result of B.~Reznick, extending the notion of first caliber to complex decompositions, is that every tight complex decomposition is also first caliber. In fact, we prove in \cref{sec_real_decomp_spherical_designs} by \autoref{lem_catalecticant_isotropic_point} that every tight decomposition cannot contain an isotropic point. As a consequence, we get the following theorem (cf.~\autoref{teo_tight_implies_first_caliber}). 
\begin{teor}
Every tight decomposition 
\[
q_n^s=\sum_{k=1}^{r}(\bfa_k\cdot\bfx)^{2s}
\]
is first caliber. In particular,
\[
(\bfa_k\cdot\bfa_k)^{s}=\frac{1}{r}\prod_{j=0}^{s-1}\frac{2j+n}{2j+1}.
\]
for every $k=1,\dots,r$.
\end{teor}
In \autoref{cha_general_decompositions_more_variables}, we provide some general decompositions for different values of $n$. In particular, we first focus on the decomposition of $q_n^2$, establishing its rank for most of the values of $n\in\bbN$, and then we provide some specific decompositions regarding ternary forms, which seem to follow a quite similar pattern.
In the case of exponent $s=2$, a quadrature formula presented by A.~H.~Stroud in \cite{Str67a} is used by B.~Reznick to prove that, for $n=4,5,6$, the real rank of $q_n^2$ is equal to $T_{n,2}+1$, while in the case of $n=7$, it gives a tight decomposition of size $T_{7,2}$. In \autoref{teo_minimal_decomposition_q_n^2} we extend this family of decompositions to an arbitrary number $n\geq 9$ of variables. This result, together with another family of decompositions for $n=8$ (cf.~\autoref{teo_minimal_decomposition_q8^2}), establishes its exact rank, hence solving the Waring rank problem, for most of the values assumed by $n$ in the case where $s=2$. We summarize this in \autoref{teo:summarize_rank_q_n^2} and get the following result.
\begin{teor}
Let $n\geq 3$. Then the following conditions hold:
\begin{enumerate}[label=(\arabic*), left= 0pt, widest=*,nosep]
\item if $n=3,7,23$, then $\rk (q_n^2)=T_{n,2}$;
\item if $n>23$ and $n=m^2-2$ for some odd number $m\in\bbN$, then $T_{n,2}\leq\rk(q_n^2)\leq T_{n,2}+1$;
\item if $n=8$, then $T_{n,2}+1\leq \rk(q_n^2)\leq T_{n,2}+9$;
\item otherwise, $\rk(q_n^2)=T_{n,2}+1$.
\end{enumerate}
\end{teor}

Apart from these general properties, B.~Reznick actually studies the real rank, also called \textit{width}, for several specific cases, including that of two variables. For this case he concludes that all the points of such minimal decompositions must be the vertices of a regular polygon inscribed in a circle whose radius depends only on $s$. Again, thanks to apolarity, we are able to generalize this fact and also to determine the exact correspondence between polynomials in the apolar ideal and complex decompositions. The latter are found to be equivalent, up to complex orthogonal transformations, to the real ones already known.

Other decompositions analyzed by B.~Reznick concern ternary forms. In particular, the power $q_3^2$ is quite special. Indeed, each of its minimal decompositions, which is of size $6$ and thus tight, must be composed of points corresponding to the vertices of a regular icosahedron inscribed in a sphere of radius $5/6$. Again, by the first caliber property, we can extend the uniqueness to the complex field (\autoref{teo_uniqueness_icosahedron}). Besides this, denoting by $\rk_{\bbR} f$ the real rank of a homogeneous polynomial $f$, he also shows that
\[
\rk_{\bbR}(q_3^3)=T_{3,3}+1=11,\qquad \rk_{\bbR}(q_3^4)=T_{3,4}+1=16.
\] 
The first case is easily obtained by previous results on the exponent $s=3$, while the second requires a more complicated approach. We will see in \cref{sec_decomp_three_variables} that the first caliber property imposes strong conditions on the angles between the different points, which can take only a finite set of values. In particular, through an analysis by the Gram matrices of $4$ points in the space $\bbC^3$, we are able to prove that
\[\rk(q_3^4)=T_{3,4}+1=16.\]
Finally, we give an upper bound on the rank. In \cite{Fla23b} an estimate of the growth of $\rk(q_n^s)$ for $s$ fixed and $n\to+\infty$ is given. In particular, we have
\[
\lim_{n\to+\infty}\log_n\bigl(\rk(q_n^s)\bigr)=s.
\]
Based on this, it is not hard to prove that for a sufficiently large $n$ the rank of $q_n^s$ is subgeneric. However, we do not know if this holds for the other cases there is currently no evidence allowing to think that the rank is subgeneric in general. Moreover, we conjecture that, fixed $n\in\bbN$, the rank is supergeneric for a sufficiently large value of $s$.

\section{Preliminaries}
\label{sec:preliminaries}
In order to analyze decompositions of powers of quadratic forms, we need some preliminary notions.
In \cref{sec_repres_modules}, basic concepts of representation theory are reviewed, emphasizing the importance of irreducible representations. All the results we present are contained in \cite{FH91}, where more detailed information about the subject can be found, as well as in many other texts, such as \cites{Pro07,Ser77}. For a comprehensive understanding of Lie groups and Lie algebras, the reader can see, for instance, \cites{AT11,GQ20, GW98, Kir08, MT11}. In particular, a complete survey of linear algebraic groups is provided in \cite{MT11}.

In \cref{section apolarity} we discuss some elements of the theory of apolarity and Waring decompositions, and give a brief overview of the concept of Waring rank. While our primary reference for the theory of apolarity is \cite{IK99}, we also include more recent insights from \cite{Dol12}. The reader can find information on various concepts of tensor rank in many standard texts. For a complete introduction, we recommend the textbook reference \cite{Lan12}. In addition, many papers also summarize critical aspects of rank and border rank of symmetric tensors, including references such as \cites{BCC+18,LO13,LT10}, to which we refer for details.

Finally, in \cref{sec_quadratic_forms} we give a short basic overview of quadratic forms. This is essential to see that quadratic forms of a given rank are equivalent. The main reference we use is \cite{Ser73}*{chapter~IV}.

\subsection{Representations and modules}
\label{sec_repres_modules}
Throughout this section, we will only consider finite dimensional vector spaces over $\bbC$. If not specified, $U$, $V$ and $W$ will denote such spaces and $G$ will denote a group.
We begin with the basic notion of group representation.
\begin{defn}
A group homomorphism $\rho\colon G\to \GL(V)$ is called a \textit{representation} of $G$, and in this context $V$ is called a \textit{$G$-module}.
\end{defn}
In cases where there is no confusion, the action of $g$ on $v$ for any $G$-module $V$ is denoted as $g\cdot v$ for all $g\in G$ and $v\in V$, without explicitly stating the associated representation.
\begin{defn}
Let $V$ be a $G$-module. A subset $U\subset V$ is said to be \textit{$G$-invariant} if $g\cdot u\in U$ for every $g\in G$ and $u\in U$. If $U$ is also a linear subspace of $V$, then $U$ is said to be a \textit{$G$-submodule} of $V$.
\end{defn}
A $G$-submodule $U$ of $V$ is naturally a $G$-module, with the rule
\[
g\cdot u=\rho(g)|_U(u),
\]
for every $u\in U$ and $g\in G$.

\begin{defn}
\label{def equivariant maps}
Let $V$ and $W$ be two $G$-modules. A \textit{$G$-equivariant
map} (or a \textit{$G$-modules homomorphism}) from $V$ to $W$ is a linear map $\alpha\colon V\to W$ that satisfies
\[
\alpha(g\cdot v)=g\cdot\alpha(v),
\]
for every $g\in G$ and $v\in V$.
If $\alpha$ is an isomorphism, it is said to be a \textit{$G$-modules isomorphism}, and the $G$-modules $V$ and $W$ are said to be \textit{isomorphic}.
\end{defn}

Furthermore, it is worth noting that many natural structures arise from $G$-modules. It is easy to verify that they are well defined.

\begin{rem}
\label{rem_representations_groups}
Let $V$ and $W$ be two $G$-modules.
\begin{enumerate}[label=(\arabic*), left= 0pt, widest=2,nosep]
\item For any $G$-equivariant map $\alpha\colon V\to W$, the kernel $\Ker\alpha$ and the image $\im\alpha$ are $G$-submodules of $V$ and $W$, respectively.
\item The vector spaces $V\oplus W$ and $V\otimes W$ are $G$-modules. The actions are defined as 
\[
g\cdot(v+w)=(g\cdot v)+(g\cdot w),\qquad g\cdot(v\otimes w)=(g\cdot v)\otimes(g\cdot w).
\]
for every $g\in G$, $v\in V$, and $w\in W$.
\item As a direct consequence of point (2), for every $d\in\bbN$, the $d$-symmetric power $S^dV$ of $V$ is a $G$-module with the action
\[
g\cdot\bigl(v_1^{k_1}\cdots v_n^{k_n}\bigr)=(g\cdot v_1)^{k_1}\cdots (g\cdot v_n)^{k_n},
\]
for every $v_1,\dots,v_n\in V$ and $k_1,\dots,k_n\in\bbN$ such that $k_1+\cdots+k_n=d$. Furthermore, the symmetric algebra $S(V)$, as a direct sum of $G$-modules, extended by linearity, is also a $G$-module.
\item If $\rho\colon G\to \GL(V)$ is a representation, then the dual vector space $V^*$ has a natural $G$-module structure given by the dual representation $\rho^*\colon G\to\GL(V^*)$. It is defined as
\[
\rho^*(g)=\transpose{\rho}(g)^{-1}
\]
for every $g\in G$, preserving the dual pairing, i.e., 
\[
(g\cdot\phi)(g\cdot v)=\phi(v)
\]
for every $g\in G$, $\phi\in V^*$, and $v\in V$.
\end{enumerate}
\end{rem}
The class of $G$-modules not containing any nontrivial $G$-submodule is highly significant.
\begin{defn}
A $G$-module $V$ is \textit{irreducible} if it does not contain a nonzero proper $G$-submodule. That is, if $W$ is a $G$-submodule of $V$, then $W$ is either the zero submodule or $W=V$.
\end{defn}
We can now present Schur's Lemma, an essential statement about irreducible modules. I. Schur first introduced this lemma in 1905 (see \cite{Sch05}), and it can also be found, for instance, in \cite{BD95}*{Theorem 1.10}, \cite{Hal15}*{Theorem 4.26}, or \cite{Ser77}*{Proposition 4}.
\begin{lem}[Schur's Lemma]
\label{Lem Schur}
Let $V$ and $W$ be irreducible $G$-modules, and let $\alpha\colon V\to W$ be a $G$-equivariant map. The following holds: 
\begin{enumerate}[label=(\arabic*), left= 0pt, widest=2,nosep]
\item if $V$ and $W$ are not isomorphic, then $\alpha=0$;
\item if $V=W$, then $\alpha=\lambda\cdot\id$ for some $\lambda\in\bbC$.
\end{enumerate}
\end{lem}
\begin{proof}
Point (1) follows directly from the fact that $\Ker\alpha$ and $\im\alpha$ are $G$-submodules. Therefore, if $V=W$ and $\lambda\in\bbC$ is an eigenvalue of $\alpha$ we can set $\alpha'=\alpha-\lambda\cdot\id$, which implies $\Ker\alpha'\neq 0$. By point (1), we conclude that $\alpha'=\alpha-\lambda\cdot\id=0$, 
which gives us, $\alpha=\lambda\cdot\id$.
\end{proof}
Analogously to algebraic groups, we can extend the concept of representations to Lie algebras, associating a linear map with each element of a Lie algebra, without requiring it to be an isomorphism.
\begin{defn}
Let $\mfg$ be a Lie algebra and let $V$ be a vector space. A map of Lie algebras
\[
\rho\colon\mfg\to\mfgl (V)\cong\End(V)
\]
is called a \textit{representation} of $\mfg$.
\end{defn}

Similar to $G$-modules, we can define subrepresentations focusing on those that do not contain any proper subrepresentations.

\begin{defn}
Let $\rho\colon\mfg\to\mfgl(V)$ be a representation of $\mfg$. If $W$ is a vector subspace of $V$ such that
\[
\rho(A)(w)\in W,
\]
for every $A\in\mfg$ and $w\in W$, then it is said to be \textit{$\mfg$-invariant}. The representation $\rho\colon\mfg\to\mfgl(V)$ is said to be \textit{irreducible} if the vector space $V$ does not contain a proper $\mfg$-invariant subspace.
\end{defn}

Of course, we can study the relations between representations of Lie groups and their Lie algebras, in particular the irreducibility of representations. The following proposition gives an insight into this relation and can also be found as an exercise in \cite{FH91}*{Exercise 8.17}.

\begin{prop}
\label{prop_invariant_representations}
Let $G$ be a connected Lie group and let $V$ be a $G$-module with the associated representation
$\rho\colon G\to\GL(V)$. Then, a subspace
$W$ of $V$ is a $G$-submodule if and only if it is invariant under the action of the Lie algebra $\mfg$ of $G$, namely, 
\[
\diff\rho_e(v)(W)\subseteq W
\] 
for every $v\in\mfg$.
\end{prop}
\begin{proof}
If $f\colon G\to H$ is a homomorphism of Lie groups, then
\[
\exp\circ\diff f_e=f\circ\exp,
\]
which guarantees the commutativity of the diagram
\[
\begin{tikzcd}
\mfg\ar[r,"\diff{f}_e"]\ar[d,"\exp"']&\mfh\ar[d,"\exp"]\\
G\ar[r,"f"]&H
\end{tikzcd}
\]
(see \cite{AT11}). Consequently, if we set $H=\GL(V)$, we find that the diagram
\begin{equation}
\label{rel_diagram_repres}
\begin{tikzcd}
\mfg\ar[r,"\diff{\rho}_e"]\ar[d,"\exp"']&\End(V)\ar[d,"\exp"]\\
G\ar[r,"\rho"]&\GL(V)
\end{tikzcd}
\end{equation}
commutes.  If we consider a subspace $W\subseteq V$ such that $\diff\rho_e(A)(W)\subseteq W$
for every $A\in \mfg$, then
we can construct a closed Lie subgroup $\GL(W)$ of $\GL(V)$. Since 
\[
\diff\rho_e(A)\big|_W\in\End(W),
\]
we must have
\[
\exp\big|_{\End(W)}\bigl(\diff\rho_e(A)\big|_W\bigr)\in\GL(W),
\]
which implies
\[
\exp\bigl(\diff\rho_e(A)\bigr)\big|_W\in\GL(W).
\]
By the commutativity of diagram \eqref{rel_diagram_repres}, we have 
\[
\rho\bigl(\exp(A)\bigr)\big|_W\in\GL(W).
\]
Since $G$ is connected, it follows from the properties of the exponential map (see e.g.~\cite{Kir08}*{Theorem 3.7}) that $G$ is generated by $\exp(\mfg)$. Therefore, we have
\[
\rho(g)\big|_W\in\GL(W)
\]
for every $g\in G$,
which means that $W$ is a $G$-submodule of $V$. 
Conversely, if $\rho(g)(w)\in W$
for every $g\in G$ and $w\in W$, then we can define a morphism $\rho\colon G\to \GL(W)$,
considering the identification
\[
\rho(g)=\rho(g)\big|_W.
\]
In particular, the diagram
\[
\begin{tikzcd}
\mfg\ar[r,"\diff{\rho}_e"]\ar[d,"\exp"']&\End(W)\ar[d,"\exp"]\\
G\ar[r,"\rho"]&\GL(W)
\end{tikzcd}
\]
is well defined and commutes. So we can also restrict the differential $\diff\rho_e$ and get
\[
\diff\rho_e(A)\big|_W\in\End{W},
\]
for every $A\in \mfg$.
This implies that $W$ is invariant under the action of $\mfg$.
\end{proof}
\begin{cor}
Let $G$ be a connected Lie group and let $V$ be a $G$-module.
Then $V$ is irreducible
over $G$ if and only if it is irreducible over $\mfg$. 
\end{cor}
Now we focus on a particular example concerning irreducible representations of Lie algebras. We analyze
the Lie algebra of the Lie group $\SL_2(\bbC)$, which is the space
\[
\mfsl_2(\bbC)=\Set{A\in \Mat_2(\bbC)|\tr A=0}.
\]
All of the following in this section is taken from \cite{FH91}*{section 11.1} and  \cite{Kir08}*{section 4.8}, to which we refer for further details.
Let us consider the basis formed by the matrices
\[
H=\begin{pNiceMatrix}
1 &0\\
0 &-1
\end{pNiceMatrix},\quad
E=\begin{pNiceMatrix}
0 &1\\
0 &0
\end{pNiceMatrix},\quad
F=\begin{pNiceMatrix}
0 &0\\
1 &0
\end{pNiceMatrix},
\]
which satisfy the equations
\begin{equation}
\label{rel_equations_sl2C}
[H,E]=2E,\quad [H,F]=-2F,\quad [E,F]=H.
\end{equation}
\begin{prop}
The Lie algebra $\mfsl_2\bbC$ is a simple algebra.
\end{prop}
\begin{proof}
Let $\mfa\subseteq\mfsl_2\bbC$ be an ideal of $\mfsl_2\bbC$ such that $\mfa\neq 0$. Consider any nonzero linear combination
\[
aH+bE+cF\in\mfa
\]
with $a,b,c\in\bbC$. We observe that if $H\in\mfa$, then
by equations \eqref{rel_equations_sl2C} we have
\[
[H,E]=2E\in\mfa,\qquad [H,F]=-2F\in\mfa
\]
and thus $\mfa=\mfsl_2\bbC$. So, if $b=c=0$, the statement is trivial. Otherwise, in the case of $b=0$, since
\[
[H,aH+bE+cF]=a[H,H]+b[H,E]+c[H,F]=2bE-2cF\in\mfa,
\]
we easily get $F\in\mfa$ and thus also 
\[
[E,F]=H\in\mfa.
\] 
Similarly, if $c=0$, we also get $H\in\mfa$ as well. Finally, if $b,c\neq 0$, we have
\[
(aH+bE+cF)+(bE-cF)=aH+2bE\in\mfa,
\]
and hence
\[
[H,aH+2bE]=a[H,H]+2b[H,E]=4bE\in\mfa.
\]
This implies, in particular, that
\[
[E,F]=H\in\mfa,
\]
proving the statement.
\end{proof}
The classical Jordan decomposition (see e.g.~\cite{Bor91}*{Proposition 4.2}) says that every endomorphism $f\colon V\to V$ of a complex vector space $V$ can be written uniquely as
\[
f=f_s+f_n,
\]
where $f_s\colon V\to V$ and $f_n\colon V\to V$ are respectively a diagonalizable endomorphism and a nilpotent endomorphism commuting with each other, namely,
\[
f_nf_s-f_sf_n=0.
\]
This concept can be extended to Lie algebras in the special case of semisimple Lie algebras. Indeed, each of them can be decomposed into a sum of two elements, preserving the diagonalizable and the nilpotent parts. The proof of this theorem, which can be found in \cite{FH91}*{Theorem 9.20, Corollary C.18}, is rather technical and we omit it. We suggest the reader to consult \cite{FH91}*{Appendix C} for more specific details and a complete proof.
\begin{teo}[Preservation of Jordan decomposition]
\label{teo_preservation_Jordan_dec}
Let $\mfg$ be a semisimple Lie algebra. For every $X\in\mfg$, there exist two element $X_s$, $X_n\in\mfg$ such that for any representation $\rho\colon\mfg\to\mfgl(V)$ we have
\[
\rho(X_s)=\rho(X)_s,\quad \rho(X_n)=\rho(X)_n.
\]
In particular, if $\rho$ is injective and $\mfg$ is a subalgebra of $\mfgl(V)$, then the diagonalizable and nilpotent parts of any element $X$ of
$\mfg$ are again in $\mfg$ and are independent of the particular representation $\rho$.
\end{teo}
Let $V$ be a finite-dimensional irreducible representation of $\mfsl_2\bbC$. By \autoref{teo_preservation_Jordan_dec}, it is clear that the action of $H$ on $V$ is diagonalizable. Therefore, we can write
\begin{equation}
\label{rel_decomposition_eigenspaces}
V=\bigoplus_{\lambda\in I}V_{\lambda},
\end{equation}
where $I$ is a finite set of complex values such that
\[
H(v)=\lambda v
\]
for every $\lambda\in I$ and $v\in V_{\lambda}$.
Now, given a vector $v\in V_{\lambda}$, we have
\begin{align*}
H\big(E(v)\big)&=E\big(H(v)\big)+[H,E](v)\\
&=E(\lambda v)+2E(v)\\
&=(\lambda+2)E(v).
\end{align*}
This means that $E$ sends each eigenvector to a different eigenspace and, in particular,
its restriction to $V_{\lambda}$ gives a morphism $E\colon V_{\lambda}\to V_{\lambda+2}$ for every $\lambda\in I$. Analogously, we get a morphism $F\colon V_{\lambda}\to V_{\lambda-2}$ such that
\[
H\big(F(v)\big)=(\lambda-2)F(v).
\]
The space
\[
\bigoplus_{k\in\bbZ}V_{\lambda_0+2k},
\]
is invariant under the action of $\mfsl_2\bbC$ and therefore, by the irreducibility of $V$, we must have
\[
V=\bigoplus_{k\in\bbZ}V_{\lambda_0+2k}.
\]
This means that the eigenvalues appearing in decomposition \eqref{rel_decomposition_eigenspaces} must be congruent modulo $2$. Furthermore, the spectrum of $H$ must consist of an unbroken string of complex numbers \[\lambda_0, \lambda_0+2,\dots,\lambda_0+2m\] for some $m\in\bbN$. More specifically, if we set $n=\lambda_0+2m$, then we can describe the action of $\mfsl_2\bbC$ on $V$ by the diagram
\begin{equation}
\label{rel_diagram_eigenspaces}
\begin{tikzcd}
0\ar[r,yshift=2.2,"E"] &V_{n-2m} \ar[l,yshift=-2.2,"F"]\ar[loop below,"H"] \ar[r,yshift=2.2,"E"]&V_{n-2m+2} \ar[l,yshift=-2.2,"F"]\ar[loop below,"H"] \ar[r,yshift=2.2,"E"]  & \cdots\ar[l,yshift=-2.2,"F"]\ar[r,yshift=2.2,"E"] & V_{n-2} \ar[loop below,"H"] \ar[l,yshift=-2.2,"F"]\ar[r,yshift=2.2,"E"] & V_n \ar[loop below,"H"] \ar[r,yshift=2.2,"E"] \ar[l,yshift=-2.2,"F"] & 0 \ar[l,yshift=-2.2,"F"]
\end{tikzcd}.
\end{equation}
Therefore, it is sufficient to determine the value of $n\in\bbC$ to determine all the eigenspaces.

We now present a series of results that allow us to determine the structure of the space $V$.
\begin{lem}
\label{lem_powersF_generates_V}
Let $v\in V_n$. Then the set $\Set{F^k(v)|k\in\bbN}$
generates the space $V$.
\end{lem}
\begin{proof}
Let us consider the subspace
\[
W=\bigl\langle F^k(v)\bigr\rangle_{k\in\bbN}\subseteq V.
\]
Then it suffices to show by the irreducibility of $V$ that $W$ is irreducible under the action of $\mfsl_2\bbC$. It is clear that $F$ preserves the space $W$. Moreover, since $F^k(v)\in V_{n-2k}$ for every $k\in\bbN$, it follows that
\begin{equation}
\label{rel_eigenvectors_powersF}
H\bigl(F^k(v)\bigr)=(n-2k)F^k(v)
\end{equation}
and thus also $H$ preserves the space $W$. Finally, to prove that 
\[
E(W)\subseteq W,
\]
we can show by induction on the power $k\in\bbN$ that
\[
E\bigl(F^k(v)\bigr)\in W
\] 
for every $k\in\bbN$.
If $k=0$, then, since $v\in V_n$, we have
$E(v)=0$,
which clearly belongs to $W$. Now, let us suppose that $E\bigl(F^{k-1}(v)\bigr)\in W$. Then we have
\begin{align*}
E\bigl(F^k(v)\bigr)&=F\Bigl(E\bigl(F^{k-1}(v)\bigr)\Bigr)+[E,F]\bigl(F^{k-1}(v)\bigr)\\
&=F\Bigl(E\bigl(F^{k-1}(v)\bigr)\Bigr)+H\bigl(F^{k-1}(v)\bigr).
\end{align*}
Since both $F$ and $H$ preserve $W$, we have that $E$ also preserves $W$, which proves that $V=W$.
\end{proof}
Considering formula \eqref{rel_eigenvectors_powersF}, we have by \autoref{lem_powersF_generates_V} that 
\[
V=\langle F^k(v)\rangle_{k\in\bbN}.
\]
So we immediately get the following corollary.
\begin{cor}
If $V_{\lambda}$ is an eigenspace of $H$, then $\dim V_{\lambda}=1$.
\end{cor}
It is  also possible to explicitly determine the image of each power $F^k(v)$ from the operator $E$.
\begin{lem}
\label{lem_formula_inductive_powersF}
The equality
\[
E\bigl(F^k(v)\bigr)=k(n-k+1)F^{k-1}(v)
\]
holds for every $k\in\bbN\setminus\{0\}$ and $v\in V_n$.
\end{lem}
\begin{proof}
We proceed by induction on $k$. If $k=1$, then
\[
E\bigl(F(v)\bigr)=F\bigl(E(v)\bigr)+[E,F](v)=H(v)=nv,
\]
where the second equality follows from $E(v)=0$. Now, let us suppose that the statement is true for $k-1$. Then by inductive hypothesis and formula \eqref{rel_eigenvectors_powersF} we have the equalities
\begin{align*}
E\bigl(F^k(v)\bigr)&=F\Bigl(E\bigl(F^{k-1}(v)\bigr)\Bigr)+[E,F]\bigl(F^{k-1}(v)\bigr)\\
&=F\bigl((k-1)(n-k+2)F^{k-2}(v)\bigr)+(n-2k+2)F^{k-1}(v)
\vphantom{F\Bigl(E\bigl(F^{k-1}(v)\bigr)}\\
&=\bigl((k-1)(n-k+2)+(n-2k+2)\bigr)F^{k-1}(v)
\vphantom{F\Bigl(E\bigl(F^{k-1}(v)\bigr)}\\
&=k(n-k+1)F^{k-1}(v)
\vphantom{F\Bigl(E\bigl(F^{k-1}(v)\bigr)}.
\qedhere
\end{align*}
\end{proof}
Since $V$ has finite dimension, it follows by decomposition \eqref{rel_decomposition_eigenspaces} and formula \eqref{rel_eigenvectors_powersF} that there must exist a natural number $m\in\bbN$ such that
\[
F^m(v)=0.
\]
In particular, if we suppose that $m$ is the smallest power of $F$ that annihilates $v$, then we have by \autoref{lem_formula_inductive_powersF}
\[
0=E\bigl(F^m(v)\bigr)=m(n-m+1)F^{m-1}(v).
\]
Therefore, since $F^{m-1}\neq 0$, we must have $n=m-1$. In particular, we conclude that $n$ is a nonnegative integer and that the set of eigenvalues of $H$ on $V$ is given by a finite sequence of integers differing by $2$ and it is symmetric with respect to $0$ in $\bbZ$. In fact, we can rewrite the diagram \eqref{rel_diagram_eigenspaces} as
\begin{equation}
\label{rel_diagram_eigenspace_final}
\begin{tikzcd}
0\ar[r,yshift=2.2,"E"] &V_{-n} \ar[l,yshift=-2.2,"F"]\ar[loop below,"H"] \ar[r,yshift=2.2,"E"]&V_{-n+2} \ar[l,yshift=-2.2,"F"]\ar[loop below,"H"] \ar[r,yshift=2.2,"E"]  & \cdots\ar[l,yshift=-2.2,"F"]\ar[r,yshift=2.2,"E"] & V_{n-2} \ar[loop below,"H"] \ar[l,yshift=-2.2,"F"]\ar[r,yshift=2.2,"E"] & V_n \ar[loop below,"H"] \ar[r,yshift=2.2,"E"] \ar[l,yshift=-2.2,"F"] & 0 \ar[l,yshift=-2.2,"F"]
\end{tikzcd}.
\end{equation}
We have thus proved that irreducible representations are unique and depend only on the dimension of $V$. From the latter we can determine the eigenvalues of the action of the operator $H$, which play a special role.
\begin{defn}
Given a representation $V^{(n)}$ of $\mfsl_2\bbC$ of dimension $n+1$, the eigenvalues of the action of the operator $H$ on $V^{(n)}$, that is, the elements of the set
\[
\{-n,-n+2,\dots,n-2,n\},
\]
are called the \textit{weights} of $V^{(n)}$.
\end{defn}
Considering the diagram \eqref{rel_diagram_eigenspace_final} we conclude that any representation of $\mfsl_2\bbC$ with distinct eigenvalues of multiplicity $1$ and having the same parity must be irreducible. In particular, we can see that irreducible representations of $\mfsl_2\bbC$ correspond to the components of the ring of polynomials in two variables and the weights correspond to monomials.
\begin{teo}
For every $n\in\bbN$, there exists a unique irreducible representation $V^{(n)}$ of $\mfsl_2\bbC$ such that $\dim{V^{(n)}}=n+1$, given by the $n$-th symmetric power of $\bbC^2$, that is
\[
V^{(n)}\cong S^n\bbC^2,
\]
and for every $k=0,\dots,n$, the weight associated to the integer number $n-2k$ is given by the monomial $x^{n-k}y^{k}$, that is
\[
H(x^{n-k}y^{k})=(n-2k)x^{n-k}y^{k}.
\]
\end{teo} 
\begin{proof}
The standard representation of $\mfsl_2\bbC$ on $\bbC^2$ can be defined by the relations on the operator $H$
\[
H(x)=x,\qquad H(y)=-y.
\]
These can be extended to the $n$-th symmetric power $S^n\bbC^2$ by considering its basis
\[
\{x^n,x^{n-1}y,\dots,xy^{n-1},y^n\}.
\]
Indeed, we have
\[
H(x^{n-k}y^{k})=(n-k)H(x)x^{n-k-1}y^k+kH(y)x^{n-k}y^{k-1}=(n-2k)x^{n-k}y^{k}.
\]
Thus, by the above considerations, we conclude that $S^n\bbC^2$ is the unique irreducible representation of $\mfsl_2\bbC$ of dimension $n+1$.
\end{proof}
\subsection{Apolarity and sums of powers}
\label{section apolarity}
The theory of
apolarity is a classical topic, first introduced by T.~Reye in 1870 in \cite{Rey70}, where the term \textit{apolar} was used for the
first time, and by J.~Rosanes in \cite{Ros73} (see
also \cite{Dol12}*{p.~75}). It was further developed by W.~F.~Meyer in \cite{Mey83}.
There are many papers in the literature that focus on the theory of apolarity. In fact, it is a widely used argument in algebraic geometry.
As classical references, it is analyzed in detail in \cites{DK93,IK99}. For a modern point of view, the reader may instead consult \cite{Dol12}.
In this section, $V$ will always denote a finite-dimensional vector space of dimension $n\in\bbN$ over $\bbC$. 
We write
\[
\calR_n=S(V),\qquad \calD_n=S(V^*),
\]
to denote the symmetric algebra of $V$ and its dual space, respectively.
For every $d\in\bbN$, the natural bilinear map 
\begin{equation}
\label{rel contraction pairing}
\begin{tikzcd}[row sep=0pt,column sep=1pc]
 \circ\colon S^dV^*\times S^dV\arrow{r} & \bbC \\
  {\hphantom{\circ\colon{}}}  (\phi_1\cdots\phi_d, v_1\cdots v_d) \arrow[mapsto]{r} & \displaystyle{\sum_{\sigma\in\mfS_d}} \phi_1(v_{\sigma(1)})\cdots\phi_d(v_{\sigma(d)})
\end{tikzcd}
\end{equation}
is known as \textit{contraction pairing} or \textit{polar pairing}. Furthermore, for every $k\leq d$, the \textit{$(k,d)$-partial polarization map} is defined as
\begin{equation}
\label{rel polarization map}
\begin{tikzcd}[row sep=0pt,column sep=1pc]
 \circ\colon S^kV^*\times S^dV\arrow{r} & S^{d-k}V\hphantom{.} \\
  {\hphantom{\circ\colon{}}}  (\phi_1\cdots\phi_k, v_1\cdots v_d) \arrow[mapsto]{r} & \displaystyle{\sum_{1\leq i_1<\cdots<i_k\leq d}} (\phi_1\cdots\phi_k)\circ (v_{i_1}\cdots v_{i_k})\prod_{j\neq i_1,\dots,i_k}v_j.
\end{tikzcd}
\end{equation}
It is possible to identify $S(V)$ and $S(V^*)$ with polynomial rings. This is a classical and well-known fact, which can be seen, for example, in \cite{CGLM08}*{section 3.1}. We denote by $\bbC[V]$ the ring of polynomial functions on $V$, i.e., the commutative $\bbC$-algebra generated by $V^*$. 
\begin{prop}
\label{prop nat isom S^d}
For every $d\in\bbN$ there is a natural isomorphism 
\[
S^dV^*\cong\bbC[V]_d.
\]
\end{prop}
\begin{proof}
Consider the $\bbC$-linear map
\[
\begin{tikzcd}[row sep=0pt,column sep=1pc]
 \Phi\colon S^dV^*\arrow{r} & \bbC[V]_d\hphantom{,} \\
  {\hphantom{\Phi\colon{}}}  \phi \arrow[mapsto]{r} & f_{\phi},
\end{tikzcd}
\]
where the function $f_{\phi}$ is defined as
\[
\begin{tikzcd}[row sep=0pt,column sep=1pc]
 f_{\phi}\colon V\arrow{r} & \bbC\hphantom{.} \\
  {\hphantom{f_{\phi}\colon{}}}  v \arrow[mapsto]{r} & \phi(v,\dots,v).
\end{tikzcd}
\]
We should check that $\Phi$ is well defined, i.e., $f_{\phi}$ is actually a polynomial function. Given a basis $\{v_1^*,\dots,v_n^*\}$ of $V^*$, we write $\phi$ as a linear combinations of products of the elements $v_1^*,\dots,v_n^*$, that is,
\begin{align}
\label{rel linear comb 1}
\nonumber\phi(w_1,\dots,w_d)&=\sum_{i_1,\dots,i_d=1}^n\phi(v_{i_1},\dots,v_{i_d})v_{i_1}^*(w_1)\cdots v_{i_d}^*(w_d)\\
&=\sum_{\sigma\in\mfS_d}\sum_{1\leq i_1\leq \cdots \leq i_d\leq n}\phi(v_{i_1},\dots,v_{i_d})v_{i_1}^*(w_{\sigma(1)})\cdots v_{i_d}^*(w_{\sigma(d)}).
\end{align}
So we have 
\begin{equation}
\label{rel linear comb 2}
f_{\phi}(v)=\phi(v,\dots,v)=d!\sum_{1\leq i_1\leq \cdots \leq i_d\leq n}\phi(v_{i_1},\dots,v_{i_d})v_{i_1}^*(v)\cdots v_{i_d}^*(v)
\end{equation}
for every $v\in V$ and hence
$f_{\phi}$ is a polynomial function, namely, $f_{\phi}\in\bbC[V]$. It remains to verify that $\Phi$ is bijective. Consider an element $\phi\in S^dV^*$ and assume that $\Phi(\phi)\equiv 0$. This means that
\[
\Phi(\phi)(v)=f_{\phi}(v)=0
\]
for every $v\in V$. Then, looking again at formula \eqref{rel linear comb 2}, we have
\[
f_{\phi}=d!\sum_{1\leq i_1\leq\cdots\leq i_d\leq n}\phi(v_{i_1},\dots,v_{i_d})v_{i_1}^*\cdots v_{i_d}^*\equiv 0.
\]
Moreover, since the set of the polynomial functions of the type
\begin{equation}
\label{formula:monomials_polynomial functions}
v_{i_1}^*\cdots v_{i_d}^*\colon V\to\bbC
\end{equation}
among the indices $1\leq i_1\leq\cdots\leq i_d\leq n$ forms a set of linear independent elements, it follows that
\[
\phi(v_{i_1},\dots,v_{i_d})=0.
\]
Since the linear combination describing $\phi$ in formula \eqref{rel linear comb 1} has the same coefficients, we conclude that $\phi\equiv 0$, that is, $\Phi$ is injective. To prove surjectivity, we consider a polynomial function $f\in\bbC[V]_d$, written as a linear combination of the functions in formula \eqref{formula:monomials_polynomial functions}, obtained from the basis $\{v_1^*,\dots,v_n^*\}$ of $V^*$. That is,
\[
f=d!\sum_{1\leq i_1\leq\cdots\leq i_d\leq n}a_{i_1,\dots, i_d}v_{i_1}^*\cdots v_{i_d}^*
\]
for some suitable coefficients $a_{i_1,\dots, i_d}\in\bbC$. If we consider the symmetric multilinear form $\psi\in S^dV^*$, defined as
\[
\psi(w_1,\dots,w_d)=\sum_{\sigma\in\mfS_d}\sum_{1\leq i_1\leq\cdots\leq i_d\leq n}a_{i_1,\dots, i_d}v_{i_1}^*(w_{\sigma(1)})\cdots v_{i_d}^*(w_{\sigma(d)}),
\]
then it is immediate to verify that $f=f_{\psi}=\Phi(\psi)$, i.e., $\Phi$ is surjective.
\end{proof}
Consider a basis $\{v_1,\dots,v_n\}$ of $V$ and its dual basis $\{v_1^*,\dots,v_n^*\}$ in $V^*$. Then, given any 
\[
v=y_1v_1+\dots+y_nv_n\in V,
\]
as each functional $v_i^*$ associates its $i$-th coordinate $y_i$ to $v$,
we can directly identify $\{v_1^*,\dots,v_n^*\}$ as the sets of coordinates $\{y_1,\dots,y_n\}$. Therefore, by \autoref{prop nat isom S^d}, we have
\[
\calR_n\simeq\bbC[x_1,\dots,x_n],\qquad \calD_n\simeq\bbC[y_1,\dots,y_n],
\]
namely, $\calR_n$ and $\calD_n$ are identified with two polynomial rings.

We can compute the image of each pair of monomials by the polarization map. That is, using the notation
\[
\bfx^{\bfdelta}=x_1^{\delta_1}\dots x_n^{\delta_n},\quad \bfy^{\bfdelta}=y_1^{\delta_1}\dots y_n^{\delta_n},\quad \bfdelta!=\delta_1!\cdots\delta_n!
\]
for every multi-index $\bfdelta=(\delta_1,\dots,\delta_n)\in\bbN^n$, we have
\[
\bfy^{\bfalpha}\circ\bfx^{\bfbeta}=
\begin{cases}
\dfrac{\bfbeta!}{(\bfbeta-\bfalpha)!}\bfx^{\bfbeta-\bfalpha},\quad &\text{if $\bfbeta-\bfalpha\geq 0$},\\[2ex]
0,\quad &\text{otherwise},
\end{cases}
\]
for every $\bfalpha,\bfbeta\in\bbN^n$. This means that, denoting by $\calR_{n,d}$ and $\calD_{n,d}$ the components of degree $d$ of $\calR_n$ and $\calD_n$, for every $d\in\bbN$, we can naturally identify the space $\calD_n$ with the space of polynomial differential operators. In particular, for any $k\leq d$ and any form $\phi\in \calD_{n,k}$, the \textit{differential operator} associated to $\phi$ is the linear map
\[
\begin{tikzcd}[row sep=0pt,column sep=1pc]
 \Der_{\phi}\colon \calR_{n,d}\arrow{r} & \calR_{n,{d-k}}\hphantom{.} \\
  {\hphantom{\Der_{\phi}\colon{}}} h \arrow[mapsto]{r} & \phi\circ h.
\end{tikzcd}
\]
\begin{defn}
The \textit{apolarity action} of $\calD_n$ on $\calR_n$ is defined by extending the polarization maps defined on the components of $\calD_n$ and $\calR_n$ by linearity, and it is given by
\[
\begin{tikzcd}[row sep=0pt,column sep=1pc]
 \circ\colon \calD_n\times\calR_n\arrow{r} & \calR_n\hphantom{.} \\
  {\hphantom{\circ\colon{}}}  (\phi,f) \arrow[mapsto]{r} & \Der_{\phi}(f).
\end{tikzcd}
\]
\end{defn}
\begin{rem}
\label{remark reversing roles}
When considering the apolarity action, we can reverse the roles of the variables $x_i$ and $y_i$, for every $i=1,\dots,n$, establishing an identification between the elements of $\calR_{n}$ and $\calD_{n}$. In particular, we can identify the contraction pairing of formula \eqref{rel contraction pairing} as a symmetric bilinear form.
\end{rem}
Now, we define two objects, which are well-known and essential for the use of the theory of apolarity.
\begin{defn}
Let $f\in\calR_{n,d}$. The \textit{catalecticant map} of $f$ is the linear map
\[
\begin{tikzcd}[row sep=0pt,column sep=1pc]
 \Cat_f\colon \calD_n\arrow{r} & \calR_n\hphantom{.} \\
  {\hphantom{\Cat_f\colon{}}} \phi \arrow[mapsto]{r} & \phi\circ f.
\end{tikzcd}
\]
For every $j\in\bbN$ the restriction of the catalecticant map to the component $\calD_{n,j}$, given by the linear map,
\[
\begin{tikzcd}[row sep=0pt,column sep=1pc]
 \Cat_{f,j}\colon \calD_{n,j}\arrow{r} & \calR_{n,{d-j}} \\
  {\hphantom{\Cat_{f,j}\colon{}}} g \arrow[mapsto]{r} & g\circ f,
\end{tikzcd}
\]
is called the \textit{$j$-th catalecticant} of $f$.
The kernel of the catalecticant map, which is the set
\[
\Ann(f)=\Set{g\in\calD|g\circ f=0},
\]
is called the \textit{apolar ideal} of $f$.
\end{defn}
Choosing any bases of $\calD_{n,j}$ and $\calR_{n,{d-j}}$, the matrix associated to $\Cat_{f,j}$ is known as the \textit{$j$-th catalecticant matrix} of $f$.

Now, by the identification we saw in \autoref{prop nat isom S^d}, we can naturally define the action of the linear group $\GL_n(\bbC)$ on both $\calR_n$ and $\calD_n$, in the same way as in \autoref{rem_representations_groups}.
\begin{defn}
\label{def_action_GLn_on_R_D}
The action of $\GL_n(\bbC)$ on $\calR_n$ is defined by naturally extending the action of $\GL_n(\bbC)$ on $\calR_{n,1}$, given by
\[
A\cdot x_j=\sum_{k=1}^nA_{kj}x_k,
\]
for every $A\in\GL_n(\bbC)$ and for every $j=1,\dots,n$. The action of $\GL_n(\bbC)$ on $\calD_n$ is defined in the same way by the dual action on $\calD_{n,1}$, given by
\[
A\cdot y_j=\sum_{k=1}^n\transpose A_{kj}^{-1}y_k,
\]
for every $A\in\GL_n(\bbC)$ and for every $j=1,\dots,n$.
\end{defn}

When dealing with the apolarity action it can be useful to consider a specific notation for a basis of monomials.

\begin{defn}
\label{def dpm}
For every $k\in\bbN$, the \textit{divided power monomials} of $\calR_{n,k}$ and $\calD_{n,k}$ are the monomials defined as 
\[
\bfx^{\bfdelta}=x_1^{[\delta_1]}\dots x_n^{[\delta_n]}=\frac{1}{\bfdelta!}x_1^{\delta_1}\dots x_n^{\delta_n},\qquad \bfy^{\pq{\bfdelta}}=y_1^{[\delta_1]}\dots y_n^{[\delta_n]}=\frac{1}{\bfdelta!}y_1^{\delta_1}\dots y_n^{\delta_n}, 
\]
where
\[
\bfdelta!=\delta_1!\cdots\delta_n!
\]
for any multi-index $\bfdelta\in\bbN^n$ such that
\[
\abs{\bfdelta}=\delta_1+\cdots+\delta_n=k.
\]
\end{defn}
The use of divided power monomials can simplify the use of the apolarity action in dealing with the coefficients produced by derivations.

The next property follows easily from the definition of the apolarity action and the representations on dual spaces.
\begin{prop}
\label{prop_apolarity_action_equivariant}
The apolarity action of $\calD_n$ on $\calR_n$ is a $\GL_n(\bbC)$-equivariant map.
\end{prop}
\begin{proof}
By the action of $\GL_n(\bbC)$ on $\calR_n$ and $\calD_n$ given in \autoref{def_action_GLn_on_R_D} and \autoref{rem_representations_groups}, we can naturally define the action of $\GL_n(\bbC)$ on $\calD_n\times\calR_n$ as
\[
A\cdot(\phi,f)=(A\cdot\phi,A\cdot f)
\]
for every $A\in\GL_n(\bbC)$, $\phi\in\calD_n$ and $f\in\calR_n$.
First we analyze the restriction to the action of $\calD_{n,1}$ on $\calR_{n,1}$. So, let $A\in\GL_n(\bbC)$ and $j,k\in\bbN$ be such that $1\leq j,k\leq n$. Then we have
\begin{align*}
(A\cdot y_j)\circ (A\cdot x_k)&=\biggl(\sum_{s=1}^n\transpose A_{sj}^{-1}y_s\biggr)\circ\biggl(\sum_{t=1}^nA_{tk}x_t\biggr)\\
&=\sum_{s=1}^n \transpose A_{sj}^{-1}\biggl(\sum_{t=1}^nA_{tk}\pdv{x_t}{x_s}\biggr)\\
&=\sum_{t=1}^nA_{jt}^{-1}A_{tk}=\delta_{jk},
\end{align*}
where
\[
\delta_{jk}=\begin{cases}
1,\quad &\text{if $j=k$},\\
0,\quad &\text{otherwise}.
\end{cases}
\]
That is,
\[
(A\cdot y_j)\circ (A\cdot x_k)=y_j\circ x_k.
\]
In particular, this means that the set $\{A\cdot y_1,\dots,A\cdot y_n\}$ corresponds to the dual basis of the basis $\{A\cdot x_1,\dots,A\cdot x_n\}$. Since the action of $A$ is a linear transformation preserving the dual pairing, we have by definition that
\[
(A\cdot \phi)\circ(A\cdot f)=A\cdot(\phi\circ f)
\]
for every $\phi\in\calD_n$ and $f\in\calR_n$.
\end{proof}
Linear forms can be viewed in general as points in ${\bbC^n}^*$.
For every point $\bfa=(a_1,\dots,a_n)\in\bbC^n$, we denote by $l_{\bfa}\in\calR_{n,1}$ the linear form
\[
l_{\bfa}=\bfa\cdot\bfx=a_1x_1+\dots+a_nx_n
\]
and we call it the \textit{linear form associated to $\bfa$}. Sometimes, when there is no risk of confusion, we will simply refer to it as the term \textit{point}.
\begin{defn}
Let $\bfa\in\bbC^n$ and let $l_{\bfa}$ be its associated linear form. Then for every $d\in\bbN$ the form $l_{\bfa}^{[d]}\in\calR_{n,d}$, defined as
\[
l_{\bfa}^{[d]}=\sum_{\alpha_1+\dots+\alpha_n=d}a_1^{\alpha_1}\cdots a_n^{\alpha_n}x_1^{[\alpha_1]}\dots x_n^{[\alpha_n]},
\]
is called the \textit{$d$-th divided power} of $l_{\bfa}$.
\end{defn}
The main properties of divided powers of linear forms are enumerated in the following proposition. Since the proof is quite simple and consists only of some technical calculations, we will omit it and we refer to \cite{IK99} for verification.
\begin{prop}[\cite{IK99}*{Proposition A.9}]
Let be $l_{\bfa}=a_1x_1+\dots+a_nx_n\in\calD_{n,1}$. Then for every $d,k\in\bbN$
\begin{enumerate}[label=(\arabic*), left= 0pt, widest=*,nosep,itemsep=1.3pt]
\item $l_{\bfa}^d=d!l_{\bfa}^{[d]}$;
\item $l_{\bfa}^{[d]}(\phi)=\phi(a_1,\dots,a_n)$ for every $\phi\in \calR_{n,d}$;
\item $A\cdot l_{\bfa}^{[d]}=(A\cdot l_{\bfa})^{[d]}$ for every $A\in \GL_n(\bbC)$;
\item $d!k!l_{\bfa}^{[d]}l_{\bfa}^{[k]}=(d+k)!l_{\bfa}^{[d+k]}$.
	\end{enumerate}
\end{prop}
Another property to consider is the fact that the image elements through the catalecticant map of the powers of a point are still powers of the same point.
\begin{lem}
\label{lem contraz forme}
Let $d,k\in\bbN$ such that $d\geq k$. Then, for every $\bfa\in\bbC^n$ and $\phi\in \calD_{n,k}$,
\[
\phi\circ l_{\bfa}^{[d]}=\phi(\bfa)l_{\bfa}^{[d-k]}.
\]
\end{lem}
\begin{proof}
By linearity it is sufficient to prove the formula only for monomials in $\calD_{n,k}$. So, given a multi-index $\bfalpha=(\alpha_1,\dots,\alpha_n)\in\bbN^n$ such that $\abs{\bfalpha}=k$, we have
\begin{align*}
y_1^{\alpha_1}\cdots y_n^{\alpha_n}\circ l_{\bfa}^{[d]}&=y_1^{\alpha_1}\cdots y_n^{\alpha_n}\circ\biggl(\sum_{\abs{\bfbeta}=d}a_1^{\beta_1}\dots a_n^{\beta_n}x_1^{[\beta_1]}\cdots x_n^{[\beta_n]}\biggr)\\
&=\sum_{\abs{\bfbeta}=d}a_1^{\beta_1}\cdots a_n^{\beta_n}x_1^{[\beta_1-\alpha_1]}\cdots x_n^{[\beta_n-\alpha_n]}
\vphantom{\biggl(\sum_{\abs{\bfbeta}=d}x_1^{[\beta_1]}\cdots x_n^{[\beta_n]}\biggr)}\\
&=a_1^{\alpha_1}\cdots a_n^{\alpha_n}\biggl(\sum_{\abs{\bfgamma}=d-k}a_1^{\gamma_1}\cdots a_n^{\gamma_n}x_1^{[\gamma_1]}\cdots x_n^{[\gamma_n]}\biggr)
\vphantom{\biggl(\sum_{\abs{\bfbeta}=d}x_1^{[\beta_1]}\cdots x_n^{[\beta_n]}\biggr)}\\
&=a_1^{\alpha_1}\cdots a_n^{\alpha_n}l_{\bfa}^{[d-k]}
\vphantom{\biggl(\sum_{\abs{\bfbeta}=d}x_1^{[\beta_1]}\cdots x_n^{[\beta_n]}\biggr)}.
\qedhere
\end{align*}
\end{proof}
\begin{cor}
For every $d,k\in\bbN$ and for every $\bfa\in\bbC^n$
\[
\rk\bigl(\Cat_{l_{\bfa}^d,k}\bigr)\leq 1.
\]
\end{cor}

Finally, we can show a classical result, crucial for our purposes, which relates homogeneous polynomials to powers of linear forms.
\begin{lem}[Apolarity lemma \cite{IK99}*{Lemma 1.15}]
\label{Lem Apo}
Let $\bfa_1,\dots,\bfa_r\in\bbC^n$, let $l_k=l_{\bfa_k}$ for every $k=1,\dots,r$, let 
\[
\calA=\{[\bfa_1],\dots,[\bfa_r]\}\subset\bbP(\bbC^n)
\]
and let $I_{\calA}$ be the homogeneous ideal in $\calD_n$ of polynomials vanishing on $\calA$. Then:
\begin{enumerate}[label=(\arabic*), left= 0pt, widest=*,nosep]
\item for every $d,k\in\bbN$ such that $d\geq k$, if $\phi\in \calD_{n,k}$, then
\[
\phi\circ\bigl(l_1^{[d]}+\cdots+l_r^{[d]}\bigr)=\phi(\bfa_1)l_1^{[d-k]}+\cdots+\phi(\bfa_r)l_r^{[d-k]};
\]
\item given any $k\in\bbN$, then
\[
(I_{\calA})_k^\perp=\bigl\langle l_1^{[k]},\dots,l_r^{[k]}\bigr\rangle,
\]
where $(I_{\calA})_k^\perp$ is the space orthogonal to $(I_{\calA})_k$ with respect to the contraction pairing 
\[
\circ\colon \calD_{n,k}\times\calR_{n,k}\to\bbC.
\]
\end{enumerate} 
\end{lem}
\begin{proof}
Point (1) follows from \autoref{lem contraz forme}.
Now, given any $\phi\in \calD_{n,k}$, we have still by \autoref{lem contraz forme} that for every $\lambda_1,\dots,\lambda_r\in\bbC$
\[
\phi\circ\biggl(\sum_{j=1}^r\lambda_jl_j^{[k]}\biggr)=\sum_{j=1}^r\lambda_j\phi(\bfa_j).
\]
This means that 
\[
\bigl\langle l_1^{[k]},\dots,l_r^{[k]}\bigr\rangle^\perp=\Set{\phi\in \calD_{n,k}|\text{$\phi(\bfa_j)=0$, $\forall j=1,\dots,r$}}=(I_{\calA})_k.
\]
Thus, point (2) follows directly from the fact that the contraction pairing is non-degenerate.
\end{proof}
The apolarity theory is useful in dealing with tensor decomposition. 

\begin{defn}
\label{def_rank_polynomial}
Let $f\in\calR_{n,d}$ be a homogeneous polynomial. For every $r\in\bbN$, 
any expression of the form
\[
f=\sum_{j=1}^rl_{\bfa_j}^d,
\]
where $l_{\bfa_1},\dots,l_{\bfa_r}\in\calR_{n,1}$, is said to be a \textit{ Waring decomposition}, or simply \textit{decomposition}, of $f$ of \textit{size} $r$.
The elements $\bfa_1,\dots,\bfa_r$ are called the \textit{points} of the decomposition.
\end{defn}
For instance (see also \cite{BBT13}*{section 2}), a formula for a minimal Waring decomposition of the quadratic form $x_1x_2\in\bbC[x_1,x_2]$ is
\[
x_1x_2=\frac{1}{4}(x_1+x_2)^2-\frac{1}{4}(x_1-x_2)^2.
\]
Similarly, a minimal decomposition of the form $x_1x_2x_3\in\bbC[x_1,x_2,x_3]$ is as follows:
\[
x_1x_2x_3=\frac{1}{24}(x_1+x_2+x_3)^3-\frac{1}{24}(x_1+x_2-x_3)^3-\frac{1}{24}(x_1-x_2+x_3)^3+\frac{1}{24}(x_1-x_2-x_3)^3.
\]
In dealing with this subject, the main problem consists of finding which is the minimum number $r\in\bbN$ such that, given a specific $f\in\calR_{n,d}$, it is possible to represent a decomposition of $f$ of size $r$.

\begin{defn}
\label{def_border_rank_polynomial}
Let $f\in\calR_{n,d}$ be a homogeneous polynomial. The \textit{Waring rank}, or \textit{symmetric tensor rank}, or simply \textit{rank}, of $f$ is the natural number
\[
\rk f=\min{\Set{r\in\bbN|\text{$f=\Sum_{j=1}^rl_{\bfa_j}^d$ : $\bfa_{j}\in \bbC^n$}}}.
\]
\end{defn}

The analogue of the Waring problem for homogeneous polynomials is known as the \textit{Big Waring problem} and aims to determine the minimum number $r$ such that a generic form $f$ of degree $d$ admits a decomposition of size $r$. By the term \textit{generic}, we mean any form belonging to a Zariski open subset in $\calR_{n,d}$.

The statement of \autoref{Lem Apo} provides an efficient method to determine a representation of a polynomial as sum of powers of linear forms. This way of writing polynomials have a special role in both classical and recent mathematics.

Now, using \autoref{lem contraz forme}, we can easily provide a well-known lower bound. This inequality is classical for binary forms and it first appeared in \cite{Syl51}. It has been then generalized for an arbitrary number of variables in \cite{IK99}*{page 11} (see also \cite{Lan12}*{Proposition 3.5.1.1}).
\begin{prop}
\label{prop lower bound}
Let $d,k\in\bbN$ be such that $d\geq k$. Then
\[
\rk f\geq\rk(\Cat_{f,k})
\]
for every $f\in\calR_{n,d}$.
\end{prop}
\begin{proof}
Let us consider the decomposition of size $r$
\[
[f]_r=l_{\bfa_1}^d+\cdots+l_{\bfa_r}^d,
\]
with $r\in\bbN$, where $\bfa_1,\dots,\bfa_r\in\bbC^n$. Then, it follows by \autoref{lem contraz forme} that for every $j=1,\dots,d$, if $\phi\in\calD_{n,k}$, we have
\begin{align*}
\Cat_{f,k}(\phi)=\phi\circ f&=\Der_{\phi}\bigl(l_{\bfa_1}^{d}+\cdots+l_{\bfa_r}^{d}\bigr)\vphantom{\frac{d!}{(d-k)!}}\\
&=\Der_{\phi}\bigl(l_{\bfa_1}^{d}\bigr)+\cdots+\Der_{\phi}\bigl(l_{\bfa_r}^{d}\bigr)\vphantom{\frac{d!}{(d-k)!}}\\
&=\frac{d!}{(d-k)!}\bigl(\phi(\bfa_1)l_1^{d-k}+\cdots+\phi(\bfa_r)l_r^{d-k}\bigr),
\end{align*}
that is, 
\[
\im (\Cat_{f,k})\subseteq\langle l_1^{d-k},\dots,l_r^{d-k}\rangle
\]
and hence
\[
\rk(\Cat_{f,k})\leq r.\qedhere
\]
\end{proof}
We recall that, for a $0$-dimensional subscheme $Z\subseteq\bbP^{n}$, its degree is defined as the maximum value assumed by the Hilbert function of the corresponding ideal $I_Z$, given by
\[
\begin{tikzcd}[row sep=0pt,column sep=1pc]
 \HF_I\colon \bbN\arrow{r} & \bbN\hphantom{.} \\
  {\hphantom{\HF_{I}\colon{}}} k \arrow[mapsto]{r} & \dim \bigl(\calD_{n,k}\big/(I_Z)_k\bigr).
\end{tikzcd}
\]
The following result, appearing in \cite{IK99}, owns a special role in dealing with Waring decompositions.
\begin{teo}[\cite{IK99}*{Theorem 1.69}]
Let $Z$ be a zero-dimensional subscheme of $\bbP^n$ with $\deg Z=r$, let $I_Z$ its associated ideal, and let
\[
\tau(I_Z)=\min\Set{k\in\bbN|\HF_{I_Z}(k)=r}.
\]
Then $\HF_{I_Z}$ is nondecreasing and stabilizes at the value $\HF_{I_Z}(k)=r$ for every $k\geq\tau(I_Z)$.
\end{teo}
Considering, for every $r\in\bbN$, the set 
\[
\calS_{n,r}^d=\Set{g\in\calR_{n,d}|\rk g\leq r},
\]
the \textit{border rank} of a polynomial $f\in\calR_{n,d}$ is the minimum  number $r\in\bbN$ such that $f\in\bar{\calS_{n,r}^d}$,
where the overline represents the closure for the Zariski topology. 
In terms of closure, we recall that, essentially by a result provided by D.~Mumford in \cite{Mum95}*{Theorem 2.33}, the closure of the set $\calS_{n,r}^d$ in Zariski topology equals the closure of the same set in Euclidean topology. Thus, we can think to the border rank of a polynomial $f\in\calR_{n,d}$ as the minimum number $r\in\bbN$ such that
\[
f=\lim_{t\to 0}\sum_{j=1}^rl_j^d(t),
\]
where $\{l_j(t)\}_{t\in\bbR}$ is a family of linear forms for every $j=1,\dots,r$.

\subsection{Quadratic forms and powers of quadrics}
\label{sec_quadratic_forms}
The notion of quadratic form is classical and appears in almost all branches of mathematics. Now we give a short overview for completeness and to justify the choice of the polynomial $x_1^2+\cdots+x_n^2$. All of the following in this section is taken from \cite{Ser73}*{chapter IV}.
\begin{defn}
Let $V$ be a vector space over $\bbC$. A function $q\colon V\to\bbC$ is called a \textit{quadratic form} on $V$ if:
\begin{enumerate}[label=(\arabic*), left= 0pt, widest=2,nosep]
\item $q(av)=a^2q(v)$ for every $a\in\bbC$ and $v\in V$;
\item the function
\[
\begin{tikzcd}[row sep=0pt,column sep=1pc]
 Q\colon V\times V\arrow{r} & \bbC \\
  {\hphantom{Q\colon{}}} (v,w) \arrow[mapsto]{r} & \dfrac{1}{2}\bigl(q(v+w)-q(v)-q(w)\bigr)
\end{tikzcd}
\]
is bilinear.
\end{enumerate}
The pair $(V,q)$ is called a \textit{quadratic module} and the bilinear symmetric form $Q$ is called the \textit{scalar product} associated to $q$.
\end{defn}
Assuming that $\dim V=m$, if $\{e_1,\dots,e_m\}$ is a basis of $V$, then we can write any vector $v\in V$ as
\[
v=\sum_{i=1}^mx_ie_i
\] 
for some $x_1,\dots,x_m\in\bbC$. So if we set $a_{ij}=Q(e_i, e_j)$, we get
\[
q(v)=\sum_{1\leq i,j\leq m}a_{ij}x_ix_j,
\]
where 
\[
A=\begin{pNiceMatrix}
a_{11}&\Cdots &a_{1m}\\
\Vdots &\Ddots &\Vdots\\
a_{m1}&\Cdots &a_{mm}
\end{pNiceMatrix}
\]
is the \textit{matrix associated} to $q$ with respect to the basis $\{x_1,\dots,x_m\}$. In particular, if we replace $v$ by the vector of coordinates $\bfx=(x_1,\dots,x_m)$, then we get the quadratic form view in terms of polynomials in $\bbC[x_1,\dots,x_m]$, that is,
\[
q(\bfx)=q(x_1,\dots,x_m)=\sum_{1\leq i,j\leq m}a_{ij}x_ix_j.
\]

The concept of orthogonality is fundamental and well known in the context of quadratic forms. We recall the main notions in the following definition.
\begin{defn}
Two vectors $v$ and $w$ are said to be \textit{orthogonal} if $Q(v,w)=0$. In particular, if a vector $v$ is orthogonal to itself, that is,
\[
q(v)=Q(v,v)=0,
\]
then it is said to be \textit{isotropic}.
A basis of $V$ is said to be \textit{orthogonal} if its elements are pairwise orthogonal. For any subset $W\subset V$, the set
\[
W^\perp=\Set{v\in V|Q(v,w)=0, \forall w\in W}
\]
is called the \textit{orthogonal space} to $W$. If $W$ is a vector subspace, then $W^\perp$ is called the \textit{orthogonal complement} of $W$. The orthogonal complement $V^\perp$ of $V$ is called the \textit{radical}, or \textit{kernel}, of $V$ and $\codim V^\perp$ is called the \textit{rank} of $q$. In particular, if $V^\perp=\{0\}$, then $q$ is said to be \textit{nondegenerate}.
\end{defn}
A well-known fact is the existence of an orthogonal basis, which is proved in the following proposition.
\begin{prop}
\label{prop:orthogonal_basis}
Every quadratic module $(V,q)$ has an orthogonal basis.
\end{prop}
\begin{proof}
We use induction on $n=\dim V$. The case where $n=0$ is clear. Now,
if $V=V^{\perp}$, then all vectors of $V$ are isotropic and every basis of $V$ is orthogonal. Otherwise, we can choose an element $v_1\in V$ such that $q(v_1)\neq 0$. In particular, if we take the hyperplane $H=v_1^{\perp}$, then we have $v_1\not\in H$ and hence
\[
V=\langle v_1\rangle\oplus H.
\]
Then, by inductive hypothesis, there exists an orthogonal basis $\{v_2,\dots,v_n\}$ of $H$.
\end{proof}
Now we introduce the concept of morphisms between quadratic modules.
\begin{defn}
\label{def:morphism_quadratic_forms}
For arbitrary quadratic modules $(V,q)$ and $(V',q')$, a linear map $f\colon V\to V'$ is a \textit{morphism of quadratic modules} of $(V,q)$ into $(V',q')$ if 
\[
q=q'\circ f.
\]
In particular, the equality
\[
Q(v,w)=Q'\bigl(f(v),f(w)\bigr)
\]
holds for every $v,w\in V$. If $f$ is an isomorphism, then the quadratic modules $(V,q)$ and $(V',q')$ are said to be \textit{isomorphic}.
\end{defn}
By \autoref{def:morphism_quadratic_forms}, we can also introduce the concept of equivalence between quadratic modules.
\begin{defn}
Two quadratic forms $q,q'\colon V\to\bbC$ are said to be \textit{equivalent} if the corresponding quadratic modules $(V,q)$ and $(V,q')$ are isomorphic.
\end{defn}
If $q$ and $q'$ are equivalent, we write $q\sim q'$. In particular, two quadratic forms are equivalent if and only if the polynomial form is the same, up to a linear change of variables.
It is not difficult to prove the following proposition about the radicals of two quadratic modules.
\begin{prop}
\label{prop:isom_radicals}
If $(V,q)$ and $(V',q')$ are two isomorphic quadratic modules, then $V^\perp\simeq V'^\perp$.
\end{prop}
\begin{proof}
If $f\colon V\to V'$ is an isomorphism of quadratic modules, then for $v\in V^{\perp}$ we have
\[
Q'\bigl(f(v),w'\bigr)=Q\bigl(v,f^{-1}(w')\bigr)=0
\]
for every $w'\in V'$, i.e., $f(V)\in V'^\perp$. By reversing the roles of $V$ and $V'$, and using $f^{-1}$ instead of $f$, we prove the inverse inclusion, which gives the statement.
\end{proof}
In particular, \autoref{prop:isom_radicals} implies that the radicals of the quadratic modules in the same equivalence class are all isomorphic. That is, the radical of an equivalence class of quadratic modules is well defined. By \autoref{prop:orthogonal_basis} we have the following result.
\begin{teo}
\label{teo:equivalence_quadratic forms}
Let $q$ be a quadratic form in $m$ variables. Then
\[
q\sim a_1x_1^2+\cdots+a_mx_m^2
\]
for some $a_1,\dots,a_m\in\bbC$.
\end{teo}
Furthermore, we can observe that the rank of a quadratic form is exactly the number of non-zero coefficients in the formula of \autoref{teo:equivalence_quadratic forms}. 
\begin{prop}
\label{prop:canonic_form_quadric}
Let $q$ any quadratic forms in $m$ variables. Then 
\[
q\sim x_1^2+\cdots+x_n^2
\]
if and only if $q$ has rank $n$.
\end{prop}
\begin{proof}
Let $(V,q)$ be a quadratic module, such that $\dim V=m$. We can write
\[
V=V^\perp\oplus U,
\]
where $\dim U=n$ and $\dim V^\perp=m-n$. Now, if we consider the quadratic module $(U,q|_U)$, then we have by \autoref{prop:isom_radicals} and \autoref{teo:equivalence_quadratic forms} together that 
\[
q|_U\sim a_1x_1^2+\cdots+a_nx_n^2,
\]
where the second quadratic form is associated to an orthogonal basis $\{u_1,\dots,u_n\}$, and $a_1,\dots,a_n\neq 0$, since $U$ is non-degenerate. Up to the scaling of the elements $u_1,\dots,u_n$, we can then write
\[
q|_U\sim x_1^2+\cdots+x_n^2.
\]
Finally, let us consider a basis $\{v_{n+1},\dots, v_m\}$ of $V^{\perp}$, so that
$\{u_1,\dots,u_n,v_{n+1},\dots,v_m\}$ is a basis of $V$. In particular, since $v_i\in V^\perp$ for every $i=n+1,\dots,m$, we have that
\[
q\sim x_1^2+\cdots+x_n^2,
\] 
which proves the statement.
\end{proof}
Now we prove that the definition of rank given for quadratic forms is consistent.
\begin{cor}
\label{cor:equivalence_quadratic_forms}
A quadratic form has rank $n$ if and only if the Waring rank of $q$ is $n$.
\end{cor}
\begin{proof}
If the rank of $q$ is $n$, then by \autoref{prop:canonic_form_quadric} a change of variables implies that $\rk q\leq n$. Now, if $\rk q<n$, then we can write
\[
q=\sum_{j=1}^{n-1}\ell_j^2,
\]
and hence, through a change of variables,
\[
q\sim a_1x_1^2+\cdots+a_{n-1}x_{n-1}^2,
\]
which by \autoref{prop:canonic_form_quadric} leads to a contradiction.
\end{proof}
In particular, \autoref{cor:equivalence_quadratic_forms} implies that every quadratic form having Waring rank equal to $n$ is equivalent to the form
\[
q_n=x_1^2+\cdots+x_n^2.
\]
Therefore, to analyze decompositions of powers of quadratic forms of rank $n$, it is sufficient to consider the forms $q_n^s$ with $s\in\bbN$. This is also the form that has been analyzed many times in the literature. Several authors gave several decompositions for different exponents and numbers of variables (see e.g.~\cite{Dic19}*{Chapter XXV, pp.~717-724}). A more recent work is due to B.~Reznick, who in \cite{Rez92}*{chapters 8-9} analyzes many cases in detail, restricting himself anyway to the case of real decompositions, also called \textit{representations} in this context. We will see many examples of these classical expressions in \cref{cha_tight_decompositions,cha_general_decompositions_more_variables}.

\section{Apolarity and powers of quadrics}
\label{cha_apolarity_quadratic_forms}
In this section, we start our analysis on the form $q_n^s$. As said above, our goal is to determine suitable decompositions and establish which are the minimal ones. 
We have seen in \cref{section apolarity} how important the use of \autoref{Lem Apo} can be in determining decompositions. For this reason, our first result concerns the description of the apolar ideal of $q_n^s$ for each $s\in\bbN$. This topic can also be found in \cite{Fla23a}*{sections 2-3}, in a shorter version. Here, we provide a more detailed explanation.

We start in \cref{sec_harmonic_polynomials} with an analysis of the space of harmonic polynomials, which is an irreducible representation of the special orthogonal group $\SO_n(\bbC)$. This last fact is a rather classical result (see \cite{GW98}*{Theorem 5.2.4}). Several results  related to quadratic forms concern harmonic polynomials. One of the most interesting is that any form can be uniquely decomposed as
a sum of harmonic polynomials multiplied by powers of $q_n$ (see \cite{GW98}*{Corollary 5.2.5}). 

Then we proceed in \cref{sec_apolar_ideal_harmonic} with the analysis of the apolarity action on $q_n^s$. After a detailed study of the catalecticant matrices of $q_n^s$, which are of full rank (a fact already known from \cite{Rez92}), we proceed to the analysis of the apolar ideal of $q_n^s$. Its determination is in fact an important result in terms of apolarity, since by \autoref{Lem Apo} it provides a way to attack the problem of decomposing the form $q_n^s$. The resulting structure is quite elegant: the apolar ideal $\Ann(q_n^s)$ is exactly the ideal generated by all the harmonic polynomials of degree $s+1$ (see \autoref{Teo Apolar ideal}).

\subsection{Harmonic polynomials}
\label{sec_harmonic_polynomials}
In order to analyze the apolar ideal of the form $q_n^s$, we need to make some considerations about the spaces of harmonic polynomials. These objects appear many times in other branches of mathematics as well. Obviously, their presence can be found in analysis, especially in the theory of harmonic functions, for which we refer to \cite{ABR01}*{Chapter 5} for further details from an analytic point of view.
It may be useful to identify the space $\calD_n$ with the space of polynomial differential operators, as we have already seen in \cref{section apolarity}. 
First, we observe that the kernel of a differential operator associated to a homogeneous polynomial $\phi\in\calD_{n,k}$, for any $k\in\bbN$, is related to the contraction pairing. In fact, it can be considered as the orthogonal complement of the space $\phi\calD_{n,{d-k}}$. The proof of the following proposition is quite simple and we omit it. It is already present in \cite{Fla23a}*{Proposition 2.1}, to which we refer to see a way to prove it.
\begin{prop}
\label{prop ker differential}
Let $k\leq d$ and let $\phi\in \calD_{n,k}$. Then
\[
\Ker(\Der_{\phi})=(\phi\calD_{n,{d-k}})^\perp\subseteq\calR_{n,d}.
\]
\end{prop}
\autoref{prop ker differential} allows us to introduce the space of harmonic polynomials by the contraction pairing. If there is no risk of confusion, we continue to exchange the roles of the variables $x_i$ and $y_i$ by passing from a space to its dual without specifying.
\begin{defn}
The vector space
\[
\calH_{n,d}=(q_n\calD_{n,{d-2}})^\perp=\Ker(\Der_{q_n})\subseteq \calR_{n,d}
\]
is the space of the \textit{harmonic polynomials} of degree $d$.
The differential operator $\Der_{q_n}$, also denoted by $\Lap$, is  the \textit{Laplace operator} and it is defined by the linear map 
\[
\begin{tikzcd}[row sep=0pt,column sep=1pc]
 \Lap\colon \calR_{n,d}\arrow{r} & \calR_{n,{d-2}}\hphantom{.} \\
  {\hphantom{\Lap\colon{}}} f \arrow[mapsto]{r} & \displaystyle\sum_{i=1}^n\pdv[2]{f}{x_i}.
\end{tikzcd}
\]
\end{defn}
Given any two forms $g_1,g_2\in \calR_n$, it is easy to verify, simply using Leibniz's rule of derivation, that
\begin{equation}
\label{rel_formula_Lap_product}
\Delta(g_1g_2)=\Delta(g_1)g_2+g_1\Delta(g_2)+2\sum_{j=1}^n\pdv{g_1}{x_j}\pdv{g_2}{x_j}.
\end{equation}
Moreover, we need to recall the Euler's formula for a polynomial $f\in \calR_{n,k}$, that is
\begin{equation}
\label{rel_Euler_formula}
\sum_{j=1}^n x_j\pdv{f}{x_j}=kf.
\end{equation}
Finally, as already observed in \cite{Fla23a}*{formula (2.4)}, we can easily see that
\begin{equation}
\label{rel_Laplace_on_q_n^s}
\Lap(q_n^s)=2s\sum_{j=1}^n\biggl(q_n^{s-1}+x_j\pdv{q_n^{s-1}}{x_j}\biggr)=2snq_n^{s-1}+4s(s-1)\sum_{j=1}^nx_j^2q_n^{s-2}=2s\bigl(n+2(s-1)\bigr)q_n^{s-1}
\end{equation}
for each $s\geq2$. Therefore, if we iterate the process, we get
\begin{equation}
\label{rel_powers_laplacian_contraction}
\Lap^k(q_n^s)=2^k\frac{s!}{(s-k)!}\biggl(\prod_{j=1}^{k}\bigl(n+2(s-j)\bigr)\biggr)q_n^{s-k}
\end{equation}
for every $k\leq s$. In particular, we can introduce a constant value depending on $n,s\in\bbN$, given by
\begin{equation}
\label{rel_contraction_q_n^s_q_n^s}
C_{n,s}=\Lap^s(q_n^s)=2^ss!\prod_{j=0}^{s-1}(n+2j)\neq 0.
\end{equation}
Furthermore, using the notation
\begin{equation}\label{formula:notation_qnk_Ansk}
q_n^{[k]}=\frac{1}{2^kk!}q_n^k,\qquad A_{n,s,k}=\prod_{j=1}^{k}\bigl(n+2(s-j)\bigr)
\end{equation}
for every $k\in\bbN$, formula \eqref{rel_powers_laplacian_contraction} takes a more concise form, which can be written as
\begin{equation}
\label{rel_powers_laplacian_contraction_alt}
q_n^k\circ q_n^{[s]}=A_{n,s,k}q_n^{[s-k]}.
\end{equation}

In addition to the apolarity action between powers of the quadratic forms, we can observe another nice property. It concerns the apolarity action of powers of $q_n$ on the products of harmonic polynomials by powers of the quadratic form, generalizing formula \eqref{rel_powers_laplacian_contraction_alt}.
\begin{lem}
\label{lem_action_qn_qnh}
Given a harmonic polynomial $h_{m}\in\calH_{n,m}$, the equality
\[
q_n^k\circ\bigl(q_n^{[d]}h_{m}\bigr)=A_{n,d+m,k}q_n^{[d-k]}h_{m}
\]
holds.
\end{lem}
\begin{proof}
We begin by the case $k=1$. Using formulas \eqref{rel_formula_Lap_product} and \eqref{rel_Euler_formula}, we get
\begin{align*}
\Delta\bigl(q_n^{[d]}h_{m}\bigr)&=\Delta\bigl(q_n^{[d]}\bigr)h_{m}+q_n^{[d]}\Delta(h_{m})+2\sum_{j=1}^n\pdv{q_n^{[d]}}{x_j}\pdv{h_{m}}{x_j}\\
&=\Delta\bigl(q_n^{[d]}\bigr)h_{m}+2q_n^{[d-1]}\sum_{j=1}^nx_j\pdv{h_{m}}{x_j}
\vphantom{\sum_{j=1}^n\pdv{q_n^{[d]}}{x_j}}\\
&=\bigl(n+2(d-1)\bigr)q_n^{[d-1]}h_m+2mq_n^{[d-1]}h_{m}
\vphantom{\sum_{j=1}^n\pdv{q_n^{[d]}}{x_j}}\\
&=\bigl(n+2(d+m-1)\bigr)q_n^{[d-1]}h_m
\vphantom{\sum_{j=1}^n\pdv{q_n^{[d]}}{x_j}}\\
&=A_{n,d+m,1}q_n^{[d-1]}h_m
\vphantom{\sum_{j=1}^n\pdv{q_n^{[d]}}{x_j}}.
\end{align*}
Thus, by iterating the process, we get
\[
\Delta^k\bigl(q_n^{[d]}h_{m}\bigr)=A_{n,d+m,k}q_n^{[d-k]}h_{m},
\]
for every $k\in\bbN$ such that $0\leq k\leq d$.
\end{proof}
The role of harmonic polynomials is very important in determining a decomposition of the whole space $\calR_{n,d}$. Indeed, it is well known that the space $\calH_{n,d}$ is an irreducible $\SO_n(\bbC)$-module for every $d\in\bbN$ (see \cite{GW98}*{Theorem 5.2.4}). The next proposition shows a decomposition of $\calR_{n,d}$ as a direct sum of irreducible representations. It is already presented in \cite{GW98}*{Corollary 5.2.5} and also in \cite{ABR01}*{Proposition 5.5}, but here we provide another proof, using apolarity.
\begin{prop}
\label{prop decomposizione armonici}
For every $d\in\bbN$,
\[
\calR_{n,d}=q_n\calR_{n,{d-2}}\oplus \calH_{n,d}
\]
and, more precisely,
\begin{equation}
\label{rel_decomp_harmonic_components}
\calR_{n,d}=\bigoplus_{j=0}^{\left\lfloor\frac{d}{2}\right\rfloor}q_n^j\calH_{n,{d-2j}}.
\end{equation}
\end{prop}
\begin{proof}
Since by duality $q_n\calR_{n,{d-2}}\simeq q_n\calD_{n,{d-2}}$, the desired decomposition can be obtained simply by proving that 
\[
q_n\calR_{n,{d-2}}\cap\calH_{n,d}=\{0\}.
\]
For every $g\in \calR_{n,{d-2}}\setminus\{0\}$, let $k\in\bbN$ be the maximum natural number such that $q_ng=q_n^kg_0$ for some $g_0\in\calR_{n,d-2k}$. Then we have that $q_n\nmid g_0$ and, using formulas \eqref{rel_Euler_formula} and \eqref{rel_powers_laplacian_contraction}, we get
\begin{align*}
\Delta(q_n^kg_0)&=\Delta(q_n^{k})g_0+q_n^k\Delta(g_0)+2\sum_{j=1}^n\pdv{q_n^k}{x_j}\pdv{g_0}{x_j}\\
&=2k\bigl(n+2(k-1)\bigr)q_n^{k-1}g_0+q_n^k\Delta(g_0)+4kq_n^{k-1}\sum_{j=1}^n x_j\pdv{g_0}{x_j}\vphantom{\sum_{j=1}^n\pdv{q_n^k}{x_j}\pdv{h_0}{x_j}}\\
&=2k\bigl(n+2(k-1)\bigr)q_n^{k-1}g_0+4k(d-2k)q_n^{k-1}g_0+q_n^k\Delta(g_0)
\vphantom{\sum_{j=1}^n\pdv{q_n^k}{x_j}\pdv{h_0}{x_j}}\\
&=2k\bigl(n+2(d-k-1)\bigr)q_n^{k-1}g_0+q_n^k\Delta(g_0)\vphantom{\sum_{j=1}^n\pdv{q_n^k}{x_j}\pdv{h_0}{x_j}}.
\end{align*}
So, if $\Delta(q_n^kg_0)=0$, then we have
\[
q_n^k\Delta(g_0)=-2k\bigl(n+2(d-k-1)\bigr)q_n^{k-1}g_0
\]
and, since $g_0\neq 0$, this implies that $q_n^k\mathrel{\bigm|} q_n^{k-1}g_0$. So, we must have $q_n\mathrel{|} g_0$, but this is absurd by the hypothesis on $g_0$. We thus proved that
\[
\calR_{n,d}=q_n\calR_{n,{d-2}}\oplus \calH_{n,d}.
\]
Proceeding by induction on $d$, we easily get the equality
\[
\calR_{n,d}=\bigoplus_{j=0}^{\left\lfloor\frac{d}{2}\right\rfloor}q_n^j\calH_{n,{d-2j}}.\qedhere
\]
\end{proof}

By \autoref{prop decomposizione armonici}, we can determine the dimension of each component of the vector space of harmonic polynomials. 
\begin{cor}
\label{cor dimens harm}
For every $d,n\in\bbN$
\[
\dim\calH_{n,d}=\dim \calR_{n,d}-\dim \calR_{n,{d-2}}=\binom{d+n-1}{n-1}-\binom{d+n-3}{n-1}.
\]
\end{cor}
We can also determine the dimension of each component recursively.
\begin{lem}
\label{lem_dim_harmonic_polynomials}
For every $d,n\geq 1$,
\[
\dim\calH_{n,d}=\dim \calR_{{n-1},{d-1}}+\dim \calR_{{n-1},d}=\dim\calH_{n,{d-1}}+\dim\calH_{n-1,d}.
\]
\end{lem}
\begin{proof}
By \autoref{cor dimens harm}, we have
\begin{align*}
\dim\calH_{n,d}&=\binom{d+n-1}{n-1}-\binom{d+n-3}{n-1}\\[1ex]
&=\frac{(d+n-1)!-d(d-1)(d+n-3)!}{d!(n-1)!}\\[1ex]
&=\big((d+n-1)(d+n-2)-d(d-1)\big)\frac{(d+n-3)!}{d!(n-1)!}\\[1ex]
&=\big(2d(n-1)+(n-1)(n-2)\big)\frac{(d+n-3)!}{d!(n-1)!}\\[1ex]
&=\big(2d+n-2\big)\frac{(d+n-3)!}{d!(n-2)!}\\[1ex]
&=\binom{(d-1)+(n-2)}{n-2}+\binom{d+n-2}{n-2}=\dim \calR_{{n-1},{d-1}}+\dim \calR_{{n-1},{d}}.
\end{align*}
Furthermore, by this last equality, we get
\begin{align*}
\dim\calH_{n,d}&=\dim \calR_{{n-1},{d-1}}+\dim \calR_{{n-1},{d}}\\
&=\dim \calR_{{n-1},{d-1}}+\dim \calR_{{n-1},{d-2}}+\dim\calH_{{n-1},d}\\
&=\dim\calH_{{n},{d-1}}+\dim\calH_{{n-1},{d}}.
\qedhere
\end{align*}
\end{proof}

With respect to the contraction pairing defined in formula \eqref{rel contraction pairing}, we observe that formula \eqref{rel_decomp_harmonic_components} provides an orthogonal decomposition of the spaces of homogeneous polynomials in each degree. In the following lemma we see how the contraction pairing works between each component of decomposition \eqref{rel_decomp_harmonic_components}.
\begin{lem}
For every $j,k\leq d\in\bbN$ such that $j\neq k$,
\[
q_n^jh_{d-2j}\circ q_n^kh_{d-2k}=0
\]
for every $h_{d-2j}\in\calH_{n,{d-2j}}$ and $h_{d-2k}\in\calH_{n,{d-2k}}$.
\end{lem}
\begin{proof}
By \autoref{remark reversing roles}, we can assume by symmetry that $j\geq k$. So we immediately get by \autoref{lem_action_qn_qnh} that
\[
q_n^jh_{d-2j}\circ q_n^kh_{d-2k}=0.\qedhere
\]
\end{proof}

\subsection{The apolar ideal of \texorpdfstring{$q_n^s$}{qns}}
\label{sec_apolar_ideal_harmonic}
In order to show how the apolar ideal $\Ann(q_n^s)$ is made, we analyze the apolarity action of an arbitrary polynomial on $q_n^s$. Again, we will move from a coordinate set to its dual without specifying it. 
We have already seen in \autoref{def_action_GLn_on_R_D} that $\calR_n$ and $\calD_n$ have a natural structure of $\GL_n(\bbC)$-modules.
For every polynomial $g\in\calR_n$, the stabilizer of $g$ is the set
\[
\bigl(\GL_n(\bbC)\bigr)_g=\set{A\in \GL_n(\bbC)|A\cdot g=g}.
\]
In particular, the stabilizer of $q_n^s$ can be described explicitly (see also \cite{Fla23a}*{Lemma 3.3}). 
\begin{lem}
\label{Lem stabilizer q^s}
Let $d\in\bbN$ and let $G=\GL_n(\bbC)$. Then the stabilizer of the form $q_n^s$ with respect to the action of $G$ on $\calR_{n,d}$ is the group
\[
G_{q_n^s}=(\bbZ/s\bbZ)\times \Oa_n(\bbC).
\]
\end{lem}
\begin{proof}
Given any $A\in G_{q_n^s}$, we have by \autoref{def_action_GLn_on_R_D} that
\[
A\cdot q_n^s=A\cdot(x_1^2+\cdots+x_n^2)^s=\bigl((A\cdot x_1)^2+\dots+(A\cdot x_n)^2\bigr)^s=(A\cdot q_n)^s
\]
and thus
\[
A\cdot q_n^s=(A\cdot q_n)^s=q_n^s,
\]
that is,
\[
\bigl((A\cdot x_1)^2+\dots+(A\cdot x_n)^2\bigr)^s=(x_1^2+\cdots+x_n^2)^s.
\]
This means, considering the $s$-th roots of unity, that
\[
A\cdot q_n=\rme^{\frac{2(j-1)\uppi\rmi}{s}}q_n
\]
for some $j\in\bbN$ such that $1\leq j\leq s$ and this implies that
\[
\rme^{-\frac{2(j-1)\uppi\rmi}{s}}A\cdot q_n=q_n.
\]
In particular, since the stabilizer of the form $q_n$ corresponds to the orthogonal group $\Oa_n(\bbC)$, we have
\[
A=\rme^{\frac{2(j-1)\uppi\rmi}{s}}B
\]
for some $B\in\Oa_n(\bbC)$, which proves the statement.
\end{proof}
We need to note another basic fact in representation theory, which is that, in general, the catalecticant map of a homogeneous polynomial $g$ is a map of $\GL_n(\bbC)_g$-modules. For a proof of this basic fact we refer to \cite{Fla23a}*{Proposition 3.1}.
\begin{prop}
\label{Prop polar map equivariant}
For any $d\in\bbN$, let $g\in \calR_{n,d}$ and let $G=\GL_n(\bbC)$. Then the catalecticant map of $g$ is a $G_g$-equivariant map. That is,
\[
\Cat_g(A\cdot f)=A\cdot \Cat_g(f)
\]
for every $A\in G_g$.
\end{prop}
Before proving that the set of $(s+1)$-harmonic polynomials  is a set of generators of the apolar ideal, we need to make some other considerations. 
\begin{lem}
\label{lem:isotropic_linear_forms}
If $l_{\bfa}\in\calD_{n,1}$ is a linear form associated with an isotropic point $\bfa\in\bbC^n$, namely such that $\bfa\cdot\bfa=0$, then the $d$-th power $l_{\bfa}^d$ is harmonic for every $d\in\bbN$.
\end{lem}
\begin{proof}
If $d=1$, the statement is clear. If $d\geq 2$, we get by \autoref{lem contraz forme} that
\[
\Lap l_{\bfa}^{[d]}=q_n\circ l_{\bfa}^{[d]}=(\bfa\cdot\bfa)l_{\bfa}^{[d-2]}=0.\qedhere
\]
\end{proof}
Let us define the polynomials $u,v\in \calD_{n,1}$, as 
\begin{equation}
\label{rel_polynomials_u1_v1}
u=y_1+\rmi y_2,\qquad v=y_1-\rmi y_2.
\end{equation}
As a direct consequence of \autoref{lem:isotropic_linear_forms}, we have that every power of the polynomials $u$ and $v$ is harmonic.
\begin{lem}
\label{Rem polinomi armoni u v}
For every $d\in\bbN$ such that $d\geq s+1$, $u^d,v^d\in\Ann(q_n^s)$.
\end{lem} 
\begin{proof}
We prove the statement only for the linear polynomial $u$, since the case of $v$ is analogous. 
We observe that
\[
(y_1+\rmi y_2)\circ(x_1+\rmi x_2)=1+\rmi^2=0
\]
and
\[
u\circ q_n^s=(y_1+\rmi y_2)\circ(x_1^2+\dots+x_n^2)^s=2s(x_1+\rmi x_2)(x_1^2+\dots+x_n^2)^{s-1}.
\]
Then, if $d\geq s+1$, we get by Leibniz's rule
\[
u^d\circ q_n^s=(y_1+\rmi y_2)^d\circ(x_1^2+\dots+x_n^2)^s=2^ss!(y_1+\rmi y_2)^{d-s}\circ(x_1+\rmi x_2)^s=0,
\]
that is, $u^d\in\Ann(q_n^s)$.
\end{proof}
\begin{rem}
\label{rem prod powers u v}
It is clear that the product of two powers of the polynomials $u$ and $v$ cannot be harmonic. In fact, for any $l,m\in\bbN$ such that $l,m\geq 1$,we have
\[
u^lv^m=(y_1^2+y_2^2)(y_1+\rmi y_2)^{l-1}(y_1-\rmi y_2)^{m-1},
\]
namely $(y_1^2+y_2^2)\mathrel{\bigm|} u^lv^m$.
\end{rem}
Now, we observe that 
\[
y_j\circ q_n^s=\pdv{q_n^s}{x_j}=2sq_n^{s-1}x_j
\]
for every $j=1,\dots,n$.
By Leibniz's rule, we can iterate this and obtain, for any $f\in \calD_{n,k}$,
\begin{equation}
\label{rel_apolarity_polynomial_powerQ}
f\circ q_n^s=2^{k}\frac{s!}{(s-k)!}q_n^{s-k}f+q_n^{s-k+1}g,
\end{equation}
for some $g\in \calR_{n,{k-2}}$, where $k\leq s$. 
Finally, we can analyze each component of the catalecticant map, in order to determine its kernel, which is the apolar ideal of $q_n^s$.
We will see that the catalecticant matrices of $q_n^s$ are full rank for every $n,s\in\bbN$. Moreover, we can represent them in their diagonal form using of harmonic polynomials. Given any $k\in\bbN$ such that $k\leq 2s$, we consider the $k$-th catalecticant map
\[
\Cat_{q_n^s,k}\colon\calD_{n,k}\to\calR_{n,2s-k}.
\]
\begin{prop}
\label{prop_catalecticant_harmonic_isomorphism}
Let $k\in\bbN$ be such that $k\leq 2s$ and $j\leq\pint*{\frac{k}{2}}$. Then
\[
\Cat_{q_n^s,k}(q_n^jh_{k-2j})\in q_n^{s-k+j}\calH_{n,k-2j}
\]
for every $h_{k-2j}\in\calH_{n,k-2j}$.
In particular, the restriction 
\[
\Cat_{q_n^s,k}\colon q_n^j\calH_{n,k-2j}\to q_n^{s-k+j}\calH_{n,k-2j}
\]
is a well-defined isomorphism of $\SO_n(\bbC)$-modules.
\end{prop}
\begin{proof}
We have already recalled that the space of harmonic polynomials $\calH_{n,s}$ is an irreducible $\SO_n(\bbC)$-module (\cite{GW98}*{Theorem 5.2.4}), and so is $q_n^j\calH_{n,k-2j}$. Now, let us consider the linear polynomials $u$ and $v$ defined in \eqref{rel_polynomials_u1_v1} and the restriction
\[
\Cat_{q_n^s,k}\bigr|_{q_n^j\calH_{n,k-2j}}\colon q_n^j\calH_{n,k-2j}\to \calR_{n,2s-k}.
\]
Then we have, by formula \eqref{rel_powers_laplacian_contraction}, that
\[
\Cat_{q_n^s,k}(q_n^ju_1^{k-2j})=(q_n^ju_1^{k-2j})\circ q_n^s=u_1^{k-2j}\circ\Biggl(2^j\frac{s!}{(s-j)!}\biggl(\prod_{t=s-j}^{s-1}(n+2t)\biggr)q_n^{s-j}\Biggr).
\]
We also see that
\[
u_1^{k-2j}\circ q_n^{s-j}=(y_1+\rmi y_2)^{k-2j}\circ q_n^{s-j}=(y_1+\rmi y_2)^{k-2j-1}\circ \bigl(2q_n^{s-j-1}(x_1+\rmi x_2)\bigr)
\]
and, since
\[
(y_1+\rmi y_2)\circ (x_1+\rmi x_2)=0,
\]
we can iterate the process, obtaining
\[
u_1^{k-2j}\circ q_n^{s-j}=2^{k-2j}q_n^{s-k+j}(x_1+\rmi x_2)^{k-2j}\in q_n^{s-k+j}\calH_{n,k-2j}.
\]
In particular, since $q_n^{s-k+j}\calH_{n,k-2j}$ is an irreducible $\SO_n(\bbC)$-module, we must have by \autoref{Lem Schur} that
\[
q_n^{s-k+j}\calH_{n,{k-2j}}\subseteq\im\bigl(\Cat_{q_n^s,k}\bigr).
\]
Therefore, for dimensional reasons, it immediately follows that
\[
\Cat_{q_n^s,k}\bigr|_{q_n^j\calH_{n,k-2j}}\colon q_n^j\calH_{n,k-2j}\to q_n^{s-k+j}\calH_{n,k-2j}
\]
is an isomorphism of $\SO_n(\bbC)$-modules.
\end{proof}
In a certain sense, the catalecticant maps of $q_n^s$ preserve the decomposition \eqref{rel_decomp_harmonic_components}. Indeed, considering the standard coordinates, for every harmonic polynomial $h\in\calH_{n,k}$, we know by formula \eqref{rel_apolarity_polynomial_powerQ} that
\[
h\circ q_n^s=2^{k}\frac{s!}{(s-k)!}q_n^{s-k}h+q_n^{s-k+1}g
\]
for some $g\in\calR_{n,k-2}$. However, since the harmonic decomposition is unique, we have by \autoref{prop_catalecticant_harmonic_isomorphism} that 
\[
h\circ q_n^s=2^{k}\frac{s!}{(s-k)!}q_n^{s-k}h.
\]
In particular, setting
\[
q_n^{[s]}=\frac{1}{2^s s!}q_n^{s}
\]
for every $s\in\bbN$, we get
\begin{equation}
\label{rel_image_harmonic_catalecticant}
h\circ q_n^{[s]}=q_n^{[s-k]}h.
\end{equation}
for every $h\in \calH_{n,k}$. Thus, we can determine a particular basis, for which every catalecticant matrix has a diagonal form. 
\begin{prop}
\label{prop_catalecticant_diagonal}
Let $\calB_{n,d}$ be a basis of $\calH_{n,d}$ for every $n,d\in\bbN$. Let $A_{n,s,k}$ be the value given by
\[
A_{n,s,k}=\prod_{j=0}^{k-1}\bigl(n+2(s-k+j)\bigr)=\prod_{j=0}^{k-1}\bigl(n+2(s-1-j)\bigr)=\prod_{j=1}^{k}\bigl(n+2(s-j)\bigr).
\]
Let also
\[
\calT_d=\bigcup_{k=0}^{\pint*{\frac{d}{2}}}\Set{\dfrac{1}{A_{n,s,k}}q_n^kh_{d-2k}|h_{d-2k}\in\calB_{n,d-2k}}
\]
and
\[
\calS_d=\bigcup_{k=0}^{\pint*{\frac{d}{2}}}\Set{q_n^{[k]}h_{d-2k}|h_{d-2k}\in\calB_{n,d-2k}^*}
\]
be bases of $\calD_{n,d}$ and $\calR_{n,d}$, respectively, for every $d=1,\dots,2s$. Then, it is possible to order the elements of the bases $\calT_d$ and $\calS_{2s-d}$ so that the entries of the $d$-th catalecticant matrix of $q_n^{[s]}$ with respect to $\calT_d$ and $\calS_{2s-d}$ are
\[
\bigl(\Cat_{q_n^{[s]},d}\bigr)_{ij}=\begin{cases}
1 &\text{if $i=j$},\\
0 &\text{otherwise}.
\end{cases}
\]
In particular, the middle catalecticant matrix corresponds to the identity matrix.
\end{prop}
\begin{proof}
It suffices to prove that, for every $1\leq d\leq 2s$, the $d$-th catalecticant map sends elements of $\calT_d$ into elements of $\calS_{2s-d}$. For each $h_{d-2k}\in\calH_{n,{d-2k}}$, we have
by formulas \eqref{rel_powers_laplacian_contraction} and \eqref{rel_image_harmonic_catalecticant} that
\begin{align*}
\frac{1}{A_{n,s,k}}q_n^kh_{d-2k}\circ q_n^{[s]}&=\biggl(\prod_{j=0}^{k-1}\frac{1}{n+2(s-k+j)}\biggr)q_n^kh_{d-2k}\circ q_n^{[s]}\\
&=h_{d-2k}\circ q_n^{[s-k]}=q_n^{[s-d+k]}h_{d-2k}
\end{align*}
and hence the statement is proved.
\end{proof}

The new form of the catalecticant matrices of $q_n^s$ allows us to determine more easily how the comoponents of the apolar ideal are made. In particular, by \autoref{prop_catalecticant_diagonal}, we have the following corollary, which is the same of \cite{Fla23a}*{Propositions 3.2 and 3.6}.
\begin{cor}
\label{cor_elementi_grado_min_s}
For every $k\in\bbN$ such that $0\leq k\leq s$,
\[
\Ker\bigl(\Cat_{q_n^s,{s-k}}\bigr)=\Ann(q_n^s)_{s-k}=\{0\}.
\]
Moreover, kor every $k\in\bbN$ such that $1\leq k\leq s+1$, 
\begin{equation}
\label{relation prop 3.6}
\Ker\bigl(\Cat_{q_n^s,s+k}\bigr)=\Ann(q_n^s)_{s+k}=\bigoplus_{j=0}^{k-1}q_n^j\calH_{n,{s+k-2j}}.
\end{equation}
\end{cor}
The fact that the catalecticant matrices of $q_n^s$ are all full-rank has already been proved by B.~Reznick, using \cite{Rez92}*{Theorem 8.15} and referring to \cite{Rez92}*{Theorem 3.7} and \cite{Rez92}*{Theorem 3.16}. Another kind of proof is given by F.~Gesmundo and J.~M.~Landsberg in \cite{GL19}*{Theorem 2.2}. 
Now we have all the elements to present our first result, that is the determination of the apolar ideal $\Ann(q_n^s)$, which 
is generated by harmonic polynomials of degree $s+1$. The proof is based on the determination of a set of generators for each component of the apolar ideal (see also \cite{Fla23a}*{Theorem 3.8}).
\begin{teo}
\label{Teo Apolar ideal}
The apolar ideal of the form $q_n^s$ is
\[
\Ann(q_n^s)=(\calH_{n,{s+1}}).
\]
\end{teo}
\begin{proof}
By \autoref{cor_elementi_grado_min_s} we simply have to prove that
\[
\Ann(q_n^s)_d=\calH_{n,{s+1}}\calD_{n,{d-s-1}}
\]
for every $d\geq s+1$. So, setting $d=s+k$ with $k\in\bbN$,  this is the same as proving that
\begin{equation}
\label{relation HD}
\bigoplus_{j=0}^{k-1}q_n^j\calH_{n,{s+k-2j}}=\calH_{n,{s+1}}\calD_{n,{k-1}},
\end{equation}
in the case where $1\leq k\leq s+1$.
In particular, by formula \eqref{relation prop 3.6} we have 
\[
\Ann(q_n^s)_{s+1}=\calH_{n,{s+1}}.
\]
Therefore, we have the inclusion
\begin{equation}
\label{relation first inclusion}
\bigoplus_{j=0}^{k-1}q_n^j\calH_{n,s+k-2j}\supseteq\calH_{n,s+1}\calD_{n,k-1}.
\end{equation}
The prove the reverse inclusion, we proceed by induction on $k$ separately for odd and even values. If $k=1$, the equality is clear. If $k=2$, the inclusion to prove is
\begin{equation}
\label{formula:first_inclusion}
\calH_{n,s+2}\oplus q_n\calH_{n,s}\subseteq\calH_{n,s+1}\calD_{n,1}.
\end{equation}
Considering the polynomials in formula \eqref{rel_polynomials_u1_v1}
and the power $u^{s+2}\in\calH_{n,s+2}$, we have
\[
u^{s+1}u\in\calH_{n,s+1}\calD_{n,1}.
\]
Thus, since $\calH_{n,s+2}$ is an irreducible $\SO_n(\bbC)$-module, it follows by \autoref{Lem Schur} that
\begin{equation}\label{formula:inclusion_harmonic_polynomials_k2}
\calH_{n,s+2}\subseteq \calH_{n,s+1}\calD_{n,1}.
\end{equation}
Now, by \autoref{rem prod powers u v} and formula \eqref{relation first inclusion}, we have that $u^{s+1}v\notin\calH_{n,s+2}$ and 
\[
u^{s+1}v\in\calH_{n,s+1}\calD_{n,1}\subseteq\calH_{n,s+2}\oplus q_n\calH_{n,s}.
\] 
Therefore, by \autoref{prop decomposizione armonici}, there are unique $h_1\in\calH_{n,s+2}$ and $ h_2\in\calH_{n,s}$, with $h_2\neq 0$, such that
\[
u^{s+1}v=h_1+q_nh_2\in \calH_{n,s+2}\oplus q_n\calH_{n,s}.
\] 
In particular,
\[
q_nh_2=u^{s+1}v-h_1\in\calH_{n,s+1}\calD_{n,1}.
\]
Again by \autoref{Lem Schur}, we have
\[
q_n\calH_{n,s}\subseteq\calH_{n,s+1}\calD_{n,1},
\]
which, together with formula \eqref{formula:inclusion_harmonic_polynomials_k2}, gives the inclusion \eqref{formula:first_inclusion}. Now, let us suppose $3\leq k\leq s$ and that equality \eqref{relation HD} holds for $k-2$. We need to show that  
\[
\bigoplus_{j=0}^{k-1}q_n^j\calH_{n,s+k-2j}=\calH_{n,s+k}\oplus q_n\biggl(\bigoplus_{j=0}^{k-2}q_n^{j}\calH_{n,s+k-2-2j}\biggr)\subseteq\calH_{n,s+1}\calD_{n,k-1}.
\]
As above, we consider the polynomial 
$u^{s+k}\in\calH_{n,s+k}$, that is the product
\[
u^{s+1}u^{k-1}\in\calH_{n,s+1}\calD_{n,k-1}.
\]
By \autoref{Lem Schur}, we conclude that
\[
\calH_{n,s+k}\subseteq\calH_{n,s+1}\calD_{n,k-1}.
\] 
Now, by inductive hypothesis, we have
\[
\bigoplus_{j=0}^{k-2}q_n^{j}\calH_{n,s+k-2-2j}\subseteq\calH_{n,s+1}\calD_{n,k-3}.
\]
Therefore, as $q_n\in \calD_{n,2}$, we have
\[
q_n\bigg(\bigoplus_{j=0}^{k-2}q_n^{j}\calH_{n,s+k-2-2j}\bigg)\subseteq\calH_{n,s+1}\calD_{n,k-1}.
\]
This gives the inclusion
\[
\bigoplus_{j=0}^{k-1}q_n^j\calH_{n,s+k-2j}\subseteq\calH_{n,s+1}\calD_{n,k-1},
\]
proving equality \eqref{relation HD}.
It remains to show that
\begin{equation*}
\Ann(q_n^s)_{d}=\calH_{n,s+1}\calD_{n,d-s-1},
\end{equation*}
for every $d\geq 2(s+1)$.
In particular, since
\[
\Ann(q_n^s)_{d}=\calD_{n,d}
\]
for every $d\geq 2(s+1)$, we simply have to prove that
\begin{equation}
\label{relation casi alti}
\calD_{n,2s+m+1}=\calH_{n,s+1}\calD_{n,s+m}
\end{equation}
for every $m\geq 1$. Now, by formulas \eqref{relation prop 3.6} and \eqref{relation HD}, we have 
\[
\Ann(q_n^s)_{2s+1}=\calH_{n,s+1}\calD_{n,s}=\bigoplus_{j=0}^{s}q_n^j\calH_{n,2s-2j+1}.
\]
This implies, by decomposition \eqref{rel_decomp_harmonic_components}, that
\[
\calH_{n,s+1}\calD_{n,s}=\calD_{n,2s+1},
\]
by which we have
\[
\calD_{n,2s+m+1}=\calD_{n,2s+1}\calD_{n,m}=\calH_{n,s+1}\calD_{n,s}\calD_{n,m}=\calH_{n,s+1}\calD_{n,s+m}
\]
for every $m\geq 1$, which proves equality \eqref{relation casi alti}.
\end{proof}

By \autoref{Teo Apolar ideal} it is quite easy to obtain a lower bound for the rank of the form $q_n^s$. In fact, as a direct consequence of \autoref{prop lower bound}, we get the following corollary.
\begin{cor}
\label{cor_cat_lower_bound}
For every $n$,$s\in\bbN$
\[
\rk(q_n^s)\geq\binom{s+n-1}{n-1}.
\]
\end{cor}

\section{Tight decompositions}
\label{cha_tight_decompositions}
\noindent The lower bound provided in \autoref{cor_cat_lower_bound} by the value
\[
T_{n,s}=\binom{s+n-1}{n-1}
\] 
leads us to ask in which cases the equality is satisfied, that is, the rank of $q_n^s$ is equal to the rank of the middle catalecticant matrix. B.~Reznick already analyzed this problem in \cite{Rez92} for real decompositions, calling the decompositions of size $T_{n,s}$ as \textit{tight decompositions} and establishing a connection with the classical language of spherical designs. Although it is not easy to extend the whole theory to the complex case, in this chapter we give some generalizations of results of B.~Reznick which are also valid for the field $\bbC$. 

In \cref{sec_real_decomp_spherical_designs} we give a summary of results for the real tight decompositions, which seem to preserve some geometric properties when extended to the complex case. In particular, as we know that in the real case the tight decompositions are formed by points of the same norm, we extend this fact to complex decompositions, proving that the $s$-th power $(\bfa\cdot\bfa)^{s}$ is the same for every point appearing in a tight decomposition of $q_n^s$, for every $s\in\bbN$. Thanks to this surprising fact, we are able to extend some results of B.~Reznick. One of the most important is that many decompositions, which turn out to be unique up to real orthogonal transformations, retain their uniqueness even when extended to the complex field. First, in \cref{sec_tight decomp_two_variables} we explicitly compute all the possible decompositions in two variables, which are unique up to complex orthogonal transformations (\autoref{Teo all complex decompositions q2s}). Then, we focus in \cref{sec_general_tight_decomp} on tight decompositions for lower powers and we also observe that, working on the complex field, there are only a few number of values of $n$ that can be assumed to get tight decompositions. In many cases, the question remains open even for real ones.

\subsection{Real decompositions and spherical designs}
\label{sec_real_decomp_spherical_designs}
\noindent When dealing with minimal real decompositions of the form $q_n^s$, the number of different values of the norm of the points of the decompositions plays a relevant role. B.~Reznick focuses on this in \cite{Rez92}*{Chapter 8}, analyzing the decompositions determined by points of the same norm, also called \textit{first caliber} decompositions, and their one-to-one correspondence with some specific combinatorial objects, known as spherical designs (see \cite{Rez92}*{Proposition 8.38}). 
Although we do not work with spherical designs, we consider these results and generalize some of them to the field of complex numbers, especially for what concerns first caliber decomposition.

Spherical designs can be defined in many different ways. B.~Reznick in \cite{Rez92} chose the one provided by
P.~Delsarte, J.-M.~Goethals, and J.~J.~Seidel in \cite{DGS77}*{Definition 5.1}. They were the first to define this concept, considering a spherical $t$-design as a finite set of points $A$ contained in the $(n-1)$-dimensional sphere $S^{n-1}$ such that, for every homogeneous polynomial $f\in\calR_n$ with $\deg f\leq t$,
\[
\frac{1}{\bigl|\rmS^{n-1}\bigr|}\int_{\rmS^{n-1}}f(\xi)\diff\omega(\xi)=\frac{1}{|A|}\sum_{\bfa\in A}f(\bfa).
\]
More details on spherical designs can be found in the literature, for which B.~Reznick also lists several texts (see \cite{Rez92}*{p.~113}) that can be consulted for more information, such as \cites{Ban84,CS99,GS79,GS81a,GS81b,Hog90,Sei84,Sei87}.
There is another equivalent definition, again given by P.~Delsarte, J.-M.~Goethals and J.~J.~Seidel in \cite{DGS77}*{Theorem 5.2}, and also used by E.~Bannai and R.~M.~Damerell in \cite{BD79} and \cite{BD80}. It tells us that any finite set of points $A\subseteq\rmS^{n-1}$ is a $t$-spherical design if and only if
\[
\sum_{\bfa\in A}h(\bfa)=0
\]
for every $h\in\calH_{n,k}$ and $k=1,\dots,t$. Although we can extend this last definition to complex points on a subset $A$ of the complexified sphere
\[
\rmS_{\bbC}^{n-1}=\Set{(x_1,\dots,x_n)\in\bbC^n|x_1^2+\cdots+x_n^2=1},
\]
this does not immediately give the same results. This is due to the several hypotheses on the points which are necessary to define another kind of objects, equally important, namely the \textit{spherical codes} (see e.g.~\cite{DGS77}*{section 4}).

We begin by introducing a name to denote the minimal decomposition of size equal to the rank of middle catalecticant matrices. This is the same term used by B.~Reznick and it was chosen in relation to the correspondence with tight spherical designs (see \cite{BD79}).
\begin{defn}
A decomposition 
\[
q_n^s=\sum_{k=1}^m(a_{k,1}x_1+\cdots+a_{k,n}x_n)^{2s}
\]
is said to be \textit{tight} if $m=T_{n,s}$, where
\[
T_{n,s}=\binom{s+n-1}{s}.
\]
\end{defn}
The work of B.~Reznick on the form $q_n^s$ involves several decompositions, some of which were known from the classical literature and then translated from the language of spherical designs. This can be explained by the fact that the real decompositions of size $m\in\bbN$ can be related to a finite set of $m$ real $n$-tuples $\bfa_1,\dots,\bfa_m$, where
\[
\bfa_k=(a_{k,1},\dots,a_{k,n})\in\bbR^n,
\]
for every $k=1,\dots,m$. Then, it is quite natural to identify these as points on multiple $n$-dimensional spheres centered on the origin. In particular, we have the following theorem, provided by B.~Reznick.
\begin{teo}[\cite{Rez92}*{Proposition 9.2}]
\label{teo_Reznick_tight_decompositions}
If $q_n^s$ has a real tight decomposition, then one of the following conditions holds:
\begin{enumerate}[label=(\arabic*), left= 0pt, widest=*,nosep]
\item $s=1$ or $n=2$;
\item $s=2$ and $n=3$;
\item $s=2$ and $n=m^2-2$ for some odd $m\in\bbN$;
\item $s=3$ and $n=3m^2-4$ for some $m\in\bbN$;
\item $s=5$ and $n=24$.
\end{enumerate}
\end{teo}
This result is very powerful in dealing with real numbers, since it shows that there is no tight decomposition for every power $s\geq 6$. Now, given a natural number $r\in\bbN$, B.~Reznick defines a \textit{$r$-th caliber representation} as a representation for which there are $r$ distinct spheres containing points of such a representation, namely, a real decomposition for which there are $r$ distinct values assumed by 
\[
\abs{\bfa_k}^{2s}=(a_{k,1}^2+\cdots+a_{k,n}^2)^{s}
\]
for $k=1,\dots,m$. We can extend in a natural way this definition to the complex field, including isotropic points.
\begin{defn}
For every $r\in\bbN$, a decomposition 
\[
q_n^s=\sum_{k=1}^m(\bfa_k\cdot\bfx)^{2s}=\sum_{k=1}^m(a_{k,1}x_1+\cdots+a_{k,n}x_n)^{2s}
\]
is said to be an \textit{$r$-th caliber decomposition} if there exist exactly $r$ values $c_1,\dots,c_r\in\bbC$ such that
\[
\{(\bfa_k\cdot\bfa_k)^s\}_{k=1,\dots,m}\in\{c_1,\dots,c_r\}.
\]
That is, there are exactly $r$ values assumed by
\[
(\bfa_k\cdot\bfa_k)^s=(a_{k,1}^2+\cdots+a_{k,n}^2)^s
\]
for $k=1,\dots,m$.
\end{defn}
The first caliber decompositions play a special role because of their notable symmetry. Most of the results given by B.~Reznick in \cite{Rez92} are based on the construction of an inner product on $\calR_{n,d}$. Given a polynomial $p\in\calR_{n,d}$ and a multi-index $\alpha\in\bbN^n$ such that $|\alpha|=d$, we introduce the values $c_{\alpha}(p)$, which are coefficients in the polynomial $p$ such that 
\[
p=\sum_{\alpha\in\calI_{n,d}}\frac{|\alpha|!}{\alpha_1!\cdots\alpha_n!}c_{\alpha}(p)\bfx^\alpha,
\]
where 
\[
\calI_{n,d}=\Set{\alpha\in\bbN^n|\abs*{\alpha}=d}.
\]
In the case of real polynomials, for every $n,d\in\bbN$, the inner product 
\[
\langle\,,\rangle\colon S^d\bbR^n\times S^d\bbR^n\to\bbR
\]
considered by B.~Reznick in \cite{Rez92}*{pp.~1-2} associates to each pair of polynomials $f,g\in S^d\bbR^n$ the real value
\[
\langle f,g\rangle=\sum_{\alpha\in\calI_{n,d}}\frac{|\alpha|!}{\alpha_1!\cdots\alpha_n!}c_{\alpha}(f)c_{\alpha}(g).
\]
This inner product is classically known as the \textit{Bombieri inner product} and it is based on the norm in spaces of polynomials known as the \textit{Bombieri norm} and introduced by B.~Beauzamy, E.~Bombieri, P.~Enflo, and H.~L.~Montgomery in \cite{BBEM90}. For the real case, this concept has also been analyzed by E.~Kostlan in \cite{Kos93}*{Section 4}, where he proves (see \cite{Kos93}*{Theorem 4.1, Theorem 4.2}) that it is invariant under the action of the group $\Oa_n(\bbR)$ of real orthogonal matrices. We can extend it to the complex case and obtain a complex symmetric bilinear form
\[
\langle\,,\rangle\colon \calR_{n,d}\times \calR_{n,d}\to\bbC,
\] 
defined in the same way. 

A nice property about this product, considered by B.~Reznick in \cite{Rez92}*{formula (1.5)}, considers the evaluation of a polynomial at a point associated with a linear form and is given by the following proposition.
\begin{prop}
\label{prop_inner_prod_ker}
For every $f\in\calR_{n,d}$ and for every $\bfa\in\bbC^n$
\begin{equation}
\langle f,(\bfa\cdot\bfx)^d\rangle=f(\bfa).
\end{equation}
\end{prop}
\begin{proof}
Since every homogeneous polynomial can be written as a sum of $d$-powers of linear forms, it is sufficient to prove the statement for the case where $f$ is a power of a linear form $(\bfb\cdot\bfx)^d$ for some $\bfb\in\bbC^n$.
Then we have
\[
\langle (\bfb\cdot\bfx)^d,(\bfa\cdot\bfx)^d\rangle=\sum_{\alpha\in\calI_{n,d}}\frac{|\alpha|!}{\alpha_1!\cdots\alpha_n!}\bfb^\alpha\bfa^\alpha=(\bfb\cdot\bfa)^d,
\]
which proves the formula.
\end{proof}
We can state, as for the previous one, that this inner product is invariant under the action of the complex orthogonal group. 
\begin{prop}
\label{prop_equivariance_Bombieri_product}
For every $f,g\in\calR_{n,d}$, the identity
\[
\pa*{A\cdot f,A\cdot g}=\pa*{f,g}
\]
holds for every $A\in\Oa_n(\bbC)$.
\end{prop}
\begin{proof}
By linearity, we can easily prove the statement for powers of linear forms. So, considering the polynomials
\[
f=(\bfa\cdot\bfx)^d,\quad g=(\bfb\cdot\bfx)^d,
\]
and an orthogonal matrix $A\in\Oa_n(\bbC)$, i.e. such that $\transpose AA=A\transpose A=I$, we get
\begin{align*}
\langle A\cdot (\bfa\cdot\bfx)^d,A\cdot (\bfb\cdot\bfx)^d\rangle&=\langle(\bfa\cdot A\bfx)^d,(\bfb\cdot A\bfx)^d\rangle=\langle(\transpose A\bfa\cdot \bfx)^d,(\transpose A\bfb\cdot \bfx)^d\rangle
\vphantom{\frac{|\alpha|!}{\alpha_1!\cdots\alpha_n!}(A\bfa)^{\alpha}(A\bfb)^\alpha}\\[1ex]
&=\sum_{\alpha\in\calI_{n,d}}\frac{|\alpha|!}{\alpha_1!\cdots\alpha_n!}(\transpose A\bfa)^{\alpha}(\transpose A\bfb)^\alpha=(\transpose A\bfa\cdot \transpose A\bfb)^d\\[2pt]
&=(\bfa\cdot\bfb)^d=\langle(\bfa\cdot\bfx)^d,(\bfb\cdot\bfx)^d\rangle\vphantom{\sum_{\alpha\in\calI_{n,d}}(A\bfa)^{\alpha}(A\bfb)^\alpha}.\qedhere
\end{align*}
\end{proof}
As a explicit example for the value of the Bombieri norm (see \cite{Rez92}*{formula (8.2) and Corollary 8.18}) we have  
\begin{equation}
\label{rel_prod_q_n^s_q_n^s}
\langle q_n^s,q_n^s\rangle=\prod_{j=0}^{s-1}\frac{2j+n}{2j+1}.
\end{equation}
Furthermore, we observe that, given a decomposition
\begin{equation}
\label{rel_arbitrary_decomposition}
q_n^s=\sum_{k=1}^m(\bfa_k\cdot\bfx)^{2s},
\end{equation}
we get by \autoref{prop_inner_prod_ker} and formula \eqref{rel_prod_q_n^s_q_n^s} the equality
\[
\prod_{j=0}^{s-1}\frac{2j+n}{2j+1}=\Bigl\langle q_n^s,\sum_{k=1}^m(\bfa_k\cdot\bfx)^{2s}\Bigr\rangle=\sum_{k=1}^m\langle q_n^s,(\bfa_k\cdot\bfx)^{2s}\rangle=\sum_{k=1}^mq_n^s(\bfa_k)=\sum_{k=1}^m(\bfa_k\cdot\bfa_k)^{s}.
\]
Thus, we rewrite the following generalization to the complex field for first caliber decompositions.
\begin{prop}
\label{prop_norm_points_decomposition}
Let
\[
q_n^s=\sum_{k=1}^m(\bfa_k\cdot\bfx)^{2s}
\]
be a first caliber decomposition of size $m\in\bbN$. Then
\[
(\bfa_k\cdot\bfa_k)^{s}=\frac{1}{m}\prod_{j=0}^{s-1}\frac{2j+n}{2j+1}
\]
for every $k=1,\dots,m$.
\end{prop} 

P.~D.~Seymour and T.~Zavlavsky prove in \cite{SZ84} that a first caliber decomposition of $q_n^s$ of size $r\in\bbN$ always exists for a sufficiently large value of $r$.
Moreover, B.~Reznick proves in \cite{Rez92}*{Corollary 8.17} that every real tight decomposition is first caliber. This result can easily be extended, but in order to prove this, we need the following lemma.
\begin{lem}
\label{lem_isotropic_points_harmonic_forms}
Let $d\geq 2$. Then the following conditions hold:
\begin{enumerate}[label=(\arabic*), left= 0pt, widest=*,nosep]
\item for every point $\bfa\in\bbC^n$, the $d$-th power of its associated linear form $l_{\bfa}=\bfa\cdot\bfx$ is harmonic if and only if $\bfa$ is isotropic in $\bbC^n$, i.e., $\bfa\cdot\bfa=0$;
\item the space $\calH_{n,d}$ is generated by the $d$-th powers of linear forms associated with isotropic points; that is,
\[
\calH_{n,d}=\bigl\langle\Set{l_{\bfa}^d\in\calR_{n,d}|\bfa\in\bbC^n:\bfa\cdot\bfa=0}\bigr\rangle.
\]
\end{enumerate}
\end{lem}
\begin{proof}
Point (1) follows directly from \autoref{lem contraz forme}. In fact, for every non-zero linear form $l_{\bfa}\in\calR_{n,1}$ we have
\[
\Lap(l_{\bfa}^d)=d(d-1)(a_1^2+\cdots+a_n^2)l_{\bfa}^{d-2}=d(d-1)(\bfa\cdot\bfa)l_{\bfa}^{d-2},
\]
which is equal to zero if and only if $\bfa\cdot\bfa=0$. To prove point (2), let us consider the space
\[
W=\bigl\langle\Set{l_{\bfa}^d\in\calR_{n,d}|\bfa\in\bbC^n:\bfa\cdot\bfa=0}\bigr\rangle.
\]
We have from point (1) that
\[
W\subseteq\calH_{n,d}.
\]
Now, for every $A\in\SO_n(\bbC)$, we have
\[
A\cdot l_{\bfa}^d=(A\cdot l_{\bfa})^d=l_{A\cdot\bfa}^d
\]
for every $\bfa\in\bbC^n$. Since $A$ is an orthogonal transformation, we have
\[
\bfa\cdot\bfa=(A\cdot\bfa)\cdot(A\cdot\bfa),
\]
which means that
\[
A\cdot h\in W
\]
for every $h\in W$.
Thus, $W$ is a $\SO_n(\bbC)$-module and since  $\calH_{n,d}$ is an irreducible $\SO_n(\bbC)$-module,
as we saw in \cref{sec_harmonic_polynomials}, we must have $W=\calH_{n,d}$.
\end{proof}
When dealing with linear forms of the kind $\bfa\cdot\bfx$ for some $\bfa\in\bbC^n$, the associated value $\bfa\cdot\bfa$ plays a relevant role. In particular, from the fact that the middle catalecticant $\Cat_{q_n^s,s}$ is full rank,  the following lemma immediately follows.
\begin{lem}
\label{lem_basic_catalecticant_minus_one}
For every $n,s\in\bbN$, let $\bfa\cdot\bfx$ be a linear form with $\bfa\in\bbC^n$ and let 
\[
f=q_n^{[s]}-(\bfa\cdot\bfx)^{[2s]}.
\]
Then
\[
T_{n,s}-1\leq \rk(\Cat_{f,s})\leq T_{n,s}.
\]
\end{lem}
\begin{proof}
We can choose a basis $\{g_1,\dots,g_{T_{n,s}}\}$ of $\calR_{n,s}$ with 
\[
g_1=(\bfa\cdot\bfx)^{[s]}. 
\]
Then we can also consider a basis $\{h_1,\dots,h_n\}$ of $\calD_{n,s}$ such that
\[
h_j\circ q_n^{[s]}=g_j
\]
for every $j=1,\dots,T_{n,s}$. In particular, for any choice of the elements $c_1,\dots,c_{T_{n,s}}\in\bbC$, let us consider the polynomial
\[
h=\sum_{j=1}^{T_{n,s}}c_jh_j\in\calD_{n,s}.
\]
Then we have
\begin{align*}
\biggl(\sum_{j=1}^{T_{n,s}}c_jh_j\biggr)\circ f&=\biggl(\sum_{j=1}^{T_{n,s}}c_jh_j\biggr)\circ q_n^{[s]}-\biggl(\sum_{j=1}^{T_{n,s}}c_jh_j\biggr)\circ (\bfa\cdot\bfx)^{[2s]}\\
&=\sum_{j=1}^{T_{n,s}}c_jg_j-\biggl(\sum_{j=1}^{T_{n,s}}c_jh_j(\bfa)\biggr)(\bfa\cdot\bfx)^{[s]}\\
&=\sum_{j=2}^{T_{n,s}}c_jg_j+\biggl(c_1-\sum_{j=1}^{T_{n,s}}c_jh_j(\bfa)\biggr)g_1.
\end{align*}
Therefore, by linear independence, $h\in\Ker(\Cat_{f,s})$ if and only if
\[
c_2=\cdots=c_{T_{n,s}}=0
\] 
and either $c_1=0$, or $h_1(\bfa)=1$. In the first case we clearly have $h=0$, while in the second one we get 
\[
\langle h_1\rangle=\Ker(\Cat_{f,s})
\] and therefore
\[
\rk(\Cat_{f,s})=T_{n,s}-1.\qedhere
\]
\end{proof}
A rather important fact is related to the middle catalecticant of the form $q_n^s$ and isotropic points. We will use the next lemma to prove that any tight decomposition of $q_n^s$ must be first caliber.
\begin{lem}
\label{lem_catalecticant_isotropic_point}
For every $n,s\in\bbN$, let $\bfa\cdot\bfx$ be a linear form such that $\bfa\in\bbC^n$ is an isotropic point, and let 
\[
f=q_n^{[s]}-(\bfa\cdot\bfx)^{[2s]}.
\]
Then the middle catalecticant $\Cat_{f,s}$ of $f$ is full rank.
\end{lem}
\begin{proof}
Let us consider an element $g\in\Ker(\Cat_{f,s})$, which, by decomposition \eqref{rel_decomp_harmonic_components}, can be written as
\[
g=\sum_{k=0}^{\pint*{\frac{s}{2}}}q_n^kh_{s-2k},
\]
where $h_{s-2k}$ is a harmonic polynomial for every $k=0,\dots,\pint*{\frac{s}{2}}$.
Then we have
\begin{equation}
\label{rel_formula_kernerl_middle_catalecticant}
q_n^kh_{s-2k}\circ (\bfa\cdot\bfx)^{[2s]}=q_n^{k-1}h_{s-2k}\circ\bigl(q_n\circ (\bfa\cdot\bfx)^{[2s]}\bigr)=0
\end{equation}
for every $k=1,\dots,\pint*{\frac{s}{2}}$. Since $\bfa$ is isotropic, then it follows by \autoref{lem_isotropic_points_harmonic_forms} that $(\bfa\cdot\bfx)^{2s}$ is harmonic, i.e.,
\[
q_n\circ (\bfa\cdot\bfx)^{2s}=0.
\]
Formula \eqref{rel_formula_kernerl_middle_catalecticant} implies that the kernel of the catalecticant map can only contain harmonic polynomials. To see this, let us suppose that $g\circ f=0$. That is, by formulas \eqref{rel_powers_laplacian_contraction_alt} and \eqref{rel_image_harmonic_catalecticant},
\[
\biggl(\sum_{k=0}^{\pint*{\frac{s}{2}}}q_n^kh_{s-2k}\biggr)\circ \bigl(q_n^{[s]}-(\bfa\cdot\bfx)^{[2s]}\bigr)=\biggl(\sum_{k=1}^{\pint*{\frac{s}{2}}}A_{n,s,k}q_n^{[k]}h_{s-2k}\biggr)+h_s\circ\bigl(q_n^{[s]}-(\bfa\cdot\bfx)^{[2s]}\bigr)=0.
\]
Then, since the polynomial
\[
h_s\circ\bigl(q_n^{[s]}-(\bfa\cdot\bfx)^{[2s]}\bigr)=h_s-h_s(\bfa)(\bfa\cdot\bfx)^{[s]}
\]
is harmonic, by the uniqueness of decomposition \eqref{rel_decomp_harmonic_components} we must have $h_{s-2k}=0$ for every $k=1\dots,\pint*{\frac{s}{2}}$. That is, 
\[
g=h_s=h_s(\bfa)(\bfa\cdot\bfy)^{[s]},
\]
which is harmonic and, in particular, evaluating $g$ in $\bfa$, we have
\[
g(\bfa)=h_s(\bfa)=h_s(\bfa)(\bfa\cdot\bfa)^{[s]}=0,
\] 
that is, $g\equiv 0$.
\end{proof}
The importance of \autoref{lem_catalecticant_isotropic_point} is that it allows to exclude isotropic points from tight decompositions. In particular,
we can now prove that every tight decomposition must be first caliber, even considering complex decompositions.
\begin{rem}\label{rem:transitivity_SOn}
Recall that the action of $\SO_n(\bbC)$ on $\bbP\bbC^n$ has two orbits, formed by isotropic and non-isotropic points, respectively. We consider two points $\bfa, \bfb\in\bbC^n$, where \[\bfa=(1,0,\dots,0),\quad \bfb=(b_1,\dots,b_n),\quad \bfb\cdot\bfb=b_1^2+\cdots+b_n^2=1\] and a complement $\{\bfv_2,\dots,\bfv_n\}$ such that $\{\bfb,\bfv_2,\dots,\bfv_n\}$ is an orthonormal basis of $\bbC^n$. Then, given the matrix
\[
A=\begin{pNiceMatrix}[margin]
|&|&&|\\
\bfb&\bfv_2&\Cdots&\bfv_n\\
|&|&&|
\end{pNiceMatrix},
\]
we have $A\bfa=\bfb$. 
Now, any isotropic point $\bfd\in\bbC^n$ can be written as $\bfd'+\rmi\bfd''$, where $\bfd'\cdot\bfd'=\bfd''\cdot\bfd''$.
Analogously, we can consider two isotropic points
\[
\bfc=(1,\rmi,0,\dots,0),\qquad \bfd=\bfd'+\rmi\bfd'',
\]
such that $\bfd'\cdot\bfd'=1$. Considering a complement $\{\bfw_3,\ldots,\bfw_n\}$ of $\{\bfc,\bfd\}$ to an orthonormal basis and the matrix
\[
B=\begin{pNiceMatrix}[margin]
|&|&|&&|\\
\bfd'&\bfd''&\bfw_3&\Cdots&\bfw_n\\
|&|&|&&|
\end{pNiceMatrix},
\]
we have $B\bfc=\bfd$.
\end{rem}
\begin{teo}
\label{teo_tight_implies_first_caliber}
Every tight decomposition 
\[
q_n^s=\sum_{k=1}^{m}(\bfa_k\cdot\bfx)^{2s},
\]
where 
\[
m=T_{n,s}=\binom{s+n-1}{s},
\]
is first caliber. In particular,
\[
(\bfa_k\cdot\bfa_k)^{s}=\frac{1}{T_{n,s}}\prod_{j=0}^{s-1}\frac{2j+n}{2j+1}.
\]
for every $k=1,\dots,T_{n,s}$.
\end{teo}
\begin{proof}
Suppose that a tight decomposition of $q_n^{[s]}$ contains a summand $(\bfa_1\cdot\bfx)^{[2s]}$, such that $\bfa_1$ is an isotropic point. In particular, by \autoref{lem_isotropic_points_harmonic_forms} we have
\[
q_n\circ(\bfa_1\cdot\bfx)^{2s}=0.
\]
Then, considering the polynomial
\[
f=q_n^{[s]}-(\bfa_1\cdot\bfx)^{[2s]},
\]
\autoref{prop lower bound} implies that $\Cat_{f,s}$
is not full rank, but by \autoref{lem_catalecticant_isotropic_point} we know that this is not possible.
Thus, no tight decomposition can contain a power of a linear form associated with an isotropic point.  Now, it remains to prove that every point of a tight decomposition has the same norm, up to roots of unity. Let us consider the tight decomposition
\[
q_n^s=\sum_{j=1}^{T_{n,s}}(\bfb_j\cdot\bfx)^{2s}
\]
with $\bfb_j\in\bbC^n$ for every $i=1,\dots,T_{n,s}$.
Now, $q_n^s$ is invariant under the action of $\SO_n(\bbC)$, which acts transitively on the set of non-isotropic points with fixed norm of $\bbC^n$. We can then assume that
\[
\bfb_1=C_0\bfe_1=C_0(1,0,\dots,0)\in\bbC^n
\]
for some $C_0\in\bbC$. If we look at the polynomial 
\[
f_1=q_n^{s}-(\bfb_1\cdot\bfx)^{2s},
\]
then we have that $\det(\Cat_{f_1,s})$
is a polynomial in the variable $C_0$ of degree $2s$. Furthermore, since the form 
\[
(\bfb_1\cdot\bfx)^{2s}=C_0^{2s}(\bfe_1\cdot\bfx)^{2s}
\] 
would not change by multiplying $C_0$ by any $2s$-th root of unity, then the roots of $\det(\Cat_{f_1,s})$ are given by a unique value up to multiplication by a $2s$-th root of unity. Thus, by the invariance of $q_n^s$ under the action of $\SO_n(\bbC)$, we get that the complex number $(\bfb_k\cdot\bfb_k)^{s}$
is the same for all $k=1,\dots,T_{n,s}$. In particular, we have by \autoref{prop_norm_points_decomposition} that this value is real, namely,
\[
(\bfb_k\cdot\bfb_k)^{s}=\frac{1}{T_{n,s}}\prod_{j=0}^{s-1}\frac{2j+n}{2j+1}.\qedhere
\]
\end{proof}

Given any tight decomposition 
\[
q_n^s=\sum_{k=1}^{T_{n,s}}(\bfb_k\cdot\bfx)^{2s},
\]
we denote by $B_{n,s}$ the value obtained in  \autoref{teo_tight_implies_first_caliber}, 
i.e.,
\begin{equation}
\label{rel_value_norm_tight}
B_{n,s}=\frac{1}{T_{n,s}}\prod_{j=0}^{s-1}\frac{2j+n}{2j+1}.
\end{equation}

\subsection{Tight decomposition in two variables}
\label{sec_tight decomp_two_variables}
In the two-variable case, the rank of the powers of the quadric $q_2$ is fully known and has been proved by B.~Reznick in \cite{Rez92}*{Theorem 9.5}, where he gives all possible real decompositions, which turn out to be unique up to a real orthogonal transformation. In this section we deal with this fact from the point of view of apolarity and extend it to the complex field, proving that the unique real decomposition is still unique for the complex case.
\label{section n=2}
We know from \autoref{Teo Apolar ideal} that the apolar ideal of $q_2^s$ is
\[
\Ann(q_2^s)=(\calH_{2,{s+1}}).
\]
So first we have to determine a basis of the space $\calH_{2,{s+1}}$. In general, for any $n\in\bbN$, we have by \autoref{cor dimens harm} that the dimension of the $d$-harmonic polynomials in $n$ variables is
\[
\dim{\calH_{n,d}}=\binom{d+n-1}{n-1}-\binom{d+n-3}{n-1},
\]
for every $d\in\bbN$.
Therefore, if we restrict to the case of two variables, we get that $\dim{\calH_{2,d}}=2$. 
Now, let us consider in $\calD_{n,1}$ the polynomials 
\[
u=\frac{y_1+\rmi y_2}{2},\qquad v=\frac{y_1-\rmi y_2}{2},
\] 
multiples of the polynomials \eqref{rel_polynomials_u1_v1}. Then we get a new basis of $\calD_{2,1}$
and thus, by a simple change of variables, we have $\calD_2=\bbC[y_1,y_2]\simeq\bbC[u,v]$.
So it follows from \autoref{Rem polinomi armoni u v} that
\[
\Ann(q_2^s)=(u^{s+1},v^{s+1}).
\]
When dealing with complex numbers, we denote by $\Im(z)$ the imaginary part of any number $z\in\bbC$.

\begin{lem}
\label{prop real roots iif u v complex}
For every $a,b\in\bbC$, let $u_1$ and $u_2$ be the complex values
\[
u_1=a+\rmi b,\qquad u_2=a-\rmi b.
\]
Then the following conditions are equivalent:
\begin{enumerate}[label=(\arabic*), left= 0pt, widest=*,nosep]
\item $a,b\in\bbR$; 
\item $\bar{u}_1=u_2$.
\end{enumerate}
\end{lem}
\begin{proof}
(1) $\Rightarrow$ (2) is trivial. Conversely, if $\bar{u}_1=u_2$, then we have
\[
a+\rmi b=\bar{a}+\rmi \bar{b}
\]
and thus
\[
2\rmi\Im(a)=a-\bar{a}=\rmi\bigl(\bar{b}-b\bigr)=-2\rmi^2\Im(b)=2\Im(b).
\]
That is,
\[
\Im(a)=\Im(b)=0,
\]
and hence $a,b\in\bbR$.
\end{proof}
\autoref{prop real roots iif u v complex} can be generalized to projective points. In particular, we can characterize the points in coordinates $\{u,v\}$ such that they correspond to real projective points in coordinates $\{y_1,y_2\}$.

\begin{lem}
\label{lem_u_v_complex}
Let $a,b\in\bbC$ be such that $(a,b)\neq (0,0)$, let $[a:b]\in\bbP^1(\bbC)$ be the projective point associated to the pair $(a,b)\in\bbC^2$, and let $u_1,u_2\in\bbC$ be the values
\[
u_1=a+\rmi b,\qquad u_2=a-\rmi b.
\]
Then the following conditions are equivalent:
\begin{enumerate}[label=(\arabic*), left= 0pt, widest=*,nosep]
\item $[a:b]\in\bbP^1(\bbR)$, i.e., $b$ can be written as a real multiple of $a$ or $a=0$;
\item $[u_1:u_2]=[u_0:\bar{u}_0]$ for a suitable $u_0\in\bbC$;
\item $|u_1|=|u_2|$.
\end{enumerate}
\end{lem}
\begin{proof}
The implications of (1) $\Rightarrow$ (2) and (2) $\Rightarrow$ (3) are trivial. So, let $u_1$ and $u_2$ be such that $|u_1|=|u_2|$. If $a=0$ or $b=0$, then the statement (1) is clear. If instead $a,b\neq 0$, then we have
\begin{align*}
|u_1|^2-|u_2|^2&=(a+\rmi b)\bigl(\bar{a}-\rmi\bar{b}\bigr)-(a-\rmi b)\bigl(\bar{a}+\rmi\bar{b}\bigr)\\
&=|a|^2+\rmi\bar{a}b-\rmi a\bar{b}+|b|^2-|a|^2-\rmi a\bar{b}+\rmi\bar{a}b-|b|^2\\
&=2\rmi\bigl(\bar{a}b-a\bar{b}\bigr)=0,
\end{align*}
that is
\[
a^2=\pt*{\frac{|a|}{|b|}}^2b^2.
\]
Therefore, we have 
\[
a=\pm\frac{|a|}{|b|}b
\]
and therefore $a$ is a real multiple of $b$.
\end{proof}
As a consequence of \autoref{lem_u_v_complex}, we get that, given the coordinate function
\[
\begin{tikzcd}[row sep=0pt,column sep=1pc]
 v\colon \bbC^2 \arrow{r} & \bbC \\
  {\hphantom{v\colon{}}} (y_1,y_2) \arrow[mapsto]{r} & y_1-\rmi y_2
\end{tikzcd}
\]
and the conjugate coordinate function 
\[
\begin{tikzcd}[row sep=0pt,column sep=1pc]
 \bar{u}\colon \bbC^2 \arrow{r} & \bbC\hphantom{,} \\
  {\hphantom{v\colon{}}} (y_1,y_2) \arrow[mapsto]{r} & \bar{y}_1-\rmi\bar{y}_2,
\end{tikzcd}
\]
a projective point $[a: b]\in\bbP^1(\bbC)$ has a real representative pair if and only if $v(a,b)=\bar{u}(a,b)$.
Without loss of generality, we will refer to real roots of a polynomial whenever this condition is satisfied. Now we can provide another proof for the determination of the rank $q_2^s$, in addition to the one proposed by B.~Reznick in \cite{Rez92}*{Theorem 9.5}, which is given in the following theorem. In this case we determine an appropriate decomposition by \autoref{Lem Apo}.
\begin{teo}
\label{Teo rank q^s n=2}
For every $s\in\bbN$
\[
\rk(q_2^s)=s+1.
\]
\end{teo}
\begin{proof}
According to point (2) of \autoref{Lem Apo}, \autoref{prop lower bound}, and \autoref{cor_elementi_grado_min_s}, to prove the statement we only need to determine a polynomial of degree $s+1$ belonging to the apolar ideal of $q_2^s$ and having $s+1$ distinct roots. 
So, let us consider the polynomial
\[
f=u^{s+1}-v^{s+1}.
\]
Then, for every $(u_0,v_0)\in\bbC^2$, we have that $f(u_0,v_0)=0$ if and only if $u_0=v_0=0$ or
\[
\pt*{\frac{u_0}{v_0}}^{s+1}=1,
\] 
that is
\[
\frac{u_0}{v_0}=\rme^{\rmi\frac{2(j-1)\uppi }{s+1}}
\]
for some $j=1,\dots,s+1$.
Consequently, we have $s+1$ distinct roots, corresponding to the projective points
\[
[u_j:v_j]=\Bigl[\rme^{\rmi\frac{2(j-1)\uppi }{s+1}}:1\Bigr]=\Bigl[\rme^{\rmi\frac{(j-1)\uppi}{s+1}}:\rme^{-\rmi\frac{(j-1)\uppi }{s+1}}\Bigr]
\]
for $j=1,\dots,s+1$.
Now, by \autoref{lem_u_v_complex}, we can write these roots using the coordinates $\{y_1,y_2\}$, as real points. In particular, considering the value
\[
\tau_j=\frac{(j-1)\uppi}{s+1}
\]
for every $j=1,\dots,s+1$, we can write the points as
\[
[y_{1,j}:y_{2,j}]=[\cos\tau_j:\sin \tau_j]
\]
for $j=1,\dots,s+1$.
We conclude that $f$ has $s+1$ distinct roots and hence $\rk(q_2^s)=s+1$.
\end{proof}
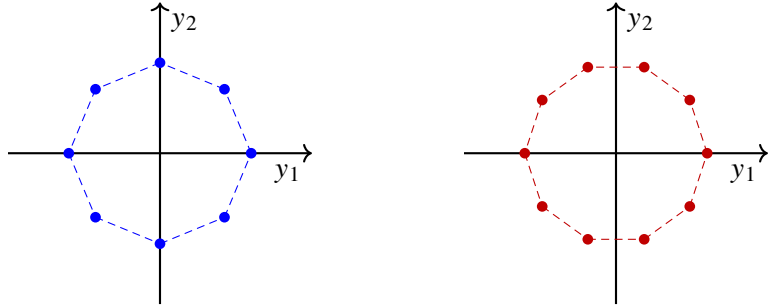
\begin{figure}[ht]
\center
\begin{tikzpicture}
\begin{scope}[shift={(-3cm,0cm)}]
\draw[thick,->] (-2,0) -- (2,0) node[anchor=north east]{$y_1$};
\draw[thick,->] (0,-2) -- (0,2) node[anchor=north west]{$y_2$};
\foreach \x in {0,45,...,315}{
\pgfpathmoveto{\pgfpointpolar{\x}{1.2cm}}
\pgfpathlineto{\pgfpointpolar{\x+45}{1.2cm}}
\color{blue}
\pgfsetdash{{3pt}{2pt}}{0pt}
\pgfusepath{fill,stroke}}
\foreach \x in {0,45,...,315}
{\pgfpathcircle{\pgfpointpolar{\x}{1.2cm}}{2pt}
\color{blue}
\pgfusepath{fill}};
\end{scope}
\begin{scope}[shift={(3cm,0cm)}]
\draw[thick,->] (-2,0) -- (2,0) node[anchor=north east]{$y_1$};
\draw[thick,->] (0,-2) -- (0,2) node[anchor=north west]{$y_2$};
\foreach \x in {0,36,...,324}{
\pgfpathmoveto{\pgfpointpolar{\x}{1.2cm}}
\pgfpathlineto{\pgfpointpolar{\x+36}{1.2cm}}
\color{darkred}
\pgfsetdash{{3pt}{2pt}}{0pt}
\pgfusepath{fill,stroke}}
\foreach \x in {0,36,...,324}
{\pgfpathcircle{\pgfpointpolar{\x}{1.2cm}}{2pt}
\color{darkred}
\pgfusepath{fill}};
\end{scope} 
\end{tikzpicture}
\caption[Examples of decompositions for the polynomials $q_2^3$ and $q_2^4$]{Examples of points of decompositions of the forms $q_2^3$ and $q_2^4$. The blue octagon on the left gives the $4$ projective points obtained as roots of the polynomial $u^4-v^4$, while the red decagon on the right gives the $5$ projective points obtained as roots of the polynomial $u^5-v^5$ (points opposite to the origin represent the same point in $\bbP^1(\bbC)$).}
\label{fig_examples_decomp_n=2}
\end{figure}

The roots of the polynomial used in the proof of \autoref{Teo rank q^s n=2} are all real and hence provide a real decomposition of $q_2^s$, whose elements correspond to the projective classes of the $2(s+1)$-th roots of unity. Equivalently, these points correspond to the vertices of a regular $2(s+1)$-gon (see \autoref{fig_examples_decomp_n=2}), inscribed in a circle with radius equal to
\[
B_{2,s}^{\frac{1}{2s}}=(s+1)^{-\frac{1}{2s}}\prod_{j=0}^{s-1}\biggl(\frac{2j+2}{2j+1}\biggr)^{\frac{1}{2s}}=2(s+1)^{-\frac{1}{2s}}\binom{2s}{s}^{-\frac{1}{2s}}. 
\]
This last fact is obtained from \autoref{teo_tight_implies_first_caliber}.

By analyzing all the polynomials of minimal degree $s+1$ with distinct real
 roots, which are contained in $\Ann(q_2^s)$, we can then determine all the real decompositions of $q_2^s$. The procedure simply consists of resolving a real system and so determining, by \autoref{Lem Apo}, the coefficients of the powers of the linear forms associated to each root.

\begin{prop}
\label{Prop polynomials distinct real roots}
Let $f=au^{s+1}+bv^{s+1}\in\Ann(q_2^s)$ be a non-zero polynomial of degree $s+1$, with $a,b\in\bbC$. Then $f$ has $s+1$ real distinct roots if and only if $|a|=|b|\neq 0$. In particular, it can be written, up to scalars, as
\[
f=u^{s+1}-\rme^{\rmi\theta}v^{s+1}
\]
for some $\theta\in[0,2\uppi)$ and its roots correspond to the projective points
\[
[y_{1,j}:y_{2,j}]=[\cos{\tau_{\theta,j}}:\sin \tau_{\theta,j}], 
\] 
where
\[
\tau_{\theta,j}=\frac{2(j-1)\uppi+\theta}{2(s+1)}
\]
for $j=1,\dots,s+1$.
\end{prop}
\begin{proof}
If the polynomial $f$ has a real root $(y_{1,0},y_{2,0})\neq(0,0)$, then 
\[
u_0=\bar{v}_0\neq 0
\]
and hence 
\[
au_0^{s+1}+b\bar{u}_0^{s+1}=0.
\]
This clearly means that $a,b\neq 0$ and $|a|=|b|$. 
Conversely, if $a$ and $b$ are complex numbers such that $|a|=|b|\neq 0$, then there exists a real number $\theta\in[0,2\uppi)$ such that $f$ can be written, up to a scalar $a$, as
\[
f=u^{s+1}-\rme^{\rmi\theta}v^{s+1}
\]
where
\[
\rme^{\rmi\theta}=-\frac{b}{a}.
\]
Thus, for every pair $(u_0,v_0)\in\bbC^2$, we have $f(u_0,v_0)=0$ if and only if $u_0=v_0=0$ or
\[
\pt*{\frac{u_0}{v_0}}^{s+1}=\rme^{\rmi\theta}.
\]
This means that there are $s+1$ roots $[u_j:v_j]\in\bbP^1(\bbC)$ for which we have
\[
u_j=\rme^{\rmi\frac{2(j-1)\uppi+\theta}{s+1}}v_j
\]
for every $j=1,\dots,s+1$. That is, the $s+1$ roots of $f$ are
\[
[u_j:v_j]=\Bigl[\rme^{\rmi\frac{2(j-1)\uppi+\theta}{s+1}}:1\Bigr]=\Bigl[\rme^{\rmi\frac{2(j-1)\uppi+\theta}{2(s+1)}}:\rme^{-\rmi\frac{2(j-1)\uppi+\theta}{2(s+1)}}\Bigr],
\]
for $j=1,\dots,s+1$.
So, introducing the notation
\[
\tau_{\theta,j}=\frac{2(j-1)\uppi+\theta}{2(s+1)}
\]
we can write
\[
[u_j:v_j]=[\rme^{\rmi\tau_{\theta,j}}:\rme^{-\rmi\tau_{\theta,j}}],
\]
for $j=1,\dots,s+1$. So by changing the variables, we get
the $s+1$ real distinct roots in standard coordinates, namely,
\[
[y_{1,j}:y_{2,j}]=[\cos{\tau_{\theta,j}}:\sin \tau_{\theta,j}],
\]
for $j=1,\dots,s+1$
\end{proof}
Note that in the case of $\theta=0$ we obtain the same polynomial that we used to prove \autoref{Teo rank q^s n=2}. In general, the roots we get correspond to the projective classes of the vertices of a regular $2(s+1)$-gon.

Now we have all the elements to obtain all the minimal real decompositions of the form $q_2^s$, exploiting all the roots of the polynomials considered in \autoref{Prop polynomials distinct real roots}.
\begin{teo}
\label{Teo real decompositions n=2}
The form $q_2^s$ has a unique real decomposition, up to orthogonal real transformations, whose terms correspond to the vertices of a regular $2s$-gon inscribed in a circle with radius equal to
\[
r=2(s+1)^{-\frac{1}{2s}}\binom{2s}{s}^{-\frac{1}{2s}}.
\]
Namely
\[
q_2^s=\sum_{j=1}^{s+1}\bigl(r\cos(\tau_{j})x_1+r\sin(\tau_{j})x_2\bigr)^{2s},
\]
where 
\[
\tau_j=\frac{(j-1)\uppi}{s+1}
\]
for every $j=1,\dots,s+1$.
\end{teo}
\begin{proof}
By \autoref{teo_tight_implies_first_caliber}, we know that every decomposition of $q_2^s$ is first caliber and every point has norm equal to
\[
r=2(s+1)^{-\frac{1}{2s}}\binom{2s}{s}^{-\frac{1}{2s}}.
\]
Moreover, by \autoref{Prop polynomials distinct real roots}, the polynomials with real distinct roots in the apolar ideal of $q_2^s$ are given by all the linear combinations of the type
\[
u^{s+1}-\rme^{\rmi\theta}v^{s+1},\qquad \theta\in[0,2\uppi).
\]
Thus, introducing the variables
\[
z_1=x_1-\rmi x_2,\qquad z_2=x_1+\rmi x_2,
\]
and considering the values
\[
\tau_{j,\theta}=\frac{(j-1)\uppi}{s+1}+\frac{\theta}{2(s+1)}
\]
for every $j=1,\dots,s+1$, we have
\begin{align*}
q_2^s=(x_1^2+x_2^2)^s=z_1^{s}z_2^{s}&=\sum_{j=1}^{s+1}\biggl(\dfrac{r}{2}\biggr)^{2s}(\rme^{\rmi\tau_{j,\theta}}z_1+\rme^{-\rmi\tau_{j,\theta}}z_2)^{2s}\\
&=\sum_{j=1}^{s+1}\biggl(\dfrac{r}{2}\biggr)^{2s}\Biggl(2\biggl(\frac{\rme^{\rmi\tau_{j,\theta}}+\rme^{-\rmi\tau_{j,\theta}}}{2}\biggr)x_1+2\biggl(\frac{\rme^{\rmi\tau_{j,\theta}}-\rme^{-\rmi\tau_{j,\theta}}}{2\rmi}\biggr)x_2\Biggr)^{2s}\\
&=\sum_{j=1}^{s+1}r^{2s}\big(\cos(\tau_{j,\theta}) x_1+\sin(\tau_{j,\theta})x_2\big)^{2s}.
\end{align*}
It remains to prove that all these decompositions are unique up to orthogonal transformation. Since the form $q_2^s$ is invariant under the action of the orthogonal group $\Oa_2(\bbR)$ (where the field $\bbR$ of real numbers must be considered for standard coordinates), we can consider the action of the matrix
\[
A_{\theta}=\begin{pmatrix}
\rme^{\rmi\frac{\theta}{2(s+1)}}&0\\
0&\rme^{-\rmi\frac{\theta}{2(s+1)}}
\end{pmatrix}.
\]
Then, we observe that
\begin{align*}
q_2^s=z_1^{s}z_2^{s}&=\sum_{j=1}^{s+1}\biggl(\dfrac{r}{2}\biggr)^{2s}(\rme^{\rmi\tau_{j,\theta}}z_1+\rme^{-\rmi\tau_{j,\theta}}z_2)^{2s}\\
&=\sum_{j=1}^{s+1}\biggl(\dfrac{r}{2}\biggr)^{2s}\bigl(\rme^{\rmi\tau_{j,0}}(A_{\theta}\cdot z_1)+\rme^{-\rmi\tau_{j,0}}(A_{\theta}\cdot z_2)\bigr)^{2s}\\
&=\sum_{j=1}^{s+1}\biggl(\dfrac{r}{2}\biggr)^{2s}(\rme^{\rmi\tau_{j}}z_1+\rme^{-\rmi\tau_{j}}z_2)^{2s}\\
&=\sum_{j=1}^{s+1}\big(r\cos(\tau_{j}) x_1+r\sin(\tau_{j})x_2\big)^{2s},
\end{align*}
where we set $\tau_j=\tau_{j,0}$ for every $j=1,\dots,s+1$.
\end{proof}
It is not so difficult to generalize the previous results to the complex case. In particular, we have to find all the polynomials of minimal degree $s+1$ with distinct roots in the apolar ideal $\Ann(q_2^s)$. Then, in the same way, we get all explicit decompositions of the form $q_2^s$. 
\begin{prop}
\label{Prop polynomials complex roots}
Let $f=au^{s+1}+bv^{s+1}\in\Ann(q_2^s)$ be a nonzero polynomial  such that $a,b\neq 0$. Then $f$ can be written, up to scalars, as
\[
f=u^{s+1}-\rme^{\rmi(\theta+\rmi k)}v^{s+1},
\]
for some $k\in\bbR$ and some $\theta\in[0,2\uppi)$. Moreover, it has $s+1$ distinct roots corresponding to
\[
[y_{1,j}:y_{2,j}]=[\cos(w_{k,\theta,j}):\sin(w_{k,\theta,j})],
\] 
where
\[
w_{k,\theta,j}=\frac{2(j-1)\uppi+\theta+\rmi k}{2(s+1)},
\]
for $j=1,\dots,s+1$.
\end{prop}
\begin{proof}
Since $f\neq 0$, we can assume, up to multiplication by a scalar, that
\[
f=u^{s+1}-c_0v^{s+1},
\]
where
\[
c_0=-\frac{b}{a}\in\bbC.
\]
Moreover, denoting by $k'\in\bbR_{>0}$ the norm of $c_0$, we can write
\[
c_0=k'\rme^{\rmi\theta}=\rme^{\rmi(\theta-\rmi\log k')}=\rme^{\rmi(\theta+\rmi k)}
\]
for some $\theta\in[0,2\uppi)$, with $k=-\log k'\in\bbR$. This gives us the desired form
\[
f=u^{s+1}-\rme^{\rmi(\theta+\rmi k)}v^{s+1}.
\]
Now, we proceed exactly as in the proof of \autoref{Prop polynomials distinct real roots}. We observe that $[u_0:v_0]\in\bbP^1(\bbC)$ is a root of $f$ if and only if
\[
\pt*{\frac{u_0}{v_0}}^{s+1}=\rme^{\rmi(\theta+\rmi k)}.
\]
So we can determine $s+1$ distinct roots of $f$, corresponding to
\[
[u_j:v_j]=\Bigl[\rme^{\rmi\frac{2(j-1)\uppi+\theta+\rmi k}{s+1}}:1\Bigr]=\Bigl[\rme^{\rmi\frac{2(j-1)\uppi+\theta+\rmi k}{2(s+1)}}:\rme^{-\rmi\frac{2(j-1)\uppi+\theta+\rmi k}{2(s+1)}}\Bigr],
\]
for $j=1,\dots,s+1$,
which can be written by introducing the term
\[
w_{k,\theta,j}=\frac{2(j-1)\uppi+\theta+\rmi k}{2(s+1)}
\]
as
\[
[u_j:v_j]=[\rme^{\rmi w_{k,\theta,j}}:\rme^{-\rmi w_{k,\theta,j}}], 
\]
for every $j=1,\dots,s+1$.
So if we rewrite each pair in standard coordinates, we get
\begin{align*}
y_{1,j}&=u_j+v_j=\rme^{\rmi w_{k,\theta,j}}+\rme^{-\rmi w_{k,\theta,j}}=2\cos(w_{k,\theta,j}),\\
y_{2,j}&=-\rmi(u_j-v_j)=-\rmi\bigl(\rme^{\rmi w_{k,\theta,j}}-\rme^{-\rmi w_{k,\theta,j}}\bigr)=2\sin(
w_{k,\theta,j}).
\end{align*}
So the roots of the polynomial $f$ are given by
\[
[y_{1,j}:y_{2,j}]=[\cos(w_{k,\theta,j}):\sin(w_{k,\theta,j})],
\]
for $j=1,\dots,s+1$.
\end{proof}
It remains to determine the other decompositions of $q_2^s$. The procedure we use is exactly the same as in the real case.
\begin{teo}
\label{Teo all complex decompositions q2s}
The form $q_2^s$ has a unique decomposition which, up to complex orthogonal transformations, corresponds to the real decomposition whose terms are given by the pairs of opposite vertices of a regular $2s$-gon inscribed in a circle of radius equal to
\[
r=2(s+1)^{-\frac{1}{2s}}\binom{2s}{s}^{-\frac{1}{2s}}.
\]
Namely,
\[
q_2^s=\sum_{j=1}^{s+1}\big(r\cos(\tau_{j}) x_1+r\sin(\tau_{j})x_2\big)^{2s},
\]
where 
\[
\tau_j=\frac{(j-1)\uppi}{s+1},
\]
for every $j=1,\dots,s+1$.
\end{teo}
\begin{proof}
By \autoref{Prop polynomials complex roots}, we can describe any minimal decomposition of $q_2^s$ as a sum of powers of linear forms corresponding to the distinct roots of a polynomial
\[
u^{s+1}-\rme^{\rmi(\theta+\rmi k)}v^{s+1},
\]
for some $\theta\in[0,2\uppi)$ and some $k\in\bbR$. This means that, given the set of variables $\{z_1,z_2\}$, introduced in the proof of \autoref{Teo real decompositions n=2}, and the values
\[
w_{k,\theta,j}=\frac{2(j-1)\uppi+\theta+\rmi k}{2(s+1)},
\]
for every $j=1,\dots,s+1$, we can write $q_2^s$ as
\begin{align*}
q_2^s=(x_1^2+x_2^2)^s=z_1^{s}z_2^{s}&=\sum_{j=1}^{s+1}\biggl(\frac{r}{2}\biggr)^{2s}(e^{iw_{k,\theta,j}}z_1+e^{-iw_{k,\theta,j}}z_2)^{2s}\\
&=\sum_{j=1}^{s+1}\biggl(\dfrac{r}{2}\biggr)^{2s}\Biggl(2\biggl(\frac{\rme^{\rmi w_{k,\theta,j}}+\rme^{-\rmi w_{k,\theta,j}}}{2}\biggr)x_1+2\biggl(\frac{\rme^{\rmi w_{k,\theta,j}}-\rme^{-\rmi w_{k,\theta,j}}}{2\rmi}\biggr)x_2\Biggr)^{2s}\\
&=\sum_{j=1}^{s+1}\bigl(r\cos(w_{k,\theta,j})x_1+r\sin(w_{k,\theta,j})x_2\bigr)^{2s}.
\end{align*}
Therefore, by the invariance of $q_2^s$ under the action of $\Oa_2(\bbC)$, we conclude by \autoref{Teo real decompositions n=2} that every minimal representation is obtained by the action of a complex orthogonal transformation on the real decomposition
\[
q_2^s=\sum_{j=1}^{s+1}\big(r\cos(\tau_{j}) x_1+r\sin(\tau_{j})x_2\big)^{2s},
\]
where 
\[
r=2(s+1)^{-\frac{1}{2s}}\binom{2s}{s}^{-\frac{1}{2s}}.
\]
That is, for every $\theta\in[0,2\pi)$ and $k\in\bbR$, we have the equality
\[
q_2^s=\sum_{j=1}^{s+1}\big(r\cos(w_{k,\theta,j}) x_1+r\sin(w_{k,\theta,j})x_2\big)^{2s}.\qedhere
\]
\end{proof}

\subsection{Tight decompositions in many variables}
\label{sec_general_tight_decomp}
By analyzing the middle catalecticant matrices we are able to extend some of the results obtained for the real case, as we did in the previous section. In particular, for the second and third powers of $q_n$, we can exclude the existence of tight decompositions for several cases. The strategy consists in trying to determine suitable decompositions by finding the possible points, all with the same norm, contained in the kernel of the central catalecticant map of $q_n^s-B_{n,s}x_1^{2s}$, which must have dimension $1$. This is exactly the same strategy used by B.~Reznick in \cite{Rez92}, but thanks to \autoref{teo_tight_implies_first_caliber} we can approach the complex case.

Considering the second power of a quadratic form, our strategy is to subtract a summand from a hypothetical tight decomposition of the form $q_n^s$, so that the rank of the middle catalecticant matrix of $q_n^s$ decreases by $1$. By \autoref{teo_tight_implies_first_caliber}, we know that the $s$-power of the value $\bfa\cdot\bfa$ is the same for every point $\bfa\in\bbC^n$ of a tight decomposition of $q_n^s$, namely
\[
B_{n,s}=(\bfa\cdot\bfa)^s=\frac{1}{T_{n,s}}\prod_{j=0}^{s-1}\frac{2j+n}{2j+1}=\frac{s!(n-1)!}{(s+n-1)!}\prod_{j=0}^{s-1}\frac{2j+n}{2j+1}.
\] 
Thus, assuming the existence of a tight decomposition, we can consider the form 
\[
\frac{1}{B_{n,s}}q_n^{s},
\]
so that every point of the decomposition has norm $1$, up to roots of unity. In particular, by the transitivity of the orthogonal group action on the non-isotropic points on the complexified sphere $\rmS^{n-1}_{\bbC}$, we can assume that one of them is the point
\[
\bfa_1=(1,0,\dots,0).
\]

Now we will use the notation 
\[
q_n^{[s]}=\frac{1}{2^ss!}q_n^s
\]
and consider the case where the exponent is $s=2$. In particular, we have
\[
B_{n,2}=\frac{2(n+2)}{3(n+1)}
\]
and it follows by \autoref{teo_tight_implies_first_caliber} that if there is a tight decomposition of $q_n^2$, the middle catalecticant matrix of the polynomial
\[
f_1=\frac{1}{B_{n,2}}q_n^{2}-(\bfa\cdot\bfx)^4,
\]
where $\bfa\in\bbC^n$ is such that $\bfa\cdot\bfa=1$, up to multiplying by a root of unity, must have rank equal to $T_{n,2}-1$. Moreover, we can determine exactly how the kernel of $\Cat_{f_1,2}$ is made.
\begin{lem}
\label{lem_kernel_tight_decomposition}
Let $n\in\bbN$ and let 
\[
f_1=\frac{1}{B_{n,2}}q_n^{2}-(\bfa\cdot\bfx)^4,
\]
for some $\bfa\in\bbC^n$ such that $\bfa\cdot\bfa=1$. Then
\[
\Ker(\Cat_{f_1,2})=\langle (n+2)(\bfa\cdot\bfy)^2-q_n\rangle.
\]
\end{lem}
\begin{proof}
Using \autoref{lem contraz forme}, formula \eqref{rel_Laplace_on_q_n^s} and the fact that $\bfa\cdot\bfa=1$, we simply observe that
\begin{align*}
\bigl((n+2)(\bfa\cdot\bfy)^2-q_n\bigr)\circ f_1&=\frac{n+2}{B_{n,2}}\bigl((\bfa\cdot\bfy)^2\circ q_n^2\bigr)-(n+2)\bigl((\bfa\cdot\bfy)^2\circ (\bfa\cdot\bfx)^4\bigr)\\
&\hphantom{{}={}}-\frac{1}{B_{n,2}}\bigl(q_n\circ q_n^{2}\bigr)+\bigl(q_n\circ(\bfa\cdot\bfx)^4\bigr)\\
&=\frac{n+2}{B_{n,2}}\bigl(4q_n+8(\bfa\cdot\bfx)^2\bigr)-12(n+2)(\bfa\cdot\bfx)^2\\
&\hphantom{{}={}}-\frac{4(n+2)}{B_{n,2}}q_n+12(\bfa\cdot\bfx)^2\\
&=\frac{8(n+2)}{B_{n,2}}(\bfa\cdot\bfx)^2-12(n+1)(\bfa\cdot\bfx)^2=0,
\end{align*}
which proves the statement.
\end{proof}
\autoref{lem_kernel_tight_decomposition} can be generalized to an arbitrary exponent $s$. We can prove that the element generating the kernel is given by a product of quadratic forms multiplied, whenever $s$ is odd, by the variable $y_1$.
\begin{lem}
\label{lem_kernel_product_quadrics}
Let $n,s\in\bbN$ and let 
\[
f_1=\frac{1}{B_{n,2}}q_n^{s}-x_1^{2s}.
\]
Then,
\[
\Ker(\Cat_{f_1,s})=\biggl\langle y_1^{\pint*{\frac{s+1}{2}}-\pint*{\frac{s}{2}}}\prod_{k=1}^{\pint*{\frac{s}{2}}}\biggl(y_1^2-a_k\sum_{j=2}^ny_j^2\biggr)\biggr\rangle,
\]
for some $a_k\in\bbC$ for every $k=1,\dots,{\pint*{\frac s2}}$.
\end{lem}
\begin{proof}
Let $g$ be a generator of $\Ker(\Cat_{f_1,s})$. Then $g$ must be invariant under the action of $\SO_{n-1}(\bbC)$  and, since it can be written as
\[
g=\sum_{k=0}^sy_1^kg_{s-k}(y_2,\dots,y_n),
\]
where $g_{s-k}\in\calD_{n-1,s-k}$ for every $k=0,\dots,n$, then also each polynomial $g_{s-k}$ must be $\SO_{n-1}(\bbC)$-invariant. Therefore, by the uniqueness of decomposition \eqref{rel_decomp_harmonic_components}, each polynomial $g_{s-k}$ must be a multiple of a power of $q_{n-1}$ and, in particular, we must have $g_{s-k}=0$ whenever $s-k$ is odd. That is,
\[
g=\sum_{j=0}^{\pint*{\frac{s}{2}}}b_jy_1^{s-2j}q_{n-1}^j,
\]
for some $b_j\in\bbC$, for every $j=1,\dots,\pint*{\frac{s}{2}}$.
\end{proof}

\autoref{lem_kernel_tight_decomposition} guarantees that, assuming that the first point is $\bfa_1=(1,0,\dots,0)$, the remaining points must be roots of the polynomial
\[
g_1=(n+2)y_1^2-q_n.
\]
In particular, this means that, given any other point of the decomposition $\bfa_2=(a_{2,1},\dots,a_{2,n})$, we must have 
\[
a_{2,1}=\pm\frac{1}{\sqrt{n+2}}.
\]
Furthermore, for every $j,k=1,\dots,T_{n,2}$ such that $j\neq k$, we have
\begin{equation}
\label{rel_formula_scalar_prodoct_exponent_2}
\bfa_j\cdot\bfa_k=\pm\frac{1}{\sqrt{n+2}}.
\end{equation}
By equation \eqref{rel_formula_Lap_product}, using the same strategy adopted by B.~Reznick in \cite{Rez92}*{pp.~130-132}, we can obtain the result of \autoref{teo_Reznick_tight_decompositions} for the case of exponent $2$, which turns out to be true for complex decompositions as well.
\begin{teo}
\label{teo_tight_decompositions_s=2}
Let $n\geq 3$ and let
\[
\frac{1}{B_{n,2}}q_n^2=\sum_{j=1}^{T_{n,2}}(\bfa_j\cdot\bfx)^4
\]
be a tight decomposition of $q_n^2$. Then $n=3$ or $n=m^2-2$ for a suitable odd number $m\in\bbN$.
\end{teo}
\begin{proof}
Suppose that a point of the decomposition is
\[
\bfa_1=(1,0,\dots,0).
\]
Then, we can assume by \autoref{lem_kernel_tight_decomposition} that the second point of the decomposition is
\[
\bfa_2=\biggl(\frac{1}{\sqrt{n+2}},\sqrt{\frac{n+1}{n+2}},0,\dots,0\biggr).
\]
So, given any other point $\bfa_k$ of the decomposition, which can be written as
\[
\bfa_k=\biggl(\frac{1}{\sqrt{n+2}},a_{k,2},\dots,a_{k,n}\biggr),
\]
for any $k=3,\dots,T_{n,2}$, we have by formula \eqref{rel_formula_scalar_prodoct_exponent_2} that
\[
\bfa_2\cdot\bfa_k=\frac{1}{n+2}+a_{k,2}\sqrt{\frac{n+1}{n+2}}=\pm\frac{1}{\sqrt{n+2}}
\]
and hence
\[
a_{k,2}=\dfrac{-1\pm\sqrt{n+2}}{\sqrt{(n+1)(n+2)}}.
\] 
For these values, we introduce the notation
\[
C_2^+=\dfrac{-1+\sqrt{n+2}}{\sqrt{(n+1)(n+2)}},\qquad C_2^-=\dfrac{-1-\sqrt{n+2}}{\sqrt{(n+1)(n+2)}}.
\]
We define the natural number $N_2^{-}\in\bbN$ as the number of elements of the decomposition with the second coordinate equal to
$C_2^-$.
Consequently, exactly $N_2^+=T_{n,s}-N_2^{-}-2$ elements must have the second coordinate equal to $C_2^+$.
Developing the form 
and solving an equation for the coefficients of the monomial $x_2^4$, we get the equation
\[
\dfrac{(n+1)^2}{(n+2)^2}+(T_{n,s}-N_2^{-}-2)\dfrac{\bigl(-1+\sqrt{n+2}\bigr)^4}{(n+1)^2(n+2)^2}+N_2^{-}\dfrac{\bigl(-1-\sqrt{n+2}\bigr)^4}{(n+1)^2(n+2)^2}=\dfrac{3(n+1)}{2(n+2)},
\]
which can be written as
\[
2(n+1)^4+2(T_{n,s}-N_2^{-}-2)\bigl(-1+\sqrt{n+2}\bigr)^4+2N_2^{-}\bigl(-1-\sqrt{n+2}\bigr)^4=3(n+1)^3(n+2),
\]
by which we get
\[
-(n+1)^3(n+4)+2(T_{n,s}-N_2^{-}-2)\bigl(-1+\sqrt{n+2}\bigr)^4+2N_2^{-}\bigl(-1-\sqrt{n+2}\bigr)^4=0.
\]
Replacing the value
\[
T_{n,2}=\binom{n+1}{2},
\]
we get
\[
-(n+1)^3(n+4)+\bigl(n(n+1)-2N_2^{-}-4\bigr)\bigl(-1+\sqrt{n+2}\bigr)^4+2N_2^{-}\bigl(-1-\sqrt{n+2}\bigr)^4=0,
\]
from which we obtain
\begin{align*}
N_2^{-}&=\dfrac{-(n+1)^3(n+4)+(n^2+n-4)\bigl(-1+\sqrt{n+2}\bigr)^4}{2\bigl(-1+\sqrt{n+2}\bigr)^4-2\bigl(-1-\sqrt{n+2}\bigr)^4}\\[1ex]
&=\dfrac{(n+1)^3(n+4)-(n^2+n-4)\bigl(n^2+10n+17-4(n+3)\sqrt{n+2}\bigr)}{16(n+3)\sqrt{n+2}}\\[1ex]
&=\dfrac{-n^3-2n^2+9n+18+(n^2+n-4)(n+3)\sqrt{n+2}}{4(n+3)\sqrt{n+2}}\\[1ex]
&=\dfrac{-(n+2)(n-3)(n+3)+(n^2+n-4)(n+3)\sqrt{n+2}}{4(n+3)\sqrt{n+2}}\\[1ex]
&=\dfrac{n^2+n-4-(n-3)\sqrt{n+2}}{4}.
\end{align*}
Since $N_2^{-}$ is a natural number, if the equality
\[
4N_2^{-}-n^2-n+4=(3-n)\sqrt{n+2}
\]
holds, then $n=3$ or $\sqrt{n+2}\in\bbN$, i.e., $n=m^2-2$ for a suitable $m\in\bbN$. So, assuming $n\geq 4$ and setting $m=\sqrt{n+2}$, we determine the value of $N_2^-$, obtaining
\[
N_2^{-}=\dfrac{(m^2-2)^2+m^2-6-m(m^2-5)}{4}=\dfrac{m^4-m^3-3m^2+5m-2}{4}=\dfrac{(m-1)^3(m+2)}{4}.
\]
By simply substituting modular values, we can check that $m\not\equiv 0\bmod 4$.
So, there are exactly $N_2^-$ points $\bfa_{3},\dots,\bfa_{N_2^{-}+2}$ such that
\[
a_{k,2}=C_2^-=-\frac{1}{m}\sqrt{\frac{m+1}{m-1}}
\]
for every $k=3,\dots,N_2^{-}+2$. 
Consequently, the number of points having $C_2^+$ as second coordinate is
\begin{align*}
N_2^+&=\binom{m^2-1}{2}-2-N_2^-=\frac{(m^2-1)(m^2-2)}{2}-2-\dfrac{(m-1)^3(m+2)}{4}\\[1ex]
&=\frac{m^4+m^3-3m^2-5m-2}{4}=\dfrac{(m-2)(m+1)^3}{4}.
\end{align*}
That is, the remaining $N_2^+$ points $\bfa_{N_2^{-}+3},\dots,\bfa_{N_2^{-}+N_2^{+}+2}$ of the decomposition have as second coordinate the value
\[
a_{k,2}=C_2^+=\frac{1}{m}\sqrt{\frac{m-1}{m+1}},
\]
for every 
$k=N_2^{-}+3,\dots,N_2^{-}+N_2^{+}+2$.
So, replacing the value $\sqrt{n+2}$ by $m$, the first two points are
\[\bfa_1=(1,0,\dots,0),\qquad \bfa_2=\biggl(\frac{1}{m},\frac{1}{m}\sqrt{m^2-1},0,\dots,0\biggr),\]
while the remaining $N_2^-+N_2^+$ ones can be written as
\[
\bfa_k=\biggl(\frac{1}{m},-\frac{1}{m}\sqrt{\frac{m+1}{m-1}},a_{k,3},\dots,a_{k,n}\biggr)
\]
for $k=4,\dots,N_2^{-}+2$, and 
\[
\bfa_k=\biggl(\frac{1}{m},\frac{1}{m}\sqrt{\frac{m-1}{m+1}},a_{k,3},\dots,a_{k,n}\biggr)
\]
$k=N_2^{-}+3,\dots,N_2^{-}+N_2^{+}+2$.
In particular, we can assume the third point of the decomposition to be
\[
\bfa_3=\biggl(\frac{1}{m},-\frac{1}{m}\sqrt{\frac{m+1}{m-1}},a_{3,3},0,\dots,0\biggr),
\]
which implies the equation
\[
a_{3,3}^2+\frac{1}{m^2}+\frac{m+1}{m^2(m-1)}=1,
\]
which means that
\[
a_{3,3}^2=\frac{(m-2)(m+1)}{m(m-1)},
\]
and hence
\[
a_{3,3}=\pm\sqrt{\frac{(m-2)(m+1)}{m(m-1)}}.
\]
Without loss of generality, by the invariance of the form $q_n^s$ under orthogonal transformations, we can assume
\[
a_{3,3}=\sqrt{\frac{(m-2)(m+1)}{m(m-1)}}
\]
and hence
\[
\bfa_3=\biggl(\frac{1}{m},-\frac{1}{m}\sqrt{\frac{m+1}{m-1}},\sqrt{\frac{(m-2)(m+1)}{m(m-1)}},0,\dots,0\biggr).
\]
Now, we must have
\[
\bfa_3\cdot\bfa_k=\pm\frac{1}{m}
\]
for every $k=4,\dots,N_2^-+N_2^++2$. In particular, if $k=4,\dots,N_2^{-}+2$, this gives the equation
\[
\biggl(\frac{1}{m},-\frac{1}{m}\sqrt{\frac{m+1}{m-1}},\sqrt{\frac{(m-2)(m+1)}{m(m-1)}},0,\dots,0\biggr)\cdot \biggl(\frac{1}{m},-\frac{1}{m}\sqrt{\frac{m+1}{m-1}},a_{k,3},\dots,a_{k,n}\biggr)=\pm\frac{1}{m},
\]
that is,
\[
\frac{1}{m^2}+\frac{m+1}{m^2(m-1)}+a_{k,3}\sqrt{\frac{(m-2)(m+1)}{m(m-1)}}=\pm\frac{1}{m},
\]
and hence
\[
a_{k,3}=\frac{-2\pm(m-1)}{\sqrt{(m+1)m(m-1)(m-2)}}.
\]
If, instead, $k=N_2^{-}+3,\dots,N_2^{-}+N_2^{+}+2$, we have the equation
\[
\biggl(\frac{1}{m},-\frac{1}{m}\sqrt{\frac{m+1}{m-1}},\sqrt{\frac{(m-2)(m+1)}{m(m-1)}},0,\dots,0\biggr)\cdot \biggl(\frac{1}{m},\frac{1}{m}\sqrt{\frac{m-1}{m+1}},a_{k,3},\dots,a_{k,n}\biggr)=\pm\frac{1}{m},
\]
which gives
\[
a_{k,3}\sqrt{\frac{(m-2)(m+1)}{m(m-1)}}=\pm\frac{1}{m},
\]
and hence
\[
a_{k,3}=\pm\sqrt{\frac{m-1}{m(m-2)(m+1)}}.
\]
Therefore, analogously to the previous case, we consider the values
\begin{gather*}
C_3^{--}=-\frac{m+1}{\sqrt{(m+1)m(m-1)(m-2)}},\quad C_3^{-+}=\frac{m-3}{\sqrt{(m+1)m(m-1)(m-2)}}\\[1ex]
C_3^{+-}=-\sqrt{\frac{m-1}{m(m-2)(m+1)}},\quad C_3^{++}=\sqrt{\frac{m-1}{m(m-2)(m+1)}}
\end{gather*}
and we define the values $N_3^{--}$, $N_3^{-+}$, $N_3^{+-}$, and $N_3^{++}$ as the numbers of points having the value $C_3^{--}$, $C_3^{-+}$, $C_3^{+-}$, and $C_3^{++}$, respectively, as third coordinate. In particular, we have
\[
N_3^{--}+N_{3}^{-+}=N_2^--1=\dfrac{(m-1)^3(m+2)}{4}-1=\dfrac{(m-2)(m^3+m^2-m+3)}{4}
\]
and
\[
N_3^{+-}+N_{3}^{++}=N_2^+=\dfrac{(m+1)^3(m-2)}{4}.
\]
Considering the coefficients of the monomial $x_1x_2x_3^2$, we must have
\begin{align*}
&-\frac{(m-2)(m+1)}{m^3(m-1)}\sqrt{\frac{m+1}{m-1}}-\frac{N_3^{--}(C_3^{--})^2}{m^2}\sqrt{\frac{m+1}{m-1}}-\frac{N_3^{-+}(C_3^{-+})^2}{m^2}\sqrt{\frac{m+1}{m-1}}+\frac{N_2^+(C_3^{++})^2}{m^2}\sqrt{\frac{m-1}{m+1}}=0,
\end{align*}
that is,
\begin{align*}
&\frac{(m-2)(m+1)}{m(m-1)}\sqrt{\frac{m+1}{m-1}}+N_3^{--}\frac{(m+1)^2}{(m+1)m(m-1)(m-2)}\sqrt{\frac{m+1}{m-1}}\\[1ex]
&\qquad +N_3^{-+}\frac{(m-3)^2}{(m+1)m(m-1)(m-2)}\sqrt{\frac{m+1}{m-1}}-N_2^+\frac{m-1}{m(m-2)(m+1)}\sqrt{\frac{m-1}{m+1}}=0,
\end{align*}
and hence
\begin{align*}
&(m-2)^2(m+1)^3+N_3^{--}(m+1)^3+N_3^{-+}(m-3)^2(m+1)-N_2^+(m-1)^3=0.
\end{align*}
Thus, substituting the value of $N_2^+$, we have
\[
(m-2)^2(m+1)^2+N_3^{--}(m+1)^2+N_3^{-+}(m-3)^2-\frac{(m+1)^2(m-1)^3(m-2)}{4}=0,
\]
and, substituting the value of $N_{3}^{-+}$, we get
\begin{align*}
&(m-2)^2(m+1)^2+N_3^{--}\bigl((m+1)^2-(m-3)^2\bigr)+\dfrac{(m-2)(m-3)^2(m^3+m^2-m+3)}{4}\\[1ex]
&\qquad -\frac{(m+1)^2(m-1)^3(m-2)}{4}=0.
\end{align*}
Then, we can solve the equation, obtaining
\begin{align*}
N_3^{--}&=\frac{(m-2)\bigl((m+1)^2(m-1)^3-4(m-2)(m+1)^2-(m-3)^2(m^3+m^2-m+3)\bigr)}{32(m-1)}\\[1ex]
&=\frac{(m-2)(m-1)(m^2-5)}{8}.
\end{align*}
In particular, since $m\in\bbN$, then we must have either $m$ odd or $m\equiv 2\bmod 8$.
Also, we get
\[
N_3^{-+}=\dfrac{(m-2)(m^3+m^2-m+3)}{4}-N_3^{--}=\frac{(m-2)(m+1)^3}{8}=\frac{N_2^+}{2}.
\]
We repeat the same procedure for the coefficients of the monomial $x_1^2x_2x_3$, which give the equation
\[
-\frac{m+1}{m^3(m-1)}\sqrt{\frac{m-2}{m}}-\frac{N_3^{--}C_3^{--}+N_3^{-+}C_3^{-+}}{m^3}\sqrt{\frac{m+1}{m-1}}+\frac{N_3^{+-}C_3^{+-}+N_3^{++}C_3^{++}}{m^3}\sqrt{\frac{m-1}{m+1}}=0,
\]
that is,
\[
(N_3^{+-}C_3^{+-}+N_3^{++}C_3^{++})\sqrt{\frac{m-1}{m+1}}=\frac{m+1}{m-1}\sqrt{\frac{m-2}{m}}+(N_3^{--}C_3^{--}+N_3^{-+}C_3^{-+})\sqrt{\frac{m+1}{m-1}}.
\]
Then, substituting the values obtained so far, we get
\begin{align*}
&\biggl(-N_3^{+-}\sqrt{\frac{m-1}{m(m-2)(m+1)}}+N_3^{++}\sqrt{\frac{m-1}{m(m-2)(m+1)}}\biggr)\sqrt{\frac{m-1}{m+1}}\\[1ex]
&\quad=\frac{m+1}{m-1}\sqrt{\frac{m-2}{m}}+\biggl(N_3^{-+}\frac{m-3}{\sqrt{(m+1)m(m-1)(m-2)}}-N_3^{--}\frac{m+1}{\sqrt{(m+1)m(m-1)(m-2)}}\biggr)\sqrt{\frac{m+1}{m-1}}
\end{align*}
and then
\begin{align*}
&\biggl(-N_3^{+-}\sqrt{\frac{m-1}{m(m-2)(m+1)}}+\biggl(\frac{(m+1)^3(m-2)}{4}-N_3^{+-}\biggr)\sqrt{\frac{m-1}{m(m-2)(m+1)}}\biggr)\sqrt{\frac{m-1}{m+1}}\\[1ex]
&\quad=\frac{m+1}{m-1}\sqrt{\frac{m-2}{m}}+\biggl(-\frac{(m-2)(m-1)(m+1)(m^2-5)}{8\sqrt{(m+1)m(m-1)(m-2)}}+\frac{(m-2)(m+1)^3(m-3)}{8\sqrt{(m+1)m(m-1)(m-2)}}\biggr)\sqrt{\frac{m+1}{m-1}},
\end{align*}
which gives
\begin{align*}
N_3^{+-}&=\frac{(m+1)^2(m-2)\bigl(2(m+1)(m-1)^2-(m+1)^2(m-3)+(m-1)(m^2-5)-8\bigl)}{16(m-1)^2}\\
&=\frac{(m+1)^3(m-2)}{8},
\end{align*}
which implies
\[
N_3^{-+}=N_3^{++}=\frac{N_2^{+}}{2}.
\]
Now, again because of the same norm of the points of the decomposition, we can repeat the same argument and rewrite, after a suitable orthogonal transformation, the first two points as 
\[
\bfa_1=(c_1,c_2,0,\dots,0),\quad \bfa_2=(c_1,-c_2,0,\dots,0).
\]
Therefore, since
\[
\abs*{\bfa_1}=\abs*{\bfa_2}=\sqrt{c_1^2+c_2^2}=1
\]
and
\[
\bfa_1\cdot\bfa_2=c_1^2-c_2^2=\frac{1}{\sqrt{n+2}}=\frac{1}{m},
\]
we easily get
\[
c_1^2=\frac{m+1}{2m},\quad c_2^2=\frac{m-1}{2m}
\]
and hence
\[
\bfa_1=\biggl(\sqrt{\frac{m+1}{2m}},\sqrt{\frac{m-1}{2m}},0,\dots,0\biggr),\quad \bfa_2=\biggl(\sqrt{\frac{m+1}{2m}},-\sqrt{\frac{m-1}{2m}},0,\dots,0\biggr).
\]
By applying \autoref{lem_kernel_tight_decomposition}, we observe that the elements in the kernel of the catalecticant of the polynomial
\[
\frac{1}{B_{n,2}}q_n^{2}-(\bfa_1\cdot\bfx)^4-(\bfa_2\cdot\bfx)^4
\]
are given by the polynomials
\[
(c_1y_1\pm c_2y_2)^2-\frac{1}{m^2-2}q_n+\frac{2}{m^2\bigl(m^2-2\bigr)}q_n,
\]
that is
\[
\frac{m+1}{2m}y_1^2\pm\frac{\sqrt{m^2-1}}{m}y_1y_2+\frac{m-1}{2m}y_2^2-\frac{1}{m^2}q_n.
\]
It is clear that the remaining points must satisfy both the equations. Therefore, since the norm of every point must be equal to one, we obtain that the point
\[
\bfa_k=(a_{k,1},\dots,a_{k,n})
\]
must satisfy the equations of the system
\[
\begin{cases}
a_{k,1}a_{k,2}=0,\\[1ex]
(m+1)a_{k,1}^2+(m-1)a_{k,2}^2=\dfrac{2}{m},
\end{cases}
\]
for every $k=3,\dots,T_{n,2}$,
that is, we must have
\begin{equation}
\label{rel_points_first_type}
a_{k,1}=\sqrt{\frac{2}{m(m+1)}},\qquad a_{k,2}=0,
\end{equation}
or
\begin{equation}
\label{rel_points_second_type}
a_{k,1}=0,\qquad a_{k,2}=\sqrt{\frac{2}{m(m-1)}}.
\end{equation}
We refer to points satisfying equations (\ref{rel_points_first_type}) and (\ref{rel_points_second_type}), respectively, as points of the first and second type.
So, considering the equation obtained by equating the coefficient of the monomial $x_1^4$ and denoting by $D_1$ and $D_2$ the number of addends respectively of the first and second type, we get
\[
\dfrac{(m+1)^2}{2m^2}+\dfrac{4D_1}{m^2(m+1)^2}=\frac{3(m^2-1)}{2m^2}
\]
and hence
\[
(m+1)^4+8D_1-3(m^2-1)(m+1)^2=0,
\]
that is
\[
D_1=\frac{(m+1)^3(m-2)}{4}.
\]
Thus, we also get
\[
D_2=T_{n,2}-D_1-2=\frac{(m^2-2)(m^2-1)}{2}-\frac{(m+1)^3(m-2)}{4}-2=\frac{(m-1)^3(m+2)}{4}
\]
for the elements of the second type.
Again up to orthogonal transformations, we can suppose the third point of the decomposition to be
\[
\bfa_3=\biggl(\sqrt{\frac{2}{m(m+1)}},0,c_3,0,\dots,0\biggr)
\]
for some $c_3\in\bbC$
and we must have
\[
\dfrac{2}{m(m+1)}+c_3^2=1,
\]
that is,
\[
c_3^2=1-\frac{2}{m(m+1)}=\frac{m^2+m-2}{m(m+1)}=\frac{(m+2)(m-1)}{m(m+1)}.
\]
For the successive points $\bfa_j$ of the first type, we must have, instead,
\[
\frac{2}{m(m+1)}+a_{j,3}\sqrt{\frac{(m+2)(m-1)}{m(m+1)}}=\pm\frac{1}{m},
\]
that is
\[
a_{j,3}\sqrt{\frac{(m+2)(m-1)}{m(m+1)}}=\frac{-2\pm(m+1)}{m(m+1)}
\]
and hence
\[
a_{j,3}=\frac{-2\pm(m+1)}{\sqrt{m(m+2)\bigl(m^2-1\bigr)}}.
\]
Now, the summands contributing to the monomial $x_1^2x_3^2$ are exactly $D_1$. So, denoting by $D_3$ the number of elements with the value
\[
\frac{-2-(m+1)}{\sqrt{m(m+2)\bigl(m^2-1\bigr)}}=-\frac{m+3}{\sqrt{m(m+2)\bigl(m^2-1\bigr)}}
\]
as third coordinate, we get
\begin{align*}
&\frac{12(m+2)(m-1)}{m^2(m+1)^2}+6D_3\frac{2(m+3)^2}{(m-1)m^2(m+1)^2(m+2)}\\[1ex]
&\quad+6\biggl(\frac{(m+1)^3(m-2)}{4}-D_3-1\biggr)\frac{2(m-1)^2}{(m-1)m^2(m+1)^2(m+2)}=\frac{3(m^2-1)}{m^2},
\end{align*}
that is,
\[
\frac{(m+2)^2(m-1)}{(m+1)^2}+D_3\frac{(m+3)^2-(m-1)^2}{(m-1)(m+1)^2}-\frac{(m-1)}{(m+1)^2}=m^2-1,
\]
and, just simplifying, we get the equation
\[
8D_3(m+1)=(m-1)^2(m^3+2m^2-m-2).
\]
Proceeding with the resolution, we obtain
\[
D_3=\frac{(m-1)^3(m+2)}{8}.
\]
Since $D_3$ must be an integer value, the modular equation
\[
(m-1)^3(m+2)\equiv 0\bmod 8
\]
must hold.
We have already stated that $m$ is odd or $m\equiv 2\bmod 8$. However, in this last case, we would have
\[
(m-1)^3(m+2)\equiv 4\bmod 8.
\]
Therefore, we conclude that $m$ cannot be even.
\end{proof}

The problem of determining for which $n\in\bbN$ tight decompositions exist is not trivial. We can admire some nice decompositions obtained mainly by classical examples of spherical designs. Apart from the case in two variables, the simplest of these is given by the decomposition
\begin{equation}
\label{rel_decomp_icos}
q_3^2=\frac{1}{6(\varphi+1)}\sum_{j=1}^3(x_j\pm\varphi x_{j-1})^4,
\end{equation}
where $\varphi$ is a root of the polynomial $x^2-x-1\in\bbR[x]$, namely
\[
\varphi=\frac{1+\sqrt{5}}{2}.
\]
Such a decomposition is made by linear forms which geometrically correspond to the vertices of a regular icosahedron, inscribed in a sphere of radius \[
(B_{3,2})^{\frac{1}{4}}=\biggl(\frac{5}{6}\biggr)^{\frac{1}{4}},
\] 
whose coordinates are given by H.~S.~M.~Coxeter in \cite{Cox73}. In particular, we can highlight an essential criterion, classically attributed to J.~Haantjes, to obtain the vertices of a regular icosahedron in the three-dimensional space $\bbR^3$, which is an important fact in relation to what we will see later in \cref{sec_decomp_three_variables}.
\begin{lem}[\cite{Haa48}]
\label{lem_suff_cond_ico}
If $12$ real distinct points $\bfa_1,\dots,\bfa_{12}\in S^2$ satisfy the conditions
\[
\bfa_i\cdot\bfa_j=\pm\frac{1}{\sqrt{5}}
\]
for all $i,j=1,\dots,12$ with $i\neq j$, then they represent the vertices of a regular icosahedron.
\end{lem}
This elegant decomposition of $q_3^2$, shown in \autoref{fig_example_icosahedron_decomp_n=3_s=2}, can also be found in
\cite{Rez92}*{Theorem 9.13}, where B.~Reznick proves its uniqueness as a real decomposition. Considering the language of spherical designs, it was already known (see \cite{DGS77}*{Example 5.16}) that the vertices of a regular icosahedron represent the unique tight $5$-spherical design in $\bbR^3$. It is not difficult to prove that decomposition \eqref{rel_decomp_icos} also represents the unique tight decomposition over the field of complex numbers, since the demonstration provided by B.~Reznick can be easily extended.
\begin{teo}
\label{teo_uniqueness_icosahedron}
Given $6$ points $\bfa_1,\dots,\bfa_6\in\bbC^3$, there is a tight decomposition
\begin{equation}
\label{rel_teo_icos_unique_decomp}
q_3^2=\sum_{j=1}^6(a_{j,1}x_1+a_{j,2}x_2+a_{j,3}x_3)^4
\end{equation}
if and only if $\bfa_1,\dots,\bfa_6$ are the vertices, up to opposite signs, of a regular icosahedron inscribed in a sphere of radius
\[
(B_{3,2})^{\frac{1}{4}}=\biggl(\frac{5}{6}\biggr)^{\frac{1}{4}},
\]
up to orthogonal complex transformations.
\end{teo}
\begin{proof}
By \autoref{teo_tight_implies_first_caliber}, we can assume that the points $\bfa_1,\dots,\bfa_6$ have norm $1$ and consider the corresponding decomposition of the rescaled form $\frac{6}{5}q_3^2$. Considering formula \eqref{rel_formula_scalar_prodoct_exponent_2}, we have that
\[
\bfa_i\cdot\bfa_j=\pm\frac{1}{\sqrt{5}}
\]
for all $i,j=1,\dots,6$ with $i\neq j$ and we can assume that $\abs{\bfa_j}=1$.
We can fix, by the invariance of the orthogonal group $\Oa_n(\bbC)$, the first point of the decomposition \eqref{rel_teo_icos_unique_decomp} as 
\[
\bfa_1=(1,0,0),
\]
so that
\[
a_{j,1}=\frac{1}{\sqrt{5}}
\]
for each $j=2,\dots,6$.
Similarly, we can set the second point as
\[
\bfa_2=\biggl(\frac{1}{\sqrt{5}},\frac{2}{\sqrt{5}},0\biggr),
\]
and, consequently, by \autoref{lem_kernel_tight_decomposition} we get that the other points must satisfy the equation
\[
5\biggl(\dfrac{1}{\sqrt{5}}y_1+\frac{2}{\sqrt{5}}y_2\biggr)^2-q_n=0.
\]
That is,
\[
\biggl(\dfrac{1}{\sqrt{5}}+2y_2\biggr)^2-1=0
\]
and hence
\[
5y_2^2+\sqrt{5}y_2-1=0,
\]
which means that
\[
a_{j,2}=\frac{-\sqrt{5}\pm 5}{10}
\]
for every $j=3,\dots,6$. Finally, we know that decomposition \eqref{rel_teo_icos_unique_decomp} is first caliber and hence we must have
\[
a_{j,1}^2+a_{j,2}^2+a_{j,3}^2=1.
\]
So we get
\[
a_{j,3}^2=\frac{4}{5}-\biggl(\frac{-\sqrt{5}\pm 5}{10}\biggr)^2=\frac{5\pm\sqrt{5}}{10},
\]
i.e.,
\[
a_{j,3}=\pm\sqrt{\frac{5\pm\sqrt{5}}{10}},
\]
for $j=3,4,5,6$.
Since these points are all real, then the condition of \autoref{lem_suff_cond_ico} is satisfied, and hence they are the vertices of a regular icosahedron, up to orthogonal transformations. Conversely, if six points $\bfa_1,\dots,\bfa_6$ are the vertices of a regular icosahedron, up to orthogonal complex transformations, then we can transform them to obtain decomposition \eqref{rel_decomp_icos}.
\end{proof}

\begin{figure}[ht]
\center
\tdplotsetmaincoords{72}{107}
\begin{tikzpicture}[tdplot_main_coords,declare function={phi=(1+sqrt(5))/2;}]

\coordinate (Vxy1) at (phi,1,0);
\coordinate (Vxy2) at (phi,-1,0);
\coordinate (Vxy3) at (-phi,1,0);
\coordinate (Vxy4) at (-phi,-1,0);
\coordinate (Vxz1) at (1,0,phi);
\coordinate (Vxz2) at (-1,0,phi);
\coordinate (Vxz3) at (1,0,-phi);
\coordinate (Vxz4) at (-1,0,-phi);
\coordinate (Vyz1) at (0,phi,1);
\coordinate (Vyz2) at (0,phi,-1);
\coordinate (Vyz3) at (0,-phi,1);
\coordinate (Vyz4) at (0,-phi,-1);

\draw[dashed, color=darkred,line width=0.7pt] (-phi,-1,0) -- (-phi,1,0);
\draw[dashed, color=darkred,line width=0.7pt] (-phi,-1,0) -- (-1,0,-phi);
\draw[dashed, color=darkred,line width=0.7pt] (-phi,-1,0) -- (-1,0,phi);
\draw[dashed, color=darkred,line width=0.7pt] (-phi,1,0) -- (-1,0,-phi);
\draw[dashed, color=darkred,line width=0.7pt] (-phi,1,0) -- (-1,0,phi);
\draw[dashed, color=darkred,line width=0.7pt] (1,0,-phi) -- (-1,0,-phi);
\draw[dashed, color=darkred,line width=0.7pt] (-1,0,-phi) -- (0,phi,-1);
\draw[dashed, color=darkred,line width=0.7pt] (-1,0,-phi) -- (0,-phi,-1);
\draw[dashed, color=darkred,line width=0.7pt] (0,phi,1) -- (-phi,1,0);
\draw[dashed, color=darkred,line width=0.7pt] (0,phi,-1) -- (-phi,1,0);
\draw[dashed, color=darkred,line width=0.7pt] (0,-phi,1) -- (-phi,-1,0);
\draw[dashed, color=darkred,line width=0.7pt] (0,-phi,-1) -- (-phi,-1,0);

\pgfsetdash{}{0pt}
\color{black}
\draw[thick,->] (0,0,0) -- (4.5,0,0) node[anchor=north east]{$y_1$};
\draw[thick,->] (0,0,0) -- (0,3,0) node[anchor=north west]{$y_2$};
\draw[thick,->] (0,0,0) -- (0,0,3) node[anchor=south]{$y_3$};
\draw[thick,dashed] (0,0,0) -- (0,0,-3);

\draw[thick, color=darkred] (phi,1,0) -- (phi,-1,0);
\draw[thick, color=darkred] (phi,1,0) -- (1,0,phi);
\draw[thick, color=darkred] (phi,1,0) -- (1,0,-phi);
\draw[thick, color=darkred] (phi,-1,0) -- (1,0,phi);
\draw[thick, color=darkred] (phi,-1,0) -- (1,0,-phi);

\draw[thick, color=darkred] (1,0,phi) -- (-1,0,phi);
\draw[thick, color=darkred] (1,0,phi) -- (0,phi,1);
\draw[thick, color=darkred] (1,0,phi) -- (0,-phi,1);
\draw[thick, color=darkred] (-1,0,phi) -- (0,phi,1);
\draw[thick, color=darkred] (-1,0,phi) -- (0,-phi,1);
\draw[thick, color=darkred] (1,0,-phi) -- (0,phi,-1);
\draw[thick, color=darkred] (1,0,-phi) -- (0,-phi,-1);
\draw[thick, color=darkred] (0,phi,1) -- (0,phi,-1);
\draw[thick, color=darkred] (0,phi,1) -- (phi,1,0);
\draw[thick, color=darkred] (0,phi,-1) -- (phi,1,0);
\draw[thick, color=darkred] (0,-phi,-1) -- (0,-phi,1);
\draw[thick, color=darkred] (phi,-1,0) -- (0,-phi,1);
\draw[thick, color=darkred] (phi,-1,0) -- (0,-phi,-1);

\foreach \x in {1,-1}
{\foreach \y in {phi,-phi}
{\pgfpathcircle{\pgfpointxyz{\x}{0}{\y}}{3pt};
\shade[ball color=darkred];
\pgfpathcircle{\pgfpointxyz{\y}{\x}{0}}{3pt};
\shade[ball color=darkred];
\pgfpathcircle{\pgfpointxyz{0}{\y}{\x}}{3pt};
\shade[ball color=darkred];}}
\end{tikzpicture}
\caption[Graphical representation of decomposition \eqref{rel_decomp_icos}]{Graphical representation of decomposition \eqref{rel_decomp_icos}, whose elements correspond, up to central simmetry, to the vertices of a regular icosahedron.}
\label{fig_example_icosahedron_decomp_n=3_s=2}
\end{figure}
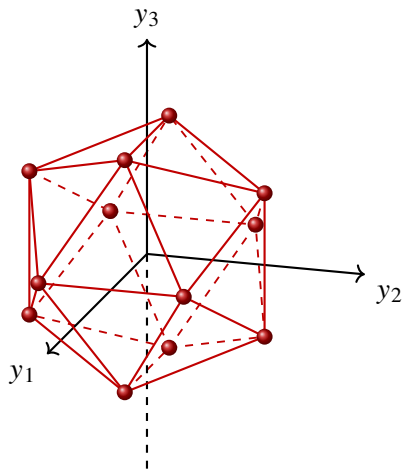
There are also two other tight decompositions representing the successive cases allowed by \autoref{teo_tight_decompositions_s=2},  for $n=7$ and $n=23$ respectively, both of which are mentioned as tight spherical designs in
\cite{DGS77}*{p.~371}. The first is given by
\begin{equation}
\label{rel_decom_tight_n=7_s=2}
q_7^2=\frac{1}{12}\sum_{j=1}^{7}(x_j\pm x_{j+1}\pm x_{j+3})^4
\end{equation}
and represents the maximal set of $28$ lines in $\bbR^7$ with mutual angles equal to the values of $\theta$ such that 
\[
\cos^2\theta=\frac{1}{9}.
\]
The second corresponds instead to a set of $276$ lines in $\bbR^{23}$ with mutual angles equal to the values of an angle $\theta$ such that 
\[
\cos^2\theta=\frac{1}{25}.
\]
This structure is known as the Leech lattice, named after J.~Leech, who introduced it in \cite{Lee67}.
In particular, decomposition \eqref{rel_decom_tight_n=7_s=2} has been analyzed, as $5$-spherical design, also by H.~Cuypers in \cite{Cuy05}, for which (real) uniqueness is proved. 

Another result, always considering real tight spherical designs, is provided by E.~Bannai, E.~Munemasa, and B.~Venkov in \cite{BMV04}*{Theorem 3.1}, where they state that by setting $n=(2m+1)^2-2$ for a suitable $m\in\bbN$, there is no tight spherical design in $\bbR^n$ in the case of $m=3,4$ or $m=2k$ for some $k\in\bbN$ such that $k\equiv 2\bmod 3$ and both $k$ and $2k+1$ are square-free. This result was further improved by G.~Nebe and B.~Venkov in \cite{NV12}.

If we set $s=3$, we can proceed exactly as in the previous case. Assuming the existence of a tight decomposition and considering the form
\[
\frac{1}{B_{n,3}}q_n^{3},
\]
we can take any point of norm $1$. So the starting point can  again be
\[
\bfa_1=(1,0,\dots,0).
\]
We have that
\[
B_{n,3}=\frac{6n(n+2)(n+4)}{15n(n+1)(n+2)}=\frac{2(n+4)}{5(n+1)}
\]
and, proceeding as in \autoref{lem_kernel_tight_decomposition}, we have that the middle catalecticant matrix of the polynomial
\[
f_1=\frac{1}{B_{n,3}}q_n^{3}-(\bfa\cdot\bfx)^6,
\]
for some $\bfa\in\bbC^n$ such that $\bfa\cdot\bfa=1$, must have rank equal to $T_{n,3}-1$. 
\begin{lem}
\label{lem_kernel_tight_decomposition_s=3}
For every $n\in\bbN$, the kernel of the middle catalecticant of the polynomial
\[
f_1=\frac{1}{B_{n,3}}q_n^{3}-(\bfa\cdot\bfx)^6
\]
is
\[
\Ker(\Cat_{f_1,3})=\langle (n+4)(\bfa\cdot\bfy)^3-3q_n(\bfa\cdot\bfy)\rangle.
\]
\end{lem}
\begin{proof}
As in the proof of \autoref{lem_kernel_tight_decomposition}, using \autoref{lem contraz forme}, formula \eqref{rel_Laplace_on_q_n^s} and the fact that $\bfa\cdot\bfa=1$, we have
\begin{align*}
\bigl((n+4)(\bfa\cdot\bfy)^3-3q_n(\bfa\cdot\bfy)\bigr)\circ f_1&=\frac{n+4}{B_{n,3}}\bigl((\bfa\cdot\bfy)^3\circ q_n^3\bigr)-(n+4)\bigl((\bfa\cdot\bfy)^3\circ (\bfa\cdot\bfx)^6\bigr)\\
&\hphantom{{}={}}-\frac{3}{B_{n,3}}\bigl(q_n(\bfa\cdot\bfy)\circ q_n^{3}\bigr)+3\bigl(q_n(\bfa\cdot\bfy)\circ(\bfa\cdot\bfx)^6\bigr)\\
&=\frac{n+4}{B_{n,3}}\bigl(72q_n(\bfa\cdot\bfx)+48(\bfa\cdot\bfx)^3\bigr)-120(n+4)(\bfa\cdot\bfx)^3\\
&\hphantom{{}={}}-\frac{72(n+4)}{B_{n,3}}q_n(\bfa\cdot\bfx)+360(\bfa\cdot\bfx)^3\\
&=\frac{48(n+4)}{B_{n,3}}(\bfa\cdot\bfx)^3-120(n+1)(\bfa\cdot\bfx)^3=0,
\end{align*}
which proves the statement.
\end{proof}
Now we provide another calculation about the values that can be taken by $n$ to obtain suitable tight decompositions.
\begin{teo}
\label{teo_tight_decompositions_n=3}
Let
\[
\frac{1}{B_{n,3}}q_n^3=\sum_{j=1}^{T_{n,3}}(\bfa_j\cdot\bfx)^6
\]
be a tight decomposition of $q_n^3$, with $n\neq 1$. Then $n\equiv 2\bmod 3$.
\end{teo}
\begin{proof}
Consider an hypothetical tight decomposition of $q_n^3$ and let us suppose that the point 
\[
\bfa_1=(1,0,\dots,0)
\]
is a point of such a decomposition. Then,
by \autoref{lem_kernel_tight_decomposition_s=3} it follows that, given another point $\bfa_2$ of this decomposition, we must have
\[
a_{2,1}=0
\]
or 
\[
a_{2,1}=\pm\sqrt{\frac{3}{n+4}}.
\]
Supposing that $C_1$ of these points have the second value as their first coordinate, we can solve the equation
\[
\frac{5(n+1)}{2(n+4)}-1=C_1\biggl(\frac{3}{n+4}\biggr)^3,
\]
i.e.,
\[
\frac{3n-3}{2(n+4)}=C_1\biggl(\frac{3}{n+4}\biggr)^3,
\]
by which we get
\begin{equation}
\label{rel_first_constant_power_3}
C_1=\frac{(n-1)(n+4)^2}{18}.
\end{equation}
Now, by equality \eqref{rel_first_constant_power_3}, the number of elements with the first coordinate equal to $0$ is
\begin{align*}
C_1'=T_{n,3}-\frac{(n-1)(n+4)^2}{18}-1&=\frac{3n(n+1)(n+2)-(n-1)(n+4)^2-18}{18}\\[1ex]
&=\frac{3\bigl(n^2+n\bigr)(n+2)-(n-1)\bigl(n^2+8n+16\bigr)-18}{18}\\[1ex]
&=\frac{3\bigl(n^3+3n^2+2n\bigr)-n^3-8n^2-16n+n^2+8n+16-18}{18}\\[1ex]
&=\frac{2n^3+2n^2-2n-2}{18}=\frac{(n-1)(n+1)^2}{9}
\end{align*}
and so we must have
\[
n\equiv 1,2,5,8\bmod 9.
\]
Since the value $C_1'$ is nonzero for every $n\geq 2$, we can assume that
\[
\bfa_2=(0,1,0,\dots,0)
\] 
and, by repeating the same procedure for this last point, we can easily see that the second coordinate of each of the other points must have the same possible values as the first coordinate. This means that we must have the same number of points whose second coordinate is equal to $0$. Now we want to know how many points have the first two coordinates other than zero. Let $C_2$ be the number of such points.
The coefficients of the monomials $x_1^4x_2^2$ and $x_1^2x_2^4$ are
\[
\frac{15(n+1)}{2(n+4)}=15C_2\biggl(\frac{3}{n+4}\biggr)^3,
\]
that is,
\[
C_2=\frac{(n+1)(n+4)^2}{54}.
\]
Then, we have
\[
n\equiv 2\bmod 3.\qedhere
\]
\end{proof}
As in the case of the exponent $2$, we can list some cases where a tight decomposition exists. In particular, for $s=3$, it exists for $n=8$ and $n=23$, and they are presented in \cite{BS81}.

\section{Estimates of the rank}
\label{cha_general_decompositions_more_variables}
\noindent As one might have seen in the previous sections, the structure of the polynomial $q_n^s$ in terms of powers of linear forms, as $n$ and $s$ change, can be quite complicated, and it is not so easy to determine. However, it is possible to obtain explicit closed formulas, depending only on $n$ for some fixed values of the exponent $s$. One of our main results concerns the forms $q_n^2$ for $n\in\bbN$.

In \cref{sec_rank_q_n_2} we give a general survey of these last polynomials, depending on the number $n$ of variables. In particular, after listing some classical and more recent decompositions, we analyze the possible minimal ones. For most of the values of $n$  it is possible to determine the exact rank, which is equal to $T_{n,2}+1$.

In \cref{sec_decomp_three_variables} we focus, instead, on the case of three variables. First, we introduce a different set of coordinates with respect to the standard ones. Then, we analyze some already known decompositions, whose points represent several polygons placed at different heights in three-dimensional space, and we also give new examples with the same pattern.

Finally, we consider in \cref{sec_equiang_lines} a strategy to exclude the existence of tight decompositions. It is based on the strong conditions imposed by these expressions, whose points must have the same norm. Indeed, by \autoref{lem_kernel_product_quadrics} it is possible to establish which are the possible mutual angles between the points of the decomposition. We focus on the case where $s=4$ and prove that $\rk(q_3^4)=16$. 

\subsection{On the rank of \texorpdfstring{$q_n^2$}{q2s}}
\label{sec_rank_q_n_2}
\noindent We have already provided an example of a tight decomposition for the exponent $2$ in the case of three variables, given by the decomposition \eqref{rel_decomp_icos}, whose points correspond to the vertices of a regular icosahedron. Although this elegant structure is rarely repeated for other values of $n$, we can see that, at least for every $n\in\bbN$ such that $n\neq 8$, the rank of $q_n^s$ is at most equal to $T_{n,2}+1$.

There are many examples of decompositions of $q_n^2$ in the classical literature. We can appreciate some simple examples. For instance, we can report two of them, due to E.~Lucas. The first is given by the equality
\begin{equation}
q_3^2=\frac{2}{3}\sum_{j=1}^3x_j^4+\frac{1}{12}\sum(x_1\pm x_2\pm x_3)^4
\end{equation}
and can be found as an exercise in \cite{Hou77}*{Question 39, p.~129}.
The second one appears in the same article as the previous one, exactly in \cite{Hou77}*{Question 38, p.~129}, but it also appeared earlier in \cite{Luc76}*{p.~101}. It concerns the case of four variables and is given by the equality
\begin{equation}
\label{decom n=4 s=2 ord=12}
q_4^2=\frac{1}{6}\sum_{j_1<j_2}(x_{j_1}\pm x_{j_2})^4,
\end{equation}
consisting of a decomposition of size $12$. Formally, this last one is exactly the same decomposition as determined by J.~Liouville and revealed by V.~A.~Lebesgue in \cite{Leb59}. This is, up to an orthogonal change of coordinates, the decomposition
\begin{equation}
q_4^2=\frac{1}{24}\sum(x_1\pm x_2\pm x_3\pm x_4)^4+\frac{2}{3}\sum_{j=1}^4x_j^4.
\end{equation}

B.~Reznick  also provides a formula, which is given in \cite{Rez92}*{formula (10.35)}, which gives us a family of real decompositions for every $n\geq 3$. That is, considering $n\neq 4$, we get
\begin{equation}
\label{Rez rep p.147}
q_n^2=\frac{1}{6}\sum_{j_1<j_2}(x_{j_1}\pm x_{j_2})^4+\frac{4-n}{3}\sum_{j=1}^nx_j^4,
\end{equation}
which is a decomposition of size $n^2$.
If instead $n=4$, since the last term of the decomposition becomes equal to zero, we get a decomposition of size $4^2-4=12$, which is exactly the decomposition \eqref{decom n=4 s=2 ord=12}.

As we can see from all of these decompositions, these examples represent some particularly symmetric arrangements of points within the space $\bbR^n$. The idea that emerges from these examples is that the dispositions of the points are in a certain sense "balanced" with respect to a center of symmetry. This aspect, as we will see later, is evident in many other decompositions. For example, besides the perfect symmetry assumed by the points of the icosahedron of decomposition \eqref{rel_decomp_icos},
we can further transform this formula. In fact, simply by an orthogonal transformation, we can obtain a disposition of points which are invariant under the action of the permutation group $S_3$ and which are arranged simmetrically in the space with respect to the axis identified by the unit vector 
\[
\pt*{\frac{1}{\sqrt{3}},\frac{1}{\sqrt{3}},\frac{1}{\sqrt{3}}}\in\bbR^3.
\] 
We will see later that this particular arrangement of points can be generalized for all the other values of $n$.

The structure of decomposition \eqref{rel_decomp_icos} can be generalized to an arbitrary number of variables $n\geq 3$, distinguishing the cases where $n$ is either an odd or an even number. For simplicity,  in the following theorems we will use the symbol $x_{n+j}$ to denote the variable $x_j$, for any $j=1,\dots,n$.
\begin{teo}
\label{teo_decom_first_caliber_complex}
For every odd number $n\in\bbN$ with $n\geq 5$, the form $q_n^2$ can be decomposed as
\begin{align}
\label{rel_first_caliber_complex_odd_decomp}
6q_n^2&=\sum_{j=1}^n\sum_{\substack{1\leq k\leq\frac{n-1}{2}\\\text{$k$ even}}}\Bigl(\varphi_{n}^{\frac{1}{4}}x_{j}\pm\varphi_{n}^{-\frac{1}{4}}x_{j+k}\Bigr)^4+\sum_{j=1}^n\sum_{\substack{1\leq k\leq\frac{n-1}{2}\\\text{$k$ odd}}}\Bigl(\varphi_{n}^{-\frac{1}{4}}x_{j}\pm\varphi_{n}^{\frac{1}{4}}x_{j+k}\Bigr)^4,
\end{align}
which is a decomposition of size $n(n-1)$, where
\[
\varphi_n=\frac{3\pm \rmi\sqrt{(n-4)(n+2)}}{n-1}.
\]
In particular, it is a first caliber decomposition, such that each point $\bfa$ appearing in the sum satisfies the relation
\[
\abs*{\bfa}^{4}=\frac{n+2}{3(n-1)}.
\]
\end{teo}
\begin{proof}
By symmetry, we only need to check the correctness of the coefficients of the monomials \[x_{j}^4,\qquad x_{j_1}^2x_{j_2}^2,\] for $j=1,\dots,n$ and $1\leq j_1<j_2\leq n$. In particular, we get that the coefficients of the monomial $x_j^4$ must be equal to $6$, and hence summing such a coefficient for each linear form of the decomposition, we get the equation
\[
(n-1)(\varphi_n^{-1}+\varphi_n)=6.
\]
In particular, this means that
\[
\varphi_n^2-\frac{6}{n-1}\varphi_n+1=0
\]
and therefore
\[
\varphi_n=\frac{3\pm \rmi\sqrt{(n-4)(n+2)}}{n-1}.
\]
For the coefficients of $x_{j_1}^2x_{j_2}^2$, by summing all the powers of the linear forms, we get the value
\[
12\varphi_n^{-\frac{1}{2}}\varphi_n^{\frac{1}{2}}=12,
\]
which confirms decomposition \eqref{rel_first_caliber_complex_odd_decomp}. It is clear that such a decomposition is first caliber, and to get the exact value of the norm raised to $4$, we simply have to compute
\[
\frac{1}{6}\bigl(\varphi_n^{\frac{1}{2}}+\varphi_n^{-\frac{1}{2}}\bigr)^2=\frac{\varphi_n+\varphi_n^{-1}+2}{6}=\frac{1}{n-1}+\frac{1}{3}=\frac{n+2}{3(n-1)}.\qedhere
\]
\end{proof}
We can find the analog of \eqref{rel_first_caliber_complex_odd_decomp} for $n$ even, but in this case the decomposition is not first caliber.
\begin{teo}
For every even number $n\in\bbN$ with $n\geq 6$, the form $q_n^2$ can be decomposed as
\begin{equation}
\label{rel_first_caliber_complex_even_decomp}
6q_n^2=\sum_{j=1}^n\sum_{\substack{1\leq k\leq\frac{n-2}{2}\\\text{$k$ even}}}\Bigl(\psi_{n}^{-\frac{1}{4}}x_{j}\pm\psi_{n}^{\frac{1}{4}}x_{j+k}\Bigr)^4+\sum_{j=1}^n\sum_{\substack{1\leq k\leq\frac{n-2}{2}\\\text{$k$ odd}}}\Bigl(\psi_{n}^{\frac{1}{4}}x_{j}\pm \psi_{n}^{-\frac{1}{4}}x_{j+k}\Bigr)^4+\sum_{j=1}^{\frac{n}{2}}\bigl(x_{j}\pm x_{j+\frac{n}{2}}\bigr)^4,
\end{equation}
which is a decomposition of size $n(n-1)$, where
\[
\psi_n=\frac{2\pm \rmi\sqrt{n(n-4)}}{n-2}.
\]
\end{teo}
\begin{proof}
We proceed in the same way as we did in \autoref{teo_decom_first_caliber_complex}, first determining the coefficient of the monomials $x_j^4$, obtained by solving the equation
\[
(n-2)(\psi_n+\psi_n^{-1})+2=6,
\]
that is,
\[
\psi_n^2-\frac{4}{n-2}\psi_n+1=0,
\]
obtaining
\[
\psi_n=\frac{2\pm \rmi\sqrt{n(n-4)}}{n-2}.
\]
For the coefficients of $x_{j_1}^2x_{j_2}^2$ the equality in both sides of the equations is trivial.
\end{proof}
Although decomposition \eqref{rel_first_caliber_complex_odd_decomp} has a quite large size, it proves that, in general, first caliber decompositions can also have complex points as summands. In particular, this shows that results in \cite{Rez92} about real tight decompositions may not be valid for the complex ones and therefore some caution is necessary.
B.~Reznick provides in \cite{Rez92}*{formula (8.35)} a decomposition of $q_n^2$ for $3\leq n\leq 7$, based on a family of integration quadrature formulas of precision $5$ revealed by A.~H.~Stroud in \cite{Str67a}, which are essentially real.
Besides proving the existence of these decompositions, we can also prove that the same formula is valid for $n\geq 9$, with the only exception, as we will see, of $n=8$. Moreover, we determine also other decomposition, containing the fourth roots of unity.

In these last cases the decompositions are no longer real, but the size remains the same, giving summations of size $T_{n,2}+1$, that are clearly not tight. In particular, these decompositions together with \autoref{teo_tight_decompositions_s=2} provide the exact rank of $q_n^2$ for many values of $n$. The set of points forming the decomposition has the special property that it is invariant under the action of the permutation group $\mfS_n$, and these points are symmetric with respect to the central axis, identified by the vector
\[
\biggl(\frac{1}{\sqrt{3}},\frac{1}{\sqrt{3}},\frac{1}{\sqrt{3}}\biggr).
\]
\begin{teo}
\label{teo_minimal_decomposition_q_n^2}
Let $n\in\bbN$ be such that $n\geq 3$ and $n\neq 8$ and let $g\in\bbC$ such that $g^4=8-n$. Then, the form $q_n^2$ can be decomposed as
\begin{equation}
\label{rel_minimal_decomposition_q_n^2}
3a_5^4q_n^2=a_1\biggl(\sum_{j=1}^n x_{j}\biggr)^4+\sum_{k=1}^n\Biggl(a_2\biggl(\sum_{j=1}^nx_{j}\biggl)+a_3x_{k}\Biggr)^4+\sum_{j_1\neq j_2}\Biggl(a_4\biggl(\sum_{j=1}^nx_{j}\biggr)+a_5(x_{j_1}+x_{j_2})\Biggr)^4,
\end{equation}
where
\begin{gather*}
a_1=8(g^4-1)\bigl(g^2\pm 2\sqrt{2}\bigr)^4,\quad a_2=2g^2\pm 2\sqrt{2},\quad a_3=\mp 2\sqrt{2}g^4-8g^2,\\ a_4=2g,\qquad a_5=\bigl(\mp 2\sqrt{2}g^3-8g\bigr).
\end{gather*}
\end{teo}
\begin{proof}
To verify the required formula, it is sufficient to equalize the coefficients of each type of monomial in both the sides of the equation, obtaining a system of $5$ equations. Indeed, the set of points appearing in formula \eqref{rel_minimal_decomposition_q_n^2} is invariant under the action of the permutation group $\mfS_n$ on the set of variables $\{x_1,\dots,x_n\}$. We first consider the equation
\begin{equation}
\label{formula:initial_equation_qn2}
q_n^2=a\biggl(\sum_{j=1}^n x_{j}\biggr)^4+\sum_{k=1}^n\Biggl(b\biggl(\sum_{j=1}^nx_{j}\biggl)+cx_{k}\Biggr)^4+\sum_{j_1\neq j_2}\Biggl(d\biggl(\sum_{j=1}^nx_{j}\biggr)+e(x_{j_1}+x_{j_2})\Biggr)^4.
\end{equation}
Expanding the form $q_n^2$, we get
\[
q_n^2=\sum_{j=1}^nx_i^4+2\sum_{j_1<j_2}x_{j_1}^2x_{j_2}^2
\]
and hence we can expand the sums of the powers of the linear forms of the right side of the equation and determine a system in $5$ variables. By symmetry, we only need to consider the monomials of distinct multi-degree.
Therefore we can gather the monomials that are in the same permutation class, which are
\begin{gather*}
x_j^4,\quad x_{j_1}^3x_{j_2},\quad x_{j_1}^2x_{j_2}^2,\quad
x_{j_1}^2x_{j_2}x_{j_3},\quad x_{j_1}x_{j_2}x_{j_3}x_{j_4}.
\end{gather*}
We obtain the system
\[
\begin{cases}
a+nb^4+4b^3c+6b^2c^2+4bc^3+c^4+\dbinom{n}{2}d^4+(n-1)\bigl(4d^3e+6d^2e^2+4de^3+e^4\bigr)=1,\\[2ex]
a+nb^4+4b^3c+3b^2c^2+bc^3+
\dbinom{n}{2}d^4+4(n-1)d^3e+3nd^2e^2+(n+2)de^3+e^4=0,\\[2ex]
a+nb^4+4b^3c+2b^2c^2+\dbinom{n}{2}d^4+4(n-1)d^3e+2(n+1)d^2e^2+4de^3+e^4=\dfrac{1}{3},\\[2ex]
a+nb^4+4b^3c+b^2c^2+\dbinom{n}{2}d^4+4(n-1)d^3e+(n+4)d^2e^2+2de^3=0,\\[2ex]
a+nb^4+4b^3c+\dbinom{n}{2}d^4+4(n-1)d^3e+6d^2e^2=0.
\end{cases}
\]
Now we notice that all the equations have a common summand in the first member, that is
\[
a+nb^4+4b^3c+\binom{n}{2}d^4+4(n-1)d^3e.
\]
Considering the last equation, associated to the monomials $x_{j_1}x_{j_2}x_{j_3}x_{j_4}$, we get
\[
a+nb^4+4b^3c+\binom{n}{2}d^4+4(n-1)d^3e=-6d^2e^2.
\] 
Substituting it the other ones, we get the system 
\[
\begin{cases}
6b^2c^2+6(n-2)d^2e^2+4(n-1)de^3+(n-1)e^4+4bc^3+c^4=1,\\
3b^2c^2+3(n-2)d^2e^2+(n+2)de^3+e^4+bc^3=0,\\
6b^2c^2+6(n-2)d^2e^2+12de^3+3e^4=1,\\
b^2c^2+(n-2)d^2e^2+2de^3=0,\\
2a+2nb^4+8b^3c+n(n-1)d^4+8(n-1)d^3e+12d^2e^2=0,
\end{cases}
\]
which is equivalent, replacing the fourth equation in the previous ones, to
\[
\begin{cases}
4(n-4)de^3+(n-1)e^4+4bc^3+c^4=1,\\
(n-4)de^3+e^4+bc^3=0,\\
3e^4=1,\\
b^2c^2+(n-2)d^2e^2+2de^3=0,\\
2a+2nb^4+8b^3c+n(n-1)d^4+8(n-1)d^3e+12d^2e^2=0.
\end{cases}
\]
Then, replacing the second equation in the first and the third equation in the previous ones, we have
\[
\begin{cases}
3c^4=8-n,\\
3e^4=1,\\
3(n-4)de^3+3bc^3+1=0,\\
b^2c^2+(n-2)d^2e^2+2de^3=0,\\
2a+2nb^4+8b^3c+n(n-1)d^4+8(n-1)d^3e+12d^2e^2=0.
\end{cases}
\]
Therefore, as $g$ is one of the fourth roots of $8-n$, we can write
\begin{equation}
\label{formula:equations_c_e}
c=\frac{g}{\sqrt[4]{3}},\quad e=\frac{\tau_4}{\sqrt[4]{3}},
\end{equation}
where $\tau_4$ is a fourth root of unity, i.e., $\tau_4^4=1$.
Now, using relations \eqref{formula:equations_c_e}, we can write the system as
\[
\begin{cases}
\tau_4\sqrt[4]{3}g^3b+(4-g^4)\sqrt[4]{3}d+\tau_4=0,\\
\sqrt[4]{3}g^2b^2+\tau_4^2\sqrt[4]{3}(6-g^4)d^2+2\tau_4^3d=0,\\
2a+2(8-g^4)b^4+8b^3c+(8-g^4)(7-g^4)d^4+8(7-g^4)d^3e+12d^2e^2=0.
\end{cases}
\]
Now, by the first equation we get
\begin{equation}
\label{formula:equations_b_d}
b=\frac{\tau_4^3\sqrt[4]{3}(g^4-4)d-1}{\sqrt[4]{3}g^3}
\end{equation}
and replacing it in the second equation we get
\[
\sqrt[4]{3}g^2\frac{\bigl(\tau_4^3\sqrt[4]{3}(g^4-4)d-1\bigr)^2}{\sqrt{3}g^6}+\tau_4^2\sqrt[4]{3}(6-g^4)d^2+2\tau_4^3d=0,
\]
which gives
\[
\tau_4^2\sqrt{3}(g^8-8g^4+16)d^2+\tau_4^2\sqrt{3}(6g^4-g^8)d^2-2\tau_4^3\sqrt[4]{3}(g^4-4)d+2\tau_4^3g^4\sqrt[4]{3}d+1=0,
\]
and hence
\[
2\tau_4^2\sqrt{3}(8-g^4)d^2+8\tau_4^3\sqrt[4]{3}d+1=0.
\]
Therefore, by solving the equation we obtain
\begin{align*}
d&=\frac{-4\tau_4^3\sqrt[4]{3}\pm\sqrt{16\tau_4^2\sqrt{3}-2\tau_4^2\sqrt{3}(8-g^4)}}{2\tau_4^2\sqrt{3}(8-g^4)}=\frac{-4\tau_4^3\pm\tau_4^3\sqrt{2g^4}}{2\tau_4^2\sqrt[4]{3}(8-g^4)}=\frac{\tau_4\bigl(-2\sqrt{2}\pm g^2\bigr)}{\sqrt{2}\sqrt[4]{3}(8-g^4)}\\
&=\frac{\tau_4\bigl(-2\sqrt{2}\pm g^2\bigr)}{\sqrt{2}\sqrt[4]{3}\bigl(-2\sqrt{2}+g^2\bigr)\bigl(-2\sqrt{2}-g^2\bigr)}=\frac{2\tau_4}{\sqrt[4]{3}\bigl(\mp 2\sqrt{2}g^2-8\bigr)}
\end{align*}
and replacing it in formula \eqref{formula:equations_b_d}, we also get
\[
b=\frac{2\sqrt[4]{3}(g^4-4)-\sqrt[4]{3}\bigl(-8\mp 2\sqrt{2}g^2\bigr)}{\sqrt{3}g^3\bigl(\mp 2\sqrt{2}g^2-8\bigr)}=\frac{2g^4\pm 2\sqrt{2}g^2}{\sqrt[4]{3}g^3\bigl(\mp 2\sqrt{2}g^2-8\bigr)}
\]
So we have
\begin{equation}
\label{formula:equations_bd2}
b=\frac{2g^2\pm 2\sqrt{2}}{\sqrt[4]{3}g\bigl(\mp 2\sqrt{2}g^2-8\bigr)},\qquad d=\frac{2g}{\sqrt[4]{3}g\bigl(\mp 2\sqrt{2}g^2-8\bigr)}.
\end{equation}
It remains to determine the possible values that $a$ can take. 
For simplicity, we can scale the values of the coefficients by multiplying both sides of equation \eqref{formula:initial_equation_qn2} by the constant term
\[
M=3\bigl(\mp 2\sqrt{2}g^3-8g\bigr)^4.
\]
So, redefining all the parameters and using those appearing in equation \eqref{rel_minimal_decomposition_q_n^2}, we obtain
\[
a_2=2g^2\pm 2\sqrt{2},\quad a_3=\mp 2\sqrt{2}g^4-8g^2,\quad a_4=2\tau_4g,\quad a_5=\tau_4\bigl(\mp 2\sqrt{2}g^3-8g\bigr),
\]
and, in particular, $M=3a_5^4$. In this case, since $a_4$ and $a_5$ are multiplied by the same fourth root of unity and they are both in the same parenthesis of formula \eqref{rel_minimal_decomposition_q_n^2}, we can suppose, without loss of generality, that $\tau_4=1$.
Now, looking at the last equation of the initial system, we see that the value of $a_1$ is
\[
a_1=\bigl((g^4-8)a_2-4a_3\bigr)a_2^3-\biggl(\frac{(g^4-8)(g^4-7)}{2}a_4^2-4(g^4-7)a_4a_5+6a_5^2\biggr)a_4^2.
\]
So all that remains is to replace equations \eqref{formula:equations_c_e} and \eqref{formula:equations_bd2} in the equation. First we observe that
\[
a_2^3=8g^6\pm 24\sqrt{2}g^4+48g^2\pm 16\sqrt{2}
\]
and
\[
(g^4-8)a_2-4a_3=2g^6\pm 10\sqrt{2}g^4+16g^2\mp 16\sqrt{2}.
\]
Then, multiplying these two values, we get
\[
\bigl((g^4-8)a_2-4a_3\bigr)a_2^3=16(g^{12}\pm 8\sqrt{2}g^{10}+44g^8\pm 48\sqrt{2}g^6+20g^4\mp 32\sqrt{2}g^2-32).
\]
Analogously, we observe that
\[
a_4^2=4g^2,\quad a_4a_5=\mp 4\sqrt{2}g^4-16g^2,\quad a_5^2=8g^6\pm 32\sqrt{2}g^4+64g^2
\]
and we get
\[
\biggl(\frac{(g^4-8)(g^4-7)}{2}a_4^2-4(g^4-7)a_4a_5+6a_5^2\biggr)a_4^2=8\bigl(g^{12}\pm 8\sqrt{2}g^{10}+41g^8\pm 40\sqrt{2}g^6+24g^4\bigr).
\]
Finally, we just have to replace these elements to get the values of $a_1$, that is
\begin{align*}
a_1&=8\bigl(g^{12}\pm 8\sqrt{2}g^{10}+47g^8\pm 56\sqrt{2}g^6+16g^4\mp 64\sqrt{2}g^2-64\bigr)\\
&=8(g^4-1)\bigl(g^{8}\pm 8\sqrt{2}g^6+48g^4\pm 64\sqrt{2}g^2+64\bigr)\\
&=8(g^4-1)\bigl(g^2\pm 2\sqrt{2}\bigr)^4,
\end{align*}
which confirms the required decomposition.
\end{proof}
The family of decompositions of \autoref{teo_minimal_decomposition_q_n^2} let us without any upper bound for the case of $8$ variables. This happens beacase in the system of the equation appearing in the proof, the coefficient $c$ must be $0$, which gives a contradiction with the structure of the decomposition. However, if we add some points in the initial formula, we obtain a decomposition for $q_8^2$ of size $T_{8,2}+9=45$.
\begin{teo}
\label{teo_minimal_decomposition_q8^2}
The form $q_8^2$ can be decomposed as
\begin{equation}
\label{rel_minimal_decomposition_q_8^2}
q_8^2=\frac{3}{256}\biggl(\sum_{j=1}^8x_{j}\biggr)^4-\frac{8}{9}\sum_{j=1}^8x_j^4+\frac{8}{9}\sum_{k=1}^8\Biggl(\frac{3}{16}{\biggl(\sum_{j=1}^8x_{j}\biggl)}+{x_{k}}\Biggr)^4+\frac{1}{3}\sum_{j_1\neq j_2}\Biggl(-\frac{3}{8}\biggl(\sum_{j=1}^8x_{j}\biggr)+(x_{j_1}+x_{j_2})\Biggr)^4.
\end{equation}
\end{teo}
\begin{proof}
Considering an equation analog to decomposition \eqref{formula:initial_equation_qn2}, write
\begin{equation}
\label{formula:equation_n=8_unknows}
q_n^2=a\biggl(\sum_{j=1}^n x_{j}\biggr)^4-f\sum_{j=1}^8x_j^4+\sum_{k=1}^n\Biggl(b{\biggl(\sum_{j=1}^nx_{j}\biggl)}+tx_{k}\Biggr)^4+\sum_{j_1\neq j_2}\Biggl(d\biggl(\sum_{j=1}^nx_{j}\biggr)+e(x_{j_1}+x_{j_2})\Biggr)^4,
\end{equation}
with coefficients $a,b,d,e,f,t\in\bbC$.
By proceeding in the same way as in the proof of \autoref{teo_minimal_decomposition_q_n^2}, we obtain a system
\[
\begin{cases}
3t^4-3f=(8-n),\\
(n-4)de^3+e^4+bt^3=0,\\
3e^4=1,\\
b^2t^2+(n-2)d^2e^2+2de^3=0,\\
2a+2nb^4+8b^3t+n(n-1)d^4+8(n-1)d^3e+12d^2e^2=0.
\end{cases}
\]
Therefore, if we take $n=8$, then we have \[e=\frac{\tau_4}{\sqrt[4]{3}},\qquad f=t^4,\]
and we obtain the system
\[
\begin{cases}
12\tau_4^3d+3\sqrt[4]{27}bt^3+\sqrt[4]{27}=0,\\
\sqrt[4]{27}b^2t^2+6\sqrt[4]{3}\tau_4^2d^2+2\tau_4^3d=0,\\
a+8b^4+4b^3t+28d^4+28d^3e+6d^2e^2=0.
\end{cases}
\]
Thus, we can write
\[
d=\frac{\sqrt[4]{27}(-3bt^3-1)\tau_4}{12},
\]
where $\tau_4\in\bbC$ is an arbitrary fourth root of unity,
and hence, replacing it in the second equation, we have
\[
3t^2(9t^4+8)b^2+6bt^3-1=0.
\]
Now, we assign the value
\[
t=\frac{\sqrt[4]{8}}{\sqrt{3}}\tau_8,
\]
where $\tau_8$ is an arbitrary primitive $8$-th root of unity,
so that $9t^4+8=0$. In particular, we have
\[
t^3=\frac{4\sqrt[4]{2}}{3\sqrt{3}}\tau_8^3,
\]
which implies
\[
b=\frac{1}{6t^3}=-\frac{\sqrt{3}}{8\sqrt[4]{2}}\tau_8,\qquad d=-\frac{\sqrt[4]{27}}{8}\tau_4
\]
Finally, we have to substitute the values of $b$ and $d$ in the third equation of the system, which is
\[
a+8b^4+4b^3t+28d^4+28d^3e+6d^2e^2=0,
\]
obtaining
\begin{align*}
&a+8\biggl(-\frac{\sqrt{3}}{8\sqrt[4]{2}}\tau_8\biggr)^4+4\biggl(-\frac{\sqrt{3}}{8\sqrt[4]{2}}\tau_8\biggr)^3\biggl(\frac{\sqrt[4]{8}}{\sqrt{3}}\tau_8\biggr)+28\biggl(-\frac{\sqrt[4]{27}}{8}\tau_4\biggr)^4+28\biggl(-\frac{\sqrt[4]{27}}{8}\tau_4\biggr)^3\biggl(\frac{\tau_4}{\sqrt[4]{3}}\biggr)\\[1ex]
&\qquad+6\biggl(-\frac{\sqrt[4]{27}}{8}\tau_4\biggr)^2\biggl(\frac{\tau_4}{\sqrt[4]{3}}\biggr)^2=0,
\end{align*}
that is,
\begin{align*}
a&=8\biggl(\frac{\sqrt{3}}{8\sqrt[4]{2}}\biggr)^4-4\biggl(\frac{\sqrt{3}}{8\sqrt[4]{2}}\biggr)^3\biggl(\frac{\sqrt[4]{8}}{\sqrt{3}}\biggr)-28\biggl(\frac{\sqrt[4]{27}}{8}\biggr)^4+28\biggl(\frac{\sqrt[4]{27}}{8}\biggr)^3\biggl(\frac{1}{\sqrt[4]{3}}\biggr)-6\biggl(\frac{\sqrt[4]{27}}{8}\biggr)^2\biggl(\frac{1}{\sqrt[4]{3}}\biggr)^2\\[1ex]
&=\frac{9}{1024}-\frac{3}{128}-\frac{189}{1024}+\frac{63}{ 128}-\frac{9}{32}=\frac{3}{256}.
\end{align*}
Substituting all the values in equation \eqref{formula:equation_n=8_unknows} and factorizing the coefficients in the summands, we obtain decomposition \eqref{rel_minimal_decomposition_q_8^2}.
\end{proof}
By \autoref{teo_tight_decompositions_s=2}, \autoref{teo_minimal_decomposition_q_n^2}, \autoref{teo_minimal_decomposition_q8^2}, and the tight decompositions for $n=7,23$ mentioned in \cite{DGS77}*{p.~371}, we can summarize the results concerning the rank of $q_n^2$.
\begin{teo}
\label{teo:summarize_rank_q_n^2}
Let $n\geq 3$. Then the following conditions hold:
\begin{enumerate}[label=(\arabic*), left= 0pt, widest=*,nosep]
\item if $n=3,7,23$, then $\rk(q_n^2)=T_{n,2}$;
\item if $n=8$, then $T_{n,2}+1\leq\rk(q_n^2)\leq T_{n,2}+9$;
\item if $n>23$ and $n=m^2-2$ for some odd number $m\in\bbN$, then $T_{n,2}\leq\rk(q_n^2)\leq T_{n,2}+1$;
\item otherwise, $\rk(q_n^2)=T_{n,2}+1$.
\end{enumerate}
\end{teo}

\subsection{Decompositions in three variables}
\label{sec_decomp_three_variables}
The decompositions presented in the previous section are not the only examples of closed formulas which are valid for multiple values of $n$. In particular, we can find several decompositions also for higher exponents.
Again, we can find some simpler decompositions of the form $q_n^3$ in the classical literature. Probably one of the less recent among them is the one for the case of $4$ variables and having size $24$, presented by A.~Kempner in \cite{Kem12}*{section 5} and equal to
\begin{equation}
q_4^3=\frac{8}{15}\sum_{j=1}^4x_j^6+\frac{1}{15}\sum_{j_1<j_2}(x_{j_1}\pm x_{j_2})^6+\frac{1}{120}\sum(x_1\pm x_2\pm x_3\pm x_4)^6.
\end{equation}
This last decomposition is a special case of the set of decompositions that can be obtained by the family of quadrature formulas provided by A.~H.~Stroud in \cite{Str67b} and revealed by B.~Reznick in \cite{Rez92}*{formula (8.33)}. These are given, for every $n\geq 3$ such that $n\neq 8$, by the formula
\begin{equation}
q_n^3=\frac{2(8-n)}{15}\sum_{j=1}^nx_j^6+\frac{1}{15}\sum_{j_1<j_2}(x_{j_1}\pm x_{j_2})^6+\frac{1}{15\cdot 2^{n-1}}\sum(x_1\pm\dots\pm x_n)^6.
\end{equation}

Another formula producing decompositions as functions of $n$ is instead given by J.~Buczyński, K.~Han, M.~Mella, and Z.~Teitler in \cite{BHMT18}*{Section 4.5}, where they provide a decomposition of $q_n^3$ of size  
\[
4\binom{n}{3}+2\binom{n}{2}+n,
\]
given by the equation
\begin{equation}
60q_n^3=\sum_{j_1<j_2<j_3}(x_{j_1}\pm x_{j_2}\pm x_{j_3})^6+2(5-n)\sum_{j_1<j_2}(x_{j_1}\pm x_{j_2})^6+2(n^2-9n+38)\sum_{j=1}^nx_j^6.
\end{equation}

For the case where $n=3$, B.~Reznick also provides two minimal decompositions for the exponents $3$ and $4$. We can see in \cite{Rez92}*{Theorem 9.28} that the form $q_3^3$ can be described by a decomposition of size $11$ with real coefficients, given by
\begin{equation}
\label{decom_n=3_s=3_std}
q_3^3=\frac{14}{27}x_1^6+\frac{7}{10}\sum_{j=2,3}x_j^6+\frac{1}{540}\sum_{j=2,3}\bigl(2x_1\pm\sqrt{3}x_j\bigr)^6+\frac{1}{540}\sum\bigl(x_1\pm\sqrt{3}x_2\pm\sqrt{3}x_3\bigr)^6.
\end{equation}
Using the same strategy of \autoref{teo_tight_decompositions_n=3}, B.~Reznick proves that this is a minimal decomposition and, in particular, we will prove that it is not unique. 

Besides the decomposition \eqref{decom_n=3_s=3_std}, B.~Reznick also analyzes a minimal decomposition for the successive exponent. Indeed, he proves in \cite{Rez92}*{formula (8.31)} that the form $q_3^4$ can be decomposed as 
\begin{equation}
\label{decom_n=3_s=4_std}
140q_3^4=3\varphi^{-4}\sum_{j=1}^3(x_j\pm\varphi x_{j-1})^8+
\sum_{j=1}^3(\varphi x_j\pm\varphi^{-1} x_{j-1})^8+\sum(x_1\pm x_2\pm x_3)^8,
\end{equation}
where 
\[
\varphi=\frac{1+\sqrt{5}}{2}.
\]
Decomposition \eqref{decom_n=3_s=4_std} consists of a summation of $16$ points, corresponding, up to symmetries, to the vertices of a regular icosahedron together with the vertices of a regular dodecahedron (see \autoref{fig_dodec_icos})
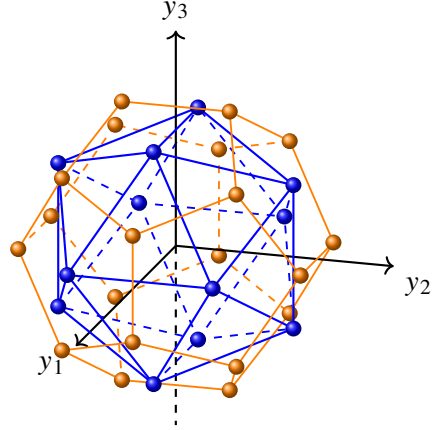
\begin{figure}[ht]
\center
\tdplotsetmaincoords{72}{107}
\begin{tikzpicture}[tdplot_main_coords,declare function={phi=(1+sqrt(5))/2;tri=sqrt(sqrt(sqrt(sqrt(3))));psi=(sqrt(5)-1)/2;psisqrt=sqrt((sqrt(5)-1)/2);phisqrt=sqrt((sqrt(5)+1)/2);}]
\begin{scope}
\coordinate (Vxy1) at (phi,1,0);
\coordinate (Vxy2) at (phi,-1,0);
\coordinate (Vxy3) at (-phi,1,0);
\coordinate (Vxy4) at (-phi,-1,0);
\coordinate (Vxz1) at (1,0,phi);
\coordinate (Vxz2) at (-1,0,phi);
\coordinate (Vxz3) at (1,0,-phi);
\coordinate (Vxz4) at (-1,0,-phi);
\coordinate (Vyz1) at (0,phi,1);
\coordinate (Vyz2) at (0,phi,-1);
\coordinate (Vyz3) at (0,-phi,1);
\coordinate (Vyz4) at (0,-phi,-1);

\draw[dashed, color=orange,line width=0.7pt] (-1.2*phi,0,1.2*psi) -- (-1.2,1.2,1.2);
\draw[dashed, color=orange,line width=0.7pt] (-1.2*phi,0,1.2*psi) -- (-1.2,-1.2,1.2);
\draw[dashed, color=orange,line width=0.7pt] (-1.2*phi,0,-1.2*psi) -- (-1.2,1.2,-1.2);
\draw[dashed, color=orange,line width=0.7pt] (-1.2*phi,0,-1.2*psi) -- (-1.2,-1.2,-1.2);

\draw[thick, color=orange,line width=0.7pt] (0,1.2*psi,-1.2*phi) -- (0,-1.2*psi,-1.2*phi);
\draw[dashed, color=orange,line width=0.7pt] (-1.2*phi,0,1.2*psi) -- (-1.2*phi,0,-1.2*psi);
\draw[thick, color=orange,line width=0.7pt] (1.2*psi,-1.2*phi,0) -- (1.2,-1.2,-1.2);
\draw[thick, color=orange,line width=0.7pt] (1.2,-1.2,-1.2) -- (0,-1.2*psi,-1.2*phi);
\draw[dashed, color=orange,line width=0.7pt] (1.2*psi,-1.2*phi,0) -- (-1.2*psi,-1.2*phi,0);
\draw[dashed, color=orange,line width=0.7pt] (-1.2,-1.2,1.2) -- (-1.2*psi,-1.2*phi,0);
\draw[dashed, color=orange,line width=0.7pt] (-1.2,-1.2,1.2) -- (0,-1.2*psi,1.2*phi);
\draw[thick, color=orange,line width=0.7pt] (-1.2,1.2,-1.2) -- (-1.2*psi,1.2*phi,0);
\draw[thick, color=orange,line width=0.7pt] (-1.2,1.2,-1.2) -- (0,1.2*psi,-1.2*phi);
\draw[dashed, color=orange,line width=0.7pt] (-1.2,-1.2,-1.2) -- (-1.2*psi,-1.2*phi,0);
\draw[dashed, color=orange,line width=0.7pt] (-1.2,-1.2,-1.2) -- (0,-1.2*psi,-1.2*phi);

\draw[dashed, color=blue,line width=0.7pt] (-phi,-1,0) -- (-phi,1,0);
\draw[dashed, color=blue,line width=0.7pt] (-phi,-1,0) -- (-1,0,-phi);
\draw[dashed, color=blue,line width=0.7pt] (-phi,-1,0) -- (-1,0,phi);
\draw[dashed, color=blue,line width=0.7pt] (-phi,1,0) -- (-1,0,-phi);
\draw[dashed, color=blue,line width=0.7pt] (-phi,1,0) -- (-1,0,phi);
\draw[dashed, color=blue,line width=0.7pt] (1,0,-phi) -- (-1,0,-phi);
\draw[dashed, color=blue,line width=0.7pt] (-1,0,-phi) -- (0,phi,-1);
\draw[dashed, color=blue,line width=0.7pt] (-1,0,-phi) -- (0,-phi,-1);
\draw[dashed, color=blue,line width=0.7pt] (0,phi,1) -- (-phi,1,0);
\draw[dashed, color=blue,line width=0.7pt] (0,phi,-1) -- (-phi,1,0);
\draw[dashed, color=blue,line width=0.7pt] (0,-phi,1) -- (-phi,-1,0);
\draw[dashed, color=blue,line width=0.7pt] (0,-phi,-1) -- (-phi,-1,0);

\pgfpathcircle{\pgfpointxyz{-1.2*(phi)}{0}{1.2*(psi)}}{3pt};
\shade[ball color=orange];
\pgfpathcircle{\pgfpointxyz{-1.2*(phi)}{0}{-1.2*(psi)}}{3pt};
\shade[ball color=orange];

\pgfpathcircle{\pgfpointxyz{-1.2}{1.2}{-1.2}}{3pt};
\shade[ball color=orange];
\pgfpathcircle{\pgfpointxyz{-1.2}{-1.2}{-1.2}}{3pt};
\shade[ball color=orange];

\draw[thick, color=blue] (-1,0,phi) -- (0,-phi,1);

\pgfsetdash{}{0pt}
\color{black}
\draw[thick,->] (0,0,0) -- (4.5,0,0) node[anchor=north east]{$y_1$};
\draw[thick,->] (0,0,0) -- (0,3,0) node[anchor=north west]{$y_2$};
\draw[thick,->] (0,0,0) -- (0,0,3) node[anchor=south]{$y_3$};
\draw[thick,dashed] (0,0,0) -- (0,0,-2.5);

\draw[thick, color=blue] (phi,1,0) -- (phi,-1,0);
\draw[thick, color=blue] (phi,1,0) -- (1,0,phi);
\draw[thick, color=blue] (phi,1,0) -- (1,0,-phi);
\draw[thick, color=blue] (phi,-1,0) -- (1,0,phi);
\draw[thick, color=blue] (phi,-1,0) -- (1,0,-phi);

\draw[thick, color=blue] (1,0,phi) -- (-1,0,phi);
\draw[thick, color=blue] (1,0,phi) -- (0,phi,1);
\draw[thick, color=blue] (1,0,phi) -- (0,-phi,1);
\draw[thick, color=blue] (-1,0,phi) -- (0,phi,1);
\draw[thick, color=blue] (1,0,-phi) -- (0,phi,-1);
\draw[thick, color=blue] (1,0,-phi) -- (0,-phi,-1);
\draw[thick, color=blue] (0,phi,1) -- (0,phi,-1);
\draw[thick, color=blue] (0,phi,1) -- (phi,1,0);
\draw[thick, color=blue] (0,phi,-1) -- (phi,1,0);
\draw[thick, color=blue] (0,-phi,-1) -- (0,-phi,1);
\draw[thick, color=blue] (phi,-1,0) -- (0,-phi,1);
\draw[thick, color=blue] (phi,-1,0) -- (0,-phi,-1);

\draw[thick, color=orange,line width=0.7pt] (1.2*psi,1.2*phi,0) -- (1.2,1.2,1.2);
\draw[thick, color=orange,line width=0.7pt] (1.2*psi,1.2*phi,0) -- (-1.2*psi,1.2*phi,0);
\draw[thick, color=orange,line width=0.7pt] (-1.2,1.2,1.2) -- (-1.2*psi,1.2*phi,0);
\draw[thick, color=orange,line width=0.7pt] (-1.2,1.2,1.2) -- (0,1.2*psi,1.2*phi);
\draw[thick, color=orange,line width=0.7pt] (1.2,1.2,1.2) -- (0,1.2*psi,1.2*phi);
\draw[thick, color=orange,line width=0.7pt] (1.2*psi,-1.2*phi,0) -- (1.2,-1.2,1.2);
\draw[thick, color=orange,line width=0.7pt] (1.2,-1.2,1.2) -- (0,-1.2*psi,1.2*phi);
\draw[thick, color=orange,line width=0.7pt] (1.2*psi,1.2*phi,0) -- (1.2,1.2,-1.2);
\draw[thick, color=orange,line width=0.7pt] (1.2,1.2,-1.2) -- (0,1.2*psi,-1.2*phi);

\draw[thick, color=orange,line width=0.7pt] (1.2*phi,0,1.2*psi) -- (1.2*phi,0,-1.2*psi);
\draw[thick, color=orange,line width=0.7pt] (1.2*phi,0,1.2*psi) -- (1.2,1.2,1.2);
\draw[thick, color=orange,line width=0.7pt] (1.2*phi,0,1.2*psi) -- (1.2,-1.2,1.2);
\draw[thick, color=orange,line width=0.7pt] (1.2*phi,0,-1.2*psi) -- (1.2,1.2,-1.2);
\draw[thick, color=orange,line width=0.7pt] (1.2*phi,0,-1.2*psi) -- (1.2,-1.2,-1.2);

\foreach \x in {1,-1}
{\foreach \y in {phi,-phi}
{\pgfpathcircle{\pgfpointxyz{\x}{0}{\y}}{3pt};
\shade[ball color=blue];
\pgfpathcircle{\pgfpointxyz{\y}{\x}{0}}{3pt};
\shade[ball color=blue];
\pgfpathcircle{\pgfpointxyz{0}{\y}{\x}}{3pt};
\shade[ball color=blue];}}

\draw[thick, color=orange,line width=0.7pt] (0,1.2*psi,1.2*phi) -- (0,-1.2*psi,1.2*phi);

\pgfpathcircle{\pgfpointxyz{1.2*(psi)}{1.2*(phi)}{0}}{3pt};
\shade[ball color=orange];
\pgfpathcircle{\pgfpointxyz{1.2}{1.2}{1.2}}{3pt};
\shade[ball color=orange];
\pgfpathcircle{\pgfpointxyz{-1.2}{1.2}{1.2}}{3pt};
\shade[ball color=orange];
\pgfpathcircle{\pgfpointxyz{1.2}{-1.2}{1.2}}{3pt};
\shade[ball color=orange];
\pgfpathcircle{\pgfpointxyz{1.2}{1.2}{-1.2}}{3pt};
\shade[ball color=orange];
\pgfpathcircle{\pgfpointxyz{-1.2}{-1.2}{1.2}}{3pt};
\shade[ball color=orange];
\pgfpathcircle{\pgfpointxyz{1.2}{-1.2}{-1.2}}{3pt};
\shade[ball color=orange];
\pgfpathcircle{\pgfpointxyz{1.2*(psi)}{1.2*(phi)}{0}}{3pt};
\shade[ball color=orange];
\pgfpathcircle{\pgfpointxyz{-1.2*(psi)}{1.2*(phi)}{0}}{3pt};
\shade[ball color=orange];
\pgfpathcircle{\pgfpointxyz{1.2*(psi)}{-1.2*(phi)}{0}}{3pt};
\shade[ball color=orange];
\pgfpathcircle{\pgfpointxyz{-1.2*(psi)}{-1.2*(phi)}{0}}{3pt};
\shade[ball color=orange];
\pgfpathcircle{\pgfpointxyz{0}{1.2*(psi)}{1.2*(phi)}}{3pt};
\shade[ball color=orange];
\pgfpathcircle{\pgfpointxyz{0}{-1.2*(psi)}{1.2*(phi)}}{3pt};
\shade[ball color=orange];
\pgfpathcircle{\pgfpointxyz{0}{1.2*(psi)}{-1.2*(phi)}}{3pt};
\shade[ball color=orange];
\pgfpathcircle{\pgfpointxyz{0}{-1.2*(psi)}{-1.2*(phi)}}{3pt};
\shade[ball color=orange];
\pgfpathcircle{\pgfpointxyz{1.2*(phi)}{0}{1.2*(psi)}}{3pt};
\shade[ball color=orange];
\pgfpathcircle{\pgfpointxyz{1.2*(phi)}{0}{-1.2*(psi)}}{3pt};
\shade[ball color=orange];
\end{scope}
\end{tikzpicture}
\caption[Graphical representation of decomposition \eqref{decom_n=3_s=4_std}]{Graphical representation of decomposition \eqref{decom_n=3_s=4_std}, whose elements correspond, up to central simmetry, to the vertices of a regular icosahedron (in blue) togheter with the vertices of a regular dodecahedron (in orange).}
\label{fig_dodec_icos}
\end{figure}

To analyze the case of three variables, it may be appropriate to use another set of coordinates, by changing the basis of the space $\calD_3=\bbC[y_1,y_2,y_3]$. Indeed, setting 
\begin{equation}
\label{rel_new_coord}
u=\frac{y_1+\rmi y_2}{2},\quad v=\frac{y_1-\rmi y_2}{2},\quad z=y_3,
\end{equation}
and thus considering the space $\calD=\bbC[u,v,z]$,
we can consider the inverse relations
\[
y_1={u+v},\quad y_2={-\rmi(u-v)},\quad y_3=z.
\]
Furthermore, we can determine the partial derivatives with respect to this new set of coordinates, i.e.,
\[
\pdv{}{u}=\pdv{}{y_1}-\rmi\pdv{}{y_2},\quad
\pdv{}{v}=\pdv{}{y_1}+\rmi\pdv{}{y_2},\quad
\pdv{}{z}=\pdv{}{y_3}.
\]
Consequently, the Laplace operator can be rewritten as
\[
\Delta=\pdv[2]{}{y_1}+\pdv[2]{}{y_2}+\pdv[2]{}{y_3}=\pdv{}{u,v}+\pdv[2]{}{z}.
\]

We have already recalled that the space of harmonic polynomials $\calH_{n,d}$ of degree $d$ is an irreducible $\SO_n(\bbC)$-module for every $d\in\bbN$ (\cite{GW98}*{Theorem 5.2.4}). 
The Lie algebra of the Lie group $\SO_3(\bbC)$ is the space
\[
\mfso_3(\bbC)=\Set{A\in\Mat_3(\bbC)|A=-\transpose{A}}
\]
and, if we consider the three matrices associated with the morphisms $H,E,F\colon \bbC^3\to\bbC^3$ with respect to the canonical basis $\{y_1,y_2,y_3\}$, given by
\begin{align*}
&\begin{cases}
H(y_1)=2\rmi y_2,\\
H(y_2)=-2\rmi y_1,\\
H(y_3)=0,
\end{cases}&
H&=\begin{pNiceMatrix}
0 & -2\rmi & 0 \\
2\rmi & 0 & 0 \\
0 & 0 & 0  
\end{pNiceMatrix},\\[1ex]
&\begin{cases}
E(y_1)=y_3,\\
E(y_2)=\rmi y_3,\\
E(y_3)=-(y_1+\rmi y_2),
\end{cases}&
E&=\begin{pNiceMatrix}
0 & 0 & -1 \\
0 & 0 & -\rmi \\
1 & \rmi & 0  
\end{pNiceMatrix},\\[1ex]
&\begin{cases}
F(y_1)=-y_3,\\
F(y_2)=\rmi y_3,\\
F(y_3)=y_1-\rmi y_2,
\end{cases}&
F&=\begin{pNiceMatrix}
0 & 0 & 1 \\
0 & 0 & -\rmi \\
-1 & \rmi & 0  
\end{pNiceMatrix},
\end{align*}
then we obtain the equations
\begin{align*}
[H,E]&=HE-EH=\begin{pNiceMatrix}
0 & 0 & -2 \\
0 & 0 & -2\rmi \\
0 & 0 & 0  
\end{pNiceMatrix}
-\begin{pNiceMatrix}
0 & 0 & 0 \\
0 & 0 & 0 \\
-2 & -2\rmi & 0  
\end{pNiceMatrix}=
\begin{pNiceMatrix}
0 & 0 & -2 \\
0 & 0 & -2\rmi \\
2 & 2\rmi & 0  
\end{pNiceMatrix}=2E,\\[2ex]
[H,F]&=HF-FH=\begin{pNiceMatrix}
0 & 0 & -2 \\
0 & 0 & 2\rmi \\
0 & 0 & 0  
\end{pNiceMatrix}
-\begin{pNiceMatrix}
0 & 0 & 0 \\
0 & 0 & 0 \\
-2 & 2\rmi & 0  
\end{pNiceMatrix}=
\begin{pNiceMatrix}
0 & 0 & -2 \\
0 & 0 & 2\rmi \\
2 & -2\rmi & 0  
\end{pNiceMatrix}=-2F,\\[2ex]
[E,F]&=EF-FE=\begin{pNiceMatrix}
1 & -\rmi & 0 \\
\rmi & 1 & 0 \\
0 & 0 & 2  
\end{pNiceMatrix}
-\begin{pNiceMatrix}
1 & \rmi & 0 \\
-\rmi & 1 & 0 \\
0 & 0 & 2  
\end{pNiceMatrix}=
\begin{pNiceMatrix}
0 & -2\rmi & 0 \\
2\rmi & 0 & 0 \\
0 & 0 & 0  
\end{pNiceMatrix}=H.\\
\end{align*}
Since these conditions correspond exactly to equations \eqref{rel_equations_sl2C}, we have, in particular, that $\mfso_3\bbC\cong\mfsl_2\bbC$, which justifies the use of the same notations.

Now, by rewriting the matrices $H$, $E$ and $F$ with respect to the new basis $\{u,v,z\}$, we clearly do not get skew-symmetric matrices, but new ones with a much simpler structure. Indeed, we have
\begin{align*}
&\begin{cases}
H(u)=2u,\\
H(v)=-2v,\\
H(z)=0,
\end{cases}&
H&=\begin{pNiceMatrix}
2 & 0 & 0 \\
0 & -2 & 0 \\
0 & 0 & 0  
\end{pNiceMatrix},\\[1ex]
&\begin{cases}
E(u)=0,\\
E(v)=z,\\
E(z)=-2u,
\end{cases}&
E&=\begin{pNiceMatrix}
0 & 0 & -2 \\
0 & 0 & 0 \\
0 & 1 & 0  
\end{pNiceMatrix},\\[1ex]
&\begin{cases}
F(u)=-z,\\
F(v)=0,\\
F(z)=2v,
\end{cases}&
F&=\begin{pNiceMatrix}
0 & 0 & 0 \\
0 & 0 & 2 \\
-1 & 0 & 0  
\end{pNiceMatrix}.
\end{align*}
Moreover, it follows from the uniqueness of irreducible representations of $\mfsl_2\bbC$ that for every $d\in\bbN$ the space $\calH_{3,d}$ of harmonic polynomials of degree $d$ in three variables is canonically isomorphic to the $2d$-th symmetric power of the standard representation $\bbC^2$, namely, $\calH_{3,d}\cong S^{2d}\bbC^2$. We can see it in the following proposition.

\begin{prop} 
\label{prop harmonic basis}
For every $d\in\bbN$, the set
\[
\calB_d=\pg*{h_{d,k}}_{k=-d,\dots,d}
\]
composed of the harmonic polynomials
\begin{equation}
\label{rel_basis_weights_harmonic_polynomials}
h_{d,k}={\binom{2d}{k+d}}^{-1}\sum_{j=0}^{\pint*{\frac{d-\abs*{k}}{2}}}(-1)^{\frac{\abs*{k}+k}{2}+j}u^{\pq*{\frac{\abs*{k}+k}{2}+j}}v^{\pq*{\frac{\abs*{k}-k}{2}+j}}z^{\pq*{d-\abs*{k}-2j}},
\end{equation}
for $k=-d,\dots,d$, is a basis of the space $\calH_{3,d}$, which corresponds in a canonical way, up to scalars, to the standard basis of $S^{2d}\bbC^2$, that is the set of monomials $\{x^{2d},x^{2d-1}y,\dots,xy^{2d-1},y^{2d}\}$.
In particular, for every $k=-d,\dots,d$ 
\[
H(h_{d,k})=2k\cdot h_{d,k}
\]
and for every $s\in\bbN$ such that $k-d\leq s\leq d-k$,
\[
E^{s}(h_{d,k})\in\langle h_{d,k+s}\rangle,\quad F^{s}(h_{d,k})\in\langle h_{d,k-s}\rangle.
\]
\end{prop}

\begin{proof}
We must first prove that for every $k\in\bbN$ the polynomial $h_{d,k}$ is an eigenvector of the linear map represented by $H$ with eigenvalue $2k$. So, we have
\[
H(h_{d,k})=\binom{2d}{k+d}^{-1}H\Biggl(\sum_{j=0}^{\pint*{\frac{d-\abs*{k}}{2}}}(-1)^{\frac{\abs*{k}+k}{2}+j}u^{\pq*{\frac{\abs*{k}+k}{2}+j}}v^{\pq*{\frac{\abs*{k}-k}{2}+j}}z^{\pq*{d-\abs*{k}-2j}}\biggr)
\]
and in the case where $k\geq 0$ we get
\begin{align*}
H(h_{d,k})&=\binom{2d}{k+d}^{-1}H\Biggl(\sum_{j=0}^{\pint*{\frac{d-k}{2}}}(-1)^{k+j}u^{\pq*{k+j}}v^{\pq*{j}}z^{\pq*{d-k-2j}}\Biggr)\\[1ex]
&=\binom{2n}{k+d}^{-1}
\sum_{j=0}^{\pint*{\frac{d-k}{2}}}
(-1)^{k+j}\frac{H\pt*{u^{k+j}}v^jz^{d-k-2j}+u^{k+j}H\pt*{v^j}z^{d-k-2j}}{j!(k+j)!(d-k-2j)!}\\[1ex]
&=\binom{2d}{k+d}^{-1}\sum_{j=0}^{\pint*{\frac{d-k}{2}}}
(-1)^{k+j}\frac{2(k+j)\pt*{u^{k+j}v^jz^{d-k-2j}}-2j\pt*{u^{k+j}v^jz^{d-k-2j}}}{j!(k+j)!(d-k-2j)!}\\[1ex]
&=2k\binom{2d}{k+d}^{-1}\sum_{j=0}^{\pint*{\frac{d-k}{2}}}(-1)^{k+j}u^{\pq*{k+j}}v^{\pq*{j}}z^{\pq*{d-k-2j}}=2k\cdot h_{d,k}.
\end{align*}
The case where $k<0$ is obtained in the same way by simply reversing the roles of $u$ and $v$.
Now, to get the statement, since for the space $S^{2d}\bbC^2$ we have 
\[
E\bigl(x^{2d}\bigr)=0,\quad E\bigl(x^{d+k}y^{d-k}\bigr)=(d-k)x^{d+k+1}y^{d-k-1},
\]
for every $k=-d,\dots,d-1$, and
\[
F\bigl(y^{2d}\bigr)=0,\quad F\bigl(x^{d+k}y^{d-k}\bigr)=(d+k)x^{d+k-1}y^{d-k+1},
\]
for every $k=-d+1,\dots,d$,
we only have to prove that 
\[
E(h_{d,d})=0,\quad E(h_{d,k})=(d-k)h_{d,k+1},
\]
for every $k=-d,\dots,d-1$, and
\[
F(h_{d,-d})=0,\quad F(h_{d,k})=(d+k)h_{d,k-1},
\]
for every $k=-d+1,\dots,d$.
Again, only considering the case where $k\geq 0$, we can write
\[
E(h_{d,k})=\binom{2d}{k+d}^{-1}E\Biggl(\sum_{j=0}^{\pint*{\frac{d-k}{2}}}(-1)^{k+j}u^{\pq*{k+j}}v^{\pq*{j}}z^{\pq*{d-k-2j}}\Biggr)
\]
and it is quite convenient to analyze separately the different cases where we have $d-k$ even or odd separately, due to the presence of an extra term. In the first case, if we set $d-k=2t$ for some $t\in\bbN$, we get
\begin{align*}
E(h_{d,k})&=\binom{2d}{2d-2t}^{-1}
\sum_{j=0}^{t}
(-1)^{d-2t+j}\frac{u^{d-2t+j}E\pt*{v^j}z^{2t-2j}+u^{d-2t+j}v^jE\pt*{z^{2t-2j}}}{j!(d-2t+j)!(2t-2j)!}\\[1ex]
&=\binom{2d}{2d-2t}^{-1}\biggl(
\sum_{j=1}^{t}
(-1)^{d-2t+j}\frac{u^{d-2t+j}v^{j-1}z^{2t-2j+1}}{(d-2t+j)!(j-1)!(2t-2j)!}\\[1ex]
&\hphantom{{}={}}-2\sum_{j=0}^{t-1}
(-1)^{d-2t+j}\frac{u^{d-2t+j+1}v^{j}z^{2t-2j-1}}{(d-2t+j)!j!(2t-2j-1)!}\biggr).
\end{align*}
which, if we set in the first summation $j'=j-1$, gives us
\begin{align*}
E(h_{d,k})&=\binom{2d}{2d-2t}^{-1}
\biggl(\sum_{j'=0}^{t-1}
(-1)^{d-2t+1+j'}\frac{u^{d-2t+1+j'}v^{j'}z^{2t-2j'-1}}{(d-2t+1+j')!j'!(2t-2j'-2)!)}\\[1ex]
&\hphantom{{}={}}+2\sum_{j=0}^{t-1}
(-1)^{d-2t+1+j}\frac{u^{d-2t+j+1}v^{j}z^{2t-2j-1}}{(d-2t+j)!j!(2t-2j-1)!}\biggr)\\[1ex]
&=\frac{(2d-2t)!(2t)!}{(2d)!}
(2d-2t+1)\sum_{j=0}^{t-1}
(-1)^{d-2t+1+j}u^{\pq*{d-2t+j+1}}v^{[j]}z^{\pq*{2t-2j-1}}\\[1ex]
&=2t\frac{(2d-2t+1)!(2t-1)!}{(2d)!}
\sum_{j=0}^{t-1}
(-1)^{d-2t+1+j}u^{\pq*{d-2t+j+1}}v^{[j]}z^{\pq*{2t-2j-1}}.
\end{align*}
Then we finally obtain, considering $k=d-2t$,
\begin{align*}
E(h_{d,k})&=(d-k)\binom{2d}{d+k+1}^{-1}
\sum_{j=0}^{\pint*{\frac{d-k-1}{2}}}
(-1)^{d-2t+1+j}u^{\pq*{k+1+j}}v^{[j]}z^{\pq*{d-k-1-2j}}\\[1ex]
&=(d-k)h_{d,k+1}.
\end{align*}
The case where $k=d$ is trivial, as we have
\[
E\bigl(u^{[d]}\bigr)=u^{[d-1]}E(u)=0.
\]
If we instead set $d-k=2t+1$ for some $t\in\bbN$ we get something similar, i.e.,
\begin{align*}
E(h_{d,k})&=\dbinom{2d}{2d-2t-1}^{-1}\sum_{j=0}^{t}
(-1)^{d-2t+j-1}\dfrac{u^{d-2t+j-1}E\pt*{v^j}z^{2t-2j+1}+u^{d-2t+j-1}v^jE\pt*{z^{2t-2j+1}}}{j!(d-2t+j-1)!(2t-2j+1)!}\\[1ex]
&=\binom{2d}{2d-2t-1}^{-1}
\biggl(\sum_{j=1}^{t}
(-1)^{d-2t+j-1}\frac{u^{d-2t+j-1}v^{j-1}z^{2t-2j+2}}{(d-2t+j-1)!(j-1)!(2t-2j+1)!}\\[1ex]
&\hphantom{{}={}} -2\sum_{j=0}^{t}
(-1)^{d-2t+j-1}\frac{u^{d-2t+j}v^{j}z^{2t-2j}}{(d-2t+j-1)!j!(2t-2j)!}\biggr).
\end{align*}
Then, if we set $j'=j-1$ in the first summation, as we did above for the previous case, we get
\begin{align*}
E(h_{d,k})&=\binom{2d}{2d-2t-1}^{-1}\biggl(
\sum_{j'=0}^{t-1}
(-1)^{d-2t+j'}\frac{u^{d-2t+j'}v^{j'}z^{2t-2j'}}{(d-2t+j')!j'!(2t-2j'-1)!}\\[1ex]
&\hphantom{{}={}} +2\sum_{j=0}^{t-1}
(-1)^{d-2t+j}\frac{u^{d-2t+j}v^{j}z^{2t-2j}}{(d-2t+j-1)!j!(2t-2j)!}+2
(-1)^{d-t}\frac{u^{d-t}v^{t}}{(d-t-1)!t!}\biggr)\\
&=\frac{(2d-2t-1)!(2t+1)!}{(2d)!}
(2d-2t)\sum_{j=0}^{t}
(-1)^{d-2t+j}u^{\pq*{d-2t+j}}v^{[j]}z^{\pq*{2t-2j}}\\[1ex]
&=(2t+1)\frac{(2d-2t)!(2t)!}{(2d)!}
\sum_{j=0}^{t}
(-1)^{d-2t+j}u^{\pq*{d-2t+j}}v^{[j]}z^{\pq*{2t-2j}}.
\end{align*}
Finally, if we set $k=d-2t-1$, we have
\begin{align*}
E(h_{d,k})&=(d-k)\binom{2d}{d+k+1}^{-1}
\sum_{j=0}^{\pint*{\frac{d-k-1}{2}}}
(-1)^{k+j+1}u^{\pq*{k+1+j}}v^{[j]}z^{\pq*{d-k-2j-1}}\\[1ex]
&=(d-k)h_{d,k+1}.
\end{align*}
The statement for the operator $F$ is obtained by reversing the roles of $u$ with $-v$.
\end{proof}
The spaces generated by the basis of harmonic polynomials defined in formulas \eqref{rel_basis_weights_harmonic_polynomials}, corresponding to the weights of representations of $\mfso_3(\bbC)$, can be visualized in \autoref{fig_weights_harmonic_polynomials}.
\begin{figure}[ht]
\[
\begin{gathered}
\adjustbox{scale=0.82,center}{
\begin{tikzcd}
\underset{\color{darkred}h_{2,2}}{\Bigl\langle\dfrac{1}{2}u^{2}\Bigr\rangle} \ar[r,yshift=2.2,"F"]  & \underset{\color{darkred}h_{2,1}}{\Bigl\langle\dfrac{1}{4}uz\Bigr\rangle} \ar[l,yshift=-2.2,"E"]\ar[r,yshift=2.2,"F"] & \underset{\color{darkred}h_{2,0}}{\Bigl\langle\dfrac{1}{12}\pt*{z^{2}-2uv}\Bigr\rangle} \ar[l,yshift=-2.2,"E"]\ar[r,yshift=2.2,"F"] & \underset{\color{darkred}h_{2,-1}}{\Bigl\langle\dfrac{1}{4}vz\Bigr\rangle} \ar[l,yshift=-2.2,"E"]\ar[r,yshift=2.2,"F"] & \underset{\color{darkred}h_{2,-2}}{\Bigl\langle\dfrac{1}{2}v^{2}\Bigr\rangle}\ar[l,yshift=-2.2,"E"]\hphantom{.}
\end{tikzcd}}\\
\adjustbox{scale=0.82,center}{
\begin{tikzcd}[column sep=small]
\underset{\color{darkred}h_{3,3}}{\Bigl\langle\dfrac{1}{6}u^3\Bigr\rangle} \ar[r,yshift=2.2,"F"]  & \underset{\color{darkred}h_{3,2}}{\Bigl\langle\dfrac{1}{12}u^2z\Bigr\rangle} \ar[l,yshift=-2.2,"E"]\ar[r,yshift=2.2,"F"] & \underset{\color{darkred}h_{3,1}}{\Bigl\langle\dfrac{1}{30}u\pt*{z^2-uv}}\Bigr\rangle \ar[l,yshift=-2.2,"E"]\ar[r,yshift=2.2,"F"] & \underset{\color{darkred}h_{3,0}}{\Bigl\langle\dfrac{1}{120}z\pt*{z^2-6uv}\Bigr\rangle} \ar[l,yshift=-2.2,"E"]\ar[r,yshift=2.2,"F"] & \underset{\color{darkred}h_{3,-1}}{\Bigl\langle\dfrac{1}{30}v\pt*{z^2-uv}\Bigr\rangle}\ar[l,yshift=-2.2,"E"]\ar[r,yshift=2.2,"F"]& \underset{\color{darkred}h_{3,-2}}{\Bigl\langle\dfrac{1}{12}v^2z\Bigr\rangle} \ar[l,yshift=-2.2,"E"]\ar[r,yshift=2.2,"F"] & \underset{\color{darkred}h_{3,-3}}{\Bigl\langle\dfrac{1}{6}v^3\Bigr\rangle}\ar[l,yshift=-2.2,"E"]\hphantom{.}
\end{tikzcd}}\\
\vdots\\
\adjustbox{scale=0.82,center}{
\begin{tikzcd}
\underset{\color{darkred}h_{n,n}}{\Bigl\langle\dfrac{1}{d!}u^d\Bigr\rangle} \ar[r,yshift=2.2,"F"]  & \underset{\color{darkred}h_{d,d-1}}{\Bigl\langle\dfrac{1}{2d(d-1)!}u^{d-1}z\Bigr\rangle} \ar[l,yshift=-2.2,"E"]\ar[r,yshift=2.2,"F"] & \cdots \ar[l,yshift=-2.2,"E"]\ar[r,yshift=2.2,"F"] & \underset{\color{darkred}h_{d,-(d-1)}}{\Bigl\langle\dfrac{1}{2d(d-1)!}v^{d-1}z\Bigr\rangle} \ar[l,yshift=-2.2,"E"]\ar[r,yshift=2.2,"F"] & \underset{\color{darkred}h_{d,-d}}{\Bigl\langle\dfrac{1}{d!}v^d\Bigr\rangle}\ar[l,yshift=-2.2,"E"].
\end{tikzcd}}
\end{gathered}
\]
\caption[Diagram representing the weights of the space of harmonic polynomials]{Diagram representing the weights of the space of harmonic polynomials defined by formulas \eqref{rel_basis_weights_harmonic_polynomials}.}
\label{fig_weights_harmonic_polynomials}
\end{figure}

We have already seen that the minimal decomposition of $q_3^2$ given by formula \eqref{rel_decomp_icos} is represented by a configuration of points forming a regular icosahedron in the three-dimensional space. This particular arrangement represents a case that, at least for real points, cannot be repeated among the higher powers of $q_n$. However, we can determine some decompositions that seem to introduce a new pattern of points, suggesting a possible repetition of what happens in decomposition \eqref{rel_decomp_icos} also for the successive cases.  
This new way of analyzing the decomposition is based on considering the points, only for the usefulness, with the new set of coordinates given by relations \eqref{rel_new_coord} and with the new basis given by the elements of the equations \eqref{rel_basis_weights_harmonic_polynomials}.
For the next cases, the use of the software system Macaulay2 (\cite{GS}) has been essential in determining the decompositions and their associated ideal of points. The variables appearing the next decompositions are the coordinates corresponding to the dual basis of the coordinates \eqref{rel_new_coord}, which we write as
\[
u'=x_1-\rmi x_2,\quad v'=x_1+\rmi x_2,\quad z'=x_3.
\]

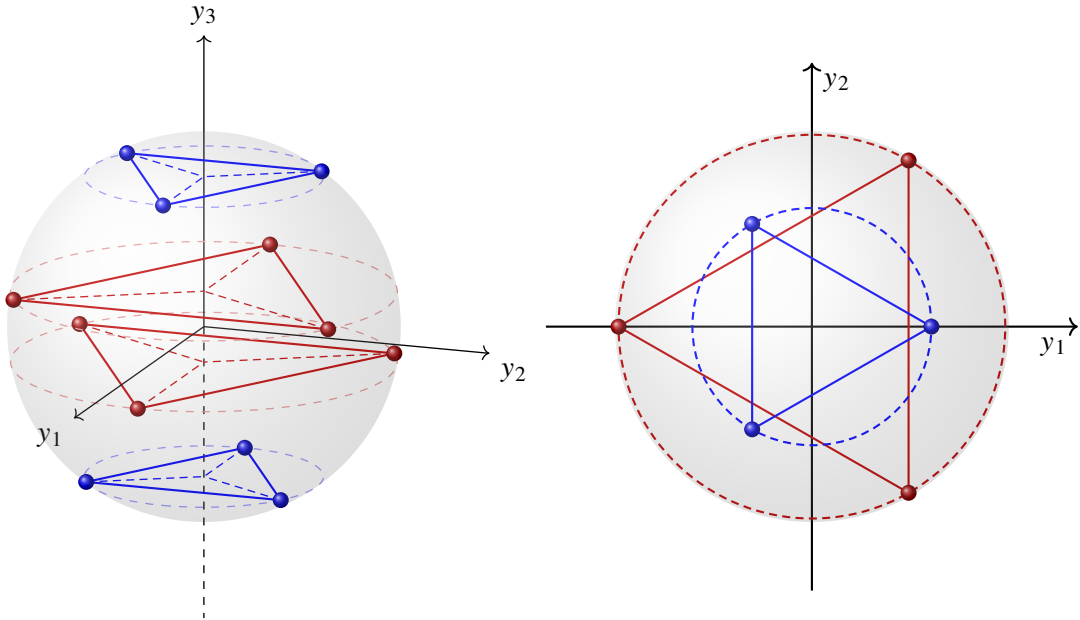
\begin{figure}[ht]
\tdplotsetmaincoords{75}{110}
\center
\begin{tikzpicture}[tdplot_main_coords,declare function={phi=(3+sqrt(5))/2;}]

\pgfmathsetmacro{\r}{(3-phi)^(1/4)*sqrt((phi^2)+4)}
\pgfmathsetmacro{\za}{phi*((3-phi)^(1/4))}
\pgfmathsetmacro{\zb}{((3-phi)*(phi)^(1/4))}

\pgfmathsetmacro{\ra}{2*(3-phi)^(1/4)}
\pgfmathsetmacro{\rb}{2*(phi)^(1/4)}

\draw[thick,line width=0.5pt,dashed] (0,0,0) -- (0,0,-4);

\pgfsetlinewidth{0.5pt}

\color{blue!40}
\draw[dashed] (0,0,-\za) circle (\ra);
  
\pgfsetdash{{3pt}{2pt}}{0pt}
\foreach \x in {60,180,300}{
\pgfpathmoveto{\pgfpointcylindrical{\x}{2*(3-phi)^(1/4)}{-(phi*(3-phi)^(1/4))}}
\pgfpathlineto{\pgfpointcylindrical{0}{0}{-(phi*(3-phi)^(1/4))}}}
\color{blue}
\pgfusepath{fill,stroke}

\pgfsetlinewidth{0.8pt}
\pgfsetdash{}{0pt}
\foreach \x in {60,180,300}{
\pgfpathmoveto{\pgfpointcylindrical{\x}{2*(3-phi)^(1/4)}{-(phi*(3-phi)^(1/4))}}
\pgfpathlineto{\pgfpointcylindrical{\x+120}{2*(3-phi)^(1/4)}{-(phi*(3-phi)^(1/4))}}
\pgfusepath{fill,stroke}
\color{blue}}

\foreach \x in {60,180,300}
{\pgfpathcircle{\pgfpointcylindrical{\x}{2*(3-phi)^(1/4)}{-(phi*(3-phi)^(1/4))}}{3pt};
\shade[ball color=blue];
};

\pgfsetlinewidth{0.5pt}
\color{darkred!40}
\draw[dashed] (0,0,-\zb) circle (\rb);

\pgfsetdash{{3pt}{2pt}}{0pt}
\foreach \x in {0,120,240}{
\pgfpathmoveto{\pgfpointcylindrical{\x}{2*(phi)^(1/4)}{-((3-phi)*(phi)^(1/4))}}
\pgfpathlineto{\pgfpointcylindrical{0}{0}{-((3-phi)*(phi)^(1/4))}}}
\color{darkred}
\pgfusepath{fill,stroke}

\pgfsetlinewidth{0.8pt}
\pgfsetdash{}{0pt}
\foreach \x in {0,120,240}{
\pgfpathmoveto{\pgfpointcylindrical{\x}{2*(phi)^(1/4)}{-((3-phi)*(phi)^(1/4))}}
\pgfpathlineto{\pgfpointcylindrical{\x+120}{2*(phi)^(1/4)}{-((3-phi)*(phi)^(1/4))}}}
\color{darkred}
\pgfusepath{fill,stroke}

\foreach \x in {0,120,240}
{\pgfpathcircle{\pgfpointcylindrical{\x}{2*(phi)^(1/4)}{-((3-phi)*(phi)^(1/4))}}{3pt};
\shade[ball color=darkred];
};

\pgfsetlinewidth{0.5pt}
\pgfsetdash{}{0pt}
\color{black}
\draw[thick,line width=0.5pt,->] (0,0,0) -- (5,0,0) node[anchor=north east]{$y_1$};
\draw[thick,line width=0.5pt,->] (0,0,0) -- (0,4,0) node[anchor=north west]{$y_2$};
\draw[thick,line width=0.5pt,->] (0,0,0) -- (0,0,4) node[anchor=south]{$y_3$};

\color{darkred!40}
\draw[dashed] (0,0,\zb) circle (\rb);

\pgfsetdash{{3pt}{2pt}}{0pt}
\foreach \x in {60,180,300}{
\pgfpathmoveto{\pgfpointcylindrical{\x}{2*(phi)^(1/4)}{((3-phi)*(phi)^(1/4))}}
\pgfpathlineto{\pgfpointcylindrical{0}{0}{((3-phi)*(phi)^(1/4))}}}
\color{darkred}
\pgfusepath{fill,stroke}

\pgfsetlinewidth{0.8pt}
\pgfsetdash{}{0pt}
\foreach \x in {60,180,300}{
\pgfpathmoveto{\pgfpointcylindrical{\x}{2*(phi)^(1/4)}{((3-phi)*(phi)^(1/4))}}
\pgfpathlineto{\pgfpointcylindrical{\x+120}{2*(phi)^(1/4)}{((3-phi)*(phi)^(1/4))}}}
\color{darkred}
\pgfusepath{fill,stroke}

\foreach \x in {60,180,300}
{\pgfpathcircle{\pgfpointcylindrical{\x}{2*(phi)^(1/4)}{((3-phi)*(phi)^(1/4))}}{3pt};
\shade[ball color=darkred];
};

\pgfsetlinewidth{0.5pt}
\color{blue!40}
\draw[dashed] (0,0,\za) circle (\ra);

\pgfsetdash{{3pt}{2pt}}{0pt}
\foreach \x in {0,120,240}{
\pgfpathmoveto{\pgfpointcylindrical{\x}{2*(3-phi)^(1/4)}{(phi*(3-phi)^(1/4))}}
\pgfpathlineto{\pgfpointcylindrical{0}{0}{(phi*(3-phi)^(1/4))}}}
\color{blue}
\pgfusepath{fill,stroke}

\pgfsetlinewidth{0.8pt}
\pgfsetdash{}{0pt}
\foreach \x in {0,120,240}{
\pgfpathmoveto{\pgfpointcylindrical{\x}{2*(3-phi)^(1/4)}{(phi*(3-phi)^(1/4))}}
\pgfpathlineto{\pgfpointcylindrical{\x+120}{2*(3-phi)^(1/4)}{(phi*(3-phi)^(1/4))}}}
\color{blue}
\pgfusepath{fill,stroke}

\foreach \x in {0,120,240}
{\pgfpathcircle{\pgfpointcylindrical{\x}{2*(3-phi)^(1/4)}{(phi*(3-phi)^(1/4))}}{3pt};
\shade[ball color=blue];
};

\shade[ball color = gray!20, opacity = 0.2] (0,0,0) circle [radius=\r cm];

\color{black}
\pgfsetdash{}{0pt}

\begin{scope}[shift={(8cm,0cm)},tdplot_screen_coords]
\pgfsetlinewidth{1pt}
\draw[thick,->] (-3.5,0) -- (3.5,0) node[anchor=north east]{$y_1$};
\draw[thick,->] (0,-3.5) -- (0,3.5) node[anchor=north west]{$y_2$};

\pgfsetlinewidth{0.8pt}
\pgfpathcircle{\pgfpointorigin}{\rb cm}
\pgfsetdash{{3pt}{2pt}}{0pt}
\color{darkred}
\pgfusepath{stroke}

\foreach \x in {60,180,300}{
\pgfpathmoveto{\pgfpointpolar{\x}{\rb cm}}
\pgfpathlineto{\pgfpointpolar{\x+120}{\rb cm}}}
\pgfsetdash{}{0pt}
\color{darkred}
\pgfusepath{stroke}

\foreach \x in {60,180,300}{
{\pgfpathcircle{\pgfpointpolar{\x}{\rb cm}}{3pt}};
\shade[ball color=darkred];}
\pgfusepath{fill}

\pgfsetdash{{3pt}{2pt}}{0pt}
\pgfpathcircle{\pgfpointorigin}{\ra cm}
\color{blue}
\pgfusepath{stroke}

\pgfsetdash{}{0pt}
\foreach \x in {0,120,240}{
\pgfpathmoveto{\pgfpointpolar{\x}{\ra cm}}
\pgfpathlineto{\pgfpointpolar{\x+120}{\ra cm}}}
\color{blue}
\pgfusepath{stroke}

\foreach \x in {0,120,240}{
\pgfpathcircle{\pgfpointpolar{\x}{\ra cm}}{3pt};
\shade[ball color=blue];}
\pgfusepath{fill}

\color{black}
\shade[ball color = gray!20, opacity = 0.2] (0,0) circle [radius=\r cm];
\end{scope}
\end{tikzpicture}
\caption[Graphical representation of decomposition \eqref{rel_decom_n=3_s=2_uno}]{Graphical representation of decomposition \eqref{rel_decom_n=3_s=2_uno} in standard coordinates $\{y_1,y_2,y_3\}$, whose elements correspond, up to central simmetry, to the vertices of two triangles, respectively in red and blue, placed at different heights.}
\label{fig_decom_3-3}
\end{figure}
We start by considering decomposition \eqref{rel_decomp_icos}, which, as already said above, results to be made, after a suitable rotation, by points forming two different triangles at different heights in the space (see \autoref{fig_decom_3-3}). Introducing the notation $\tau_j$ to denote the $j$-th root of unity equal to
\[
\tau_j=\rme^{\frac{2\uppi\rmi}{j}},
\]
for every $j\in\bbN$, we can therefore rewrite decomposition \eqref{rel_decomp_icos} as
\begin{equation}
\label{rel_decom_n=3_s=2_uno}
54q_3^2=\sum_{j=1}^3(3-\varphi)\bigl(\tau_{6}^{2(j-1)}u'+\tau_{6}^{-2(j-1)}v'+\varphi z'\bigr)^4+\sum_{j=1}^3\varphi\bigl(\tau_{6}^{2j-1}u'+\tau_{6}^{-(2j-1)}v'+(3-\varphi) z'\bigr)^4,
\end{equation}
where 
\[
\varphi=\frac{3+\sqrt{5}}{2}.
\]
In particular, the ideal of the points forming such decomposition presents a rather simple set of generators, consisting of four polynomials symmetric with respect to the variables $u$ and $v$, that is,
\begin{equation}
\label{rel_ideal_decomp_3-3}
\calI_{\text{\eqref{rel_decom_n=3_s=2_uno}}}=\pt*{\begin{aligned}
&h_{3,3}+4\sqrt{5}h_{3,0},
&2h_{3,-2}+\sqrt{5}h_{3,1},\\
&h_{3,-3}-4\sqrt{5}h_{3,0},
&2h_{3,2}-\sqrt{5}h_{3,-1}
\end{aligned}
},
\end{equation}
which effectively has a symmetric structure with respect to the basis of harmonic polynomials we determined.

If we now look at decomposition \eqref{decom_n=3_s=3_std}, we can observe that the disposition of the points presents some analogies with the disposition of the points of decomposition \eqref{rel_decom_n=3_s=2_uno}. In fact, we observe that the points related to decomposition \eqref{decom_n=3_s=3_std} are distributed over three squares placed at different heights, symmetrically with respect to the central axis $y_1$. Since one of the squares is placed at height $0$, only two vertices of this square need to be considered as points of the decomposition (see \autoref{fig_decom_2-4-4-1} with $y_1$ interchanged with $y_3$).
\begin{figure}[ht]
\tdplotsetmaincoords{73}{110}
\center
\begin{tikzpicture}[tdplot_main_coords,declare function={h=(540)^(1/6);}]

\pgfmathsetmacro{\r}{sqrt(7)}
\pgfmathsetmacro{\za}{1}
\pgfmathsetmacro{\zb}{2}
\pgfmathsetmacro{\zd}{h*((14/27)^(1/6))}

\pgfmathsetmacro{\ra}{sqrt(6)}
\pgfmathsetmacro{\rb}{sqrt(3)}
\pgfmathsetmacro{\rc}{h*((7/10)^(1/6))}

\draw[thick,line width=0.5pt,dashed] (0,0,0) -- (0,0,-4);

\pgfsetlinewidth{0.5pt}

\pgfpathcircle{\pgfpointcylindrical{0}{0}{-\zd}}{3pt};
\shade[ball color=yellow];

\color{blue!40}
\draw[dashed] (0,0,-\zb) circle (\rb);
  
\pgfsetdash{{3pt}{2pt}}{0pt}
\foreach \x in {0,90,180,270}{
\pgfpathmoveto{\pgfpointcylindrical{\x}{\rb}{-\zb}}
\pgfpathlineto{\pgfpointcylindrical{0}{0}{-\zb}}}
\color{blue}
\pgfusepath{fill,stroke}

\pgfsetlinewidth{0.8pt}
\pgfsetdash{}{0pt}
\foreach \x in {0,90,180,270}{
\pgfpathmoveto{\pgfpointcylindrical{\x}{\rb}{-\zb}}
\pgfpathlineto{\pgfpointcylindrical{\x+90}{\rb}{-\zb}}
\pgfusepath{fill,stroke}
\color{blue}}

\foreach \x in {0,90,180,270}
{\pgfpathcircle{\pgfpointcylindrical{\x}{\rb}{-\zb}}{3pt};
\shade[ball color=blue];
};

\pgfsetlinewidth{0.5pt}
\color{darkred!40}
\draw[dashed] (0,0,-\za) circle (\ra);

\pgfsetdash{{3pt}{2pt}}{0pt}
\foreach \x in {45,135,225,315}{
\pgfpathmoveto{\pgfpointcylindrical{\x}{\ra}{-\za}}
\pgfpathlineto{\pgfpointcylindrical{0}{0}{-\za}}}
\color{darkred}
\pgfusepath{fill,stroke}

\pgfsetlinewidth{0.8pt}
\pgfsetdash{}{0pt}
\foreach \x in {45,135,225,315}{
\pgfpathmoveto{\pgfpointcylindrical{\x}{\ra}{-\za}}
\pgfpathlineto{\pgfpointcylindrical{\x+90}{\ra}{-\za}}}
\color{darkred}
\pgfusepath{fill,stroke}

\foreach \x in {45,135,225,315}
{\pgfpathcircle{\pgfpointcylindrical{\x}{\ra}{-\za}}{3pt};
\shade[ball color=darkred];
};

\pgfsetlinewidth{0.5pt}
\pgfsetdash{}{0pt}
\color{black}
\draw[thick,line width=0.5pt,->] (0,0,0) -- (5.5,0,0) node[anchor=north east]{$y_1$};
\draw[thick,line width=0.5pt,->] (0,0,0) -- (0,4,0) node[anchor=north west]{$y_2$};
\draw[thick,line width=0.5pt,->] (0,0,0) -- (0,0,4) node[anchor=south]{$y_3$};

\pgfpathcircle{\pgfpointcylindrical{0}{0}{\zd}}{3pt};
\shade[ball color=yellow];

\color{green!40}
\draw[dashed] (0,0,0) circle (\rc);
  
\pgfsetdash{{3pt}{2pt}}{0pt}
\foreach \x in {0,90,180,270}{
\pgfpathmoveto{\pgfpointcylindrical{\x}{\rc}{0}}
\pgfpathlineto{\pgfpointcylindrical{0}{0}{0}}}
\color{green}
\pgfusepath{fill,stroke}

\pgfsetlinewidth{0.8pt}
\pgfsetdash{}{0pt}
\foreach \x in {0,90,180,270}{
\pgfpathmoveto{\pgfpointcylindrical{\x}{\rc}{0}}
\pgfpathlineto{\pgfpointcylindrical{\x+90}{\rc}{0}}
\pgfusepath{fill,stroke}
\color{green}}

\foreach \x in {0,90,180,270}
{\pgfpathcircle{\pgfpointcylindrical{\x}{\rc}{0}}{3pt};
\shade[ball color=green];
};

\color{darkred!40}
\draw[dashed] (0,0,\za) circle (\ra);

\pgfsetdash{{3pt}{2pt}}{0pt}
\foreach \x in {45,135,225,315}{
\pgfpathmoveto{\pgfpointcylindrical{\x}{\ra}{\za}}
\pgfpathlineto{\pgfpointcylindrical{0}{0}{\za}}}
\color{darkred}
\pgfusepath{fill,stroke}

\pgfsetlinewidth{0.8pt}
\pgfsetdash{}{0pt}
\foreach \x in {45,135,225,315}{
\pgfpathmoveto{\pgfpointcylindrical{\x}{\ra}{\za}}
\pgfpathlineto{\pgfpointcylindrical{\x+90}{\ra}{\za}}}
\color{darkred}
\pgfusepath{fill,stroke}

\foreach \x in {45,135,225,315}
{\pgfpathcircle{\pgfpointcylindrical{\x}{\ra}{\za}}{3pt};
\shade[ball color=darkred];
};

\pgfsetlinewidth{0.5pt}
\color{blue!40}
\draw[dashed] (0,0,\zb) circle (\rb);

\pgfsetdash{{3pt}{2pt}}{0pt}
\foreach \x in {0,90,180,270}{
\pgfpathmoveto{\pgfpointcylindrical{\x}{\rb}{\zb}}
\pgfpathlineto{\pgfpointcylindrical{0}{0}{\zb}}}
\color{blue}
\pgfusepath{fill,stroke}

\pgfsetlinewidth{0.8pt}
\pgfsetdash{}{0pt}
\foreach \x in {0,90,180,270}{
\pgfpathmoveto{\pgfpointcylindrical{\x}{\rb}{\zb}}
\pgfpathlineto{\pgfpointcylindrical{\x+90}{\rb}{\zb}}}
\color{blue}
\pgfusepath{fill,stroke}

\foreach \x in {0,90,180,270}
{\pgfpathcircle{\pgfpointcylindrical{\x}{\rb}{\zb}}{3pt};
\shade[ball color=blue];
};

\shade[ball color = gray!20, opacity = 0.2] (0,0,0) circle [radius=\r cm];

\color{black}
\pgfsetdash{}{0pt}

\begin{scope}[shift={(8cm,0cm)},tdplot_screen_coords]
\pgfsetlinewidth{1pt}
\draw[thick,->] (-3.5,0) -- (3.5,0) node[anchor=north east]{$y_1$};
\draw[thick,->] (0,-3.5) -- (0,3.5) node[anchor=north west]{$y_2$};

\pgfsetdash{{3pt}{2pt}}{0pt}
\pgfpathcircle{\pgfpointorigin}{\rc cm}
\color{green}
\pgfusepath{stroke}

\pgfsetdash{}{0pt}
\foreach \x in {0,90,180,270}{
\pgfpathmoveto{\pgfpointpolar{\x}{\rc cm}}
\pgfpathlineto{\pgfpointpolar{\x+90}{\rc cm}}}
\color{green}
\pgfusepath{stroke}

\foreach \x in {0,90,180,270}{
\pgfpathcircle{\pgfpointpolar{\x}{\rc cm}}{3pt};
\shade[ball color=green];}
\pgfusepath{fill}

\pgfsetlinewidth{0.8pt}
\pgfpathcircle{\pgfpointorigin}{\ra cm}
\pgfsetdash{{3pt}{2pt}}{0pt}
\color{darkred}
\pgfusepath{stroke}

\foreach \x in {45,135,225,315}{
\pgfpathmoveto{\pgfpointpolar{\x}{\ra cm}}
\pgfpathlineto{\pgfpointpolar{\x+90}{\ra cm}}}
\pgfsetdash{}{0pt}
\color{darkred}
\pgfusepath{stroke}

\foreach \x in {45,135,225,315}{
{\pgfpathcircle{\pgfpointpolar{\x}{\ra cm}}{3pt}};
\shade[ball color=darkred];}
\pgfusepath{fill}

\pgfsetdash{{3pt}{2pt}}{0pt}
\pgfpathcircle{\pgfpointorigin}{\rb cm}
\color{blue}
\pgfusepath{stroke}

\pgfsetdash{}{0pt}
\foreach \x in {0,90,180,270}{
\pgfpathmoveto{\pgfpointpolar{\x}{\rb cm}}
\pgfpathlineto{\pgfpointpolar{\x+90}{\rb cm}}}
\color{blue}
\pgfusepath{stroke}

\foreach \x in {0,90,180,270}{
\pgfpathcircle{\pgfpointpolar{\x}{\rb cm}}{3pt};
\shade[ball color=blue];}
\pgfusepath{fill}

\pgfpathcircle{\pgfpointpolar{0}{0}}{3pt};
\shade[ball color=yellow];

\color{black}
\shade[ball color = gray!20, opacity = 0.2] (0,0) circle [radius=\r cm];
\end{scope}
\end{tikzpicture}
\caption[Graphical representation of decompositions \eqref{decom_n=3_s=3_std} and \eqref{rel_decom_2-4-4-1}]{Graphical representation of decompositions \eqref{decom_n=3_s=3_std} and \eqref{rel_decom_2-4-4-1} in standard coordinates $\{y_1,y_2,y_3\}$, whose elements correspond, up to central simmetry, to the vertices of three squares, respectively in green, red and blue, placed at different heights, with adjuntive and an additional point in the axis $y_3$, in yellow.}
\label{fig_decom_2-4-4-1}
\end{figure}
The decomposition written in new coordinates, reversing the roles of $y_1$ and $y_3$, again takes into account the roots of unity, which have size $8$ and are given by
\begin{align}
\label{rel_decom_2-4-4-1}
\nonumber q_3^3&=\frac{14}{27}z'^6+\frac{7}{640}\sum_{j=1}^2\bigl(\tau_8^{2(j-1)}u'+\tau_8^{-2(j-1)}v'\bigr)^6+\frac{1}{1280}\sum_{j=1}^4\biggl(\tau_8^{2(j-1)}u'+\tau_8^{-2(j-1)}v'+\sqrt{\frac{16}{3}}z'\biggr)^6\\
&\hphantom{{}={}}+\frac{1}{160}\sum_{j=1}^4\biggl(\tau_8^{2j-1}u'+\tau_8^{-(2j-1)}v'+\sqrt{\frac{2}{3}}z'\biggr)^6.
\end{align}
Again, we observe that ideal associated to such a decomposition is structured in a similar way to the ideal \eqref{rel_ideal_decomp_3-3}. Indeed, it presents generators which are obtained by linear combination of elements placed at "distance" $4$ in the diagram we have seen in \autoref{fig_weights_harmonic_polynomials}, that is,
\[
\calI_{\text{\eqref{rel_decom_2-4-4-1}}}=\pt*{\begin{aligned}
&h_{4,4}-h_{4,-4},
&h_{4,2}-h_{4,-2},\\
&h_{4,3}+3h_{4,-1},
&h_{4,-3}+3h_{4,1}
\end{aligned}
}.
\]
This gives a symmetric structure with respect to the variables $u$ and $v$, due to the symmetric disposition of the points along the axis $y_3$.

By \autoref{teo_tight_decompositions_n=3}, it follows that decomposition \eqref{rel_decom_2-4-4-1} is minimal. However, we can prove that, in general, minimal decompositions of $q_n^s$ are not unique, not even up to orthogonal transformations. Indeed, by doing some computations, we can obtain another decomposition of $q_3^3$, which is also minimal, but which is not in the orbit under the action of $\Oa_n(\bbC)$ of decomposition \eqref{rel_decom_2-4-4-1}. Indeed, it is structured as two set of five points each, forming the vertices of two regular pentagons at different heights, with an additional point (see \autoref{fig_decom_5-5-1}). That is,
\begin{align}
\label{rel_decomp_5-5-1}
\nonumber q_3^3&=\frac{35}{54}z'^6+\biggl(\frac{1}{500}\beta^2+\frac{1}{750}\biggr)\sum_{j=1}^5\bigl(\tau_{10}^{2(j-1)}u'+\tau_{10}^{-2(j-1)}v'+\alpha z'\bigr)^6\\
&\hphantom{{}={}}+\biggl(\frac{1}{500}\alpha^2+\frac{1}{750}\biggr)\sum_{j=1}^5\bigl(\tau_{10}^{2j-1}u'+\tau_{10}^{-(2j-1)}v'+\beta z'\bigr)^6,
\end{align}
where $\alpha,\beta\in\bbR$ are constant values satisfying the equations
\[
\alpha\beta=\frac{2}{3},\qquad \alpha^2+\beta^2=\frac{11}{3},
\]
namely
\[
\alpha=\sqrt{\frac{11+\sqrt{105}}{6}},\qquad\beta=\sqrt{\frac{11- \sqrt{105}}{6}}.
\]

\begin{figure}[ht]
\tdplotsetmaincoords{73}{110}
\center
\begin{tikzpicture}[tdplot_main_coords,declare function={alpha=sqrt((11+sqrt(105))/6);},declare function={beta=((11-sqrt(105))/4)*sqrt((11+sqrt(105))/6);}]

\pgfmathsetmacro{\r}{0.88*(4375/13)^(1/6)}
\pgfmathsetmacro{\rk}{0.88*(10250/13)^(1/6)}
\pgfmathsetmacro{\za}{0.88*alpha*(beta^2+(2/3))^(1/6)}
\pgfmathsetmacro{\zb}{0.88*beta*(alpha^2+8)^(1/6)}
\pgfmathsetmacro{\zc}{0.88*beta*(3*alpha-1)^(1/8)}

\pgfmathsetmacro{\ra}{0.88*2*(beta^2+(2/3))^(1/6)}
\pgfmathsetmacro{\rb}{0.88*2*(alpha^2+8)^(1/6)}
\pgfmathsetmacro{\rc}{0.88*2*(3*alpha-1)^(1/8)}

\draw[thick,line width=0.5pt,dashed] (0,0,0) -- (0,0,-4);

\pgfpathcircle{\pgfpointcylindrical{0}{0}{-\r}}{3pt};
\shade[ball color=yellow];

\pgfsetlinewidth{0.5pt}
\color{darkred!40}
\draw[dashed] (0,0,-\za) circle (\ra);

\pgfsetlinewidth{0.5pt}  

\pgfsetdash{{3pt}{2pt}}{0pt}
\foreach \x in {36,108,180,252,324}{
\pgfpathmoveto{\pgfpointcylindrical{\x}{\ra}{-\za}}
\pgfpathlineto{\pgfpointcylindrical{0}{0}{-\za}}}
\color{darkred}
\pgfusepath{fill,stroke}

\pgfsetlinewidth{0.8pt}
\pgfsetdash{}{0pt}
\foreach \x in {36,108,180,252,324}{
\pgfpathmoveto{\pgfpointcylindrical{\x}{\ra}{-\za}}
\pgfpathlineto{\pgfpointcylindrical{\x+72}{\ra}{-\za}}}
\color{darkred}
\pgfusepath{fill,stroke}

\foreach \x in {36,108,180,252,324}
{\pgfpathcircle{\pgfpointcylindrical{\x+72}{\ra}{-\za}}{3pt};
\shade[ball color=darkred];
};

\color{blue!40}
\draw[dashed] (0,0,-\zb) circle (\rb);
  
\pgfsetlinewidth{0.5pt}  
  
\pgfsetdash{{3pt}{2pt}}{0pt}
\foreach \x in {0,72,144,216,288}{
\pgfpathmoveto{\pgfpointcylindrical{\x}{\rb}{-\zb}}
\pgfpathlineto{\pgfpointcylindrical{0}{0}{-\zb}}}
\color{blue}
\pgfusepath{fill,stroke}

\pgfsetlinewidth{0.8pt}
\pgfsetdash{}{0pt}
\foreach \x in {0,72,144,216,288}{
\pgfpathmoveto{\pgfpointcylindrical{\x}{\rb}{-\zb}}
\pgfpathlineto{\pgfpointcylindrical{\x+72}{\rb}{-\zb}}
\pgfusepath{fill,stroke}
\color{blue}}

\foreach \x in {0,72,144,216,288}
{\pgfpathcircle{\pgfpointcylindrical{\x}{\rb}{-\zb}}{3pt};
\shade[ball color=blue];
};

\pgfsetlinewidth{0.5pt}

\pgfsetlinewidth{0.5pt}
\pgfsetdash{}{0pt}
\color{black}
\draw[thick,line width=0.5pt,->] (0,0,0) -- (5.5,0,0) node[anchor=north east]{$y_1$};
\draw[thick,line width=0.5pt,->] (0,0,0) -- (0,4,0) node[anchor=north west]{$y_2$};
\draw[thick,line width=0.5pt,->] (0,0,0) -- (0,0,4) node[anchor=south]{$y_3$};

\color{blue!40}
\draw[dashed] (0,0,\zb) circle (\rb);
  
\pgfsetlinewidth{0.5pt}    
  
\pgfsetdash{{3pt}{2pt}}{0pt}
\foreach \x in {36,108,180,252,324}{
\pgfpathmoveto{\pgfpointcylindrical{\x}{\rb}{\zb}}
\pgfpathlineto{\pgfpointcylindrical{0}{0}{\zb}}}
\color{blue}
\pgfusepath{fill,stroke}

\pgfsetlinewidth{0.8pt}
\pgfsetdash{}{0pt}
\foreach \x in {36,108,180,252,324}{
\pgfpathmoveto{\pgfpointcylindrical{\x}{\rb}{\zb}}
\pgfpathlineto{\pgfpointcylindrical{\x+72}{\rb}{\zb}}
\pgfusepath{fill,stroke}
\color{blue}}

\foreach \x in {36,108,180,252,324}
{\pgfpathcircle{\pgfpointcylindrical{\x}{\rb}{\zb}}{3pt};
\shade[ball color=blue];
};

\color{darkred!40}
\draw[dashed] (0,0,\za) circle (\ra);

\pgfsetlinewidth{0.5pt}  

\pgfsetdash{{3pt}{2pt}}{0pt}
\foreach \x in {0,72,144,216,288}{
\pgfpathmoveto{\pgfpointcylindrical{\x}{\ra}{\za}}
\pgfpathlineto{\pgfpointcylindrical{0}{0}{\za}}}
\color{darkred}
\pgfusepath{fill,stroke}

\pgfsetlinewidth{0.8pt}
\pgfsetdash{}{0pt}
\foreach \x in {0,72,144,216,288}{
\pgfpathmoveto{\pgfpointcylindrical{\x}{\ra}{\za}}
\pgfpathlineto{\pgfpointcylindrical{\x+72}{\ra}{\za}}}
\color{darkred}
\pgfusepath{fill,stroke}

\foreach \x in {0,72,144,216,288}
{\pgfpathcircle{\pgfpointcylindrical{\x}{\ra}{\za}}{3pt};
\shade[ball color=darkred];
};

\pgfpathcircle{\pgfpointcylindrical{0}{0}{\r}}{3pt};
\shade[ball color=yellow];

\shade[ball color = gray!20, opacity = 0.2] (0,0,0) circle [radius=\r cm];

\shade[ball color = gray!10, opacity = 0.2] (0,0,0) circle [radius=\rk cm];

\color{black}
\pgfsetdash{}{0pt}

\begin{scope}[shift={(8cm,0cm)},tdplot_screen_coords]
\pgfsetlinewidth{1pt}
\draw[thick,->] (-3.5,0) -- (3.5,0) node[anchor=north east]{$y_1$};
\draw[thick,->] (0,-3.5) -- (0,3.5) node[anchor=north west]{$y_2$};

\pgfsetdash{{3pt}{2pt}}{0pt}
\pgfpathcircle{\pgfpointorigin}{\rb cm}
\color{blue}
\pgfusepath{stroke}

\pgfsetdash{}{0pt}
\foreach \x in {36,108,180,252,324}{
\pgfpathmoveto{\pgfpointpolar{\x}{\rb cm}}
\pgfpathlineto{\pgfpointpolar{\x+72}{\rb cm}}}
\color{blue}
\pgfusepath{stroke}

\foreach \x in {36,108,180,252,324}{
\pgfpathcircle{\pgfpointpolar{\x}{\rb cm}}{3pt};
\shade[ball color=blue];}
\pgfusepath{fill}

\pgfsetlinewidth{0.8pt}
\pgfpathcircle{\pgfpointorigin}{\ra cm}
\pgfsetdash{{3pt}{2pt}}{0pt}
\color{darkred}
\pgfusepath{stroke}

\foreach \x in {0,72,144,216,288}{
\pgfpathmoveto{\pgfpointpolar{\x}{\ra cm}}
\pgfpathlineto{\pgfpointpolar{\x+72}{\ra cm}}}
\pgfsetdash{}{0pt}
\color{darkred}
\pgfusepath{stroke}

\foreach \x in {0,72,144,216,288}{
\pgfpathcircle{\pgfpointpolar{\x}{\ra cm}}{3pt};
\shade[ball color=darkred];}
\pgfusepath{fill}

\pgfpathcircle{\pgfpointpolar{0}{0}}{3pt};
\shade[ball color=yellow];

\color{black}
\shade[ball color = gray!20, opacity = 0.2] (0,0) circle [radius=\r cm];
\shade[ball color = gray!10, opacity = 0.2] (0,0) circle [radius=\rk cm];
\end{scope}
\end{tikzpicture}
\caption[Graphical representation of decomposition \eqref{rel_decomp_5-5-1}]{Graphical representation of decomposition \eqref{rel_decomp_5-5-1} in standard coordinates $\{y_1,y_2,y_3\}$, whose elements correspond, up to central simmetry, to the vertices of two pentagons, respectively in blue, and red, placed at different heights, with an additional point in the axis $y_3$, in yellow.}
\label{fig_decom_5-5-1}
\end{figure}
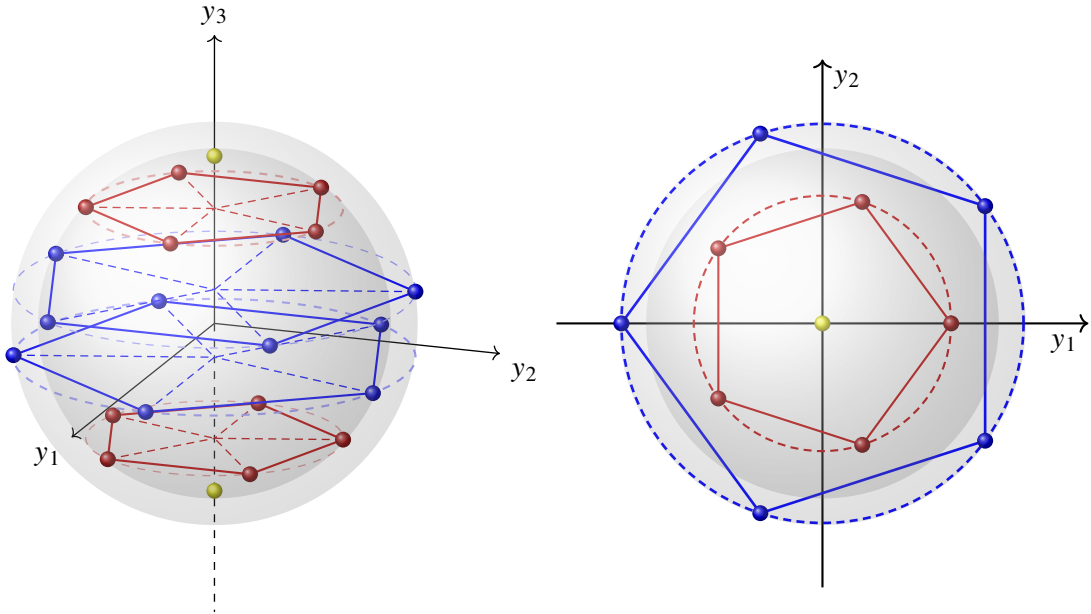

Again, we get our ideal generated by linear combination of elements having "distance" $5$ in the diagram of \autoref{fig_weights_harmonic_polynomials}, namely
\[
\calI_{\text{\eqref{rel_decomp_5-5-1}}}=\pt*{\begin{aligned}
&h_{4,-4}+3\sqrt{21}h_{4,1},
&7h_{4,-2}+\sqrt{21}h_{4,3},\\
&h_{4,4}-3\sqrt{21}h_{4,-1},
&7h_{4,2}-\sqrt{21}h_{4,-3}
\end{aligned}
}.
\]
So we have proved the following proposition.
\begin{prop}
Minimal decompositions of $q_3^3$ are not unique, even up to orthogonal transformations.
\end{prop}

Both decompositions \eqref{rel_decom_2-4-4-1} and \eqref{rel_decomp_5-5-1} are real, but we determine another pattern of non-real points.
This decomposition is quite complicated and the use of Macaulay2 (see \cite{GS}) was essential to determine it. Nevertheless, the result is quite elegant and, again, the distribution of the points in the space appears balanced.
Denoting by $\alpha_1$, $\alpha_2$ and $\alpha_3$ the three roots of the polynomial
\[
z^3-3z^2-3z+2\in\bbC[z],
\]
which are
\begin{align*}
\alpha_1&=1+2\biggl(\dfrac{3+\rmi\sqrt{23}}{2}\biggr)^{-\frac{1}{3}}+\biggl(\dfrac{3+\rmi\sqrt{23}}{2}\biggr)^{\frac{1}{3}},\\[1ex]
\alpha_2&=1-\bigl(1-\rmi\sqrt{3}\bigr)\biggl(\dfrac{3+\rmi\sqrt{23}}{2}\biggr)^{-\frac{1}{3}}-\frac{1+\rmi\sqrt{3}}{2}\biggl(\dfrac{3+\rmi\sqrt{23}}{2}\biggr)^{\frac{1}{3}},\\[1ex]
\alpha_3&=1-\bigl(1+\rmi\sqrt{3}\bigr)\biggl(\dfrac{3+\rmi\sqrt{23}}{2}\biggr)^{-\frac{1}{3}}-\frac{1-\rmi\sqrt{3}}{2}\biggl(\dfrac{3+\rmi\sqrt{23}}{2}\biggr)^{\frac{1}{3}},
\end{align*}
we get the relations 
\[
\alpha_1+\alpha_2+\alpha_3=3,\qquad \alpha_1\alpha_2\alpha_3=-2,
\]
\begin{figure}[ht]
\tdplotsetmaincoords{73}{110}
\center
\begin{tikzpicture}[tdplot_main_coords,declare function={phi=(3+sqrt(5))/2;},declare function={psi=(3-sqrt(5))/2;}]

\pgfmathsetmacro{\r}{(1875)^(1/8)}
\pgfmathsetmacro{\za}{phi*(8-3*phi)^(1/8)}
\pgfmathsetmacro{\zb}{(3)^(1/8)}
\pgfmathsetmacro{\zc}{psi*(3*phi-1)^(1/8)}

\pgfmathsetmacro{\ra}{2*(8-3*phi)^(1/8)}
\pgfmathsetmacro{\rb}{2*(3)^(1/8)}
\pgfmathsetmacro{\rc}{2*(3*phi-1)^(1/8)}

\draw[thick,line width=0.5pt,dashed] (0,0,0) -- (0,0,-4);

\pgfpathcircle{\pgfpointcylindrical{0}{0}{-\r}}{3pt};
\shade[ball color=yellow];

\pgfsetlinewidth{0.5pt}
\color{darkred!40}
\draw[dashed] (0,0,-\za) circle (\ra);

\pgfsetlinewidth{0.5pt}  

\pgfsetdash{{3pt}{2pt}}{0pt}
\foreach \x in {36,108,180,252,324}{
\pgfpathmoveto{\pgfpointcylindrical{\x}{\ra}{-\za}}
\pgfpathlineto{\pgfpointcylindrical{0}{0}{-\za}}}
\color{darkred}
\pgfusepath{fill,stroke}

\pgfsetlinewidth{0.8pt}
\pgfsetdash{}{0pt}
\foreach \x in {36,108,180,252,324}{
\pgfpathmoveto{\pgfpointcylindrical{\x}{\ra}{-\za}}
\pgfpathlineto{\pgfpointcylindrical{\x+72}{\ra}{-\za}}}
\color{darkred}
\pgfusepath{fill,stroke}

\foreach \x in {36,108,180,252,324}
{\pgfpathcircle{\pgfpointcylindrical{\x+72}{\ra}{-\za}}{3pt};
\shade[ball color=darkred];
};

\color{green!40}
\draw[dashed] (0,0,-\zb) circle (\rb);
  
\pgfsetlinewidth{0.5pt}  
  
\pgfsetdash{{3pt}{2pt}}{0pt}
\foreach \x in {0,72,144,216,288}{
\pgfpathmoveto{\pgfpointcylindrical{\x}{\rb}{-\zb}}
\pgfpathlineto{\pgfpointcylindrical{0}{0}{-\zb}}}
\color{green}
\pgfusepath{fill,stroke}

\pgfsetlinewidth{0.8pt}
\pgfsetdash{}{0pt}
\foreach \x in {0,72,144,216,288}{
\pgfpathmoveto{\pgfpointcylindrical{\x}{\rb}{-\zb}}
\pgfpathlineto{\pgfpointcylindrical{\x+72}{\rb}{-\zb}}
\pgfusepath{fill,stroke}
\color{green}}

\foreach \x in {0,72,144,216,288}
{\pgfpathcircle{\pgfpointcylindrical{\x}{\rb}{-\zb}}{3pt};
\shade[ball color=green];
};

\pgfsetlinewidth{0.5pt}

\color{blue!40}
\draw[dashed] (0,0,-\zc) circle (\rc);
  
\pgfsetlinewidth{0.5pt}    
  
\pgfsetdash{{3pt}{2pt}}{0pt}
\foreach \x in {36,108,180,252,324}{
\pgfpathmoveto{\pgfpointcylindrical{\x}{\rc}{-\zc}}
\pgfpathlineto{\pgfpointcylindrical{0}{0}{-\zc}}}
\color{blue}
\pgfusepath{fill,stroke}

\pgfsetlinewidth{0.8pt}
\pgfsetdash{}{0pt}
\foreach \x in {36,108,180,252,324}{
\pgfpathmoveto{\pgfpointcylindrical{\x}{\rc}{-\zc}}
\pgfpathlineto{\pgfpointcylindrical{\x+72}{\rc}{-\zc}}
\pgfusepath{fill,stroke}
\color{blue}}

\foreach \x in {36,108,180,252,324}
{\pgfpathcircle{\pgfpointcylindrical{\x}{\rc}{-\zc}}{3pt};
\shade[ball color=blue];
};

\pgfsetlinewidth{0.5pt}
\pgfsetdash{}{0pt}
\color{black}
\draw[thick,line width=0.5pt,->] (0,0,0) -- (5.5,0,0) node[anchor=north east]{$y_1$};
\draw[thick,line width=0.5pt,->] (0,0,0) -- (0,4,0) node[anchor=north west]{$y_2$};
\draw[thick,line width=0.5pt,->] (0,0,0) -- (0,0,4) node[anchor=south]{$y_3$};

\pgfsetlinewidth{0.5pt}
\color{blue!40}
\draw[dashed] (0,0,\zc) circle (\rc);

\pgfsetlinewidth{0.5pt}  

\pgfsetdash{{3pt}{2pt}}{0pt}
\foreach \x in {0,72,144,216,288}{
\pgfpathmoveto{\pgfpointcylindrical{\x}{\rc}{\zc}}
\pgfpathlineto{\pgfpointcylindrical{0}{0}{\zc}}}
\color{blue}
\pgfusepath{fill,stroke}

\pgfsetlinewidth{0.8pt}
\pgfsetdash{}{0pt}
\foreach \x in {0,72,144,216,288}{
\pgfpathmoveto{\pgfpointcylindrical{\x}{\rc}{\zc}}
\pgfpathlineto{\pgfpointcylindrical{\x+72}{\rc}{\zc}}}
\color{blue}
\pgfusepath{fill,stroke}

\foreach \x in {0,72,144,216,288}
{\pgfpathcircle{\pgfpointcylindrical{\x}{\rc}{\zc}}{3pt};
\shade[ball color=blue];
};

\color{green!40}
\draw[dashed] (0,0,\zb) circle (\rb);
  
\pgfsetlinewidth{0.5pt}    
  
\pgfsetdash{{3pt}{2pt}}{0pt}
\foreach \x in {36,108,180,252,324}{
\pgfpathmoveto{\pgfpointcylindrical{\x}{\rb}{\zb}}
\pgfpathlineto{\pgfpointcylindrical{0}{0}{\zb}}}
\color{green}
\pgfusepath{fill,stroke}

\pgfsetlinewidth{0.8pt}
\pgfsetdash{}{0pt}
\foreach \x in {36,108,180,252,324}{
\pgfpathmoveto{\pgfpointcylindrical{\x}{\rb}{\zb}}
\pgfpathlineto{\pgfpointcylindrical{\x+72}{\rb}{\zb}}
\pgfusepath{fill,stroke}
\color{green}}

\foreach \x in {36,108,180,252,324}
{\pgfpathcircle{\pgfpointcylindrical{\x}{\rb}{\zb}}{3pt};
\shade[ball color=green];
};

\color{darkred!40}
\draw[dashed] (0,0,\za) circle (\ra);

\pgfsetlinewidth{0.5pt}  

\pgfsetdash{{3pt}{2pt}}{0pt}
\foreach \x in {0,72,144,216,288}{
\pgfpathmoveto{\pgfpointcylindrical{\x}{\ra}{\za}}
\pgfpathlineto{\pgfpointcylindrical{0}{0}{\za}}}
\color{darkred}
\pgfusepath{fill,stroke}

\pgfsetlinewidth{0.8pt}
\pgfsetdash{}{0pt}
\foreach \x in {0,72,144,216,288}{
\pgfpathmoveto{\pgfpointcylindrical{\x}{\ra}{\za}}
\pgfpathlineto{\pgfpointcylindrical{\x+72}{\ra}{\za}}}
\color{darkred}
\pgfusepath{fill,stroke}

\foreach \x in {0,72,144,216,288}
{\pgfpathcircle{\pgfpointcylindrical{\x}{\ra}{\za}}{3pt};
\shade[ball color=darkred];
};

\pgfpathcircle{\pgfpointcylindrical{0}{0}{\r}}{3pt};
\shade[ball color=yellow];

\shade[ball color = gray!20, opacity = 0.2] (0,0,0) circle [radius=\r cm];

\color{black}
\pgfsetdash{}{0pt}

\begin{scope}[shift={(8cm,0cm)},tdplot_screen_coords]
\pgfsetlinewidth{1pt}
\draw[thick,->] (-3.5,0) -- (3.5,0) node[anchor=north east]{$y_1$};
\draw[thick,->] (0,-3.5) -- (0,3.5) node[anchor=north west]{$y_2$};

\foreach \x in {0,72,144,216,288}{
{\pgfpathcircle{\pgfpointpolar{\x}{\rc cm}}{3pt}};
\shade[ball color=darkred];}
\pgfusepath{fill}

\pgfsetdash{{3pt}{2pt}}{0pt}
\pgfpathcircle{\pgfpointorigin}{\rc cm}
\color{blue}
\pgfusepath{stroke}

\pgfsetdash{}{0pt}
\foreach \x in {0,72,144,216,288}{
\pgfpathmoveto{\pgfpointpolar{\x}{\rc cm}}
\pgfpathlineto{\pgfpointpolar{\x+72}{\rc cm}}}
\color{blue}
\pgfusepath{stroke}

\foreach \x in {0,72,144,216,288}{
\pgfpathcircle{\pgfpointpolar{\x}{\rc cm}}{3pt};
\shade[ball color=blue];}
\pgfusepath{fill}

\pgfsetdash{{3pt}{2pt}}{0pt}
\pgfpathcircle{\pgfpointorigin}{\rb cm}
\color{green}
\pgfusepath{stroke}

\pgfsetdash{}{0pt}
\foreach \x in {36,108,180,252,324}{
\pgfpathmoveto{\pgfpointpolar{\x}{\rb cm}}
\pgfpathlineto{\pgfpointpolar{\x+72}{\rb cm}}}
\color{green}
\pgfusepath{stroke}

\foreach \x in {36,108,180,252,324}{
\pgfpathcircle{\pgfpointpolar{\x}{\rb cm}}{3pt};
\shade[ball color=green];}
\pgfusepath{fill}

\pgfsetlinewidth{0.8pt}
\pgfpathcircle{\pgfpointorigin}{\ra cm}
\pgfsetdash{{3pt}{2pt}}{0pt}
\color{darkred}
\pgfusepath{stroke}

\foreach \x in {0,72,144,216,288}{
\pgfpathmoveto{\pgfpointpolar{\x}{\ra cm}}
\pgfpathlineto{\pgfpointpolar{\x+72}{\ra cm}}}
\pgfsetdash{}{0pt}
\color{darkred}
\pgfusepath{stroke}

\foreach \x in {0,72,144,216,288}{
\pgfpathcircle{\pgfpointpolar{\x}{\ra cm}}{3pt};
\shade[ball color=darkred];}
\pgfusepath{fill}

\pgfpathcircle{\pgfpointpolar{0}{0}}{3pt};
\shade[ball color=yellow];

\color{black}
\shade[ball color = gray!20, opacity = 0.2] (0,0) circle [radius=\r cm];
\end{scope}
\end{tikzpicture}
\caption[Graphical representation of decompositions \eqref{decom_n=3_s=4_std} and \eqref{rel_decom_5-5-5-1}]{Graphical representation of decompositions \eqref{decom_n=3_s=4_std} and \eqref{rel_decom_5-5-5-1} in standard coordinates, whose elements correspond, up to central simmetry, to the vertices of three pentagons, respectively in blue, green, and red, placed at different heights, with an additional point in the axis $y_3$, in yellow.}
\label{fig_decom_5-5-5-1}
\end{figure}
by which is possible to determine the decomposition
\begin{align}
\label{rel_decom_2-3-3-3}
\nonumber q_3^3&=-\frac{1}{20}u'^6-\frac{1}{20}v'^6+\frac{1}{6210}\bigl(19\alpha_1^{-1}-11\alpha_1+36\bigr)\sum_{j=1}^3\bigl(\tau_3^{j-1}u'+\tau_3^{-(j-1)}v'+\alpha_1z'\bigr)^6\\
&\nonumber \hphantom{{}={}}+\frac{1}{6210}\bigl(19\alpha_2^{-1}-11\alpha_2+36\bigr)\sum_{j=1}^3\bigl(\tau_3^{j-1}u'+\tau_3^{-(j-1)}v'+\alpha_2z'\bigr)^6\\
&\hphantom{{}={}}+\frac{1}{6210}\bigl(19\alpha_3^{-1}-11\alpha_3+36\bigr)\sum_{j=1}^3\bigl(\tau_3^{j-1}u'+\tau_3^{-(j-1)}v'+\alpha_3z'\bigr)^6.
\end{align}
The associated ideal is instead given by
\[
\calI_{\text{\eqref{rel_decom_2-3-3-3}}}=\pt*{\begin{aligned}
&5h_{4,0}+h_{4,3},
&h_{4,1}+h_{4,-2},\\
&5h_{4,0}-h_{4,-3},
&h_{4,-1}-h_{4,2}
\end{aligned}
}.
\]
\begin{rem}
As already observed, \autoref{teo_tight_decompositions_n=3} implies that decomposition \eqref{rel_decom_2-3-3-3} is minimal. An important fact about decomposition \eqref{rel_decom_2-3-3-3} concerns the \textit{Waring locus} of homogeneous polynomials. In particular, the Waring locus of a polynomial $f\in\calR_{n,d}$ consists of the set of the linear forms, viewed as points in $\bbP({\bbC^n}^*)$, which appear in a minimal decomposition of $f$. Its complement is known as the \textit{forbidden locus} of $f$. Now, by the transitivity of the action of $\SO_n(\bbC)$ (see \autoref{rem:transitivity_SOn}) on the two orbits of isotropic and non-isotropic points, decomposition \eqref{rel_decom_2-3-3-3} implies that any point of $\bbC^n$ can appear in a minimal decomposition of the form $q_3^3$, proving that its forbidden locus is empty. This provides a counterexample to a conjecture made by E.~Carlini, M.~V.~Catalisano, and A.~Oneto in \cite{CCO17}*{section 5.1}, where they conjecture that the forbidden locus is non-empty for any polynomial $f\in\calR_{n,d}$.
\end{rem}

For completeness, we also include a transformation of decomposition \eqref{decom_n=3_s=4_std}, obtained by a rotation for which we set a point on the axis $y_3$. We obtain points located on three different pentagons placed at different heights, symmetrically with respect to the central axis (see \autoref{fig_decom_5-5-5-1}). Thus, we can write
\begin{align}
\label{rel_decom_5-5-5-1}
\nonumber q_3^4&=\frac{15}{28}z'^8+\frac{\varphi^{-2}}{3500}\sum_{j=1}^5\bigl(\tau_{10}^{2(j-1)}u'+\tau_{10}^{-2(j-1)}v'+\varphi z'\bigr)^8+\frac{3}{3500}\sum_{j=1}^5\bigl(\tau_{10}^{2j-1}u'+\tau_{10}^{-(2j-1)}v'+z'\bigr)^8\\
&\hphantom{{}={}}+\frac{\varphi^2}{3500}\sum_{j=1}^5\bigl(\tau_{10}^{2(j-1)}u'+\tau_{10}^{-2(j-1)}v'+\varphi^{-1} z'\bigr)^8,
\end{align}
where
\[
\varphi=\frac{3+\sqrt{5}}{2},\qquad \varphi^{-1}=\frac{3-\sqrt{5}}{2}
\]
with the ideal of points equal to
\[
\calI_{\text{\eqref{rel_decom_5-5-5-1}}}=\pt*{\begin{aligned}
&7h_{5,1}+h_{5,-4},
&4h_{5,2}-3h_{5,-3},
&\quad h_{5,5}+h_{5,-5},\\
&7h_{5,-1}-h_{5,4},
&4h_{5,-2}+3h_{5,3}
\end{aligned}
}.
\]

The case where $s=5$ is the first power of unknown rank for the case of three variables.

\subsection{Equiangular lines and rank lower bounds}
\label{sec_equiang_lines}
In this section we focus on decompositions of the powers of ternary quadratic forms. 
In particular, we generalize the lower bound valid for the real decompositions in the cases where $s=3,4$. Indeed, we observe that the property of being first caliber decompositions imposes strong conditions between points. A specific analysis on this fact allows to prove that, also for complex numbers,
\[
\rk q_3^4=16.
\]
We consider the point of view of angles between points, which can be uniquely defined up to the sign. Indeed, there are always two supplementary angles between two lines in the space and the restriction of the first caliber decomposition effectively imposes a strong bound on the disposition of the points on the sphere. In the next definition we give an alternative version of the concept of angles between two complex points in $\bbC^n$. Instead of the standard Hermitian product, we are considering the product defined by
\begin{equation}
\label{formula:inner_product_coordinates}
\bfa\cdot\bfb=\sum_{j=1}^na_jb_j
\end{equation}
for every $\bfa,\bfb\in\bbC^n$. We say that a point $\bfa\in\bbC^n$ is \textit{unitary} if $\bfa\cdot\bfa=1$. In particular, for any non-isotropic point $\bfb\in\bbC^n$, we can find two associated unitary vectors
\[
\pm(\bfb\cdot\bfb)^{-\frac{1}{2}}\bfb
\]
and hence, any unitary vector identifies only one complex projective non-isotropic point.

The Gram matrix of a set of vectors $\{\bfv_1,\dots,\bfv_n\}$ in a vector space $V$ endowed with a bilinear form is the matrix $G$ of order $n$ defined by the entries
\[
G_{i,j}=\bfv_i\cdot\bfv_j
\]
for every $i,j=1,\dots,n$. In particular, the utility of the Gram determinant is due to the fact that, when computed on the $n(n+1)/2$  scalar products
of $n+1$ unitary vectors $\bfa_1,\dots,\bfa_{n+1}\in\bbC^{n}$, with respect to the inner product of formula \eqref{formula:inner_product_coordinates}, it vanishes identically.
Fixing $n=3$ and considering four unitary points $\bfa_1,\bfa_2,\bfa_3,\bfa_4\in\bbC^3$, the Gram determinant is given by the polynomial in six variables equal to
\begin{equation}
\label{rel_Gram_determinant}
g(c_{12},\ldots, c_{34})=\det
\begin{pNiceMatrix}
      1&{c}_{12}&{c}_{13}&{c}_{14}\\
      {c}_{12}&1&{c}_{23}&{c}_{24}\\
      {c}_{13}&{c}_{23}&1&{c}_{34}\\
      {c}_{14}&{c}_{24}&{c}_{34}&1\end{pNiceMatrix},
\end{equation}
where
\[
c_{jk}=\bfa_j\cdot\bfa_k\in\bbC
\]
for every $j,k=1,2,3,4$ with $j\neq k$.

We can check that there is a subgroup of order $24$ of the permutation group $\mfS_6$ that leaves $g$
invariant. This means that by permuting the variables, the determinant can take at most $6!/24=30$ possible distinct values
and one can find $30$ permutations representing each lateral class to evaluate the determinant.
In fact, the Gram matrix in this case factors as $\transpose{A}A$, where $A$ is the matrix of order $4$ with $\bfa_1,\bfa_2,\bfa_3,\bfa_4$ as columns.
\begin{teo}
\label{teo_equiangular_one_value}
Let $\bfa_1,\bfa_2,\bfa_3,\bfa_4\in\bbC^3$ be four distinct unitary points such that there exists a value $a\in\bbC$ such that \[(\bfa_j\cdot\bfa_k)^2=a^2\] for $1\leq j< k\leq 4$.
Then 
\[
a^2\in\biggl\{\frac{1}{9},\frac{1}{5},1\biggr\}.
\]
\end{teo}
\begin{proof}
We have $\bfa_j\cdot\bfa_k$
for $1\leq j< k\leq 4$. To get the statement, we simply have to compute the Gram determinant given by formula \eqref{rel_Gram_determinant} in the case of all the variables equal to $\pm a$. Now, in the case where $\bfa_j\cdot\bfa_k=a$ for every $j<k$, we have
\[
\det\begin{pNiceMatrix}
      1&a&a&a\\
      a&1&a&a\\
      a&a&1&a\\
      a&a&a&1\end{pNiceMatrix}=-(a-1)^3(3a+1).
\]
By the definition of Gram matrix, we simply have to compute matrices taking into account how many points have mutual product equal to $a$ or $-a$. In other words, we can consider only representing matrices in equivalence classes obtained by changing sign or permuting rows and columns. Thus, if only two points have a mutual product equal to $-a$, we can assume, without loss of generality, that
$\bfa_1\cdot\bfa_2=-a$
and compute the determinant
\[
\det\begin{pNiceMatrix}
      1&-a&a&a\\
      -a&1&a&a\\
      a&a&1&a\\
      a&a&a&1\end{pNiceMatrix}
      =(a^2-1)(5a^2-1).
\]
We can follow the same reasoning for the other cases. In the case of two pairs of points with the opposite angle, we have to distinguish the case of these pairs involving the same point and the case of disjoint pairs. That is, we can suppose that
$\bfa_1\cdot\bfa_2=\bfa_1\cdot\bfa_3=-a$,
getting
\[
\det\begin{pNiceMatrix}
      1&-a&-a&a\\
      -a&1&a&a\\
      -a&a&1&a\\
      a&a&a&1\end{pNiceMatrix}=(a^2-1)(5a^2-1),
\]
or assume that
$\bfa_1\cdot\bfa_2=\bfa_3\cdot\bfa_4=-a$,
obtaining
\[
\det\begin{pNiceMatrix}
      1&-a&a&a\\
      -a&1&a&a\\
      a&a&1&-a\\
      a&a&-a&1\end{pNiceMatrix}=-(a+1)^3(3a-1).
\]
Similarly, if we have three pairs of points with mutual product equal to $-a$, we have to analyze what is the value of the determinant if the three pairs involve the same point or not. We have only two possibilities. Assuming that
$\bfa_1\cdot\bfa_2=\bfa_1\cdot\bfa_3=\bfa_1\cdot\bfa_4=-a$,
we also must have
$\bfa_2\cdot\bfa_3=\bfa_2\cdot\bfa_4=\bfa_3\cdot\bfa_4=a$,
and
\[
\det\begin{pNiceMatrix}
      1&-a&-a&-a\\
      -a&1&a&a\\
      -a&a&1&a\\
      -a&a&a&1\end{pNiceMatrix}
      =-(a-1)^3(3a+1).
\]
If we suppose, instead,
$\bfa_1\cdot\bfa_2=\bfa_1\cdot\bfa_3=\bfa_2\cdot\bfa_3=-a$,
we must also have
$\bfa_1\cdot\bfa_4=\bfa_2\cdot\bfa_4=\bfa_3\cdot\bfa_4=a$,
which is equivalent to the previous case unless the roles of $\bfa_1$ and $\bfa_4$ are reversed. So the only case left is given by equations
$\bfa_1\cdot\bfa_2=\bfa_2\cdot\bfa_3=\bfa_3\cdot\bfa_4=-a$,
obtaining
\[
\det\begin{pNiceMatrix}
      1&-a&a&a\\
      -a&1&-a&a\\
      a&-a&1&-a\\
      a&a&-a&1\end{pNiceMatrix}
      =(a^2-1)(5a^2-1).
\]
So we proved that the Gram determinant can be $0$ only if
\[
a^2\in\pg*{1,\frac{1}{5},\frac{1}{9}},
\]
which gives the statement.
\end{proof}
We have seen in \autoref{lem_kernel_product_quadrics} how the kernel of the middle catalecticant matrix of the polynomial
\[
f_1=q_n^s-\lambda x_1^{2s}
\]
is made.
In particular, we have already observed that for every $n\in\bbN$ there exists a unique value $\lambda_n\in\bbC$ such that 
\[
\dim\Bigl(\Ker\bigl(\Cat_{f_1,s}\bigr)\Bigr)=1.
\]
For the cases of $s=2,3$, we stated in the proofs of \autoref{teo_tight_decompositions_s=2} and \autoref{teo_tight_decompositions_n=3} that this value must be the norm assumed by each point of any tight decomposition. 
We proceed as in \autoref{teo_equiangular_one_value}, to highlight which values are taken by the products obtained by a set of four points, given two possible complex values.
\begin{teo}
\label{teo_two_values_angles}
Let $\bfa_1,\bfa_2,\bfa_3,\bfa_4\in\bbC^3$ be four distinct unitary points such that there exist two values $a$, $b\in\bbC$ such that, for $1\leq j< k\leq 4$, $(\bfa_j\cdot\bfa_k)^2=a^2$ or $(\bfa_j\cdot\bfa_k)^2=b^2$. Then either $a^2=1$ or $b^2=1$ or one the pairs $(a,b),(b,a)\in\bbC^2$ must be a root of at least one of the following polynomials $g_1,\dots,g_{11}\in\bbC[x_1,x_2]$:
\begin{alignat*}{2}
&g_1=4x_1^2\pm x_1x_2\pm x_1\pm x_2-1,\quad &&g_2=4x_1^2+x_2^2-1,\\ 
&g_3=x_1^3-x_1\pm 4x_1^2x_2\pm{(3x_1^2+2x_2^2-1)},\quad
&&g_4=2x_1\pm x_2\pm 1,\\ 
&g_5=5x_1x_2^2-x_1\pm{(2x_1^2+3x_2^2-1)},\quad &&g_6=2x_1^2+x_2^2-1\pm 2x_1x_2,\\
&g_7=3x_2^2-1\pm 2x_1,\quad &&g_8=x_1^2+x_2^2-1\pm x_1x_2\pm x_1\pm x_2,\\
&g_9=x_1^2+x_2^2-1\pm 3x_1x_2\pm x_1\pm x_2,\quad
&& g_{10}=x_1^3-x_1\pm 4x_1x_2^2\pm(3x_1^2+2x_2^2-1),\\
&g_{11}=x_1^4+3x_1^2x_2^2+x_2^4-3x_1^2-3x_2^2+1\pm {(2x_1^3x_2-2x_1x_2^3)}.\quad
\end{alignat*}
\end{teo}
\begin{proof}
As in the proof of \autoref{teo_equiangular_one_value}, we only have to consider the number of possible products between the different points and this allows us not to consider every permutation of the different values in each matrix. The square of the products can assume the two values $a^2$, $b^2\in\bbC$. We first consider the case where $b$ appears only once in the products between the four points, for which we have several possibilities. We can suppose, without loss of generality, that 
$\bfa_1\cdot\bfa_2=b$.
The first polynomial, equal to the determinant of the Gram matrix in which all others product equals $a$, is given by
\[
\det\begin{pNiceMatrix}
      1&b&a&a\\
      b&1&a&a\\
      a&a&1&a\\
      a&a&a&1\end{pNiceMatrix}
      =-(a-1)(b-1)(4a^2-ab-a-b-1).
\]
If instead we assume that one of the products is equal to $-a$, then we have to distinguish the case of this product involving $\bfa_1$ or $\bfa_2$ from the case where
$\bfa_3\cdot\bfa_4=-a$.
Assuming that one among $\bfa_1\cdot\bfa_3$, $\bfa_1\cdot\bfa_4$, $\bfa_2\cdot\bfa_3$, $\bfa_2\cdot\bfa_4$ is equal to $-a$, we obtain, up to sign and permutation of the sets $\{\bfa_1,\bfa_2\}$ and $\{\bfa_3,\bfa_4\}$, the determinant
\[
\det\begin{pNiceMatrix}
      1&b&-a&a\\
      b&1&a&a\\
      -a&a&1&a\\
      a&a&a&1\end{pNiceMatrix}=(a^2-1)(4a^2+b^2-1).
\]
If instead $\bfa_3\cdot\bfa_4=-a$, then we have
\[
\det\begin{pNiceMatrix}
      1&b&a&a\\
      b&1&a&a\\
      a&a&1&-a\\
      a&a&-a&1\end{pNiceMatrix}
      =(a+1)(b-1)(4a^2+ab+a-b-1).
\]
In the case of two of the products equal to $-a$, we have other different cases. If both of the products involve the same point $\bfa_1$ or $\bfa_2$, we can assume that it is $\bfa_1$. Then we have
$\bfa_1\cdot\bfa_3=\bfa_1\cdot\bfa_4=-a$
and the polynomial is given by
\[
\det\begin{pNiceMatrix}
      1&b&-a&-a\\
      b&1&a&a\\
      -a&a&1&a\\
      -a&a&a&1\end{pNiceMatrix}
      =(a-1)(b+1)(4a^2+ab-a+b-1).
\]
Otherwise, we could have
$\bfa_1\cdot\bfa_3=\bfa_2\cdot\bfa_4=-a$,
or, equivalently,
$
\bfa_1\cdot\bfa_4=\bfa_2\cdot\bfa_3=-a$,
obtaining
\[
\det\begin{pNiceMatrix}
      1&b&-a&a\\
      b&1&a&-a\\
      -a&a&1&a\\
      a&-a&a&1\end{pNiceMatrix}
      =-(a+1)(b+1)(4a^2-ab+a+b-1),
\]
Also, we have the case where $\bfa_1\cdot\bfa_3=\bfa_2\cdot\bfa_3=-a$, or replacing $\bfa_3$ by $\bfa_4$, which gives
\[
\det\begin{pNiceMatrix}
      1&b&-a&a\\
      b&1&-a&a\\
      -a&-a&1&a\\
      a&a&a&1\end{pNiceMatrix}
      =(a+1)(b-1)(4a^2+ab+a-b-1).
\]
Finally, if one of the two products equal to $-a$ does not involve the points $\bfa_1$ and $\bfa_2$, that is,
$\bfa_3\cdot\bfa_4=-a$,
then we can assume $\bfa_1\cdot\bfa_3=-a$, possibly permuting the sets $\{\bfa_1,\bfa_2\}$ and $\{\bfa_3,\bfa_4\}$, we get
\[
\det\begin{pNiceMatrix}
      1&b&-a&a\\
      b&1&a&a\\
      -a&a&1&-a\\
      a&a&-a&1\end{pNiceMatrix}=(a^2-1)(4a^2+b^2-1).
\]
Now we analyze the various cases with two values assumed by the square of products equal to $b^2$  and four values equal to $a^2$. We start by the simplest case, assuming, up to permutations of the four points, that 
$\bfa_1\cdot\bfa_2=\bfa_1\cdot\bfa_3=b$, and the other products to be equal to $a$.
We get
\[
\det\begin{pNiceMatrix}
      1&b&b&a\\
      b&1&a&a\\
      b&a&1&a\\
      a&a&a&1\end{pNiceMatrix}
      =(a-1)(a^3-4a^2b+3a^2+2b^2-a-1).
\]
In the case where a product is equal to $-a$, If this involves the point $\bfa_1$, then we have
$\bfa_1\cdot\bfa_4=-a$,
and we get the polynomial
\[
\det\begin{pNiceMatrix}
      1&b&b&-a\\
      b&1&a&a\\
      b&a&1&a\\
      -a&a&a&1\end{pNiceMatrix}
      =(a-1)(a^3+4a^2b+3a^2+2b^2-a-1).
\]
For the cases of the product equal to $-a$ not involving $\bfa_1$, we first analyze the case where
$\bfa_2\cdot\bfa_3=-a$,
which gives
\[
\det\begin{pNiceMatrix}
      1&b&b&a\\
      b&1&-a&a\\
      b&-a&1&a\\
      a&a&a&1\end{pNiceMatrix}
      =(a+1)(a^3+4a^2b-3a^2-2b^2-a+1).
\]
Moreover, we have to consider the case where $\bfa_2\cdot\bfa_4=-a$ or, equivalently, $\bfa_3\cdot\bfa_4=-a$ , obtaining
\[
\det\begin{pNiceMatrix}
      1&b&b&a\\
      b&1&a&-a\\
      b&a&1&a\\
      a&-a&a&1\end{pNiceMatrix}
      =(a+1)(a^3  + 4ab^2  - 3a^2  - 2b^2  - a + 1).
\]
Using the same reasoning, we compute the determinant for the cases where two products are equal to $-a$. In the cases where $\bfa_1$ is involved, we have
$\bfa_1\cdot\bfa_4=\bfa_2\cdot\bfa_3=-a$, which gives
the polynomial
\[
\det\begin{pNiceMatrix}
      1&b&b&-a\\
      b&1&-a&a\\
      b&-a&1&a\\
      -a&a&a&1\end{pNiceMatrix}
      =(a+1)(a^3-4a^2b-3a^2-2b^2-a+1)
\]
or $\bfa_1\cdot\bfa_4=\bfa_2\cdot\bfa_4=-a$, equivalent to the case where $\bfa_2$ is replaced by $\bfa_3$, which gives
\[
\det\begin{pNiceMatrix}
      1&b&b&-a\\
      b&1&a&-a\\
      b&a&1&a\\
      -a&-a&a&1\end{pNiceMatrix}
      =(a+1)(a^3+4ab^2-3a^2-2b^2-a+1).
\]
Finally, if the point $\bfa_1$ is not involved in the two products equal to $-a$, we have
$\bfa_2\cdot\bfa_3=\bfa_2\cdot\bfa_4=-a$, which is equivalent to the case where $\bfa_2\cdot\bfa_3=\bfa_3\cdot\bfa_4=-a$, which gives
\[
\det\begin{pNiceMatrix}
      1&b&b&a\\
      b&1&-a&-a\\
      b&-a&1&a\\
      a&-a&a&1\end{pNiceMatrix}=(a-1)(a^3+4ab^2+3a^2+2b^2-a-1).
\]
Now, we suppose that the two initial products are opposite, that is,
$
\bfa_1\cdot\bfa_2=-\bfa_1\cdot\bfa_3=b$.
We can repeat the same computations as in the previous cases, just replacing one element $b$ by $-b$ in the matrices. Thus, in the case of four products equal to $a$, we obtain the polynomial
\[
\det\begin{pNiceMatrix}
      1&b&-b&a\\
      b&1&a&a\\
      -b&a&1&a\\
      a&a&a&1\end{pNiceMatrix}
      =(a-1)(a^3+4ab^2+3a^2+2b^2-a-1).
\]
In the cases with  one product equal to $-a$, instead, we get the three polynomials
\begin{gather*}
\det\begin{pNiceMatrix}
      1&b&-b&-a\\
      b&1&a&a\\
      -b&a&1&a\\
      -a&a&a&1\end{pNiceMatrix}
      =(a-1)(a^3+4ab^2+3a^2+2b^2-a-1),\\[1ex]
\det\begin{pNiceMatrix}
      1&b&-b&a\\
      b&1&-a&a\\
      -b&-a&1&a\\
      a&a&a&1\end{pNiceMatrix}
      =(a+1)(a^3+4ab^2-3a^2-2b^2-a+1),\\[1ex]
\det\begin{pNiceMatrix}
      1&b&-b&a\\
      b&1&a&-a\\
      -b&a&1&a\\
      a&-a&a&1\end{pNiceMatrix}
      =(a+1)(a^3-4a^2b-3a^2-2b^2-a + 1).
\end{gather*}
Finally, for the cases where two products are equal to $-a$, we get the polynomials
\begin{gather*}
\det\begin{pNiceMatrix}
      1&b&-b&-a\\
      b&1&-a&a\\
      -b&-a&1&a\\
      -a&a&a&1\end{pNiceMatrix}
      =(a+1)(a^3+4ab^2-3a^2-2b^2-a+1),\\[1ex]
      \det\begin{pNiceMatrix}
      1&b&-b&-a\\
      b&1&a&-a\\
      -b&a&1&a\\
      -a&-a&a&1\end{pNiceMatrix}=(a+1)(a^3+4a^2b-3a^2-2b^2-a+1),\\[1ex]
\det\begin{pNiceMatrix}
      1&b&-b&a\\
      b&1&-a&-a\\
      -b&-a&1&a\\
      a&-a&a&1\end{pNiceMatrix}
      =(a-1)(a^3+4a^2b+3a^2+2b^2-a-1).
\end{gather*}
Again, for the cases of two values of the square of the product equal to $b^2$, we analyze the case where they involve two disjoint pairs of points. That is, we assume, up to permutation, that
$\bfa_1\cdot\bfa_2=\bfa_3\cdot\bfa_4=b$.
If all of the other four products are equal to $a$, then we have
\[
\det\begin{pNiceMatrix}
      1&b&a&a\\
      b&1&a&a\\
      a&a&1&b\\
      a&a&b&1\end{pNiceMatrix}=-(b-1)^2(2a-b-1)(2a+b+1).
\]
If only one of them is equal to $-a$, instead, we can suppose by symmetry that it is the case where $\bfa_1\cdot\bfa_3=-a$ and we get
\[
\det\begin{pNiceMatrix}
      1&b&-a&a\\
      b&1&a&a\\
      -a&a&1&b\\
      a&a&b&1\end{pNiceMatrix}=(2a^2-2ab+b^2-1)(2a^2+2ab+b^2-1).
\]
In the case of two products equal to $-a$, we have two cases, up to permutations of the four points, which are 
$\bfa_1\cdot\bfa_3=\bfa_1\cdot\bfa_4=-a$,
corresponding to
\[
\det\begin{pNiceMatrix}
      1&b&-a&-a\\
      b&1&a&a\\
      -a&a&1&b\\
      -a&a&b&1\end{pNiceMatrix}=(b-1)(b+1)(4a^2+b^2-1),
\]
and
$\bfa_1\cdot\bfa_4=\bfa_2\cdot\bfa_3=-a$, which gives
\[
\det\begin{pNiceMatrix}
      1&b&a&-a\\
      b&1&-a&a\\
      a&-a&1&b\\
      -a&a&b&1\end{pNiceMatrix}=-(b+1)^2(2a-b+1)(2a+b-1).
\]
For the cases where $\bfa_1\cdot\bfa_2=-\bfa_3\cdot\bfa_4=b$, we have again to replace an element $b$ in the matrices by $-b$. We obtain the polynomials
\begin{gather*}
\det\begin{pNiceMatrix}
      1&b&a&a\\
      b&1&a&a\\
      a&a&1&-b\\
      a&a&-b&1\end{pNiceMatrix}=(b-1)(b+1)(4a^2+b^2-1),\\[1ex]
\det\begin{pNiceMatrix}
      1&b&-a&a\\
      b&1&a&a\\
      -a&a&1&-b\\
      a&a&-b&1\end{pNiceMatrix}=(2a^2-2ab+b^2-1)(2a^2+2ab+b^2-1),\\[1ex]
\det\begin{pNiceMatrix}
      1&b&-a&-a\\
      b&1&a&a\\
      -a&a&1&-b\\
      -a&a&-b&1\end{pNiceMatrix}=-(b+1)^2(2a-b+1)(2a+b-1),\\[1ex]
\det\begin{pNiceMatrix}
      1&b&a&-a\\
      b&1&-a&a\\
      a&-a&1&-b\\
      -a&a&-b&1\end{pNiceMatrix}=(b^2-1)(4a^2+b^2-1).
\end{gather*}
It remains to analyze the cases of three values of the square of the  product equal to $b^2$. We start by supposing that
$\bfa_1\cdot\bfa_2=\bfa_1\cdot\bfa_3=\bfa_1\cdot\bfa_4=b$. That is,
\[
\det\begin{pNiceMatrix}
      1&b&b&b\\
      b&1&a&a\\
      b&a&1&a\\
      b&a&a&1\end{pNiceMatrix}=-(a-1)^2(3b^2-2a-1).
\]
If instead one of these last three products is opposite to the others, that is, for example,
$\bfa_1\cdot\bfa_2=-b$,
we get
\[
\det\begin{pNiceMatrix}
      1&-b&b&b\\
      -b&1&a&a\\
      b&a&1&a\\
      b&a&a&1\end{pNiceMatrix}=(a-1)(5ab^2+2a^2+3b^2-a-1).
\]
Doing the same with $\bfa_2\cdot\bfa_3=-a$,
we get
\[
\det\begin{pNiceMatrix}
      1&b&b&b\\
      b&1&-a&a\\
      b&-a&1&a\\
      b&a&a&1\end{pNiceMatrix}=(a+1)(5ab^2-2a^2-3b^2-a+1)
\]
Also if we consider both conditions together, that is,
$\bfa_1\cdot\bfa_2=-b$ and $\bfa_2\cdot\bfa_3=-a$,
we have
\begin{align*}
\det\begin{pNiceMatrix}
      1&-b&b&b\\
      -b&1&-a&a\\
      b&-a&1&a\\
      b&a&a&1\end{pNiceMatrix}=(a+1)(5ab^2-2a^2-3b^2-a+1).
\end{align*}
Finally, we have the case of three products equal to $\pm b$, but not all of them involving the same point, i.e., first assuming
$\bfa_1\cdot\bfa_2=\bfa_2\cdot\bfa_3=\bfa_3\cdot\bfa_4=b$,
we have
\[
\det\begin{pNiceMatrix}
      1&b&a&a\\
      b&1&b&a\\
      a&b&1&b\\
      a&a&b&1\end{pNiceMatrix}=(a^2-3ab+b^2+a+b-1)(a^2+ab+b^2-a-b-1).
\]
If instead one of the three products is equal to $-b$, that is, for instance,
$\bfa_1\cdot\bfa_2=-b$,
we get
\[
\det\begin{pNiceMatrix}
      1&-b&a&a\\
      -b&1&b&a\\
      a&b&1&b\\
      a&a&b&1\end{pNiceMatrix}=(a^4-2a^3b+3a^2b^2+2ab^3+b^4-3a^2-3b^2+1).
\]
The last two cases are obtained as above, first setting
$\bfa_1\cdot\bfa_3=-a$,
and then
$\bfa_1\cdot\bfa_2=-b$ and $\bfa_1\cdot\bfa_3=-a$,
giving respectively
\[
\det\begin{pNiceMatrix}
      1&b&-a&a\\
      b&1&b&a\\
      -a&b&1&b\\
      a&a&b&1\end{pNiceMatrix}=(a^4+2a^3b+3a^2b^2-2ab^3+b^4-3a^2-3b^2+1)
\]
and
\[
\det\begin{pNiceMatrix}
      1&-b&-a&a\\
      -b&1&b&a\\
      -a&b&1&b\\
      a&a&b&1\end{pNiceMatrix}=(a^2-ab+b^2-a+b-1)(a^2+3ab+b^2+a-b-1).
\]
All of the other determinants are with three products equal to $\pm b$ are obtained by reversing $a$ and $b$, or replacing $a$ with $-a$ and $b$ with $-b$. Therefore, we get the statement.
\end{proof}
Thanks to \autoref{teo_tight_decompositions_n=3} and decomposition \eqref{decom_n=3_s=3_std}, we know that the rank of $q_3^3$, which is $11$, cannot be equal to its border rank, which by \cite{Fla23a}*{Theorem 4.5} is equal to $10$.  However, this result can still be obtained with \autoref{teo_two_values_angles}, just by doing some computations.
\begin{cor}
$\rk(q_3^3)>\brk(q_3^3)=10$.
\end{cor}
\begin{proof}
By \autoref{teo_two_values_angles} and \autoref{lem_kernel_tight_decomposition_s=3}, it is sufficient to replace the values by
\begin{equation}
\label{rel_values_angles_s=3_first_case}
x_1=0,\qquad x_2=\pm\sqrt{\frac{3}{7}}
\end{equation}
and 
\begin{equation}
\label{rel_values_angles_s=3_second_case}
x_1=\pm\sqrt{\frac{3}{7}},\qquad x_2=0
\end{equation}
in polynomials $g_1,\dots,g_{11}\in\bbC[x_1,x_2]$, defined in \autoref{teo_two_values_angles}. We first substitute values given in formula \eqref{rel_values_angles_s=3_first_case}, obtaining
\begin{alignat*}{2}
&g_1\biggl(0,\pm\sqrt{\frac{3}{7}}\biggr)=-1\pm\sqrt{\frac{3}{7}},\qquad &&g_2\biggl(0,\pm\sqrt{\frac{3}{7}}\biggr)=-\frac{4}{7},\\[1ex]
&g_3\biggl(0,\pm\sqrt{\frac{3}{7}}\biggr)=\pm\frac{1}{7},\qquad
&&g_4\biggl(0,\pm\sqrt{\frac{3}{7}}\biggr)=\pm \frac{3}{7}\pm 1,\\[1ex]
&g_5\biggl(0,\pm\sqrt{\frac{3}{7}}\biggr)=\pm\frac{2}{7},\qquad &&g_6\biggl(0,\pm\sqrt{\frac{3}{7}}\biggr)=-\frac{4}{7},\\[1ex]
&g_7\biggl(0,\pm\sqrt{\frac{3}{7}}\biggr)=\frac{2}{7},\qquad &&g_8\biggl(0,\pm\sqrt{\frac{3}{7}}\biggr)=-\frac{4}{7}\pm \sqrt{\frac{3}{7}}\\[1ex]
&g_9\biggl(0,\pm\sqrt{\frac{3}{7}}\biggr)=-\frac{4}{7}\pm \sqrt{\frac{3}{7}},\qquad
&&g_{10}\biggl(0,\pm\sqrt{\frac{3}{7}}\biggr)=\pm \frac{1}{7},\\[1ex]
&g_{11}\biggl(0,\pm\sqrt{\frac{3}{7}}\biggr)=-\frac{5}{49}.
\end{alignat*}
Then, we proceed in the same way with values of formula \eqref{rel_values_angles_s=3_second_case}, obtaining
\begin{alignat*}{2}
&g_1\biggl(\pm\sqrt{\frac{3}{7}},0\biggr)=\frac{5}{7}\pm \sqrt{\frac{3}{7}},\qquad &&g_2\biggl(\pm\sqrt{\frac{3}{7}},0\biggr)=\frac{5}{7},\\[1ex]
&g_3\biggl(\pm\sqrt{\frac{3}{7}},0\biggr)=\pm \frac{2}{7}\pm\frac{4}{7}\sqrt{\frac{3}{7}},\qquad
&&g_4\biggl(\pm\sqrt{\frac{3}{7}},0\biggr)=\pm 1\pm 2\sqrt{\frac{3}{7}},\\[1ex] 
&g_5\biggl(\pm\sqrt{\frac{3}{7}},0\biggr)=\pm\frac{1}{7}\pm \sqrt{\frac{3}{7}},\qquad &&g_6\biggl(\pm\sqrt{\frac{3}{7}},0\biggr)=-\frac{1}{7},\\[1ex]
&g_7\biggl(\pm\sqrt{\frac{3}{7}},0\biggr)=- 1\pm 2\sqrt{\frac{3}{7}},\qquad &&g_8\biggl(\pm\sqrt{\frac{3}{7}},0\biggr)=-\frac{4}{7}\pm \sqrt{\frac{3}{7}}\\[1ex]
&g_9\biggl(\pm\sqrt{\frac{3}{7}},0\biggr)=-\frac{4}{7}\pm \sqrt{\frac{3}{7}},\qquad
&&g_{10}\biggl(\pm\sqrt{\frac{3}{7}},0\biggr)=\pm \frac{2}{7}\pm\frac{4}{7}\sqrt{\frac{3}{7}},\\[1ex]
&g_{11}\biggl(\pm\sqrt{\frac{3}{7}},0\biggr)=-\frac{5}{49}.
\end{alignat*}
Since all the values are nonzero, it follows that there cannot be a tight decomposition of $q_3^3$.
\end{proof}
To proceed in the same way for the form $q_3^4$, we need to know which are the possible values of the angles between points of a possible tight decomposition. So we have to find the kernel of the catalecticant map of the polynomial 
\[
f_1=\frac{1}{B_{n,4}}q_n^{4}-(\bfa\cdot\bfa)^8,
\]
where $\bfa\in\bbC^n$ is such that $\bfa\cdot\bfa=1$, recalling (see formula \eqref{rel_value_norm_tight} that
\[
B_{n,4}=\frac{8(n+4)(n+6)}{35(n+1)(n+3)}.
\]
Analogously to \autoref{lem_kernel_tight_decomposition} and \autoref{lem_kernel_tight_decomposition_s=3}, we get the following lemma for the exponent $s=4$. 
\begin{lem}
\label{lem_values_kernel_s=4}
Let $n\in\bbN$ and let 
\[
f_1=\frac{1}{B_{n,4}}q_n^{4}-(\bfa\cdot\bfx)^8,
\]
for some $\bfa\in\bbC^n$ such that $\bfa\cdot\bfa=1$. Then
\begin{align*}
&\Ker\bigl(\Cat_{f_1,4}\bigr)=\biggl\langle(n+6)(\bfa\cdot\bfy)^4-6q_n(\bfa\cdot\bfy)^2+\frac{3}{n+4}q_n^2\biggr\rangle\\
&\qquad=\biggl\langle \Biggl((n+6)(\bfa\cdot\bfy)^2-\biggl(3+\sqrt{\frac{6(n+3)}{n+4}}\biggr)q_n\Biggr)\Biggl((n+6)(\bfa\cdot\bfy)^2-\biggl(3-\sqrt{\frac{6(n+3)}{n+4}}\biggr)q_n\Biggr)\biggr\rangle.
\end{align*}
\end{lem}
\begin{proof}
Let
\[
g_1=(n+6)(\bfa\cdot\bfy)^4-6q_n(\bfa\cdot\bfy)^2+\frac{3}{n+4}q_n^2.
\]
Using \autoref{lem contraz forme}, formula \eqref{rel_powers_laplacian_contraction} and the fact that $\bfa\cdot\bfa=1$, we simply observe that
\begin{align*}
g_1\circ f_1&=\biggl((n+6)(\bfa\cdot\bfy)^4-6q_n(\bfa\cdot\bfy)^2+\frac{3}{n+4}q_n^2\biggr)\circ \biggl(\frac{1}{B_{n,4}}q_n^{4}-(\bfa\cdot\bfx)^8\biggr)\\
&=\frac{n+6}{B_{n,4}}\bigl((\bfa\cdot\bfy)^4\circ q_n^4\bigr)-(n+6)\bigl((\bfa\cdot\bfy)^4\circ (\bfa\cdot\bfx)^8\bigr)-\frac{6}{B_{n,4}}\bigl(q_n(\bfa\cdot\bfy)^2\circ q_n^4\bigr)\\
&\hphantom{{}={}}+6\bigl(q_n(\bfa\cdot\bfy)^2\circ(\bfa\cdot\bfx)^8\bigr)+\frac{3}{B_{n,4}(n+4)}\bigl(q_n^2\circ q_n^4\bigr)-\frac{3}{n+4}\bigl(q_n^2\circ(\bfa\cdot\bfx)^8\bigr)\\
&=\frac{48(n+6)}{B_{n,4}}\bigl(8(\bfa\cdot\bfx)^4+24q_n(\bfa\cdot\bfx)^2+3q_n^2\bigr)-1680(n+6)(\bfa\cdot\bfx)^4\\
&\hphantom{{}={}}-\frac{288(n+6)}{B_{n,4}}\bigl(4q_n(\bfa\cdot\bfx)^2+q_n^2\bigr)+10080(\bfa\cdot\bfx)^4+\frac{144(n+6)}{B_{n,4}}q_n^2-\frac{5040}{n+4}(\bfa\cdot\bfx)^4\\
&=\frac{384(n+6)}{B_{n,4}}(\bfa\cdot\bfx)^4-1680(n+6)(\bfa\cdot\bfx)^4+10080(\bfa\cdot\bfx)^4-\frac{5040}{n+4}(\bfa\cdot\bfx)^4\\
&=\frac{1680(n+1)(n+3)-1680n(n+4)-5040}{n+4}(\bfa\cdot\bfx)^4=0,
\end{align*}
which proves the statement.
\end{proof}
So, we can now prove that the rank of $q_3^4$ is equal to $16$. Potentially, this is a technique that can be extended even to higher cases. It would be interesting to establish a general rule regarding the angles between the different points of the decomposition in all the successive cases. To do this, it would be necessary to determine a general formula to obtain the values of all of the angles between the points of the decomposition.
\begin{cor}
$\rk\bigl(q_3^4\bigr)=16$.
\end{cor}
\begin{proof}
As in the proof of the previous corollary, by substituting $n=3$ in the formula obtained in \autoref{lem_values_kernel_s=4}, it is sufficient to substitute the values
\begin{equation}
\label{rel_values_angles_s=4_first_case}
x_1=\pm\sqrt{\frac{\sqrt{7}+2}{3\sqrt{7}}}=\pm\sqrt{\frac{7+2\sqrt{7}}{21}},\quad x_2=\pm\sqrt{\frac{\sqrt{7}-2}{3\sqrt{7}}}=\pm\sqrt{\frac{7-2\sqrt{7}}{21}}
\end{equation}
and 
\begin{equation}
\label{rel_values_angles_s=4_second_case}
x_1=\pm\sqrt{\frac{\sqrt{7}-2}{3\sqrt{7}}}=\pm\sqrt{\frac{7-2\sqrt{7}}{21}},\quad x_2=\pm\sqrt{\frac{\sqrt{7}+2}{3\sqrt{7}}}=\pm\sqrt{\frac{7+2\sqrt{7}}{21}}
\end{equation}
to polynomial $g_1,\dots,g_{10}\in\bbC[x_1,x_2]$, defined in \autoref{teo_two_values_angles}. For computations, it can be useful to observe that, in any case,
\[
x_1x_2=\pm\frac{1}{\sqrt{21}},\quad x_1^2+x_2^2=\frac{2}{3}.
\]
Replacing the values in formulas \eqref{rel_values_angles_s=4_first_case} and \eqref{rel_values_angles_s=4_second_case}, we get
\begin{align*}
&g_1\Biggl(\pm\sqrt{\frac{7+2\sqrt{7}}{21}},\pm\sqrt{\frac{7-2\sqrt{7}}{21}}\Biggr)=\frac{7+8\sqrt{7}\pm\sqrt{21}\pm\sqrt{147+42\sqrt{7}}\pm\sqrt{147-42\sqrt{7}}}{21},\\[1ex]
&g_2\Biggl(\pm\sqrt{\frac{7+2\sqrt{7}}{21}},\pm\sqrt{\frac{7-2\sqrt{7}}{21}}\Biggr)=\frac{2}{3}+\frac{2\sqrt{7}}{7},\\[1ex]
&g_3\Biggl(\pm\sqrt{\frac{7+2\sqrt{7}}{21}},\pm\sqrt{\frac{7-2\sqrt{7}}{21}}\Biggr)=\frac{2\sqrt{7+2\sqrt{7}}\bigl(\sqrt{147}-7\sqrt{21}\pm 42\bigr)\pm 42\bigl(7+\sqrt{7}\bigr)}{441},\\[1ex]
&g_4\Biggl(\pm\sqrt{\frac{7+2\sqrt{7}}{21}},\pm\sqrt{\frac{7-2\sqrt{7}}{21}}\Biggr)=\pm 2\sqrt{\frac{7+2\sqrt{7}}{21}}\pm\sqrt{\frac{7-2\sqrt{7}}{21}}\pm 1,\\[1ex]
&g_5\Biggl(\pm\sqrt{\frac{7+2\sqrt{7}}{21}},\pm\sqrt{\frac{7-2\sqrt{7}}{21}}\Biggr)=\pm\sqrt{\frac{7+2\sqrt{7}}{21}}\biggl(\frac{14-10\sqrt{7}}{21}\biggr)\pm{\frac{14-2\sqrt{7}}{21}},\\[1ex] 
&g_6\Biggl(\pm\sqrt{\frac{7+2\sqrt{7}}{21}},\pm\sqrt{\frac{7-2\sqrt{7}}{21}}\Biggr)=\frac{2\bigl(\sqrt{7}\pm\sqrt{21}\bigr)}{21},\\[1ex]
&g_7\Biggl(\pm\sqrt{\frac{7+2\sqrt{7}}{21}},\pm\sqrt{\frac{7-2\sqrt{7}}{21}}\Biggr)=-\frac{2\sqrt{7}}{7}\pm 2\sqrt{\frac{7+2\sqrt{7}}{21}},\\[1ex]
&g_8\Biggl(\pm\sqrt{\frac{7+2\sqrt{7}}{21}},\pm\sqrt{\frac{7-2\sqrt{7}}{21}}\Biggr)=\frac{-7\pm\sqrt{21}\pm\sqrt{147+42\sqrt{7}}\pm\sqrt{147-42\sqrt{7}}}{21},\\[1ex]
&g_9\Biggl(\pm\sqrt{\frac{7+2\sqrt{7}}{21}},\pm\sqrt{\frac{7-2\sqrt{7}}{21}}\Biggr)=\frac{-7\pm 3\sqrt{21}\pm\sqrt{147+42\sqrt{7}}\pm\sqrt{147-42\sqrt{7}}}{21},\\[1ex]
&g_{10}\Biggl(\pm\sqrt{\frac{7+2\sqrt{7}}{21}},\pm\sqrt{\frac{7-2\sqrt{7}}{21}}\Biggr)=\frac{2}{63}\biggl(\pm 3\bigl(7+\sqrt{7}\bigr)\pm\Bigl(\bigl(4-2\sqrt{7}\bigr)\pm\bigl(1-\sqrt{7}\bigr)\Bigr)\sqrt{3\bigl(7+2\sqrt{7}\bigr)}\biggr)\\[1ex]
&g_{11}\Biggl(\pm\sqrt{\frac{7+2\sqrt{7}}{21}},\pm\sqrt{\frac{7-2\sqrt{7}}{21}}\Biggr)=-\frac{8\bigl(4\pm \sqrt{3}\bigr)}{63},\\[1ex]
&g_1\Biggl(\pm\sqrt{\frac{7-2\sqrt{7}}{21}},\pm\sqrt{\frac{7+2\sqrt{7}}{21}}\Biggr)=\frac{7-8\sqrt{7}\pm\sqrt{21}\pm\sqrt{147-42\sqrt{7}}\pm\sqrt{147+42\sqrt{7}}}{21},\\[1ex]
&g_2\Biggl(\pm\sqrt{\frac{7-2\sqrt{7}}{21}},\pm\sqrt{\frac{7+2\sqrt{7}}{21}}\Biggr)=\frac{2}{3}-\frac{2\sqrt{7}}{7},\\[1ex]
&g_3\Biggl(\pm\sqrt{\frac{7-2\sqrt{7}}{21}},\pm\sqrt{\frac{7+2\sqrt{7}}{21}}\Biggr)=\frac{2\sqrt{7-2\sqrt{7}}\bigl(-\sqrt{147}-7\sqrt{21}\pm 42\bigr)\pm 42\bigl(7-\sqrt{7}\bigr)}{441},\\[1ex]
&g_4\Biggl(\pm\sqrt{\frac{7-2\sqrt{7}}{21}},\pm\sqrt{\frac{7+2\sqrt{7}}{21}}\Biggr)=\pm 2\sqrt{\frac{7-2\sqrt{7}}{21}}\pm\sqrt{\frac{7+2\sqrt{7}}{21}}\pm 1,\\[1ex]
&g_5\Biggl(\pm\sqrt{\frac{7-2\sqrt{7}}{21}},\pm\sqrt{\frac{7+2\sqrt{7}}{21}}\Biggr)=\pm\sqrt{\frac{7-2\sqrt{7}}{21}}\biggl(\frac{14+10\sqrt{7}}{21}\biggr)\pm{\frac{14+2\sqrt{7}}{21}},\\[1ex] 
&g_6\Biggl(\pm\sqrt{\frac{7-2\sqrt{7}}{21}},\pm\sqrt{\frac{7+2\sqrt{7}}{21}}\Biggr)=\frac{2\bigl(-\sqrt{7}\pm\sqrt{21}\bigr)}{21},\\[1ex]
&g_7\Biggl(\pm\sqrt{\frac{7-2\sqrt{7}}{21}},\pm\sqrt{\frac{7+2\sqrt{7}}{21}}\Biggr)=\frac{2\sqrt{7}}{7}\pm 2\sqrt{\frac{7-2\sqrt{7}}{21}},\\[1ex]
&g_8\Biggl(\pm\sqrt{\frac{7-2\sqrt{7}}{21}},\pm\sqrt{\frac{7+2\sqrt{7}}{21}}\Biggr)=\frac{-7\pm\sqrt{21}\pm\sqrt{147-42\sqrt{7}}\pm\sqrt{147+42\sqrt{7}}}{21},\\[1ex]
&g_9\Biggl(\pm\sqrt{\frac{7-2\sqrt{7}}{21}},\pm\sqrt{\frac{7+2\sqrt{7}}{21}}\Biggr)=\frac{-7\pm 3\sqrt{21}\pm\sqrt{147-42\sqrt{7}}\pm\sqrt{147+42\sqrt{7}}}{21},\\[1ex]
&g_{10}\Biggl(\pm\sqrt{\frac{7-2\sqrt{7}}{21}},\pm\sqrt{\frac{7+2\sqrt{7}}{21}}\Biggr)=\frac{2}{63}\biggl(\pm 3(7-\sqrt{7})\pm \Bigl(\bigl(4+2\sqrt{7}\bigr)\pm \bigl(1+\sqrt{7}\bigr)\Bigr)\sqrt{3\bigl(7-2\sqrt{7}\bigr)}\biggr)\\[1ex]
&g_{11}\Biggl(\pm\sqrt{\frac{7-2\sqrt{7}}{21}},\pm\sqrt{\frac{7+2\sqrt{7}}{21}}\Biggr)=-\frac{8\bigl(4\pm \sqrt{3}\bigr)}{63}.
\end{align*}
Since all the values are non-zero, it follows that it cannot exists a tight decomposition of $q_3^4$.
\end{proof}
Despite these low exponent cases, determining the explicit rank of the form $q_n^s$ for higher values of $s$ remains a difficult problem. However, it is possible to analyze the asymptotic behavior of this value.
In \cite{Fla23b} it is proved that
\begin{equation}
\label{formula:asympt_n}
\lim_{n\to+\infty}\log_n\bigl(\rk(q_n^s)\bigr)=s.
\end{equation}
This means that the rank of $q_n^s$ grows as $n^s$ for $n\to+\infty$. Furthermore, it is proved in \cite{Fla23b} that the rank is subgeneric for any $n>(2s-1)^2$, and we conjecture that, fixed $n\in\bbN$, the rank is supergeneric for a sufficiently large value of $s$. In addition, assuming that the factors of the form appearing in \autoref{lem_kernel_product_quadrics} are all pairwise distinct, we believe that, by summing other $n-2$ suitable points to $q_n^s$, it is possible to obtain $n-1$ forms of degree $s$ such that the ideal generated by these forms is a complete intersection ideal corresponding to a set of $s^{n-1}$ points. In particular, we believe that,
for every $n,s\in\bbN$,
\[
\rk(q_n^s)\leq s^{n-1}+2(n-1),
\]
and hence,
\[
\lim_{s\to+\infty}\log_s\bigl(\rk(q_n^s)\bigr)=n-1.
\]

\newpage
\section*{Acknowledgements}
This paper was written while the author was a research fellow at \textit{Università degli Studi di Firenze} and it is based on the author's Ph.D.~thesis, completed at \textit{Alma Mater Studiorum -- Università di Bologna} under the patient and valuable supervision of Alessandro Gimigliano and Giorgio Ottaviani, to whom the author would like to express his gratitude for their support. Sincere thanks are due to Enrique Arrondo and Jaros{\l}aw Buczy\'{n}ski for their availability and help in the central and the final parts of this work, and also for numerous comparisons and suggestions in the drafting of the final version of the Ph.D.~thesis. The author is also grateful to Fulvio Gesmundo for countless scientific discussions related to this argument and other problems. Finally, the author would like to thank the anonymous referee for the detailed analysis of the paper and the numerous suggestions. 
The author is a member of the research group \textit{Gruppo Nazionale per le Strutture Algebriche Geometriche e Affini} (GNSAGA) of \textit{Istituto Nazionale di Alta Matematica} (INdAM). The author has been supported by the scientific project \textit{Multilinear Algebraic Geometry} of the program \textit{Progetti di ricerca di Rilevante Interesse Nazionale} (PRIN), Grant Assignment Decree No.~973, adopted on 06/30/2023 by the Italian Ministry of University and Research (MUR) and by the project \textit{Thematic Research Programmes}, Action I.1.5 of the program \textit{Excellence Initiative -- Research University} (IDUB) of the Polish Ministry of Science and Higher Education.

\bibliographystyle{amsalpha}
\addcontentsline{toc}{section}{References}
\bibliography{DecompositionsOfPowersOfQuadricsVersion6.bib}

\end{document}